\documentclass{amsart}

\usepackage[T1]{fontenc}
\usepackage{enumerate, amsmath, amsfonts, amssymb, amsthm, mathrsfs, wasysym, graphics, graphicx, url, hyperref, hypcap, shuffle, xargs, multicol, overpic, pdflscape, multirow, hvfloat, minibox, accents, array, multido, xifthen, a4wide, ae, aecompl, blkarray, pifont, mathtools, etoolbox, dsfont, stmaryrd}
\usepackage[dvipsnames]{xcolor}
\usepackage{marginnote}
\hypersetup{colorlinks=true, citecolor=darkblue, linkcolor=darkblue}
\usepackage[all]{xy}
\usepackage[bottom]{footmisc}
\usepackage{tikz, tikz-cd, tikz-qtree}
\usetikzlibrary{trees, decorations, decorations.markings, shapes, arrows, matrix, calc, fit, intersections, patterns, angles, snakes}
\usepackage[external]{forest}
\graphicspath{{figures/}{figures/nodes/}}
\makeatletter\def\input@path{{figures/}}\makeatother
\usepackage{caption}
\usepackage[noabbrev,capitalise]{cleveref}
\usepackage[export]{adjustbox}
\usepackage{ulem}
\usepackage{xargs}

\definecolor{darkblue}{rgb}{0,0,0.7} % darkblue color
\definecolor{green}{RGB}{57,181,74} % green color
\definecolor{violet}{RGB}{147,39,143} % violet color

\newcommand{\red}{\color{red}} % red command
\newcommand{\blue}{\color{blue}} % blue command
\newcommand{\darkblue}{\color{darkblue}} % darkblue command

\usepackage{todonotes}
\newtheorem{theorem}{Theorem}[section]
\newtheorem*{theorem*}{Theorem}
\newtheorem{corollary}[theorem]{Corollary}
\newtheorem{proposition}[theorem]{Proposition}
\newtheorem{lemma}[theorem]{Lemma}

\newtheorem{construction}[theorem]{Construction}

\theoremstyle{definition}
\newtheorem{definition}[theorem]{Definition}
\newtheorem{example}[theorem]{Example}
\newtheorem{remark}[theorem]{Remark}

\newtheorem{notation}[theorem]{Notation}

\crefname{notation}{Notation}{Notations}
\crefname{problem}{Problem}{Problems}

\newcommand{\R}{\mathbb{R}} % reals
\newcommand{\N}{\mathbb{N}} % naturals
\renewcommand{\b}[1]{{\boldsymbol{#1}}} % bold letters
\newcommand{\Un}{\mathrm{U}} % unordered IJ
\newcommand{\Or}{\mathrm{O}} % ordered IJ
\newcommand{\PT}{\mathrm{PT}} %planar trees
\newcommand{\K}{\mathrm{K}} %associahedra
\newcommand{\J}{\mathrm{J}} %multiplihedra

\newcommand{\set}[2]{\left\{ #1 \;\middle|\; #2 \right\}} % set notation
\newcommand{\bigset}[2]{\big\{ #1 \;\big|\; #2 \big\}} % big set notation
\newcommand{\ssm}{\smallsetminus} % small set minus
\newcommand{\dotprod}[2]{\left\langle \, #1 \; \middle| \; #2 \, \right\rangle} % dot product
\newcommand{\one}{\b{1}} % the all one vector
\newcommand{\eqdef}{\mbox{\,\raisebox{0.2ex}{\scriptsize\ensuremath{\mathrm:}}\ensuremath{=}\,}} % :=
\newcommand{\card}[1]{\##1} % cardinality

\DeclareMathOperator{\Inv}{Inv} % inversions
\DeclareMathOperator{\Ima}{Im} %Image d'une fonction
\DeclareMathOperator{\std}{std}

\newcommand{\bluea}[1]{\textcolor{MidnightBlue}{\boldsymbol{\left(\right.}} #1 \textcolor{MidnightBlue}{\boldsymbol{\left.\right)}}} 
\newcommand{\reda}[1]{\textcolor{Red!60}{\boldsymbol{\left(\right.}} #1 \textcolor{Red!60}{\boldsymbol{\left.\right)}}} 
\newcommand{\purplea}[1]{\textcolor{Purple!80}{\boldsymbol{\left[\right.}} #1 \textcolor{Purple!80}{\boldsymbol{\left.\right]}}} 

\newcommand{\ie}{\textit{i.e.}~} % id est
\newcommand{\resp}{resp.~} % id est
\newcommand{\eg}{\textit{e.g.}~} % exempli gratia
\newcommand{\ordinal}{\textsuperscript{th}} % th for ordinals
\newcommand{\ordinalst}{\textsuperscript{st}} % st for ordinals
\newcommand{\defn}[1]{\textsl{\darkblue #1}} % emphasis of a definition
\newcommand{\imagebot}[1]{\vbox{\hbox{#1}\null}} % image aligned bot
\newcommand{\OEIS}[1]{\cite[{\rm \href{http://oeis.org/#1}{\texttt{#1}}}]{OEIS}}

\renewcommand{\b}[1]{\boldsymbol{#1}} % bold
\newcommandx{\arrangement}[1][1 = A]{\mathcal{#1}} % arrangement
\newcommand{\HH}{\mathbb{H}} % hyperplane
\newcommandx{\braidArrangement}[1][1 = n]{\arrangement[B]_{#1}} % braid arrangement
\newcommandx{\multiBraidArrangement}[2][1 = n, 2 = \ell]{\arrangement[B]_{#1}^{#2}} % (l,n)-braid arrangement
\newcommandx{\rank}[1][1 = \arrangement]{\operatorname{rk}(#1)} % rank
\newcommandx{\facePoset}[1][1 = \arrangement]{\mathsf{Fa}(#1)} % face poset
\newcommandx{\fPol}[2][1 = \arrangement, 2 = x]{\b{f}_{#1}(#2)} % face polynomial
\newcommandx{\bPol}[2][1 = \arrangement, 2 = x]{\b{b}_{#1}(#2)} % bounded face polynomial
\newcommandx{\flatPoset}[1][1 = \arrangement]{\mathsf{Fl}(#1)} % intersection poset
\newcommandx{\charPol}[2][1 = \arrangement, 2 = y]{\b{\chi}_{#1}(#2)} % characteristic polynomial
\newcommandx{\mobPol}[3][1 = \arrangement, 2 = x, 3 = y]{\b{\mu}_{#1}(#2, #3)} % Mobius polynomial
\newcommandx{\weirdPol}[2][1 = \arrangement, 2 = x]{\b{\pi}_{#1}(#2)} % weird polynomial
\newcommandx{\partitionPoset}[1][1 = n]{\b{\Pi}_{#1}} % partition poset
\newcommandx{\forestPoset}[2][1 = n, 2 = \ell]{\b{\Phi}_{#1}^{#2}} % partition forest poset
\newcommandx{\rainbowForests}[2][1 = n, 2 = \ell]{\b{\Psi}_{#1}^{#2}} % rainbow forests
\newcommandx{\rainbowTrees}[2][1 = n, 2 = \ell]{\b{\rm{T}}_{#1}^{#2}} % rainbow trees
\newcommandx{\Perm}[1][1 = n]{\mathsf{Perm}(#1)} % permutahedron
\newcommandx{\Asso}[1][1 = n]{\mathsf{Asso}(#1)} % associahedron
\newcommandx{\Cube}[1][1 = n]{\mathsf{Cube}(#1)} % cube
\newcommandx{\Simplex}[1][1 = n]{\mathsf{Simplex}(#1)} % simplex
\makeatletter
\def\rightharpoonupfill@{\arrowfill@\relbar\relbar\rightharpoonup}
\newcommand{\overrightharpoonup}{%
\mathpalette{\overarrow@\rightharpoonupfill@}}
\makeatother
\newcommandx{\order}[1]{\smash{\overrightharpoonup{#1}}} % ordered
\newcommandx{\orderedPartitionPoset}[1][1 = n]{\order{\b{\Pi}}_{#1}} % oriented partition poset
\newcommandx{\orderedForestPoset}[2][1 = n, 2 = \ell]{\order{\b{\Phi}}_{#1}^{#2}} % oriented partition forest poset

\newcommand{\SU}{\mathrm{SU}}
\newcommand{\LA}{\mathrm{LA}}
\newcommand{\SUD}{\triangle^{\mathrm{SU}}}
\newcommand{\LAD}{\triangle^{\mathrm{LA}}}
\newcommand{\SCP}{\mathrm{SCP}}
\newcommand{\PolySub}{\mathsf{Poly}}
\newcommand{\Ainf}{\ensuremath{\mathrm{A}_\infty}}
\newcommand{\op}{\mathrm{op}}
\newcommand{\tr}{\mathrm{tr}}
\newcommand{\id}{\mathrm{id}}

\newcommand{\divcube}[1]{\Box_{#1}}
\newcommandx{\hour}[1][1 = v]{\raisebox{0.065cm}{\rotatebox[origin=c]{90}{$\bowtie$}}_{#1}}

\usepackage{tkz-euclide}
\usepackage{xstring}

\def\splicelist#1{
\StrCount{#1}{,}[\numofelem]
\ifnum\numofelem>0\relax
     \StrBehind[\numofelem]{#1}{,}[\mylast]%
\else
    \let\mylast#1%
\fi
}

\newcommand{\hedge}[3][very thick,color=black]{
\splicelist{#2}
\foreach \vertex [remember=\vertex as \succvertex
    (initially \mylast)] in {#2}{
    \coordinate (\succvertex-next) at ($(\succvertex)!#3!90:(\vertex)$);
    \coordinate (\vertex-previous) at ($(\vertex)!#3!-90:(\succvertex)$);
    \draw[#1] (\succvertex-next) --  (\vertex-previous);
}
\foreach \vertex in {#2}{
    \tkzDrawArc[#1](\vertex,\vertex-next)(\vertex-previous)
}
}

\makeatletter
\def\part{\@startsection{part}{1}%
\z@{.7\linespacing\@plus\linespacing}{.8\linespacing}%
{\LARGE\sffamily\centering}}
\makeatother

\makeatletter
\def\l@part{\@tocline{1}{8pt}{0pc}{}{}}
\def\l@section{\@tocline{1}{4pt}{0pc}{}{}}
\makeatother
\let\oldtocpart=\tocpart
\renewcommand{\tocpart}[2]{\sc\large\oldtocpart{#1}{#2}}
\let\oldtocsection=\tocsection
\renewcommand{\tocsection}[2]{\bf\oldtocsection{#1}{#2}}
\let\oldtocsubsubsection=\tocsubsubsection
\renewcommand{\tocsubsubsection}[2]{\quad\oldtocsubsubsection{#1}{#2}}

\title{Cellular diagonals of permutahedra}

\thanks{
BDO was partially supported by the French ANR grants ALCOHOL (ANR-19-CE40-0006), CARPLO (ANR-20-CE40-0007), HighAGT (ANR-20-CE40-0016) and S3 (ANR-20-CE48-0010). 
GLA and KS were supported by the Australian Research Council Future Fellowship FT210100256. 
KS was supported by an Australian Government Research Training Program (RTP) Scholarship.
VP was partially supported by the French ANR grant CHARMS (ANR-19-CE40-0017) and by the French--Austrian project PAGCAP (ANR-21-CE48-0020 \& FWF I 5788).
}

\author[B. Delcroix-Oger]{B\'er\'enice Delcroix-Oger}
\address[B\'{e}r\'{e}nice Delcroix-Oger]{Institut Montpelli\'erain Alexander Grothendieck, Universit\'e de Montpellier, France}
\email{berenice.delcroix-oger@umontpellier.fr}
\urladdr{\url{https://oger.perso.math.cnrs.fr/}}

\author[G. Laplante-Anfossi]{Guillaume Laplante-Anfossi}
\address[Guillaume Laplante-Anfossi]{School of Mathematics and Statistics, The University of Melbourne, Victoria, Australia}
\email{guillaume.laplanteanfossi@unimelb.edu.au}
\urladdr{\url{https://guillaumelaplante-anfossi.github.io/}}

\author[V. Pilaud]{Vincent Pilaud}
\address[Vincent Pilaud]{CNRS \& LIX, \'Ecole Polytechnique, Palaiseau, France}
\email{vincent.pilaud@lix.polytechnique.fr}
\urladdr{\url{http://www.lix.polytechnique.fr/~pilaud/}}

\author[K. Stoeckl]{Kurt Stoeckl}
\address[Kurt Stoeckl]{School of Mathematics and Statistics, The University of Melbourne, Victoria, Australia}
\email{kstoeckl@student.unimelb.edu.au}
\urladdr{\url{https://kstoeckl.github.io/}}

\begin{document}

\begin{abstract}
We provide a systematic enumerative and combinatorial study of geometric cellular diagonals on the permutahedra. 

In the first part of the paper, we study the combinatorics of certain hyperplane arrangements obtained as the union of $\ell$ generically translated copies of the classical braid arrangement.
Based on Zaslavsky's theory, we derive enumerative results on the faces of these arrangements involving combinatorial objects named partition forests and rainbow forests.
This yields in particular nice formulas for the number of regions and bounded regions in terms of exponentials of generating functions of Fuss-Catalan numbers.
By duality, the specialization of these results to the case $\ell = 2$ gives the enumeration of any geometric diagonal of the permutahedron.

In the second part of the paper, we study diagonals which respect the operadic structure on the family of permutahedra.
We show that there are exactly two such diagonals, which are moreover isomorphic.
We describe their facets by a simple rule on paths in partition trees, and their vertices as pattern-avoiding pairs of permutations.
We show that one of these diagonals is a topological enhancement of the Sanbeblidze--Umble diagonal, and unravel a natural lattice structure on their sets of facets.

In the third part of the paper, we use the preceding results to show that there are precisely two isomorphic topological cellular operadic structures on the families of operahedra and multiplihedra, and exactly two infinity-isomorphic geometric universal tensor products of homotopy operads and A-infinity morphisms.
\end{abstract}

\vspace*{-1.4cm}
\maketitle

\vspace*{-.6cm}
\centerline{\includegraphics[scale=.75]{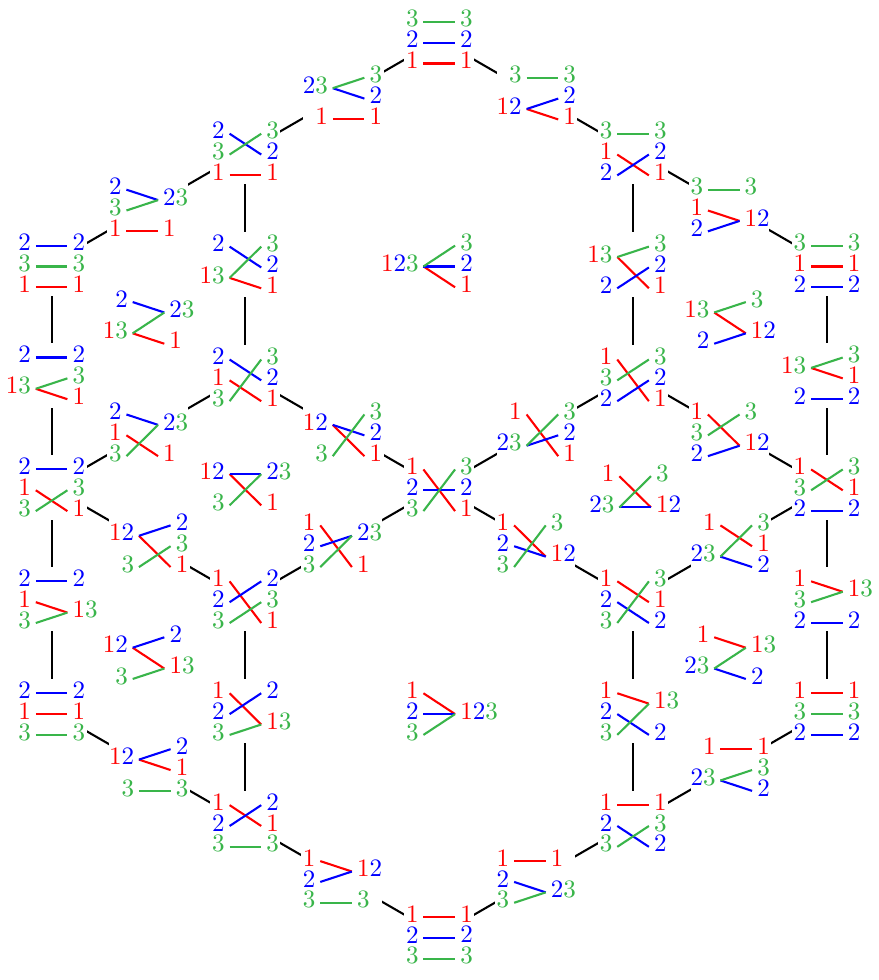}}
\vspace*{-.2cm}

\newpage
\enlargethispage{1cm}
\tableofcontents

%%%%%%%%%%%%%%%%%%%%%%%%%%%%%%%%%%%%%%

\newpage
\section*{Introduction} 
\label{s:introduction}

The purpose of this article is to study \emph{cellular diagonals} on the \emph{permutahedra}, which are cellular maps homotopic to the usual \emph{thin diagonal} $\triangle : P \to P \times P,~ x \mapsto (x,x)$ (\cref{def:thinDiagonal,def:cellularDiagonal}).
Such diagonals, and in particular coherent families that we call \emph{operadic diagonals} (\cref{def:operadicDiagonal}), are of interest in algebraic geometry and topology: via the theory of Fulton--Sturmfels~\cite{FultonSturmfels}, they give explicit formulas for the cup product on Losev--Manin toric varieties~\cite{LosevManin}; they define universal tensor products of permutadic $\Ainf$-algebras~\cite{LodayRonco-permutads,Markl}; they define a coproduct on permutahedral sets, which are models of two-fold loop spaces~\cite{SaneblidzeUmble}, and their study is needed to pursue the work of H. J. Baues aiming at defining explicit combinatorial models for higher iterated loop spaces~\cite{Baues}; using the canonical projections to the operahedra, associahedra and multiplihedra, they define universal tensor products of homotopy operads, $\Ainf$-algebras and $\Ainf$-morphisms, respectively~\cite{LaplanteAnfossi,LaplanteAnfossiMazuir}.

Cellular diagonals for face-coherent families of polytopes are a fundamental object in algebraic topology. 
The Alexander--Whitney diagonal for simplices~\cite{EilenbergMacLane}, and the Serre map for cubes~\cite{Serre}, allow one to define the cup product in singular simplicial and cubical cohomology. 
These two diagonals are also needed in the study of iterated loop spaces~\cite{Baues}, while other diagonals are needed in the study of the homology of fibered spaces~\cite{Saneblidze-freeLoopFibration,SaneblidzeRivera, Proute}. 
In another direction, cellular diagonals allow one to define universal tensor products in homotopical algebra. 
The seminal case of the \emph{associahedra} has a rich history: the first algebraic diagonal was found by S.~Saneblidze and R.~Umble~\cite{SaneblidzeUmble}, followed by a second one by M.~Markl and S.~Shnider~\cite{MarklShnider}, which was conjectured to coincide with the first one. 
This has recently been shown to hold~\cite{SaneblidzeUmble-comparingDiagonals}, while a topological enhancement of the \emph{magical formula} of~\cite{MarklShnider} was provided by N.~Masuda, H.~Thomas, A.~Tonks and B.~Vallette~\cite{MasudaThomasTonksVallette}.

In~\cite{MasudaThomasTonksVallette}, the authors re-introduced the powerful technique of Fulton--Sturmfels~\cite{FultonSturmfels}, which came from the theory of fiber polytopes of~\cite{BilleraSturmfels}, to define a topological cellular diagonal of the associahedra.
We shall call such a diagonal a \emph{geometric diagonal} (\cref{def:geometricDiagonal}).
There are two remarkable features of this diagonal (or more precisely this family of diagonals, one for the Loday associahedron in each dimension).
First, it respects the operadic structure of the associahedra (in fact, forces a unique topological cellular operad structure on them!), that is, the fact that each face of an associahedron is isomorphic to a product of lower-dimensional associahedra.
Second, it satisfies the \emph{magical formula} of J.-L. Loday: the faces in the image of the diagonal are given by the pairs of faces which are comparable in the Tamari order (see \cref{sec:cellularDiagonals,rem:magicalFormula} for a precise statement). 
This magical formula for the associahedra recently lead to new enumerative results for Tamari intervals~\cite{BostanChyzakPilaud}.

Building on~\cite{MasudaThomasTonksVallette}, a general theory of geometric diagonals was developed in~\cite{LaplanteAnfossi}.
In particular, a combinatorial formula describing the image of the diagonal of any polytope was given~\cite[Thm.~1.26]{LaplanteAnfossi}.
The topological operad structure of~\cite{MasudaThomasTonksVallette} on the associahedra was generalized to the family of \emph{operahedra}, which comprise the family of permutahedra, and encodes the notion of homotopy operad.
Cellular diagonals of the operahedra do \emph{not} satisfy the magical formula, and the combinatorial difficulty of describing their image is what prompted the development of the theory in~\cite{LaplanteAnfossi}.
In fact, there is an interesting dichotomy between the families of polytopes which satisfy the magical formula (simplices, cubes, freehedra, associahedra) and those who do not (permutahedra, multiplihedra, operahedra).

Since the operahedra are \emph{generalized permutahedra}~\cite{Postnikov}, their operadic diagonals are completely determined by the operadic diagonals of permutahedra (see~\cite[Sect.~1.6]{LaplanteAnfossi}), which is the purpose of study of the present paper.

The first cellular diagonal of the permutahedra was obtained at the algebraic level by S.~Sanebli\-dze and R.~Umble~\cite{SaneblidzeUmble}.
We shall call this diagonal the \emph{original $\SU$ diagonal}. 
The first topological cellular diagonal of the permutahedra was defined in~\cite{LaplanteAnfossi}, we shall call it the \emph{geometric $\LA$ diagonal}.
Both of these families of diagonals are \emph{operadic}, \ie they respect the product structure on the faces of permutahedra (this property is called ``comultiplicativity" in~\cite{SaneblidzeUmble}).
More precisely, the algebraic structure encoded by the permutahedra is that of permutadic $\Ainf$-algebra~\cite{LodayRonco-permutads,Markl}.

The toric varieties associated with the permutahedra are called Losev--Manin varieties, introduced in~\cite{LosevManin}.
At this level, the operadic structure is that of a reconnectad~\cite{DotsenkoKeilthyLyskov}. 
The cohomology ring structure was studied by A.~Losev and Y.~Manin, and quite extensively since then, see for instance~\cite{BergstromMinabe, Lin}. 
Our current work brings a completely combinatorially explicit description of the cup product; it would be interesting to know if this new description can lead to new results, or how it can be used to recover existing ones.

The first part of the paper derives enumerative results for the iterations of any geometric diagonal of the permutahedra. 
According to the Fulton--Sturmfels formula~\cite{FultonSturmfels} (see \cref{prop:diagonalCommonRefinement} and \cref{rem:Fulton--Sturmfels}), this amounts to the study of hyperplane arrangements made of generically translated copies of the braid arrangement.
The second part studies in depth the combinatorics of operadic diagonals of the permutahedra, providing in particular a topological enhancement of the original $\SU$ diagonal, while the third part derives consequences of this combinatorial study in the field of homotopical algebra.
We now proceed to introduce separately each part in more detail.

%%%%%%%%%%%%%%%

\subsection*{\cref{part:multiBraidArrangements}. Combinatorics of multiple braid arrangements}

As the dual of the permutahedron~$\Perm$ is the classical braid arrangement~$\braidArrangement$, the dual of a diagonal of the permutahedron~$\Perm$ is a hyperplane arrangement~$\multiBraidArrangement[n][2]$ made of $2$ generically translated copies of the braid arrangement~$\braidArrangement$.
In the first part of the paper, we therefore study the combinatorics of the $(\ell,n)$-braid arrangement~$\multiBraidArrangement$, defined as the union of $\ell$ generically translated copies of the braid arrangement~$\braidArrangement$ (\cref{def:multiBraidArrangement}).
We are mainly interested in the~$\ell = 2$ case for the enumeration of the faces of the diagonals of the permutahedron~$\Perm$, but the general $\ell$ case is not much harder and corresponds algebraically to the enumeration of the faces of cellular $\ell$-gonals of the permutahedron~$\Perm$.
%We observe that the flats of~$\multiBraidArrangement$ are in bijection with forests of partitions and that the faces of~$\multiBraidArrangement$ are in bijection with certain forests of ordered partitions.

\cref{sec:flatPoset} is dedicated to the combinatorial description of the flat poset of~$\multiBraidArrangement$ and its enumerative consequences.
We first observe that the flats of~$\multiBraidArrangement$ are in bijection with \emph{$(\ell,n)$-partition forests}, defined as $\ell$-tuples of (unordered) partitions of~$[n]$ whose intersection hypergraph is a hyperforest (\cref{def:partitionForests}).
As this description is independent of the translations of the different copies (as long as these translations are generic), we obtain by T.~Zaslavsky's theory that the number of \mbox{$k$-dimensional} faces and of bounded faces of~$\multiBraidArrangement$ only depends on~$k$, $\ell$, and~$n$. % (and not on the specific translation vectors choosen to construct~$\multiBraidArrangement$).
In fact, we obtain the following formula for the M\"obius polynomial of the $(\ell,n)$-braid arrangement~$\multiBraidArrangement$ in terms of pairs of $(\ell,n)$-partition forests.

\begin{theorem*}[\cref{thm:MobiusPolynomialMultiBraidArrangement}]
The M\"obius polynomial of the $(\ell,n)$-braid arrangement~$\multiBraidArrangement$ is given by
\[
\mobPol[\multiBraidArrangement] = x^{n-1-\ell n} y^{n-1-\ell n} \sum_{\b{F} \le \b{G}} \prod_{i \in [\ell]} x^{\card{F_i}} y^{\card{G_i}} \prod_{p \in G_i} (-1)^{\card{F_i[p]}-1} (\card{F_i[p]}-1)! \; ,
\]
where~$\b{F} \le \b{G}$ ranges over all intervals of the $(\ell,n)$-partition forest poset, and~$F_i[p]$ denotes the restriction of the partition~$F_i$ to the part~$p$ of~$G_i$.
\end{theorem*}

This formula is not particularly easy to handle, but it turns out to simplify to very elegant formulas for the number of vertices, regions, and bounded regions of the $(\ell,n)$-braid arrangement~$\multiBraidArrangement$.
Namely, using an alternative combinatorial description of the $(\ell,n)$-partition forests in terms of $(\ell, n)$-rainbow forests and a colored analogue of the classical Pr\"ufer code for permutations, we first obtain the number of vertices of the $(\ell,n)$-braid arrangement~$\multiBraidArrangement$.

\begin{theorem*}[\cref{thm:verticesMultiBraidArrangement}]
The number of vertices of the $(\ell,n)$-braid arrangement~$\multiBraidArrangement$ is
\[
f_0(\multiBraidArrangement) = \ell \big( (\ell-1) n + 1 \big)^{n-2}.
\]
\end{theorem*}

This result can even be refined according to the dimension of the flats of the different copies intersected to obtain the vertices of the $(\ell,n)$-braid arrangement~$\multiBraidArrangement$.

\begin{theorem*}[\cref{thm:verticesRefinedMultiBraidArrangement}]
For any~$k_1, \dots, k_\ell$ such that~$0 \le k_i \le n-1$ for~$i \in [\ell]$ and~${\sum_{i \in [\ell]} k_i = n-1}$, the number of vertices~$v$ of the $(\ell,n)$-braid arrangement~$\multiBraidArrangement$ such that the smallest flat of the $i$\ordinal{} copy of~$\braidArrangement$ containing~$v$ has dimension~$n-k_i-1$ is given by
\[
n^{\ell-1} \binom{n-1}{k_1, \dots, k_\ell} \prod_{i \in [\ell]} (n-k_i)^{k_i-1}.
\]
\end{theorem*}

We then consider the regions of the $(\ell,n)$-braid arrangement~$\multiBraidArrangement$.
We first obtain a very simple exponential formula for its characteristic polynomial.

\begin{theorem*}[\cref{thm:characteristicPolynomialMultiBraidArrangement}]
The characteristic polynomial~$\charPol[\multiBraidArrangement]$ of the $(\ell,n)$-braid arrangement~$\multiBraidArrangement$ is given~by
\[
\charPol[\multiBraidArrangement] = \frac{(-1)^n n!}{y} \, [z^n] \, \exp \bigg( - \sum_{m \ge 1} \frac{F_{\ell,m} \, y \, z^m}{m} \bigg) ,
\]
where~$\displaystyle F_{\ell,m} \eqdef \frac{1}{(\ell-1)m+1} \binom{\ell m}{m}$ is the Fuss-Catalan number.
\end{theorem*}

Evaluating the characteristic polynomial at~$y = -1$ and~$y = 1$ respectively, we obtain by T.~Zaslavsky's theory the numbers of regions and bounded regions of the $(\ell,n)$-braid arrangement~$\multiBraidArrangement$.

\begin{theorem*}[\cref{thm:regionsMultiBraidArrangement}]
The numbers of regions and of bounded regions of the $(\ell,n)$-braid arrangement~$\multiBraidArrangement$ are given by
\begin{align*}
f_{n-1}(\multiBraidArrangement) 
& = n! \, [z^n] \exp \Bigg( \sum_{m \ge 1} \frac{F_{\ell,m} \, z^m}{m} \Bigg) \\
\text{and}\qquad
b_{n-1}(\multiBraidArrangement)
%& = - n! \, [z^n] \exp \Bigg( - \sum_{m \ge 1} \frac{F_{\ell,m} \, z^m}{m} \Bigg) 
& = (n-1)! \, [z^{n-1}] \exp \bigg( (\ell-1) \sum_{m \ge 1} F_{\ell,m} \, z^m \bigg),
\end{align*}
where~$\displaystyle F_{\ell,m} \eqdef \frac{1}{(\ell-1)m+1} \binom{\ell m}{m}$ is the Fuss-Catalan number.
\end{theorem*}

Finally, \cref{sec:facePoset} is dedicated to the combinatorial description of the face poset of~$\multiBraidArrangement$.
We observe that the faces of~$\multiBraidArrangement$ are in bijection with certain \emph{ordered} $(\ell,n)$-partition forests, defined as $\ell$-tuples of ordered partitions of~$[n]$ whose underlying unordered partitions form an (unordered) $(\ell,n)$-partition forest (\cref{def:orderedPartitionForest}).
Here, which ordered $(\ell,n)$-partition forests actually appear as faces of~$\multiBraidArrangement$ depends on the choice of the translations of the different copies.
We provide a combinatorial description of the possible orderings of a $(\ell,n)$-partition forest compatible with some given translations in terms of certain paths in the forest (\cref{prop:PFtoOPF1,prop:PFtoOPF2}), and a combinatorial characterization of the ordered partition forests which appear for some given translations in terms of the circuits of a certain oriented graph (\cref{prop:characterizationOPFs}).

%%%%%%%%%%%%%%%

\subsection*{\cref{part:diagonalsPermutahedra}. Diagonals of permutahedra}

We present cellular diagonals, the Fulton--Sturmfels method, the magical formula and specialize the results of Part I to the permutahedra in \cref{sec:cellularDiagonals}.
Then, we initiate in \cref{sec:operadicDiagonals} the study of operadic diagonals (\cref{def:operadicDiagonal}).
These are families of diagonals of the permutahedra which are compatible with the property that faces of permutahedra are product of lower-dimensional permutahedra. 

\begin{theorem*}[\cref{thm:unique-operadic,thm:bijection-operadic-diagonals}]
There are exactly four operadic geometric diagonals of the permutahedra, the geometric $\LA$ and $\SU$ diagonals and their opposites, and only the first two respect the weak order on permutations.
Moreover, their cellular images are isomorphic as posets.
\end{theorem*}

It turns out that the facets and vertices of operadic diagonals admit elegant combinatorial descriptions. 
The following is a consequence of a general geometrical result, that holds for any diagonal (\cref{prop:PFtoOPF1}).

\begin{theorem*}[\cref{thm:facet-ordering}]
A pair of ordered partitions $(\sigma,\tau)$ forming a partition tree is a facet of the $\LA$ (\resp $\SU$) geometric diagonal if and only if the minimum (\resp maximum) of every directed path between two consecutive blocks of $\sigma$ or $\tau$ is oriented from $\sigma$ to $\tau$ (\resp from $\tau$ to $\sigma$).
\end{theorem*}

\pagebreak
Vertices of operadic diagonals are pairs of permutations, and form a strict subset of intervals of the weak order. 
They admit an analogous description in terms of pattern-avoidance. 

\begin{theorem*}[\cref{thm:patterns}]
A pair of permutations of $[n]$ is a vertex of the $\LA$ (\resp $\SU$) diagonal if and only if for any~$k\geq 1$ and for any $I=\{i_1, \dots, i_k\},J=\{j_1, \dots, j_k\} \subset [k]$ such that $i_1=1$ (\resp $j_k=2k$) it avoids the patterns 
\begin{align}
	(j_1 i_1 j_2 i_2 \cdots j_k i_k,\ i_2 j_1 i_3 j_2 \cdots i_k j_{k-1} i_1 j_k), \tag{LA} \\
	\text{ \resp } (j_1 i_1 j_2 i_2 \cdots j_k i_k, \ i_1 j_k i_2 j_1 \cdots i_{k-1} j_{k-2}i_k j_{k-1}), \tag{SU}
\end{align}
For each $k \ge 1$, there are $\binom{2k-1}{k-1,k}(k-1)!k!$ such patterns, which are $(21,12)$ for $k=1$, and the following for $k=2$
\begin{itemize}
	\item $\LA$ avoids 
	$(3142,2314), (4132,2413),
	(2143,3214), (4123,3412),
	(2134,4213), (3124,4312)$,
	\item $\SU$ avoids 
	$(1243,2431),(1342,3421),
	(2143,1432),(2341,3412),
	(3142,1423),(3241,2413)$.
\end{itemize}
\end{theorem*}

In \cref{sec:shifts}, we introduce \emph{shifts} that can be performed on the facets of operadic diagonals.
These allow us to show that the geometric $\SU$ diagonal is a topological enhancement of the original $\SU$ diagonal. 

\begin{theorem*}[\cref{thm:recover-SU}]
The original and geometric $\SU$ diagonals coincide. 
\end{theorem*}

The proof of this result, quite technical, proceeds by showing the equivalence between $4$ different descriptions of the diagonal: the original, $1$-shift, $m$-shift and geometric $\SU$ diagonals (\cref{subsec:topological-SU}).
This brings a positive answer to~\cite[Rem.~2.19]{LaplanteAnfossi}, showing that the original $\SU$ diagonal can be recovered from a choice of chambers in the fundamental hyperplane arrangement of the permutahedron. 
Our formulas for the number of facets also agrees with the experimental count made in~\cite{VejdemoJohansson}.
Moreover, it provides a new proof that all known diagonals on the associahedra coincide~\cite{SaneblidzeUmble-comparingDiagonals}.
Indeed, since the family of vectors inducing the geometric $\SU$ diagonal all have strictly decreasing coordinates, the diagonal induced on the associahedron is given by the magical formula~\cite[Thm.~2]{MasudaThomasTonksVallette}, see also~\cite[Prop.~3.8]{LaplanteAnfossi}.

The above theorem also allows us to translate the different combinatorial descriptions of the facets of operadic diagonals from one to the other, compiled in the following table. 

\begin{figure}[h!]
	\begin{center}
	\begin{tabular}{c|c|c}
	Description & $\SU$ diagonal & $\LA$ diagonal \\
	\hline
	Original & \cite{SaneblidzeUmble} & \cref{def:classical-LA} \\
	Geometric & \cref{thm:minimal} & \cite{LaplanteAnfossi} \\
	Path extrema & \cref{thm:facet-ordering} & \cref{thm:facet-ordering} \\
	$1$-shifts & \cref{def:classical-SU} & \cref{def:classical-LA} \\
	$m$-shifts & \cref{def:classical-SU} & \cref{def:classical-LA} \\
	Lattice & \cref{prop:shift lattice} & \cref{prop:shift lattice} \\
	Cubical & \cite{SaneblidzeUmble} & \cref{prop:LA-cubical} \\
	Matrix & \cite{SaneblidzeUmble} & \cref{subsec:matrix} 
	\end{tabular}
	\end{center}
\end{figure}

In \cref{sec:Shift-lattice}, we show that the facets of operadic diagonals are disjoint unions of lattices, that we call the \emph{shift lattices}.
These lattices are isomorphic to a product of chains, and are indexed by the permutations of $[n]$.
Moreover, while the pairs of facets of operadic diagonals are intervals of the facial weak order (\cref{sec:facial-weak-order}), the shift lattices are not sub-lattices of this order's product (see \cref{fig: Inversion and lattice counter example}). 

Finally, we present the alternative cubical (\cref{sec:Cubical}) and matrix (\cref{subsec:matrix}) descriptions of the $\SU$ diagonal from \cite{SaneblidzeUmble,SaneblidzeUmble-comparingDiagonals}, providing proofs of their equivalence with the other descriptions, and giving their $\LA$ counterparts. 
The existence of this cubical description, based on a subdivision of the cube combinatorially isomorphic to the permutahedron, finds its conceptual root in the bar-cobar resolution of the associative permutad. 
Indeed, this resolution is encoded by the dual subdivision of the permutahedron, which is cubical since $\Perm$ is a simple polytope, and a diagonal can be obtained from the classical Serre diagonal via retraction, in the same fashion as for the associahedra, see \cite{MarklShnider, Loday-diagonal} and \cite[Sec. 5.1]{LaplanteAnfossiMazuir}.

%%%%%%%%%%%%%%%

\pagebreak
\subsection*{\cref{part:higherAlgebraicStructures}. Higher algebraic structures}

In this shorter third part of the paper, we derive some higher algebraic consequences of the preceding results, which were the original motivation for the present study.
They concern the \emph{operahedra}, a family of polytopes indexed by planar trees, which encode (non-symmetric non-unital) homotopy operads \cite{LaplanteAnfossi}, and the \emph{multiplihedra}, a family of polytopes indexed by $2$-colored nested linear trees, which encode $\Ainf$-morphisms \cite{LaplanteAnfossiMazuir}.
Both of these admit realizations ``\`a la Loday", which generalize the Loday realizations of the associahedra. 
The faces of an operahedron are in bijection with \emph{nestings}, or parenthesization, of the corresponding planar tree, while the faces of a multiplihedron are in bijection with $2$-colored nestings of the corresponding linear tree. 
The main results concerning the operahedra are summarized as follows. 

\begin{theorem*}[\cref{thm:operahedra,thm:top-iso,thm:infinity-iso}] 
There are exactly 
\begin{enumerate}
	\item two geometric operadic diagonals of the Loday operahedra, the $\LA$ and $\SU$ diagonals, 
	\item two geometric topological cellular colored operad structures on the Loday operahedra,
	\item two geometric universal tensor products of homotopy operads,
\end{enumerate}
which agree with the generalized Tamari order on fully nested trees. 
Moreover, the two topological operad structures are isomorphic, and the two tensor products are not strictly isomorphic, but are related by an $\infty$-isotopy. 
\end{theorem*}

As the associahedra and the permutahedra are part of the family of operahedra, we get analogous results for $\Ainf$-algebras and permutadic $\Ainf$-algebras. 
The main results concerning the multiplihedra are summarized as follows. 

\begin{theorem*}[{\cref{thm:multiplihedra,thm:top-iso-2,thm:infinity-iso-2}}]
	There are exactly 
	\begin{enumerate}
	\item two geometric operadic diagonals of the Forcey multiplihedra, the $\LA$ and $\SU$ diagonals,
	\item two geometric topological cellular operadic bimodule structures (over the Loday associahedra) on the Forcey multiplihedra,
	\item two compatible geometric universal tensor products of $\Ainf$-algebras and $\Ainf$-morphisms,
	\end{enumerate}
	which agree with the Tamari-type order on atomic $2$-colored nested linear trees. 
	Moreover, the two topological operadic bimodule structures are isomorphic, and the two tensor products are not strictly isomorphic, but are related by an $\infty$-isotopy. 
\end{theorem*}

Here, by the adjective ``geometric", we mean diagonal, operadic structure and tensor product which are obtained geometrically on the polytopes via the Fulton--Sturmfelds method. 
By ``universal", we mean formulas for the tensor products which apply to \emph{any} pair of homotopy operads or $\Ainf$-morphisms.

However, the isomorphisms of topological operads (\resp operadic bimodules) takes place in a category of polytopes $\PolySub$ for which the morphisms are \emph{not} affine maps \cite[Def. 4.13]{LaplanteAnfossi}, and it does \emph{not} commute with the diagonal maps (\cref{ex:iso-not-Hopf,ex:iso-not-Hopf-2}).
Moreover, the pairs of faces in the image of the two operadic diagonals are in general not in bijection (see \cref{ex:operahedra-LA-SU,ex:multiplihedra-LA-SU}), yielding different (but $\infty$-isomorphic) tensor products of homotopy operads (\resp \mbox{$\Ainf$-morphisms}).

%%%%%%%%%%%%%%%%%%%%%%%%%%%%%%%%%%%%%%

\section*{Acknowledgements}

We are indebted to Matthieu Josuat-Vergès for taking part in the premises of this paper, in particular for conjecturing the case $\ell = 2$ of \cref{thm:verticesRefinedMultiBraidArrangement}, for working on the proof of \cref{thm:patterns}, and for implementing some sage code used during the project.
GLA is grateful to Hugh Thomas for raising the question to count the facets of the $\LA$ diagonal and for preliminary discussions during a visit at the LACIM in the summer of 2021 where the project started, and to the Max Planck Institute for Mathematics in Bonn where part of this work was carried out.
We are grateful to Sylvie Corteel for pointing out the relevance of the dual perspective for the purpose of counting.
VP is grateful to the organizers (Karim Adiprasito, Alexey Glazyrin, Isabella Novic, and Igor Pak) of the workshop ``Combinatorics and Geometry of Convex Polyhedra'' held at the Simons Center for Geometry and Physics in March 2023 for the opportunity to present a preliminary version of our enumerative results, and to  Pavel Galashin for asking about the characteristic polynomial of the multiple braid arrangement during this presentation.
We thank Samson Saneblidze and Ron Umble for useful correspondence.

%%%%%%%%%%%%%%%%%%%%%%%%%%%%%%%%%%%%%%
%%%%%%%%%%%%%%%%%%%%%%%%%%%%%%%%%%%%%%

\clearpage
\part{Combinatorics of multiple braid arrangements}
\label{part:multiBraidArrangements}

In this first part, we study the combinatorics of hyperplane arrangements obtained as unions of generically translated copies of the braid arrangement.
In \cref{sec:arrangements}, we first recall some classical facts on the enumeration of hyperplane arrangements (\cref{subsec:arrangements}), present the classical braid arrangement (\cref{subsec:braidArrangement}), and define our multiple braid arrangements (\cref{subsec:multiBraidArrangement}).
Then in \cref{sec:flatPoset}, we describe their flat posets in terms of partition forests (\cref{subsec:partitionForests}) and rainbow forests (\cref{subsec:rainbowForests}), from which we derive their M\"obius polynomials (\cref{subsec:MobiusPolynomialMultiBraidArrangement}), and some surprising formulas for their numbers of vertices (\cref{subsec:verticesMultiBraidArrangement}) and regions (\cref{subsec:regionsMultiBraidArrangement}).
Finally, in \cref{sec:facePoset}, we describe their face posets in terms of ordered partition forests (\cref{subsec:orderedPartitionForests}), and explore some combinatorial criteria to describe the ordered partition forests that appear as faces of a given multiple braid arrangement (\cref{subsec:PFtoOPF,subsec:criterionOPF}).

%%%%%%%%%%%%%%%%%%%%%%%%%%%%%%%%%%%%%%

\section{Recollection on hyperplane arrangements and braid arrangements}
\label{sec:arrangements}

%%%%%%%%%%%%%%%

\subsection{Hyperplane arrangements}
\label{subsec:arrangements}

We first briefly recall classical results on the combinatorics of affine hyperplane arrangements, in particular the enumerative connection between their intersection posets and their face lattices due to T.~Zaslavsky~\cite{Zaslavsky}.

\begin{definition}
A finite affine real \defn{hyperplane arrangement} is a finite set~$\arrangement$ of affine hyperplanes in~$\R^d$.
\end{definition}

\begin{definition}
A \defn{region} of~$\arrangement$ is a connected component of~$\R^d \ssm \bigcup_{H \in \arrangement} H$.
The \defn{faces} of~$\arrangement$ are the closures of the regions of~$\arrangement$ and all their intersections with a hyperplane of~$\arrangement$.
The \defn{face poset} of~$\arrangement$ is the poset~$\facePoset$ of faces of~$\arrangement$ ordered by inclusion.
The \defn{$f$-polynomial}~$\fPol$ and \defn{$b$-polynomial}~$\bPol$ of~$\arrangement$ are the polynomials
\[
\fPol \eqdef \sum_{k = 0}^d f_k(\arrangement) \, x^k
\qquad\text{and}\qquad
\bPol \eqdef \sum_{k = 0}^d b_k(\arrangement) \, x^k ,
\]
where~$f_k(\arrangement)$ denotes the number of $k$-dimensional faces of~$\arrangement$, while~$b_k(\arrangement)$ denotes the number of bounded $k$-dimensional faces of~$\arrangement$.
\end{definition}

\begin{definition}
A \defn{flat} of~$\arrangement$ is a non-empty affine subspace of~$\R^d$ that can be obtained as the intersection of some hyperplanes of~$\arrangement$.
The \defn{flat poset} of~$\arrangement$ is the poset~$\flatPoset$ of flats of~$\arrangement$ ordered by reverse inclusion.
\end{definition}

\begin{definition}
\label{def:MobiusPolynomial}
The \defn{M\"obius polynomial}~$\mu_{\arrangement}(x,y)$ of~$\arrangement$ is the polynomial defined by
\[
\mobPol \eqdef \sum_{F \le G} \mu_{\flatPoset}(F,G) \, x^{\dim(F)} \, y^{\dim(G)},
\]
where~$F \le G$ ranges over all intervals of the flat poset~$\flatPoset$, and~$\mu_{\flatPoset}(F,G)$ denotes the \defn{M\"obius function} on the flat poset~$\flatPoset$ defined as usual by
\[
\mu_{\flatPoset}(F, F) = 1
\qquad\text{and}\qquad
\sum_{F \le G \le H} \mu_{\flatPoset}(F,G) = 0
\]
for all~$F < H$ in~$\flatPoset$.
\end{definition}

\begin{remark}
Our definition of the M\"obius polynomial slightly differs from that of~\cite{Zaslavsky} as we use the dimension of~$F$ instead of its codimension, in order to simplify slightly the following statement.
\end{remark}

\begin{theorem}[{\cite[Thm.~A]{Zaslavsky}}]
\label{thm:Zaslavsky}
The $f$-polynomial, the $b$-polynomial, and the M\"obius polynomial of the hyperplane arrangement~$\arrangement$ are related by
\[
\fPol = \mobPol[\arrangement][-x][-1]
\qquad\text{and}\qquad
\bPol = \mobPol[\arrangement][-x][1].
\]
\end{theorem}

\begin{example}
\begin{figure}
%	\begin{overpic}[scale=.9]{intersectionPoset}
%		\put(72.5, -2){$1$}
%		\put(51, 10){$-1$}
%		\put(61, 10){$-1$}
%		\put(71, 10){$-1$}
%		\put(82, 10){$-1$}
%		\put(94, 10){$-1$}
%		\put(51.5, 32){$2$}
%		\put(66, 32){$1$}
%		\put(80, 32){$1$}
%		\put(93.5, 32){$2$}
%	\end{overpic}
%	\caption{A hyperplane arrangement (left) and its intersection poset with its M\"obius function (right).}
	\centerline{\includegraphics[scale=.9]{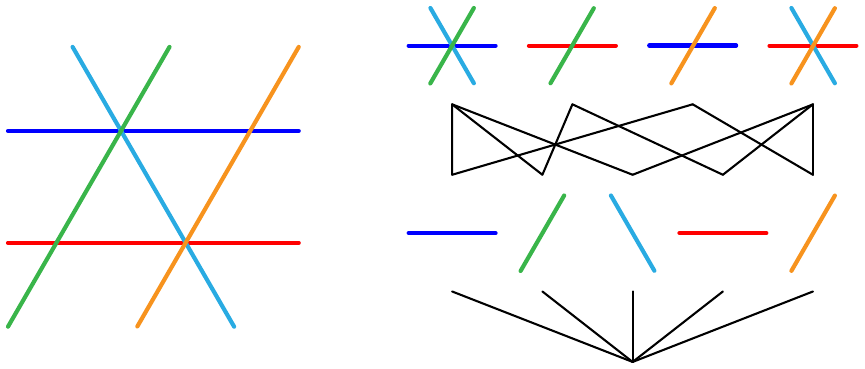}}
	\caption{A hyperplane arrangement (left) and its intersection poset (right).}
	\label{fig:arrangement}
\end{figure}
For the arrangement~$\arrangement$ of $5$ hyperplanes of \cref{fig:arrangement}, we have
\[
\mobPol = x^2y^2 - 5x^2y + 6x^2 + 5xy - 10x + 4 ,
\]
so that
\[
\fPol = \mobPol[\arrangement][-x][-1] = 12 \, x^2 + 15 \, x + 4 
\qquad \text{and}\qquad
\bPol = \mobPol[\arrangement][-x][1] = 2 \, x^2 + 5 \, x + 4 .
\]
\end{example}

\begin{remark}
\label{rem:characteristicPolynomial}
The coefficient of~$x^d$ in the M\"obius polynomial~$\mobPol$ gives the more classical \defn{characteristic polynomial}
\[
\charPol \eqdef [x^d] \, \mobPol = \sum_F \mu_{\flatPoset}(\R^d,F) \, y^{\dim(F)} .
\]
By \cref{thm:Zaslavsky}, we thus have
\[
f_d(\arrangement) = (-1)^d \, \charPol[\arrangement][-1] 
\qquad\text{and}\qquad
b_d(\arrangement) = (-1)^d \, \charPol[\arrangement][1].
\]
\end{remark}

%%%%%%%%%%%%%%%

\subsection{The braid arrangement}
\label{subsec:braidArrangement}

We now briefly recall the classical combinatorics of the braid arrangement.
See \cref{fig:facePosetBraidArrangement3,fig:intersectionPosetBraidArrangement3,fig:intersectionPosetBraidArrangement4} for illustrations when~$n = 3$ and~$n = 4$.

%\begin{figure}
%	\centerline{\includegraphics[scale=.9]{figures/intersectionPosetBraidArrangement3Full}}
%	\caption{The braid arrangement $\braidArrangement[3]$ (left), its flat poset~$\flatPoset[{\braidArrangement[3]}]$ (middle), and the partition poset~$\partitionPoset[3]$ (right).}
%	\label{fig:braidArrangement3}
%\end{figure}

\afterpage{
\begin{figure}
	\centerline{\includegraphics[scale=.6]{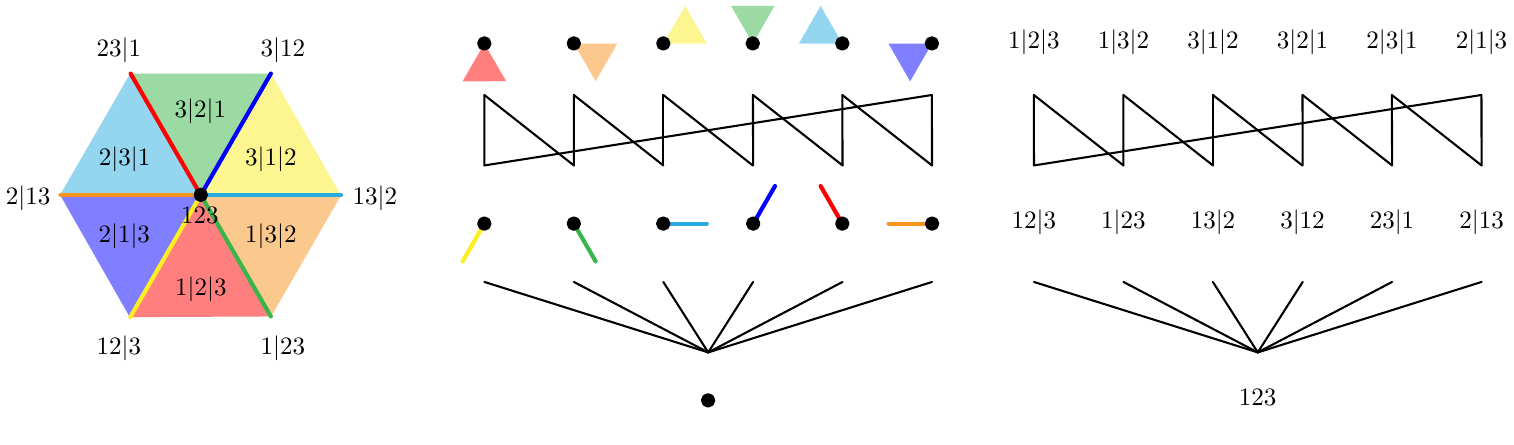}}
	\caption{The face poset~$\facePoset[{\braidArrangement[3]}]$ of the braid arrangement $\braidArrangement[3]$ (left), where faces are represented as cones (middle) or as ordered set partitions of~$[3]$ (right).}
	\label{fig:facePosetBraidArrangement3}
\end{figure}
}

\afterpage{
\begin{figure}
	\centerline{\includegraphics[scale=.6]{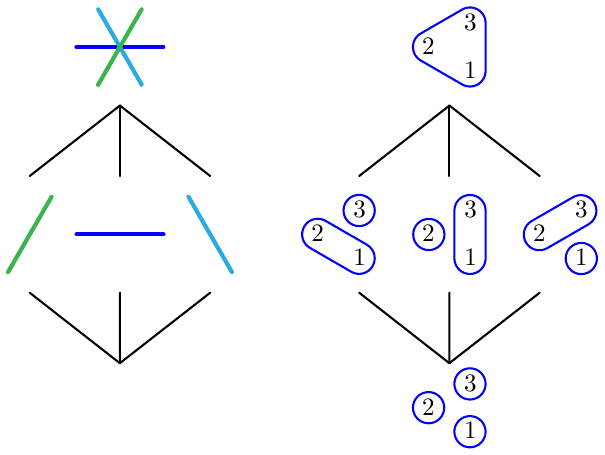}}
	\caption{The flat poset~$\flatPoset[{\braidArrangement[3]}]$ of the braid arrangement $\braidArrangement[3]$, where flats are represented as intersections of hyperplanes (left) or as set partitions of~$[3]$ (right).}
	\label{fig:intersectionPosetBraidArrangement3}
\end{figure}
}

\afterpage{
\begin{figure}
	\centerline{\includegraphics[scale=.26]{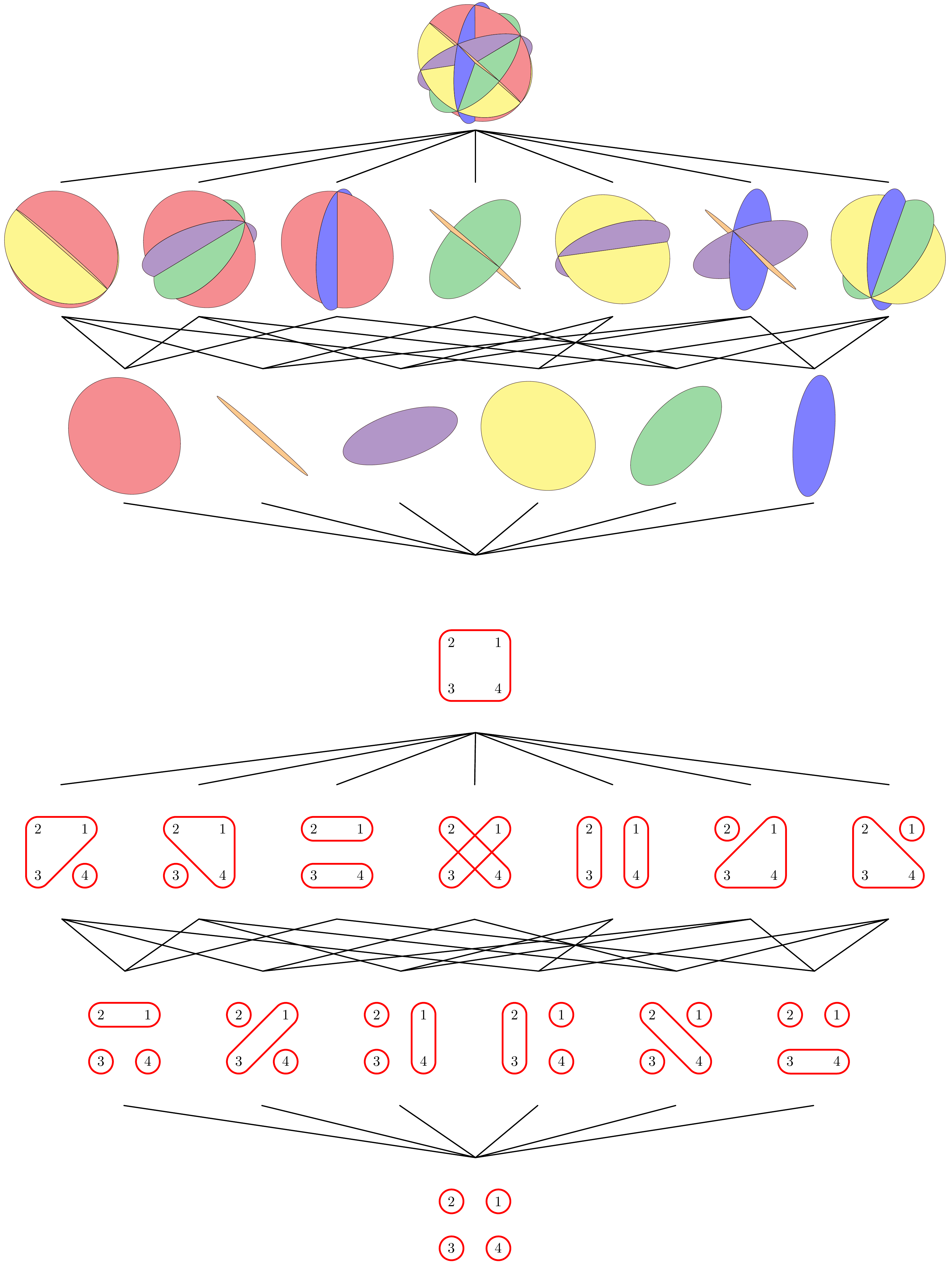}}
	\caption{The flat poset~$\flatPoset[{\braidArrangement[4]}]$ of the braid arrangement $\braidArrangement[4]$, where flats are represented as intersections of hyperplanes (top) or as set partitions of~$[4]$ (bottom).}
	\label{fig:intersectionPosetBraidArrangement4}
\end{figure}
}

\begin{definition}
Fix~$n \ge 1$ and denote by~$\HH$ the hyperplane of~$\R^n$ defined by~$\sum_{s \in [n]} x_s = 0$.
The \defn{braid arrangement}~$\braidArrangement$ is the arrangement of the hyperplanes~$\set{\b{x} \in \HH}{x_s = x_t}$ for all~${1 \le s < t \le n}$.
\end{definition}

\begin{remark}
\label{rem:essential}
Note that we have decided to work in the space~$\HH$ rather than in the space~$\R^n$.
The advantage is that the braid arrangement~$\braidArrangement$ in~$\HH$ is essential, so that we can speak of its rays.
Working in~$\R^n$ would change rays to walls, and would multiply all M\"obius polynomials by a factor~$xy$.
\end{remark}

The combinatorics of the braid arrangement~$\braidArrangement$ is well-known.
The descriptions of its face and flat posets involve both ordered and unordered set partitions.
To avoid confusions, we will always mark with an arrow the ordered structures (ordered set partitions, ordered partition forests, etc.).
Hence, the letter $\pi$ denotes an unordered set partition (the order is irrelevant, neither inside each part, nor between two distinct parts), while~$\order{\pi}$ denotes an ordered set partition (the order inside each part is irrelevant, but the order between distinct parts is relevant).

The braid arrangement~$\braidArrangement$ has a $k$-dimensional face
\[
\Phi(\order{\pi}) \eqdef \set{\b{x} \in \R^n}{x_s \le x_t \text{ for all $s,t$ such that the part of~$s$ is weakly before the part of~$t$ in $\order{\pi}$}}
\]
for each ordered set partition~$\order{\pi}$ of~$[n]$ into~$k+1$ parts, or equivalently, for each surjection from~$[n]$ to~$[k+1]$.
The face poset~$\facePoset[\braidArrangement]$ is thus isomorphic to the refinement poset~$\orderedPartitionPoset$ on ordered set partitions, where an ordered partition~$\order{\pi}$ is smaller than an ordered partition~$\order{\omega}$ if each part of~$\order{\pi}$ is the union of an interval of consecutive parts in~$\order{\omega}$.
In particular, it has a single vertex corresponding to the ordered partition~$[n]$, $2^n-2$ rays corresponding to the proper nonempty subsets of~$[n]$ (ordered partitions of~$[n]$ into~$2$ parts), and $n!$ regions corresponding to the permutations of~$[n]$ (ordered partitions of~$[n]$ into~$n$ parts).
As an example, \cref{fig:facePosetBraidArrangement3} illustrates the face poset of the braid arrangement~$\braidArrangement[3]$.

The braid arrangement~$\braidArrangement$ has a $k$-dimensional~flat
\[
\Psi(\pi) \eqdef \set{\b{x} \in \R^n}{x_s = x_t \text{ for all  $s, t$ which belong to the same part of~$\pi$}}
\]
for each unordered set partition~$\pi$ of~$[n]$ into $k+1$ parts.
The flat poset~$\flatPoset[\braidArrangement]$ is thus isomorphic to the refinement poset~$\partitionPoset$ on set partitions of~$[n]$, where a partition~$\pi$ is smaller than a partition~$\omega$ if each part of~$\pi$ is contained in a part of~$\omega$.
For instance, \cref{fig:intersectionPosetBraidArrangement3,fig:intersectionPosetBraidArrangement4} illustrate the flat posets of the braid arrangements~$\braidArrangement[3]$ and~$\braidArrangement[4]$.
Note that the refinement in~$\orderedPartitionPoset$ and in~$\partitionPoset$ are in opposite direction.

The M\"obius function of the set partitions poset~$\partitionPoset$ is given by
\[
\mu_{\partitionPoset}(\pi, \omega) = \prod_{p \in \omega} (-1)^{\card{\pi[p]}-1}(\card{\pi[p]}-1)! \ ,
\]
where~$\pi[p]$ denotes the restriction of the partition~$\pi$ to the part~$p$ of the partition~$\omega$, and $\card{\pi[p]}$ denotes its number of parts.
See for instance~\cite{Birkhoff, Rota}.
The M\"obius polynomial of the braid arrangement~$\braidArrangement$ is given by
\[
\mobPol[\braidArrangement] = \sum_{k \in [n]} x^{k-1} S(n,k) \prod_{i \in [k-1]} (y-i) ,
\]
where~$S(n,k)$ denotes the Stirling number of the second kind \OEIS{A008277}, \ie the number of set partitions of~$[n]$ into~$k$ parts.
For instance
\begin{align*}
\mobPol[{\braidArrangement[1]}] & = 1 \\
\mobPol[{\braidArrangement[2]}] & = x y - x + 1 = x (y - 1) + 1 \\
\mobPol[{\braidArrangement[3]}] & = x^2 y^2 - 3 x^2 y + 2 x^2 + 3 x y - 3 x + 1 = x^2 (y - 1) (y - 2) + 3 x (y - 1) + 1\\
\mobPol[{\braidArrangement[4]}] & = x^3 y^3 - 6 x^3 y^2 + 11 x^3 y - 6 x^3 + 6 x^2 y^2 - 18 x^2 y + 12 x^2 + 7 x y - 7 x + 1 \\
& = x^3 (y - 1) (y - 2) (y - 3) + 6 x^2 (y - 1) (y - 2) + 7 x (y - 1) + 1.
\end{align*}
In particular, the characteristic polynomial of the braid arrangement~$\braidArrangement$ is given by
\[
\charPol[\braidArrangement] = (y-1) (y-2) \dots (y-n-1).
\]
Working in~$\R^n$ rather than in~$\HH$ would lead to an additional~$y$ factor in this formula, which might be more familiar to the reader.
See \cref{rem:essential}.

Finally, we will consider the evaluation of the M\"obius polynomial~$\mobPol[\braidArrangement]$ at~$y = 0$:
\[
\weirdPol[n] \eqdef \mobPol[\braidArrangement][x][0] = \sum_{k \in [n]} (-1)^{k-1} \, (k-1)! \, S(n,k) \, x^{k-1}.
\] 
The coefficients of this polynomial are given by the sequence \OEIS{A028246}.
We just observe here that it is connected to the $\b{f}$-polynomial of~$\braidArrangement$.

\begin{lemma}
We have~$\weirdPol[n] = (1-x) \, \fPol[\braidArrangement]$.
\end{lemma}

\begin{proof}
This lemma is equivalent to the equality
\begin{equation*}
\sum_{k=1}^n (-1)^{k-1} (k-1)! \, S(n,k) \, x^{k-1} = (1-x) \sum_{k=1}^{n-1} (-1)^{k-1} \, k! \, S(n-1,k) \, x^{k-1}.
\end{equation*}
Distributing $(1-x)$ in the right hand side gives:
\begin{gather*}
(1-x) \sum_{k=1}^{n-1} (-1)^{k-1} \, k! \, S(n-1,k) \, x^{k-1} \\
= \sum_{k=1}^{n-1} k! \, S(n-1,k) \, (-x)^{k-1} + \sum_{k=1}^{n-1} k! \, S(n-1,k) \, (-x)^{k} \\
= \sum_{k=1}^{n-1} k! \, S(n-1,k) \, (-x)^{k-1} + \sum_{k=2}^{n} (k-1)! \, S(n-1,k-1) \, (-x)^{k-1} + (n-1)! \, S(n-1, n-1) \, (-x)^{n-1} \\
= S(n-1,1) \, (-x)^0 + \sum_{k=2}^{n-1} (k-1)! \, \big( S(n-1,k-1) + k \, S(n-1,k) \big) \, (-x)^{k-1}.
\end{gather*}
The result thus follows from the inductive formula on Stirling numbers of the second kind
\[
S(n+1,k) = k \, S(n,k) + S(n,k-1)
\]
for $0<k<n$.
\end{proof}

%%%%%%%%%%%%%%%

\subsection{The $(\ell,n)$-braid arrrangement}
\label{subsec:multiBraidArrangement}

\enlargethispage{.7cm}
We now focus on the following specific hyperplane arrangements, illustrated in \cref{fig:multiBraidArrangements}.
We still denote by~$\HH$ the hyperplane of~$\R^n$ defined by~$\sum_{s \in [n]} x_s = 0$.

\begin{figure}[t]
	\centerline{
	\begin{tabular}{c@{\hspace{.7cm}}c@{\hspace{.7cm}}c@{\hspace{.7cm}}c}
		\includegraphics[scale=.4]{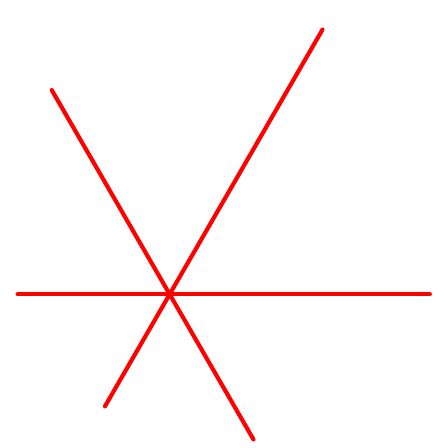}
		&
		\includegraphics[scale=.4]{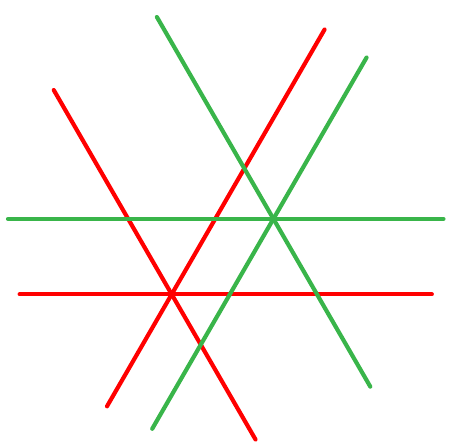}
		&
		\includegraphics[scale=.4]{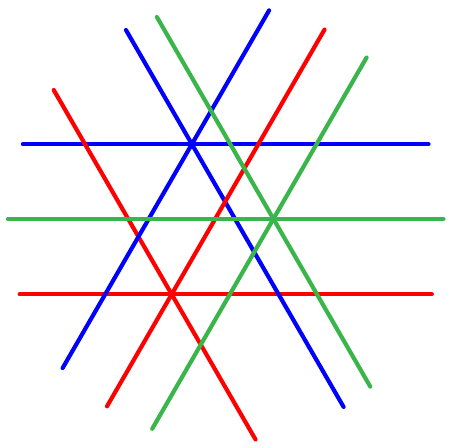}
		&
		\includegraphics[scale=.4]{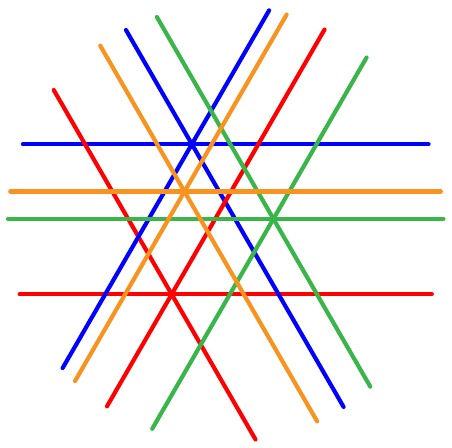}
		\\
		$\ell = 1$ & $\ell = 2$ & $\ell = 3$ & $\ell = 4$
	\end{tabular}
	}
	\caption{The $(\ell,3)$-braid arrangements for~$\ell \in [4]$.}
	\label{fig:multiBraidArrangements}
\end{figure}

\begin{definition}
\label{def:multiBraidArrangementPrecise}
For any integers~$\ell,n \geq 1$, and any matrix~$\b{a} \eqdef (a_{i,j}) \in M_{\ell,n-1}(\R)$, the \defn{$\b{a}$-braid arrangement}~$\multiBraidArrangement(\b{a})$ is the arrangement of hyperplanes~$\set{\b{x} \in \HH}{x_s - x_t = A_{i,s,t}}$ for all~${1 \le s < t \le n}$ and~$i \in [\ell]$, where $A_{i,s,t} \eqdef \smash{\sum_{s \le j < t} a_{i,j}}$.
\end{definition}

In other words, the $\b{a}$-braid arrangement~$\multiBraidArrangement(\b{a})$ is the union of~$\ell$ copies of the braid arrangement~$\braidArrangement$ translated according to the matrix~$\b{a}$.
Of course, the $\b{a}$-braid arrangement~$\multiBraidArrangement(\b{a})$ highly depends on~$\b{a}$.
In this paper, we are interested in the case where~$\b{a}$ is generic in the following sense.

\begin{definition}
A matrix~$\b{a} \eqdef (a_{i,j}) \in M_{\ell,n-1}(\R)$ is \defn{generic} if for any~$i_1, \dots, i_k \in [\ell]$ and distinct $r_1, \dots, r_k \in [n]$, the equality~$\sum_{j \in [k]} A_{i_j, r_{j-1}, r_j} = 0$ implies~$i_1 = \dots = i_k$ (with the notation~$A_{i,s,t} \eqdef \smash{\sum_{s \le j < t} a_{i,j}}$ and the convention~$r_0 = r_k$).
\end{definition}

We will see that many combinatorial aspects of~$\multiBraidArrangement(\b{a})$, in particular its flat poset and thus its M\"obius, $f$- and $b$-polynomials, are in fact independent of the matrix~$\b{a}$ as long as it is generic.
We therefore consider the following definition.

\begin{definition}
\label{def:multiBraidArrangement}
The \defn{$(\ell,n)$-braid arrangement}~$\multiBraidArrangement$ is the arrangement in~$\HH$ obtained as the union of $\ell$ generically translated copies of the braid arrangement~$\braidArrangement$ (that is, any $\b{a}$-braid arrangement for some generic matrix~$\b{a} \in M_{\ell,n-1}(\R)$).
\end{definition}

%\begin{remark}
%\label{rem:multiBraidArrangement}
%In practice, consider the hyperplanes~$\set{\b{x} \in \HH}{x_s - x_t = A_{i,s,t}}$ for all~$1 \le s < t \le n$ and~$i \in [\ell]$, where $A_{i,s,t} \eqdef \smash{\sum_{s \le j < t} a_{i,j}}$ for an arbitrary generic matrix~$(a_{i,j}) \in M_{\ell,n-1}(\R)$.
%\end{remark}

The objective of \cref{part:multiBraidArrangements} is to explore the combinatorics of these multiple braid arrangements.
We have split our presentation into two sections:
\begin{itemize}
\item In \cref{sec:flatPoset}, we describe the flat poset~$\flatPoset[\multiBraidArrangement]$ of the $(\ell,n)$-braid arrangement~$\multiBraidArrangement$ in terms of $(\ell,n)$-partition forests (\cref{subsec:partitionForests}) and labeled $(\ell,n)$-rainbow forests (\cref{subsec:rainbowForests}), which enables us to derive its M\"obius, $f$- and $b$- polynomials (\cref{subsec:MobiusPolynomialMultiBraidArrangement}), from which we extract interesting formulas for the number of vertices (\cref{subsec:verticesMultiBraidArrangement}) and regions (\cref{subsec:regionsMultiBraidArrangement}). Note that all these results are independent of the translation matrix.
\item In \cref{sec:facePoset}, we describe the face poset~$\facePoset[\multiBraidArrangement(\b{a})]$ of the $\b{a}$-braid  arrangement~$\multiBraidArrangement(\b{a})$ in terms of ordered $(\ell,n)$-partition forests (\cref{subsec:orderedPartitionForests}). In contrast to the flat poset, this description of the face poset depends on the translation matrix~$\b{a}$. For a given choice of~$\b{a}$, we describe in particular the ordered $(\ell,n)$-partitions forests with a given underlying (unordered) $(\ell,n)$-partition forest (\cref{subsec:PFtoOPF}). We then give a criterion to decide whether a given ordered $(\ell,n)$-partition forest corresponds to a face of~$\multiBraidArrangement(\b{a})$ (\cref{subsec:criterionOPF}).
\end{itemize}

\begin{remark}
Note that each hyperplane of the $(\ell,n)$-braid arrangement~$\multiBraidArrangement$ is orthogonal to a root~$\b{e}_i-\b{e}_j$ of the type~$A$ root system.
Many such arrangements have been studied previously, for instance, the \defn{Shi arrangement}~\cite{Shi1, Shi2}, the \defn{Catalan arrangement}~\cite[Sect.~7]{PostnikovStanley}, the \defn{Linial arrangement}~\cite[Sect.~8]{PostnikovStanley}, the \defn{generic arrangement} of~\cite[Sect.~5]{PostnikovStanley}, or the \defn{discriminantal arrangements} of~\cite{ManinSchechtman,BayerBrandt}.
We refer to the work of A.~Postnikov and R.~Stanley~\cite{PostnikovStanley} and of O.~Bernardi~\cite{Bernardi} for much more references.
However, in all these examples, either the copies of the braid arrangement are perturbed, or they are translated non-generically.
We have not been able to find the $(\ell,n)$-braid arrangement~$\multiBraidArrangement$ properly treated in the literature.
\end{remark}

\begin{remark}
Part of our discussion on the $(\ell,n)$-braid arrangement~$\multiBraidArrangement$ could actually be developed for a hyperplane arrangement~$\arrangement^\ell$ obtained as the union of~$\ell$ generically translated copies of an arbitrary linear hyperplane arrangement~$\arrangement$.
Similarly to \cref{prop:flatPosetMultiBraidArrangement}, the flat poset~$\flatPoset[\arrangement^\ell]$ is isomorphic to the lower set of the $\ell$\ordinal{} Cartesian power of the flat poset~$\flatPoset$ induced by the $\ell$-tuples whose meet in the flat poset~$\flatPoset$ is the bottom element~$\b{0}$ (these are sometimes called strong antichains) and which are minimal for this property.
Similar to \cref{thm:MobiusPolynomialMultiBraidArrangement}, this yields a general formula for the M\"obius polynomial of~$\arrangement^\ell$ in terms of the M\"obius function of the flat poset~$\flatPoset$.
Here, we additionally benefit from the nice properties of the M\"obius polynomial of the braid arrangement~$\braidArrangement$ to obtain appealing formulas for the vertices, regions and bounded regions of the $(\ell,n)$-braid arrangement~$\multiBraidArrangement$ (see \cref{thm:verticesMultiBraidArrangement,thm:verticesRefinedMultiBraidArrangement,thm:characteristicPolynomialMultiBraidArrangement,thm:regionsMultiBraidArrangement}).
We have therefore decided to restrict our attention to the $(\ell,n)$-braid arrangement~$\multiBraidArrangement$.
\end{remark}

%%%%%%%%%%%%%%%%%%%%%%%%%%%%%%%%%%%%%%

\section{Flat poset and enumeration of~$\multiBraidArrangement$}
\label{sec:flatPoset}

In this section, we describe the flat poset of the $(\ell,n)$-braid arrangement~$\multiBraidArrangement$ in terms of \mbox{$(\ell,n)$-partition} forests and derive explicit formulas for its $f$-vector.
Remarkably, the flat poset (and thus the M\"obius, $f$- and $b$- polynomials) of~$\multiBraidArrangement$ is independent of the translation vectors as long as they are generic.

%%%%%%%%%%%%%%%

\subsection{Partition forests}
\label{subsec:partitionForests}

We first introduce the main characters of this section, which will describe the combinatorics of the flat poset of the $(\ell,n)$-braid arrangement~$\multiBraidArrangement$ of \cref{def:multiBraidArrangement}.

\begin{definition}
\label{def:intersectionHypergraph}
The \defn{intersection hypergraph} of a $\ell$-tuple~$\b{F} \eqdef (F_1, \dots, F_\ell)$ of set partitions of~$[n]$ is the $\ell$-regular $\ell$-partite hypergraph on all parts of all the partitions~$F_i$ for~${i \in [\ell]}$, with a hyperedge connecting the parts containing~$j$  for each~$j \in [n]$.
\end{definition}

\begin{definition}
\label{def:partitionForests}
An \defn{$(\ell,n)$-partition forest} (\resp \defn{$(\ell,n)$-partition tree}) is a $\ell$-tuple~$\b{F} \eqdef (F_1, \dots, F_\ell)$ of set partitions of~$[n]$ whose intersection hypergraph is a hyperforest (\resp hypertree).
See \cref{fig:forests}.
The \defn{dimension} of~$\b{F}$ is~$\smash{{\dim(\b{F}) \eqdef n - 1 - \ell n + \sum_{i \in [\ell]} \card{F_i}}}$.
The \defn{$(\ell,n)$-partition forest poset} is the poset~$\forestPoset$ on $(\ell,n)$-partition forests ordered by componentwise refinement.
\begin{figure}[b]
	\centerline{\includegraphics[scale=.9]{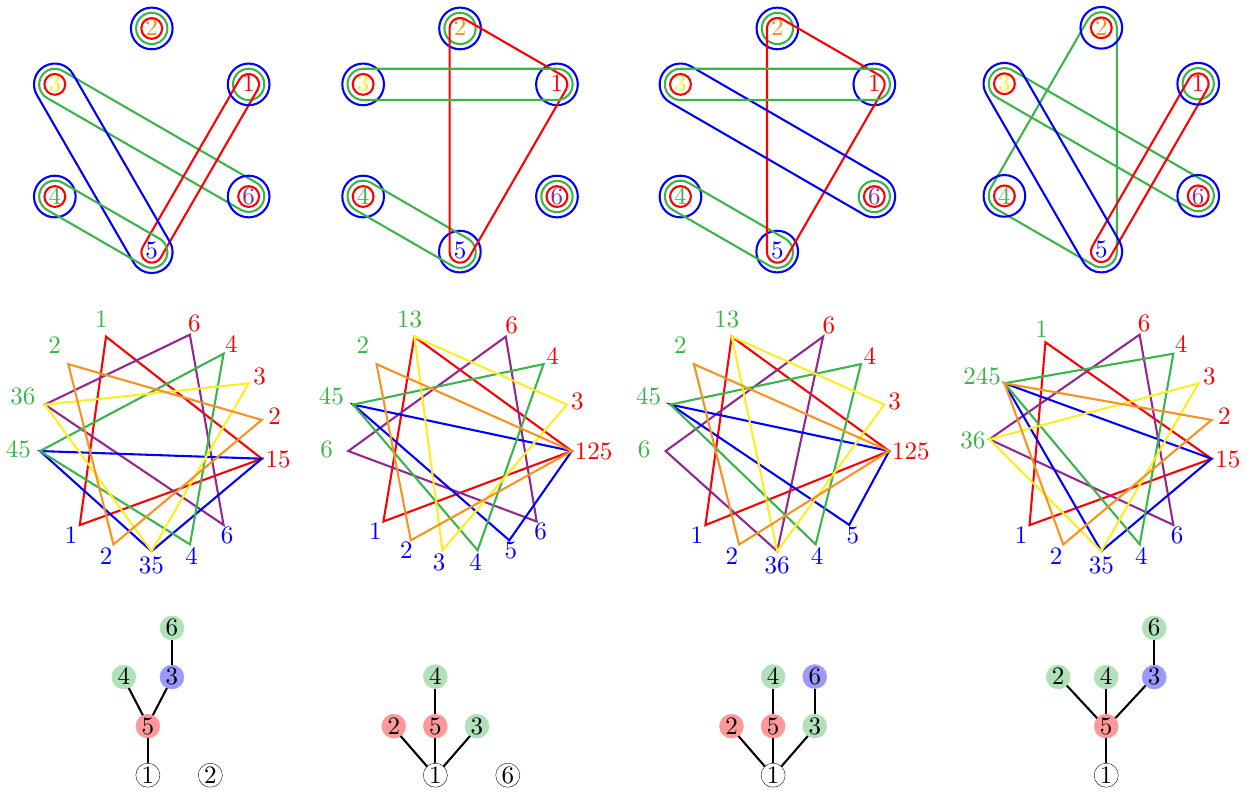}}
	\caption{Some $(3,6)$-partition forests (top) with their intersection hypergraphs (middle) and the corresponding labeled $(3,6)$-rainbow forests (bottom). The last two are trees. The order of the colors in the bottom pictures is red, green, blue.}
	\label{fig:forests}
\end{figure}
\end{definition}

In other words, $\forestPoset$ is the lower set of the $\ell$\ordinal{} Cartesian power of the partition poset~$\partitionPoset$ induced by $(\ell,n)$-partition forests.
Note that the maximal elements of~$\forestPoset$ are the $(\ell, n)$-partition trees.

The following statement is illustrated in \cref{fig:intersectionPosetMultiBraidArrangement32}.
\begin{figure}
	\centerline{\includegraphics[scale=.9]{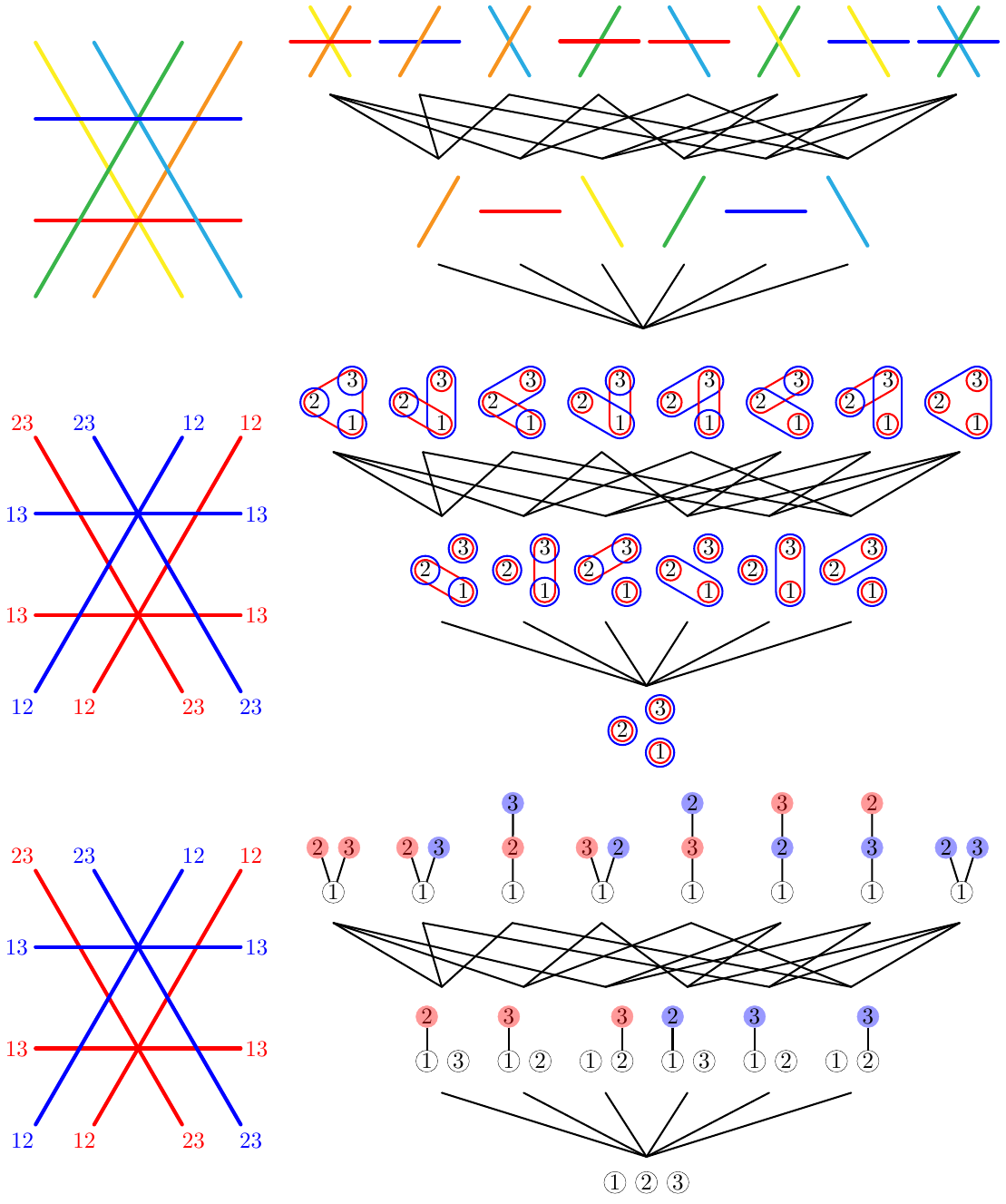}}
	\caption{The $(2,3)$-braid arrangement $\multiBraidArrangement[3][2]$ (left), and its flat poset (right), where flats are represented as intersections of hyperplanes (top), as $(2,3)$-partitions forests (middle), and as labeled $(2,3)$-rainbow forests (bottom).}
	\label{fig:intersectionPosetMultiBraidArrangement32}
\end{figure}

\begin{proposition}
\label{prop:flatPosetMultiBraidArrangement}
The flat poset~$\flatPoset[\multiBraidArrangement]$ of the $(\ell,n)$-braid arrangement~$\multiBraidArrangement$ is isomorphic to the $(\ell,n)$-partition forest poset.
\end{proposition}

\begin{proof}
Consider that~$\multiBraidArrangement$ is the $\b{a}$-braid arrangement~$\multiBraidArrangement(\b{a})$ for some generic matrix~$\b{a}$.
%Consider the hyperplanes~$\set{\b{x} \in \HH}{x_s - x_t = A_{i,s,t}}$ described in \cref{def:multiBraidArrangementPrecise}.
In view of our discussion in \cref{subsec:braidArrangement}, observe that, for each~${i \in [\ell]}$, each set partition~$\pi$ of~$[n]$ corresponds to a $(\card{\pi})$-dimensional flat
\[
\Psi_i(\pi) \eqdef \set{\b{x} \in \HH}{x_s - x_t = A_{i,s,t} \text{ for all $s,t$ in the same part of $\pi$} }
\]
of the $i$\ordinal{} copy of the braid arrangement~$\braidArrangement$.
The flats of the $(\ell,n)$-braid arrangement~$\multiBraidArrangement$ are thus all of the form
\[
\Psi(\b{F}) \eqdef \bigcap_{i \in [\ell]} \Psi_i(F_i)
\]
for certain $\ell$-tuples~$\b{F} \eqdef (F_1, \dots, F_\ell)$ of set partitions of~$[n]$.
Since the matrix~$\b{a}$ is generic, $\Psi(\b{F})$ is non-empty if and only if the intersection hypergraph of~$\b{F}$ is acyclic.
Moreover, $\Psi(\b{F})$ is included in~$\Psi(\b{G})$ if and only if~$\b{F}$ refines~$\b{G}$ componentwise.
Hence, the flat poset of~$\multiBraidArrangement$ is isomorphic to the $(\ell,n)$-partition forest poset.
Finally, notice that the codimension of the flat~$\Psi(\b{F})$ is the sum of the codimensions of the flats~$\Psi_i(F_i)$ for~$i \in [\ell]$, so that~$\dim(\b{F}) \eqdef n - 1 - \ell n + \sum_{i \in [\ell]} \card{F_i} $ is indeed the dimension of the flat~$\Psi(\b{F})$.
\end{proof}

%%%%%%%%%%%%%%%

\subsection{M\"obius polynomial}
\label{subsec:MobiusPolynomialMultiBraidArrangement}

We now derive from \cref{def:MobiusPolynomial,prop:flatPosetMultiBraidArrangement} the M\"obius polynomial of the $(\ell,n)$-braid arrangement~$\multiBraidArrangement$.

\begin{theorem}
\label{thm:MobiusPolynomialMultiBraidArrangement}
The M\"obius polynomial of the $(\ell,n)$-braid arrangement~$\multiBraidArrangement$ is given by
\[
\mobPol[\multiBraidArrangement] = x^{n-1-\ell n} y^{n-1-\ell n} \sum_{\b{F} \le \b{G}} \prod_{i \in [\ell]} x^{\card{F_i}} y^{\card{G_i}} \prod_{p \in G_i} (-1)^{\card{F_i[p]}-1} (\card{F_i[p]}-1)! \; ,
\]
where~$\b{F} \le \b{G}$ ranges over all intervals of the $(\ell,n)$-partition forest poset~$\forestPoset$, and~$F_i[p]$ denotes the restriction of the partition~$F_i$ to the part~$p$ of~$G_i$.
\end{theorem}

\begin{proof}
Observe that for~$\b{F} \eqdef (F_1, \dots, F_\ell)$ and~$\b{G} \eqdef (G_1, \dots, G_\ell)$ in~$\forestPoset$, we have
\[
[\b{F}, \b{G}] = \prod_{i \in [\ell]} [F_i, G_i] \simeq \prod_{i \in [\ell]} \prod_{p \in G_i} \partitionPoset[{\card{F_i[p]}}].
\]
Recall that the M\"obius function is multiplicative:
\(
\mu_{P \times Q} \big( (p,q), (p’,q’) \big) = \mu_P(p,p’) \cdot \mu_Q(q,q’),
\)
for all~$p, p' \in P$ and~$q, q' \in Q$.
Hence, we obtain that
\[
\mu_{\forestPoset}(\b{F}, \b{G}) = \prod_{i \in [\ell]} \prod_{p \in G_i} (-1)^{\card{F_i[p]}-1} (\card{F_i[p]}-1)! .
\]
Hence, we derive from \cref{def:MobiusPolynomial,prop:flatPosetMultiBraidArrangement} that
\begin{align*}
\mobPol[\multiBraidArrangement] 
& = \sum_{\b{F} \le \b{G}} \mu_{\forestPoset}(\b{F}, \b{G}) \, x^{\dim(\b{F})} \, y^{\dim(\b{G})} \\
& = x^{n-1-\ell n} y^{n-1-\ell n} \sum_{\b{F} \le \b{G}} \prod_{i \in [\ell]} x^{\card{F_i}} y^{\card{G_i}} \prod_{p \in G_i} (-1)^{\card{F_i[p]}-1} (\card{F_i[p]}-1)! .
\qedhere
\end{align*}
\end{proof}

By using the polynomial
\[
\weirdPol[n] \eqdef \mobPol[\braidArrangement][x][0] = \sum_{k \in [n]} (-1)^{k-1} \, (k-1)! \, S(n,k) \, x^{k-1}
\]
introduced at the end of \cref{subsec:braidArrangement}, the M\"obius polynomial~$\mobPol[\multiBraidArrangement]$ can also be expressed as follows.

\pagebreak
\begin{proposition} 
\label{prop:alternativeFormulaMobiusPolynomialMultiBraidArrangement}
The M\"obius polynomial of the $(\ell,n)$-braid arrangement~$\multiBraidArrangement$ is given by
\[
\mobPol[\multiBraidArrangement] = x^{(n-1)(1-\ell)} \sum_{G \in \forestPoset} y^{n-1-\ell n+\sum_{i \in [\ell]} \card{G_i}}  \prod_{i \in [\ell]} \weirdPol[\card{G_i}].
\]
\end{proposition}

\begin{proof}
As already mentioned, the $(\ell,n)$-partition forest poset~$\forestPoset$ is a lower set of the $\ell$\ordinal{} Cartesian power of the partition poset~$\partitionPoset$.
In other words, given a $(\ell,n)$-partition forest $\b{G} \eqdef (G_1, \dots, G_\ell)$, any $\ell$-tuple~$\b{F} \eqdef (F_1, \dots, F_\ell)$ of partitions satisfying~$F_i \le_{\partitionPoset[n]} G_i$ for all~$i \in [\ell]$ is a $(\ell,n)$-partition forest.
Hence, we obtain from \cref{def:MobiusPolynomial,prop:flatPosetMultiBraidArrangement} that
\begin{align*}
\mobPol[\multiBraidArrangement]
& = \sum_{G \in \forestPoset} y^{n-\ell n - 1 + \sum_{i \in [\ell]} \card{G_i}} \prod_{i \in [\ell]} \sum_{F_i \leq_{\partitionPoset[n]} G_i } \mu_{\partitionPoset[n]}(F_i,G_i) \, x^{n-\ell n - 1 + \sum_{i \in [\ell]} \card{F_i}}, \\
& = \sum_{G \in \forestPoset} y^{n-\ell n - 1 + \sum_{i \in [\ell]} \card{G_i}} x^{(n-1)(1-\ell)} \prod_{i \in [\ell]} \sum_{\pi_i \in \partitionPoset[\card{G_i}]}  \mu_{\partitionPoset[\card{G_i}]}(\pi_i,\hat{1}) \, x^{\card{\pi_i}-1}, 
\end{align*}
where $\hat{1}$ denotes the maximal element in $\partitionPoset[\card{G_i}]$ and $\pi_i$ is obtained from $F_i$ by merging elements in the same part of $G_i$.
The result follows since~$\weirdPol[\card{G_i}] = \sum_{\pi_i \in \partitionPoset[\card{G_i}]}  \mu_{\partitionPoset[\card{G_i}]}(\pi_i,\hat{1}) \, x^{\card{\pi_i}-1}$.
%Hence, we obtain
%\[
%\mobPol[\multiBraidArrangement] = \sum_{G \in \forestPoset} y^{n-\ell n - 1 + \sum_{i \in [\ell]} \card{G_i}} x^{(n-1)(1-\ell)} \prod_{i \in [\ell]} \weirdPol[\card{G_i}].
%\qedhere
%\]
\end{proof}

From \cref{thm:Zaslavsky,thm:MobiusPolynomialMultiBraidArrangement}, we thus obtain the face numbers and bounded face numbers of~$\multiBraidArrangement$, whose first few values are gathered in \cref{table:fvectorMultiBraidArrangements}.

\begin{table}[t]
	\centerline{\scalebox{.8}{
		\begin{tabular}{c@{\hspace{.7cm}}c@{\hspace{.7cm}}c@{\hspace{.7cm}}c}
			$\ell = 1$ & $\ell = 2$ & $\ell = 3$ & $\ell = 4$
			\\[.2cm]
			\begin{tabular}[t]{c|cccc|c}
				$n \backslash k$ & $0$ & $1$ & $2$ & $3$ & $\Sigma$ \\
				\hline
				$1$ & $1$ &&&& $1$ \\
				$2$ & $2$ & $1$ &&& $3$ \\
				$3$ & $6$ & $6$ & $1$ && $13$ \\
				$4$ & $24$ & $36$ & $14$ & $1$ & $75$
			\end{tabular}
			&
			\begin{tabular}[t]{c|cccc|c}
				$n \backslash k$ & $0$ & $1$ & $2$ & $3$ & $\Sigma$ \\
				\hline
				$1$ & $1$ &&&& $1$ \\
				$2$ & $3$ & $2$ &&& $5$ \\
				$3$ & $17$ & $24$ & $8$ && $49$ \\
				$4$ & $149$ & $324$ & $226$ & $50$ & $749$
			\end{tabular}
			&
			\begin{tabular}[t]{c|cccc|c}
				$n \backslash k$ & $0$ & $1$ & $2$ & $3$ & $\Sigma$ \\
				\hline
				$1$ & $1$ &&&& $1$ \\
				$2$ & $4$ & $3$ &&& $7$ \\
				$3$ & $34$ & $54$ & $21$ && $109$ \\
				$4$ & $472$ & $1152$ & $924$ & $243$ & $2791$
			\end{tabular}
			&
			\begin{tabular}[t]{c|cccc|c}
				$n \backslash k$ & $0$ & $1$ & $2$ & $3$ & $\Sigma$ \\
				\hline
				$1$ & $1$ &&&& $1$ \\
				$2$ & $5$ & $4$ &&& $9$ \\
				$3$ & $57$ & $96$ & $40$ && $193$ \\
				$4$ & $1089$ & $2808$ & $2396$ & $676$ & $6969$
			\end{tabular}
			\\[2.3cm]
			\begin{tabular}[t]{c|cccc|c}
				$n \backslash k$ & $0$ & $1$ & $2$ & $3$ & $\Sigma$ \\
				\hline
				$1$ & $1$ &&&& $1$ \\
				$2$ & $0$ & $1$ &&& $1$ \\
				$3$ & $0$ & $0$ & $1$ && $1$ \\
				$4$ & $0$ & $0$ & $0$ & $1$ & $1$
			\end{tabular}
			&
			\begin{tabular}[t]{c|cccc|c}
				$n \backslash k$ & $0$ & $1$ & $2$ & $3$ & $\Sigma$ \\
				\hline
				$1$ & $1$ &&&& $1$ \\
				$2$ & $1$ & $2$ &&& $3$ \\
				$3$ & $5$ & $12$ & $8$ && $25$ \\
				$4$ & $43$ & $132$ & $138$ & $50$ & $363$
			\end{tabular}
			&
			\begin{tabular}[t]{c|cccc|c}
				$n \backslash k$ & $0$ & $1$ & $2$ & $3$ & $\Sigma$ \\
				\hline
				$1$ & $1$ &&&& $1$ \\
				$2$ & $2$ & $3$ &&& $5$ \\
				$3$ & $16$ & $36$ & $21$ && $73$ \\
				$4$ & $224$ & $684$ & $702$ & $243$ & $1853$
			\end{tabular}
			&
			\begin{tabular}[t]{c|cccc|c}
				$n \backslash k$ & $0$ & $1$ & $2$ & $3$ & $\Sigma$ \\
				\hline
				$1$ & $1$ &&&& $1$ \\
				$2$ & $3$ & $4$ &&& $7$ \\
				$3$ & $33$ & $72$ & $40$ && $145$ \\
				$4$ & $639$ & $1944$ & $1980$ & $676$ & $5239$
			\end{tabular}
		\end{tabular}
	}}
%	\vspace{.3cm}
	\caption{The face numbers (top) and the bounded face numbers (bottom) of the $(\ell,n)$-braid arrangements for~$\ell, n \in [4]$.}
	\label{table:fvectorMultiBraidArrangements}
\end{table}

\begin{corollary}
\label{coro:fbvectorsMultiBraidArrangement}
The $f$- and $b$-polynomials of the $(\ell,n)$-braid arrangement~$\multiBraidArrangement$ are given by
\begin{align*}
\fPol[\multiBraidArrangement] & = x^{n-1-\ell n}\sum_{\b{F} \le \b{G}} \prod_{i \in [\ell]} x^{\card{F_i}} \prod_{p \in G_i} (\card{F_i[p]}-1)!\\
\text{and}\qquad
\bPol[\multiBraidArrangement] & = (-1)^\ell x^{n-1-\ell n} \sum_{\b{F} \le \b{G}} \prod_{i \in [\ell]} x^{\card{F_i}} \prod_{p \in G_i} -(\card{F_i[p]}-1)! ,
\end{align*}
where~$\b{F} \le \b{G}$ ranges over all intervals of the $(\ell,n)$-partition forest poset~$\forestPoset$, and~$F_i[p]$ denotes the restriction of the partition~$F_i$ to the part~$p$ of~$G_i$.
\end{corollary}

\begin{example}
For~$n = 1$, we have
\[
\mobPol[{\multiBraidArrangement[1][\ell]}] = \fPol[{\multiBraidArrangement[1][\ell]}] = \bPol[{\multiBraidArrangement[1][\ell]}] = 1.
\]
For~$n = 2$, we have
\[
\mobPol[{\multiBraidArrangement[2][\ell]}] = xy-\ell x+\ell,
\quad
\fPol[{\multiBraidArrangement[2][\ell]}] = (\ell+1)x+\ell
\quad\text{and}\quad
\bPol[{\multiBraidArrangement[2][\ell]}] = (\ell-1)x+\ell.
\]
The case~$n = 3$ is already more interesting.
Consider the set partitions~$P \eqdef \big\{ \{1\}, \{2\}, \{3\} \big\}$, $Q_i \eqdef \big\{ \{i\}, [3] \ssm \{i\} \big\}$ for~$i \in [3]$, and~$R \eqdef \big\{ [3] \big\}$.
Observe that the $(\ell,3)$-partition forests are all of the form
\begin{gather*}
\b{F} \eqdef P^\ell,
\quad
\b{G}_i^p \eqdef P^p Q_i P^{\ell-p-1}, % \text{ for } p \le \ell-1 \text{ and } i \in [3],
\quad
\b{H}_{i,j}^{p,q} \eqdef P^p Q_i P^{\ell-p-q-2} Q_j P^q \;\text{($i \ne j$)} % \text{ for } p + q \le \ell-2 \text{ and } i \ne j \in [3]
\quad\text{or}\quad
\b{K}^p \eqdef P^p R P^{\ell-p-1}. % \text{ for } p \le \ell-1.
\end{gather*}
(where we write a tuple of partitions of~$[3]$ as a word on~$\{P, Q_1, Q_2, Q_3, R\}$). %, and $p$ and $q$ are such that the total length is~$\ell$).
%\begin{gather*}
%\b{F} \eqdef (\underbrace{P, \dots, P}_\ell), \\
%\b{G}_i^p \eqdef (\underbrace{P, \dots, P}_p, Q_i, \underbrace{P, \dots, P}_{\ell-p-1}) \text{ for } p \le \ell-1 \text{ and } i \in [3],
%\\
%\b{H}_{i,j}^{p,q} \eqdef (\underbrace{P, \dots, P}_p, Q_i, \underbrace{P, \dots, P}_{\ell-p-q-2}, Q_j, \underbrace{P, \dots, P}_q) \text{ for } p + q \le \ell-2 \text{ and } i \ne j \in [3],
%\\
%\text{or}\quad
%\b{K}^p \eqdef (\underbrace{P, \dots, P}_p, R, \underbrace{P, \dots, P}_{\ell-p-1}) \text{ for } p \le \ell-1.
%\end{gather*}
Moreover, the cover relations in the $(\ell,3)$-partition forest poset are precisely the relations
\[
\b{F} \le \b{G}_i^p \hspace{-.4cm} \begin{array}{l} \rotatebox[origin=c]{45}{$\le$} \raisebox{.2cm}{$\b{H}_{i,j}^{p,q}$} \\[.1cm] \quad \le \b{K}^p \\[.1cm] \rotatebox[origin=c]{-45}{$\le$} \raisebox{-.2cm}{$\b{H}_{j,i}^{\ell-q-1, \ell-p-1}$} \end{array}
\]
for~$i \ne j$ and~$p, q$ such that~$p + q \le \ell-2$.
Hence, we have
\begin{align*}
\mobPol[{\multiBraidArrangement[3][\ell]}] & = x^2 y^2 - 3 \ell x^2 y + \ell (3 \ell - 1) x^2 + 3 \ell x y - 3 \ell (2 \ell - 1) x + \ell (3 \ell - 2) , \\
\fPol[{\multiBraidArrangement[3][\ell]}] & = (3 \ell^2 + 2 \ell + 1) x^2 + 6 \ell^2 x + \ell (3 \ell - 2), \\
\text{and}\qquad
\bPol[{\multiBraidArrangement[3][\ell]}] & = (3 \ell^2 - 4 \ell + 1) x^2 + 6 \ell (\ell - 1) x + \ell (3 \ell - 2).
\end{align*}
Observe that~$3 \ell^2 + 2 \ell + 1$ is~\OEIS{A056109}, that~$\ell (3 \ell - 2)$ is~\OEIS{A000567}, and that ${3 \ell^2 - 4 \ell + 1}$ is~\OEIS{A045944}.
%\vincenti{There is a weird connection between the first and the last. Namely, $3 \ell^2 - 4 \ell + 1 = 3 (\ell - 1)^2 + 2 (\ell - 1)$. Is there a bijective explanation on the arrangements?}
\end{example}

%%%%%%%%%%%%%%%

\subsection{Rainbow forests}
\label{subsec:rainbowForests}

In order to obtain more explicit formulas for the number of vertices and regions of the $(\ell,n)$-braid arrangement~$\multiBraidArrangement$ in \cref{subsec:verticesMultiBraidArrangement,subsec:regionsMultiBraidArrangement}, we now introduce another combinatorial model for $(\ell,n)$-partition forests which is more adapted to their enumeration.

\begin{definition}
\label{def:rainbowForest}
An \defn{$\ell$-rainbow coloring} of a rooted plane forest~$F$ is an assignment of colors of~$[\ell]$ to the non-root nodes of~$F$ such that
\begin{enumerate}[(i)]
\item there is no monochromatic edge,
\item the colors of siblings are increasing from left to right.
\end{enumerate}
We denote by~$\|F\|$ the number of nodes of~$F$ and by~$\card{F}$ the number of trees of the forest~$F$ (\ie its number of connected components).
An \defn{$(\ell,n)$-rainbow forest} (\resp \defn{tree}) is a \mbox{$\ell$-rainbow} colored forest (\resp tree) with $\|F\| = n$ nodes.
We denote by~$\rainbowForests$ (\resp $\rainbowTrees$) the set of $(\ell,n)$-rainbow forests (\resp trees), and set~$\rainbowForests[][\ell] \eqdef \bigsqcup_n \rainbowForests$ (\resp $\rainbowTrees[][\ell] \eqdef \bigsqcup_n \rainbowTrees$).
\end{definition}

For instance, we have listed the $14$ $(2,4)$-rainbow trees in \cref{fig:rainbowTrees}\,(top).
This figure actually illustrates the following statement.

\begin{lemma}
\label{lem:FussCatalan}
The $(\ell,m)$-rainbow trees are counted by the \defn{Fuss-Catalan number}
\[
\card{\rainbowTrees[m][\ell]} = F_{\ell,m} \eqdef \frac{1}{(\ell-1)m+1} \binom{\ell m}{m} \qquad \text{\OEIS{A062993}}.
\]
\end{lemma}

\begin{proof}
We can transform a $\ell$-rainbow tree~$R$ to an $\ell$-ary tree~$T$ as illustrated in \cref{fig:rainbowTrees}.
Namely, the parent of a node~$N$ in~$T$ is the previous sibling colored as~$N$ in~$R$ if it exists, and the parent of~$N$ in~$R$ otherwise.
This classical map is a bijection from $\ell$-rainbow trees to $\ell$-ary trees, which are counted by the Fuss-Catalan numbers~\cite{Klarner, HiltonPedersen}.
\begin{figure}
	\centerline{\includegraphics[scale=.7]{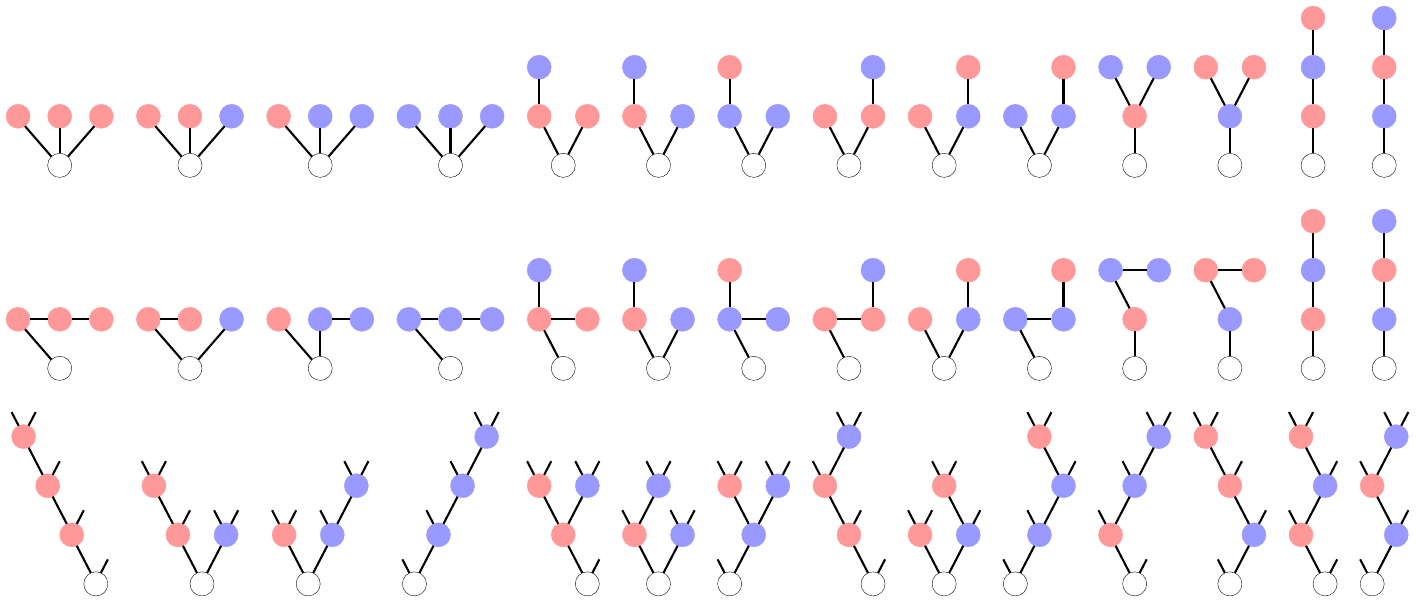}}
	\caption{The $14$ $(2,4)$-rainbow trees (top) and $14$ binary trees (bottom), and the simple bijection between them (middle). The order of the colors is red, blue.}
	\label{fig:rainbowTrees}
\end{figure}
\end{proof}

\begin{remark}
\label{rem:functionalEquationFussCatalan}
Recall that the corresponding generating function~$F_\ell(z) \eqdef \sum_{m \ge 0} F_{\ell,m} \, z^m$ satisfies the functional equation
\[
F_\ell(z) = 1 + z \, F_\ell(z)^\ell.
\]
\end{remark}

\begin{table}
	\centerline{\scalebox{.8}{
		\begin{tabular}[t]{c|ccccccccc}
			$m \backslash \ell$ & $1$ & $2$ & $3$ & $4$ & $5$ & $6$ & $7$ & $8$ & $9$ \\
			\hline
			$1$ & $1$ & $1$ & $1$ & $1$ & $1$ & $1$ & $1$ & $1$ & $1$ \\
			$2$ & $1$ & $2$ & $3$ & $4$ & $5$ & $6$ & $7$ & $8$ & $9$ \\
			$3$ & $1$ & $5$ & $12$ & $22$ & $35$ & $51$ & $70$ & $92$ & $117$ \\
			$4$ & $1$ & $14$ & $55$ & $140$ & $285$ & $506$ & $819$ & $1240$ & $1785$ \\
			$5$ & $1$ & $42$ & $273$ & $969$ & $2530$ & $5481$ & $10472$ & $18278$ & $29799$ \\
			$6$ & $1$ & $132$ & $1428$ & $7084$ & $23751$ & $62832$ & $141778$ & $285384$ & $527085$ \\
			$7$ & $1$ & $429$ & $7752$ & $53820$ & $231880$ & $749398$ & $1997688$ & $4638348$ & $9706503$ \\
			$8$ & $1$ & $1430$ & $43263$ & $420732$ & $2330445$ & $9203634$ & $28989675$ & $77652024$ & $184138713$ \\
			$9$ & $1$ & $4862$ & $246675$ & $3362260$ & $23950355$ & $115607310$ & $430321633$ & $1329890705$ & $3573805950$
		\end{tabular}
	}}
%	\vspace{.3cm}
	\caption{The Fuss-Catalan numbers~$F_{\ell,m} = \frac{1}{(\ell-1)m+1} \binom{\ell m}{m}$ for~$\ell,m \in [9]$. See \OEIS{A062993}.}
\end{table}

\begin{definition}
For a $(\ell,n)$-rainbow forest~$F$, we define
\[
\omega(F) \eqdef \prod_{i \in [\ell]} \prod_{N \in F} \card{C_i(N)}! ,
\]
where~$N$ ranges over all nodes of~$F$ and~$C_i(N)$ denotes the children of~$N$ colored by~$i$.
\end{definition}

\begin{definition}
\label{def:labelingRainbowForest}
A \defn{labeling} of a $(\ell,n)$-rainbow forest~$F$ is a bijective map from the nodes of~$F$ to~$[n]$ such that
\begin{enumerate}[(i)]
\item the label of each root is minimal in its tree,
\item the labels of siblings with the same color are increasing from left to right.
\end{enumerate}
\end{definition}

\begin{lemma}
\label{lem:labelingRainbowForest}
The number~$\lambda(F)$ of labelings of a $(\ell,n)$-rainbow forest~$F$ is given by
\[
\lambda(F) = \frac{n!}{\omega(F) \prod\limits_{T \in F} \|T\|} .
\]
\end{lemma}

\begin{proof}
Out of all~$n!$ bijective maps from the nodes of~$F$ to~$[n]$, only~$1/\prod_{T \in F} \|T\|$ satisfy Condition~(i) of \cref{def:labelingRainbowForest}, and only $1/\prod_{i \in [\ell]} \prod_{N \in F} \card{C_i(N)}! = 1/\omega(F)$ satisfy Condition~(ii) of \cref{def:labelingRainbowForest}.
\end{proof}

The following statement is illustrated in \cref{fig:forests}.

\begin{proposition}
\label{prop:bijectionForests}
There is a bijection from $(\ell,n)$-partition forests to labeled $(\ell,n)$-rainbow forests, such that if the partition forest~$\b{F}$ is sent to the labeled rainbow forest~$F$, then
\[
\dim(\b{F}) = \card{F}-1
\qquad\text{and}\qquad
\mu_{\forestPoset}(\HH, \b{F}) = (-1)^{n-\card{F}} \, \omega(F).
\]
\end{proposition}

\begin{proof}
From a labeled $(\ell,n)$-rainbow forest~$F$, we construct a $(\ell,n)$-partition forest~$\b{F} \eqdef (F_1, \dots, F_\ell)$ whose $i$\ordinal{} partition~$F_i$ has a part~$\{N\} \cup C_i(N)$ for each node~$N$ of~$F$ not colored~$i$.
Condition~(i) of \cref{def:rainbowForest} ensures that each $F_i$ is indeed a partition.

Conversely, start from a $(\ell,n)$-partition forest~$\b{F} \eqdef (F_1, \dots, F_i)$.
Consider the colored clique graph~$K_{\b{F}}$ on~$[n]$ obtained by replacing each part in~$F_i$ by a clique of edges colored by~$i$.
For each~$1 < j \le n$, there is a unique shortest path in~$K_{\b{F}}$ from the vertex~$j$ to the smallest vertex in the connected component of~$j$.
Define the parent~$p$ of~$j$ to be the next vertex along this path, and color the node~$j$ by the color of the edge between~$j$ and~$p$.
This defines a labeled $(\ell,n)$-rainbow forest~$F$.

Finally, observe that
\begin{gather*}
\dim(\b{F}) = n - 1 - \ell n + \sum_{i \in [\ell]} \card{F_i} = \card{F}-1, \qquad\text{and} \\
\mu_{\forestPoset}(\HH, \b{F}) = \prod\limits_{i \in [\ell]} \prod\limits_{p \in F_i} (-1)^{\card{p}-1} (\card{p}-1) = \prod\limits_{i \in [\ell]} \prod\limits_{N \in F} (-1)^{\card{C_i(N)}} \card{C_i(N)}! = (-1)^{n-\card{F}} \, \omega(F).
\qedhere
\end{gather*}
\end{proof}

We now transport via this bijection the partial order of the flat lattice on rainbow forests.
For a node~$a$ of a forest~$F$, we denote by $\operatorname{Root}(a)$ the root of the tree of~$F$ containing~$a$.
The following statement is illustrated in \cref{fig:CoverRelRF}, choosing $c$ to be green, $a$ to be $5$ and $b$ to be $7$.

\begin{figure}
	\centerline{\includegraphics[scale=1]{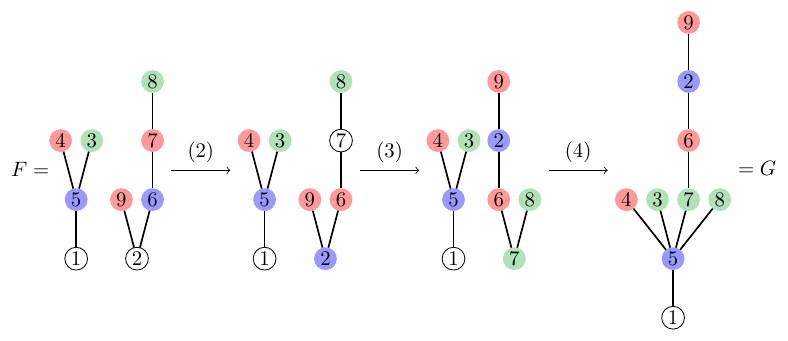}}
	\vspace{-.3cm}
	\caption{A covering relation described in \cref{prop:CoverRelRF}, choosing $c$ to be green, $a$ to be $5$ and $b$ to be $7$.}
	\label{fig:CoverRelRF}
\end{figure}

\begin{proposition}
\label{prop:CoverRelRF}
In the flat poset~$\flatPoset[\multiBraidArrangement]$ labeled by rainbow forests using~\cref{prop:flatPosetMultiBraidArrangement,prop:bijectionForests}, a rainbow forest~$F$ is covered by a rainbow forest~$G$ if and only $G$ can be obtained from $F$ by:
\begin{enumerate}
\item choosing a color $c$, and two vertices $a$ and $b$ not colored with~$c$ and with~$\operatorname{Root}(a)<\operatorname{Root}(b)$,
\item shifting the colors along the path from~$\operatorname{Root}(b)$ to~$b$, so that each node along this path is now colored by the former color of its child and~$b$ is not colored anymore,
\item rerooting at~$b$ the tree containing~$b$ at~$b$, and coloring $b$ with~$c$,
\item adding an edge~$(a,b)$ and replacing the edge~$(b,e)$ by an edge~$(a,e)$ for each child $e$ of~$b$ colored with~$c$.
\end{enumerate}
\end{proposition}

\begin{proof}
Let us first remark that the graph obtained by these operations is indeed a rainbow forest.
First, we add an edge between two distinct connected components, so that the result is indeed acyclic.
Moreover, the condition on the color of $a$ and on the deletion of edges between $b$ and vertices of color $c$ ensures that we do not add an edge between two vertices of the same color.
Note that the parent of $b$ inherits the color of $b$ which is not $c$.

Let us recall that the cover relations in the flat poset~$\flatPoset[\multiBraidArrangement]$ are given in terms of $(\ell,n)$-partition forests by choosing a partition $\pi$ of the partition tuple (which corresponds directly to choosing a color), choosing two parts $\pi_a$ and $\pi_b$ in the partition $\pi$, and merging them, without creating a loop in the intersection hypergraph.

By choosing two vertices in different connected components of the rainbow forest, we are sure that the intersection hypergraph obtained by adding an edge is still acyclic.

The last point that has to be explained is the link between the condition on the color of~$a$ and~$b$ and merging two parts in the same partition.
If one of the two nodes, say $a$ for instance is of color~$c$, then it belongs to the same part of $\pi$ as its parent $z$.
The merging is the same if we choose~$z$ which is not colored $c$.
Moreover, as $b$ is in a different connected component, the corresponding two parts are distinct in $\pi$.
Finally, a part is just a corolla so the merging corresponds to building a corolla with $a$, $b$ and their children of color $c$.
\end{proof}

We finally recast \cref{prop:alternativeFormulaMobiusPolynomialMultiBraidArrangement} in terms of rainbow forests.

\begin{proposition} 
The M\"obius polynomial of the $(\ell,n)$-braid arrangement~$\multiBraidArrangement$ is given by
\[
\mobPol[\multiBraidArrangement] = x^{(n-1)(1-\ell)} \sum_{G \in \rainbowForests} y^{n-1+\card{E(G)}}  \prod_{i \in [\ell]} \weirdPol[n-\card{E(G,i)}].
\]
\end{proposition}

\begin{remark}
To further simplify this expression, we would need to count the number of rainbow forests with a prescribed number of colored edges.
However, this number does not admit a known multiplicative formula, up to our knowledge. When there is only one color, the corresponding sequence (counting non-colored forests on $n$ nodes and $k$ edges, rooted in the minimal label of each connected component) is \OEIS{A138464}.
\end{remark}

%%%%%%%%%%%%%%%

\subsection{Enumeration of vertices of $\multiBraidArrangement$}
\label{subsec:verticesMultiBraidArrangement}

We now use the labeled $(\ell,n)$-rainbow forests of \cref{subsec:rainbowForests} to derive more explicit formulas for the number of vertices of the $(\ell,n)$-braid arrangement~$\multiBraidArrangement$.
The first few values are gathered in \cref{table:verticesMultiBraidArrangement}.
 
\begin{table}
	\centerline{\scalebox{.8}{
		\begin{tabular}[t]{c|cccccccc}
			$n \backslash \ell$ & $1$ & $2$ & $3$ & $4$ & $5$ & $6$ & $7$ & $8$ \\ % & $9$ \\
			\hline
		$1$ & $1$ & $1$ & $1$ & $1$ & $1$ & $1$ & $1$ & $1$ \\ % & $1$ \\
		$2$ & $1$ & $2$ & $3$ & $4$ & $5$ & $6$ & $7$ & $8$ \\ % & $9$ \\
		$3$ & $1$ & $8$ & $21$ & $40$ & $65$ & $96$ & $133$ & $176$ \\ % & $225$ \\
		$4$ & $1$ & $50$ & $243$ & $676$ & $1445$ & $2646$ & $4375$ & $6728$ \\ % & $9801$ \\
		$5$ & $1$ & $432$ & $3993$ & $16384$ & $46305$ & $105456$ & $208537$ & $373248$ \\ % & $620289$ \\
		$6$ & $1$ & $4802$ & $85683$ & $521284$ & $1953125$ & $5541126$ & $13119127$ & $27350408$ \\ % & $51883209$ \\
		$7$ & $1$ & $65536$ & $2278125$ & $20614528$ & $102555745$ & $362797056$ & $1029059101$ & $2500000000$ \\ % & $5415228513$ \\
		$8$ & $1$ & $1062882$ & $72412707$ & $976562500$ & $6457339845$ & $28500625446$ & $96889010407$ & $274371577992$ % \\ & $678770015625$ \\
%		$9$ & $1$ & $20000000$ & $2681615217$ & $53971714048$ & $474659385665$ & $2614905943296$ & $10657046640625$ & $35184372088832$ & $99426586671873$
		\end{tabular}
	}}
%	\vspace{.3cm}
	\caption{The numbers $f_0(\multiBraidArrangement) = \ell \big( (\ell-1) n + 1 \big)^{n-2}$ of vertices of~$\multiBraidArrangement$ for~$\ell,n \in [8]$.}
	\label{table:verticesMultiBraidArrangement}
\end{table}

\begin{theorem}
\label{thm:verticesMultiBraidArrangement}
The number of vertices of the $(\ell,n)$-braid arrangement~$\multiBraidArrangement$ is
\[
f_0(\multiBraidArrangement) = \ell \big( (\ell-1) n + 1 \big)^{n-2}.
\]
\end{theorem}

\begin{proof}
By \cref{prop:flatPosetMultiBraidArrangement,prop:bijectionForests}, we just need to count the labeled $(\ell,n)$-rainbow trees.
A common reasoning for counting Cayley trees is the use of its Prüfer code defined by recursively pruning the smallest leaf while writing down the label of its parent.
This bijection can be adapted to colored Cayley trees by writing down the label of the parent colored by the color of the pruned leaf.
This leads to a bijection with certain colored words of length $n-1$.
Namely, there are two possibilities:
\begin{itemize}
\item either the pruned leaf is attached to the node~$1$ and it can have all $\ell$ colors,
\item or it is attached to one of the $n-1$ other nodes and it can only have $\ell-1$ colors.
\end{itemize}
Note that the last letter in the Prüfer code (obtained by removing the last edge) is necessarily the root $1$, with $\ell$ possible different colors.
Hence, there are 
\[
\big( \ell+(n-1)(\ell-1) \big)^{n-2} \ell = \ell \big( (\ell-1) n + 1 \big)^{n-2}
\]
such words.
Similar ideas were used in~\cite{Lewis}.
\end{proof}

We can refine the formula of \cref{thm:verticesMultiBraidArrangement} according to the dimension of the flats of the different copies intersected to obtain the vertices of the $(\ell,n)$-braid arrangement~$\multiBraidArrangement$.

\begin{theorem}
\label{thm:verticesRefinedMultiBraidArrangement}
For any~$k_1, \dots, k_\ell$ such that~$0 \le k_i \le n-1$ for~$i \in [\ell]$ and~${\sum_{i \in [\ell]} k_i = n-1}$, the number of vertices~$v$ of the $(\ell,n)$-braid arrangement~$\multiBraidArrangement$ such that the smallest flat of the $i$\ordinal{} copy of~$\braidArrangement$ containing~$v$ has dimension~$n-k_i-1$ is given by
\[
n^{\ell-1} \binom{n-1}{k_1, \dots, k_\ell} \prod_{i \in [\ell]} (n-k_i)^{k_i-1}.
\]
\end{theorem}

\begin{proof}
By \cref{prop:flatPosetMultiBraidArrangement,prop:bijectionForests}, we just need to count the labeled $(\ell,n)$-rainbow trees with~$k_i$ nodes colored by~$i$.
Forgetting the labels, the $(\ell,n)$-rainbow trees with~$k_i$ nodes colored by~$i$ are precisely the spanning trees of the complete multipartite graph~$K_{k_1, \dots, k_\ell, 1}$ (where the last~$1$ stands for the uncolored root).
Using a Pr\"ufer code similar to that of the proof of \cref{thm:verticesMultiBraidArrangement}, R.~Lewis proved in~\cite{Lewis} that the latter are counted by~${n^{\ell-1} \prod_{i \in [\ell]} (n-k_i)^{k_i-1}}$.
Finally, the possible labelings are counted by the multinomial coefficient~$\binom{n-1}{k_1, \dots, k_\ell}$.
\end{proof}

%%%%%%%%%%%%%%%

\subsection{Enumeration of regions and bounded regions of $\multiBraidArrangement$}
\label{subsec:regionsMultiBraidArrangement}

\enlargethispage{.2cm}
We finally use the labeled $(\ell,n)$-rainbow forests of \cref{subsec:rainbowForests} to derive more explicit formulas for the number of regions and bounded regions of the $(\ell,n)$-braid arrangement~$\multiBraidArrangement$.
The first few values are gathered in \cref{table:regionsMultiBraidArrangement,table:boundedRegionsMultiBraidArrangement}.
We first compute the characteristic polynomial of~$\multiBraidArrangement$.

\afterpage{
\begin{table}
	\centerline{\scalebox{.8}{
		\begin{tabular}[t]{c|cccccccc}
			$n \backslash \ell$ & $1$ & $2$ & $3$ & $4$ & $5$ & $6$ & $7$ & $8$ \\ % & $9$ \\
			\hline
			$1$ & $1$ & $1$ & $1$ & $1$ & $1$ & $1$ & $1$ & $1$ \\ % & $1$ \\
			$2$ & $2$ & $3$ & $4$ & $5$ & $6$ & $7$ & $8$ & $9$ \\ % & $10$ \\
			$3$ & $6$ & $17$ & $34$ & $57$ & $86$ & $121$ & $162$ & $209$ \\ % & $262$ \\
			$4$ & $24$ & $149$ & $472$ & $1089$ & $2096$ & $3589$ & $5664$ & $8417$ \\ % & $11944$ \\
			$5$ & $120$ & $1809$ & $9328$ & $29937$ & $73896$ & $154465$ & $287904$ & $493473$ \\ % & $793432$ \\
			$6$ & $720$ & $28399$ & $241888$ & $1085157$ & $3442816$ & $8795635$ & $19376064$ & $38323753$ \\ % & $69841072$ \\
			$7$ & $5040$ & $550297$ & $7806832$ & $49075065$ & $200320816$ & $625812385$ & $1629858672$ & $3720648337$ \\ % & $7686190000$ \\
			$8$ & $40320$ & $12732873$ & $302346112$ & $2666534049$ & $14010892416$ & $53536186825$ & $164859458688$ & $434390214657$ % \\ & $1017282905344$ \\
%			$9$ & $362880$ & $343231361$ & $13682809216$ & $169423639713$ & $1146173002496$ & $5357227099105$ & $19506923076096$ & $59328538244801$ & $157507267166848$
		\end{tabular}
	}}
%	\vspace{.3cm}
	\caption{The numbers $f_{n-1}(\multiBraidArrangement)$ of regions of~$\multiBraidArrangement$ for~$\ell,n \in [8]$.}
	\label{table:regionsMultiBraidArrangement}
\end{table}
}

\afterpage{
\begin{table}
	\centerline{\scalebox{.8}{
		\begin{tabular}[t]{c|cccccccc}
			$n \backslash \ell$ & $1$ & $2$ & $3$ & $4$ & $5$ & $6$ & $7$ & $8$ \\ % & $9$ \\
			\hline
			$1$ & $1$ & $1$ & $1$ & $1$ & $1$ & $1$ & $1$ & $1$ \\ % & $1$ \\
			$2$ & $0$ & $1$ & $2$ & $3$ & $4$ & $5$ & $6$ & $7$ \\ % & $8$ \\
			$3$ & $0$ & $5$ & $16$ & $33$ & $56$ & $85$ & $120$ & $161$ \\ % & $208$ \\
			$4$ & $0$ & $43$ & $224$ & $639$ & $1384$ & $2555$ & $4248$ & $6559$ \\ % & $9584$ \\
			$5$ & $0$ & $529$ & $4528$ & $17937$ & $49696$ & $111745$ & $219024$ & $389473$ \\ % & $644032$ \\
			$6$ & $0$ & $8501$ & $120272$ & $663363$ & $2354624$ & $6455225$ & $14926176$ & $30583847$ \\ % & $57255488$ \\
			$7$ & $0$ & $169021$ & $3968704$ & $30533409$ & $138995776$ & $464913325$ & $1268796096$ & $2996735329$ \\ % & $6353133184$ \\
			$8$ & $0$ & $4010455$ & $156745472$ & $1684352799$ & $9841053184$ & $40179437975$ & $129465630720$ & $352560518527$ % \\ & $846588258944$ \\
%			$9$ & $0$ & $110676833$ & $7216242688$ & $108413745057$ & $813420601856$ & $4055310777025$ & $15431698810368$ & $48461340225473$ & $131823966149632$
		\end{tabular}
	}}
%	\vspace{.3cm}
	\caption{The numbers $b_{n-1}(\multiBraidArrangement)$ of bounded regions of~$\multiBraidArrangement$ for~$\ell,n \in [8]$.}
	\label{table:boundedRegionsMultiBraidArrangement}
\end{table}
}

\begin{theorem}
\label{thm:characteristicPolynomialMultiBraidArrangement}
The characteristic polynomial~$\charPol[\multiBraidArrangement]$ of the $(\ell,n)$-braid arrangement~$\multiBraidArrangement$ is given~by
\[
\charPol[\multiBraidArrangement] = \frac{(-1)^n n!}{y} \, [z^n] \, \exp \bigg( - \sum_{m \ge 1} \frac{F_{\ell,m} \, y \, z^m}{m} \bigg) ,
\]
where~$\displaystyle F_{\ell,m} \eqdef \frac{1}{(\ell-1)m+1} \binom{\ell m}{m}$ is the Fuss-Catalan number.
\end{theorem}

\begin{proof}
By \cref{thm:MobiusPolynomialMultiBraidArrangement,prop:bijectionForests}, the characteristic polynomial~$\charPol[\multiBraidArrangement]$ is
\[
\charPol[\multiBraidArrangement] = \sum_{\b{F} \in \forestPoset} \mu_{\forestPoset}(\HH, \b{F}) \, y^{\dim(\b{F})} = \sum_{F \in \rainbowForests} \lambda(F) \, (-1)^{n-\card{F}} \, \omega(F) \, y^{\card{F}-1}.
\]
From \cref{lem:labelingRainbowForest}, we observe that
\[
\frac{\lambda(F) \, \omega(F) \, (-y)^{\card{F}} \, z^{\|F\|}}{\|F\|!} = \prod_{T \in F} \frac{-y \, z^{\|T\|}}{\|T\|} ,
\]
where $T$ ranges over the trees of~$F$.
Now using that rainbow forests are exactly sets of rainbow trees, we obtain that
\[
\sum_{F \in \rainbowForests[][\ell]} \frac{ \lambda(F) \, \omega(F) \, (-y)^{\card{F}} \, z^{\|F\|}}{\|F\|!} = \sum_{F \in \rainbowForests[][\ell]} \prod_{T \in F} \frac{-y \, z^{\|T\|}}{\|T\|} = \exp \bigg( \sum_{T \in \rainbowTrees[][\ell]} \frac{-y \, z^{\|T\|}}{\|T\|} \bigg).
\]
From \cref{lem:FussCatalan}, we obtain that
\[
\exp \bigg( \sum_{T \in \rainbowTrees[][\ell]} \frac{-y \, z^{\|T\|}}{\|T\|} \bigg) = \exp \bigg( - \sum_{m \ge 1} \frac{F_{\ell,m} \, y \, z^m}{m} \bigg).
\]
We conclude that
\begin{align*}
\charPol[\multiBraidArrangement] 
& = \sum_{F \in \rainbowForests}  \lambda(F) \, (-1)^{n-\card{F}} \, \omega(F) \, y^{\card{F}-1} \\
& = \frac{(-1)^n \, n!}{y} [z^n] \sum_{F \in \rainbowForests[][\ell]} \frac{ \lambda(F) \, \omega(F) \, (-y)^{\card{F}} \, z^{\|F\|}}{\|F\|!} \\
& = \frac{(-1)^n n!}{y} \, [z^n] \, \exp \bigg( - \sum_{m \ge 1} \frac{F_{\ell,m} \, y \, z^m}{m} \bigg).
\qedhere
\end{align*}
\end{proof}

From the characteristic polynomial of~$\multiBraidArrangement$ and \cref{rem:characteristicPolynomial}, we obtain its numbers of regions and bounded regions.

\begin{theorem}
\label{thm:regionsMultiBraidArrangement}
The numbers of regions and of bounded regions of the $(\ell,n)$-braid arrangement~$\multiBraidArrangement$ are given by
\begin{align*}
f_{n-1}(\multiBraidArrangement) 
& = n! \, [z^n] \exp \Bigg( \sum_{m \ge 1} \frac{F_{\ell,m} \, z^m}{m} \Bigg) \\
\text{and}\qquad
b_{n-1}(\multiBraidArrangement)
%& = - n! \, [z^n] \exp \Bigg( - \sum_{m \ge 1} \frac{F_{\ell,m} \, z^m}{m} \Bigg) 
& = (n-1)! \, [z^{n-1}] \exp \bigg( (\ell-1) \sum_{m \ge 1} F_{\ell,m} \, z^m \bigg),
\end{align*}
where~$\displaystyle F_{\ell,m} \eqdef \frac{1}{(\ell-1)m+1} \binom{\ell m}{m}$ is the Fuss-Catalan number.
\end{theorem}

\begin{proof}
By \cref{rem:characteristicPolynomial}, we obtain from \cref{thm:characteristicPolynomialMultiBraidArrangement} that
\begin{align*}
f_{n-1}(\multiBraidArrangement) & = (-1)^{n-1} \charPol[\multiBraidArrangement][-1] = n! \, [z^n]  \, \exp \bigg( \sum_{m \ge 1} \frac{F_{\ell,m} \, z^m}{m} \bigg), \\
% \text{and}\qquad
b_{n-1}(\multiBraidArrangement) & = (-1)^{n-1} \charPol[\multiBraidArrangement][1] = - n! \, [z^n]  \, \exp \bigg( - \sum_{m \ge 1} \frac{F_{\ell,m} \, z^m}{m} \bigg).
\end{align*}
To conclude, we thus just need to observe that
\(
U_\ell(z) = \frac{\partial}{\partial z} V_\ell(z)
\)
where
\[
U_\ell(z) \eqdef \exp \bigg( (\ell-1) \sum_{m \ge 1} F_{\ell,m} \, z^m \bigg)
\qquad\text{and}\qquad
V_\ell(z) \eqdef - \exp \bigg( - \sum_{m \ge 1} \frac{F_{\ell,m} \, z^m}{m} \bigg).
\]
For this, consider the generating functions
\[
F_\ell(z) \eqdef \sum_{m \ge 0} F_{\ell,m} \, z^m
\qquad\text{and}\qquad
G_\ell(z) \eqdef \sum_{m \ge 1} \frac{F_{\ell,m} \, z^m}{m}.
\]
Recall from \cref{rem:functionalEquationFussCatalan} that~$F_\ell(z)$ satisfies the functional equation
\[
F_\ell(z) = 1 + z \, F_\ell(z)^\ell.
\]
We thus obtain that
\[
F_\ell'(z) \big( 1 - \ell \, z \, F_\ell(z)^{\ell-1} \big) = F_\ell(z)^\ell
\quad\text{and}\quad
F_\ell(z) \big( 1 - \ell \, z \, F_\ell(z)^{\ell-1} \big) = 1 - (\ell-1) \, z \, F_\ell(z)^\ell.
\]
Combining these two equations, we get
\begin{equation}
\label{eq:diff}
F_\ell(z)^{\ell+1} = F_\ell'(z) \big( 1 - (\ell-1) \, z \, F_\ell(z)^\ell \big).
\end{equation}
Observe now that
\begin{equation}
\label{eq:GF}
z \, G_\ell'(z) = F_\ell(z) - 1 = z \, F_\ell(z)^\ell
\qquad\text{and}\qquad
G_\ell''(z) = \ell \, F_\ell(z)^{\ell-1} \, F_\ell'(z).
\end{equation}
Hence
\[
U_\ell(z) = \exp \big( (\ell-1) \, (F_\ell(z) - 1) \big) = \exp \big( (\ell-1) \, z \, G_\ell’(z) \big)
\]
and
\[
V_\ell'(z) = \frac{\partial}{\partial z}  - \exp \big( \! - G_\ell(z) \big) = G_\ell'(z) \exp \big( -G_\ell(z) \big).
\]
Consider now the function
\[
W_\ell(z) = V_\ell'(z) / U_\ell(z) = G_\ell'(z) \exp \big( \! - G_\ell(z) - (\ell-1) \, z \, G_\ell'(z) \big).
\]
Clearly, $W_\ell(0) = 1$.
Moreover, using~\eqref{eq:GF}, we obtain that its derivative is
\begin{align*}
W_\ell'(z)
& = \Big( G_\ell''(z) \big(1 - (\ell-1) \, z \, G_\ell'(z) \big) - \ell \, G_\ell'(z)^2 \Big) \exp \big( \! - G_\ell(z) - (\ell-1) \, z \, G_\ell'(z) \big) \\
& = \ell \, F_\ell(z)^{\ell-1} \Big( F_\ell'(z) \big( 1 - (\ell-1) \, z \, F_\ell(z)^\ell \big) - F_\ell(z)^{\ell+1} \Big) \exp \big( \! - G_\ell(z) - (\ell-1) \, z \, G_\ell'(z) \big),
\end{align*}
which vanishes by~\eqref{eq:diff}.
\end{proof}

%%%%%%%%%%%%%%%%%%%%%%%%%%%%%%%%%%%%%%

\section{Face poset and combinatorial description of~$\multiBraidArrangement(\b{a})$}
\label{sec:facePoset}

In this section, we describe the face poset of the $\b{a}$-braid arrangement~$\multiBraidArrangement(\b{a})$ in terms of ordered $(\ell,n)$-partition forests.
This section highly depends on the choice of the translation matrix~$\b{a}$.

%%%%%%%%%%%%%%%

\subsection{Ordered partition forests}
\label{subsec:orderedPartitionForests}

We now introduce the combinatorial objects that will be used to encode the faces of the $\b{a}$-braid arrangement~$\multiBraidArrangement(\b{a})$ of \cref{def:multiBraidArrangementPrecise}.

\begin{definition}
\label{def:orderedPartitionForest}
An \defn{ordered $(\ell,n)$-partition forest} (\resp \defn{tree}) is an $\ell$-tuple~$\order{\b{F}} \eqdef (\order{F_1}, \dots, \order{F_\ell})$ of ordered set partitions of~$[n]$ such that the corresponding $\ell$-tuple~$\b{F} \eqdef (F_1, \dots, F_\ell)$ of unordered set partitions of~$[n]$ forms an $(\ell,n)$-partition forest (\resp tree).
The \defn{ordered $(\ell,n)$-partition forest poset} is the poset~$\orderedForestPoset$ on ordered $(\ell,n)$-partition forests ordered by componentwise refinement.
In other words, $\orderedForestPoset$ is the subposet of the $\ell$\ordinal{} Cartesian power of the ordered partition poset~$\orderedPartitionPoset$ induced by ordered $(\ell,n)$-partition forests.
Note that the maximal elements of~$\orderedForestPoset$ are the ordered $(\ell, n)$-partition trees.
\end{definition}

The following statement is the analogue of \cref{prop:flatPosetMultiBraidArrangement}, and is illustrated in \cref{fig:B23a,fig:B23b}.

\begin{proposition}
\label{prop:facePosetMultiBraidArrangement}
The face poset~$\facePoset[\multiBraidArrangement(\b{a})]$ of the $\b{a}$-braid arrangement~$\multiBraidArrangement(\b{a})$ is isomorphic to an upper set~$\orderedForestPoset(\b{a})$ of the ordered $(\ell,n)$-partition forest poset~$\orderedForestPoset$.
\end{proposition}

\begin{proof}
The proof is based on that of \cref{prop:flatPosetMultiBraidArrangement}.
A face of~$\multiBraidArrangement(\b{a})$ is an intersection of faces of the $\ell$ copies of~$\multiBraidArrangement$, hence corresponds to an $\ell$-tuple of ordered partitions of~$[n]$.
Moreover, the flats supporting these faces intersect, so that the corresponding unordered partitions must form an $(\ell,n)$-partition forest.
Hence, each face of~$\multiBraidArrangement(\b{a})$ corresponds to a certain ordered $(\ell,n)$-partition forest.
Moreover, the inclusion of faces of~$\multiBraidArrangement(\b{a})$ translates to the componentwise refinement on ordered partitions.
Finally, by genericity, it is immediate that we obtain an upper set of this componentwise refinement order.
\end{proof}

\begin{figure}
	\centerline{\includegraphics[scale=.9]{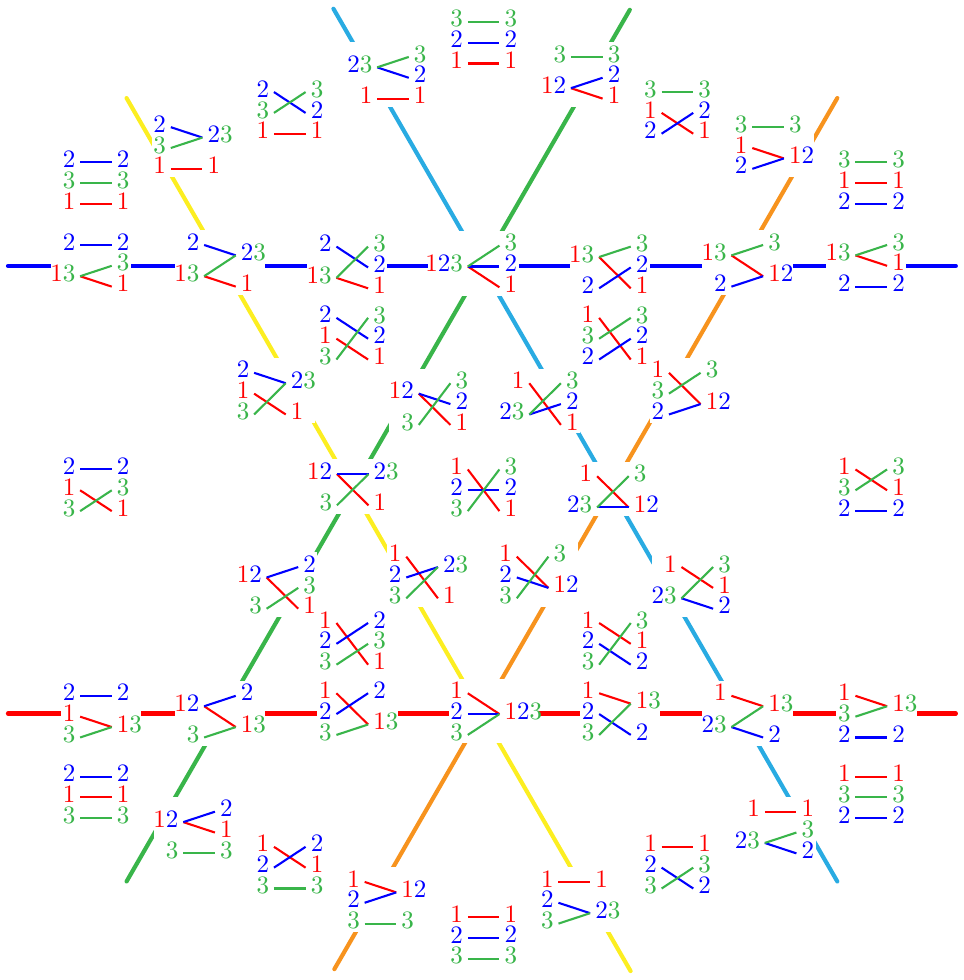}}
	\caption{Labelings of the faces of the arrangement~$\multiBraidArrangement[3][2](\b{a})$ for~$\b{a} = \begin{bmatrix} 0 & 0 \\ -1 & -1 \end{bmatrix}$.}
	\label{fig:B23a}
\end{figure}

\begin{figure}
	\centerline{\includegraphics[scale=.9]{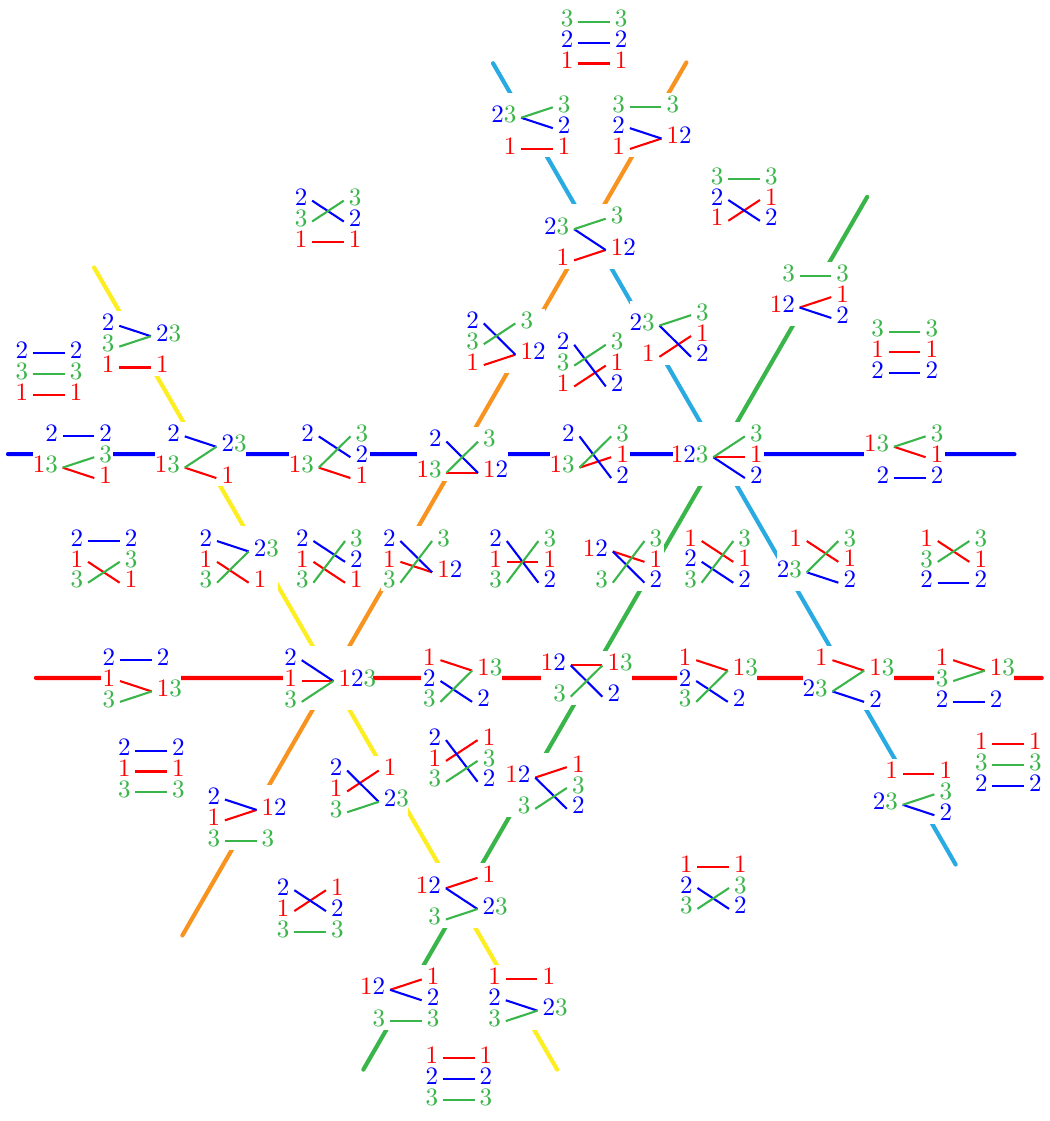}}
	\caption{Labelings of the faces of the arrangement~$\multiBraidArrangement[3][2](\b{a})$ for~$\b{a} = \begin{bmatrix} 0 & 0 \\ 1 & -2 \end{bmatrix}$.}
	\label{fig:B23b}
\end{figure}

We now fix a generic translation matrix~$\b{a} \eqdef (a_{i,j})$ and still denote by~$A_{i,s,t} \eqdef \smash{\sum_{s \le j < t} a_{i,j}}$ for all~$1 \le s < t \le n$ and~$i \in [\ell]$ (and often write~$A_{i,t,s}$ for~$-A_{i,s,t}$).
The objective of this section is to describe
\begin{itemize}
\item the ordered $(\ell,n)$-partitions forests of the upper set~$\orderedForestPoset(\b{a})$ with a given underlying (unordered) $(\ell,n)$-partition forest (\cref{subsec:PFtoOPF}),
\item a criterion to decide whether a given ordered $(\ell,n)$-partition forest belongs to the upper set~$\orderedForestPoset(\b{a})$, \ie corresponds to a face of~$\multiBraidArrangement(\b{a})$ (\cref{subsec:criterionOPF}).
\end{itemize}

%%%%%%%%%%%%%%%

\subsection{From partition forests to ordered partition forests}
\label{subsec:PFtoOPF}

In this section, we describe the ordered $(\ell,n)$-partitions forests~$\order{\b{F}}$ of the upper set~$\orderedForestPoset(\b{a})$ with a given underlying $(\ell,n)$-partition forest~$\b{F}$.
We denote by~$cc(\b{F})$ the connected components of~$\b{F}$, meaning the partition of~$[n]$ given by the hyperedge labels of the connected components of the intersection hypergraph of~$\b{F}$.
We first observe that the choice of~$\b{a}$ fixes the order of the parts in a common connected component of~$\b{F}$.

\begin{proposition}
\label{prop:PFtoOPF1}
Consider a $(\ell,n)$-partition forest~$\b{F} \eqdef (F_1, \dots, F_\ell)$, and two integers~$s,t \in [n]$ labeling two hyperedges in the same connected component of the intersection hypergraph of~$\b{F}$.
Assume that the unique path from~$s$ to~$t$ in the hypergraph of~$\b{F}$ passes through the hyperedges labeled by~$s = r_0, \dots, r_q = t$ and through parts of the partitions~$F_{i_1}, \dots, F_{i_q}$.
Then for any ordered $(\ell,n)$-partition forest~$\order{\b{F}} \eqdef (\order{F_1}, \dots, \order{F_\ell})$ of the upper set~$\orderedForestPoset(\b{a})$ with underlying $(\ell,n)$-partition forest~$\b{F}$ and any~$i \in [\ell]$, the order of~$s$ and~$t$ in~$\order{F}_i$ is given by the sign of~$A_{i,s,t} - \sum_{p \in [q]} A_{i_p, r_{p-1}, r_p}$.
\end{proposition}

\begin{proof}
Consider any point~$\b{x}$ in the face of~$\multiBraidArrangement(\b{a})$ corresponding to~$\order{\b{F}}$.
Along the path from~$s$ to~$t$, we have~$x_{r_{p-1}} - x_{r_p} = A_{i_p, r_{p-1}, r_p}$ for each~$p \in [q]$.
Hence, we obtain that
\[
x_s - x_t = \sum_{p \in [q]}  (x_{r_{p-1}} - x_{r_p}) = \sum_{p \in [q]} A_{i_p, r_{p-1}, r_p}.
\]
The order of~$s,t$ in~$\order{F}_i$ is given by the sign of~$A_{i,s,t} - (x_s - x_t)$, hence of~$A_{i,s,t} - \sum_{p \in [q]} A_{i_p, r_{p-1}, r_p}$.
\end{proof}

We now describe the different ways to order the parts in distinct connected components of~$\b{F}$.
For this, we need the following posets.

\begin{definition}
Consider a $(\ell,n)$-partition forest~$\b{F}$ and denote by~$cc(\b{F})$ the connected components of~$\b{F}$.
For each pair~$s,t \in [n]$ in distinct connected components of~$cc(\b{F})$, we define the chain~$<_{s,t}$ on the $\ell$ triples~$(i,s,t)$ for~$i \in [\ell]$ given by the order of the values~$A_{i,s,t}$.
The \defn{inversion poset}~$\Inv(\b{F}, \b{a})$ is then the poset obtained by quotienting the disjoint union of the chains~$<_{s,t}$ (for all~$s,t \in [n]$ in distinct connected components of~$cc(\b{F})$) by the equivalence relation~$(i,s,t) \equiv (i,s',t')$ if~$s$ and~$s'$ belong to the same part of~$F_i$ and~$t$ and~$t'$ belong to the same part of~$F_i$.
We say that a subset~$X$ of~$\Inv(\b{F}, \b{a})$ is antisymmetric if~$(i,s,t) \in X \iff (i,t,s) \notin X$.
\end{definition}

\begin{proposition}
\label{prop:PFtoOPF2}
The ordered $(\ell,n)$-partition forests of the upper set~$\orderedForestPoset(\b{a})$ with a given underlying $(\ell,n)$-partition forest~$\b{F}$ are in bijection with the antisymmetric lower sets of the inversion poset~$\Inv(\b{F}, \b{a})$.
\end{proposition}

\begin{proof}
Consider an ordered $(\ell,n)$-partition forest~$\order{\b{F}}$ of the upper set~$\orderedForestPoset(\b{a})$.
Let~$\b{x}$ be any point of the face of~$\multiBraidArrangement(\b{a})$ corresponding to~$\order{\b{F}}$.
For each pair~$s,t \in [n]$ in distinct connected components of~$cc(\b{F})$, let~$I_{s,t}(\order{\b{F}})$ be the set of indices~$i \in [\ell]$ such that~$x_s - x_t < A_{i,s,t}$.
Note that~$I_{s,t}$ is by definition a lower set of the chain~$<_{s,t}$ of~$\Inv(\b{F}, \b{a})$.
Hence, $I(\order{F}) \eqdef \bigcup_{s,t} I_{s,t} / {\equiv}$ is a lower set of~$\Inv(\b{F}, \b{a})$.
Moreover, it is clearly antisymmetric since
\[
(i,s,t) \in I(\order{F}) \iff x_s - x_t < A_{i,s,t} \iff x_t - x_s > A_{i,t,s} \iff (i,t,s) \notin I(\order{F}).
\]

Conversely, given an antisymmetric lower set~$I$ of~$\Inv(\b{F}, \b{a})$, we can reconstruct an ordered $(\ell,n)$-partition forest~$\order{\b{F}}$ by ordering each pair~$s,t \in [n]$ in~$\order{F}_i$ 
\begin{itemize}
\item according to \cref{prop:PFtoOPF1} (hence independently of~$I$) if $s$ and~$t$ belong to the same connected component of~$\b{F}$,
\item according to~$I$ if~$s$ and~$t$ belong to distinct connected components of~$\b{F}$. Namely, we place the block of~$\order{F}_i$ containing~$s$ before the block of~$\order{F}_i$ containing~$t$ if and only if~$(i,s,t) \in I$.
\end{itemize}
It is then straightforward to check that the resulting ordered $(\ell,n)$-partition forest belongs to the upper set~$\orderedForestPoset(\b{a})$, by exhibiting a point~$\b{x}$ in of the corresponding face of~$\multiBraidArrangement(\b{a})$.
\end{proof}

%%%%%%%%%%%%%%%

\subsection{A criterion for ordered partition forests}
\label{subsec:criterionOPF}

We now consider a given ordered $(\ell,n)$-partition forest~$\order{F}$ and provide a criterion to decide if it belongs to the upper set~$\orderedForestPoset(\b{a})$ corresponding to the faces of~$\multiBraidArrangement(\b{a})$.
For this, we need the following directed graph associated to~$\order{\b{F}}$.

\begin{definition}
For an ordered partition~$\order{\pi} \eqdef \order{\pi}_1 | \cdots | \order{\pi}_k$ of~$[n]$, we denote by~$D_{\order{\pi}}$ the directed graph on~$[n]$ with an arc~$\max(\order{\pi}_j) \to \min(\order{\pi}_{j+1})$ for each~$j \in [k-1]$ and a cycle~${x_1 \to \dots \to x_p \to x_1}$ for each part~$\order{\pi}_j = \{x_1 < \dots < x_p\}$.
Note that~$D_{\order{\pi}}$ has~$n$ vertices and~$n + k$ arcs.
For an ordered $(\ell,n)$-partition forest~$\order{\b{F}} \eqdef (\order{F_1}, \dots, \order{F_\ell})$, we denote by~$D_{\order{\b{F}}}$ the superposition of the directed graphs~$D_{\order{F}_i}$ for~$i \in [\ell]$, where the arcs of~$D_{\order{F}_i}$ are labeled by~$i$.
\end{definition}

\begin{proposition}
\label{prop:characterizationOPFs}
An ordered $(\ell,n)$-partition forest~$\order{\b{F}}$ belongs to the upper set~$\orderedForestPoset(\b{a})$ if and only if $\sum_{\alpha \in \gamma} A_{i(\alpha), s(\alpha), t(\alpha)} \ge 0$ for any (simple) oriented cycle~$\gamma$ in~$D_{\order{\b{F}}}$, where each arc~$\alpha \in \gamma$ has label~$i(\alpha)$, source~$s(\alpha)$, and target~$t(\alpha)$.
\end{proposition}

\begin{proof}
Consider an ordered $(\ell,n)$-partition forest~$\order{\b{F}} \eqdef (\order{F}_1, \dots, \order{F}_\ell)$.
For each~$i \in [\ell]$, denote by
\begin{itemize}
\item $m_i$ the number of arcs of~$D_{\order{F}_i}$
\item $M_i$ the incidence matrix of~$D_{\order{F}_i}$, with $m_i$ rows and $n$ columns, with a row for each arc~$\alpha$ of~$D_{\order{F}_i}$ containing a $-1$ in column~$s(\alpha)$, a $1$ in column~$t(\alpha)$, and $0$ elsewhere,
\item $\b{z}_i$ the column vector in~$\R^{m_i}$ with a row for each arc~$\alpha$ of~$D_{\order{F}_i}$ containing the value~$A_{i(\alpha), s(\alpha), t(\alpha)}$.
\end{itemize}
Then a point~$\b{x} \in \R^n$ belongs to the face of the $i$\ordinal{} braid arrangement corresponding to~$\order{F}_i$ if and only if it satisfies~$M_i \, \b{x} \le \b{z}_i$.
Hence, $\order{\b{F}}$ appears as a face of the $\b{a}$-braid arrangement if and only if there exists~$\b{x} \in \R^n$ such that~$M \, \b{x} \le \b{z}$, where~$M$ is the $(m \times n)$-matrix (where~$m \eqdef \sum_{i \in [\ell]} m_i$), obtained by piling the matrices~$M_i$ for~$i \in [\ell]$ and similarly, $\b{z}$ is the column vector obtained by piling the vectors~$\b{z}_i$.
A direct application of the Farkas lemma (see \eg \cite[Prop.~1.7]{Ziegler}), there exists~$\b{x} \in \R^n$ such that~$M \, \b{x} \le \b{z}$ if and only if~$\b{c} \b{z} \ge 0$ for any~$\b{c} \in (\R^m)^*$ with~$\b{c} \ge \b{0}$ and~$\b{c} M = \b{0}$.
Now it is classical that the left kernel of the incidence matrix of a directed graph is generated by its circuits (non-necessarily oriented cycles), and that the positive cone in this left kernel is generated by its oriented cycles.
\end{proof}

\begin{remark}
Note that we made some arbitrary choices here by choosing the arc from~$\max(\order{\pi}_j)$ to~$\min(\order{\pi}_{j+1})$ between two consecutive parts~$\order{\pi}_j$ and~$\order{\pi}_{j+1}$ and a cycle inside each part~$\order{\pi}_j$ (while we said that the order in each part is irrelevant).
We could instead have considered all arcs connecting two elements of two consecutive parts, or two elements inside the same part.
Our choices just limit the amount of oriented cycles in~$D_{\order{\pi}}$.
\end{remark}

%%%%%%%%%%%%%%%%%%%%%%%%%%%%%%%%%%%%%%
%%%%%%%%%%%%%%%%%%%%%%%%%%%%%%%%%%%%%%

\clearpage
\part{Diagonals of permutahedra}
\label{part:diagonalsPermutahedra}

In this second part, we study the combinatorics of the diagonals of the permutahedra.
In \cref{sec:cellularDiagonals}, we first recall the definition and some known facts about cellular diagonals of polytopes (\cref{subsec:cellularDiagonalsPolytopes}), which we immediately specialize to the classical permutahedron (\cref{sec:cellularDiagonalsPermutahedra}), and connect to \cref{part:multiBraidArrangements} to derive enumerative statements on the diagonals of permutahedra (\cref{subsec:enumerationDiagonalPermutahedra}).
In \cref{sec:operadicDiagonals}, we consider two particular diagonals, the $\LA$ and $\SU$ diagonals (\cref{subsec:LASUdiagonal}), we show that these are the only two operadic diagonals which induce the weak order (\cref{subsec:operadicProperty}), and that they are isomorphic (\cref{subsec:isos-LA-SU}). 
Using results from \cref{sec:facePoset,sec:cellularDiagonals}, we then characterize their facets in terms of paths in $(2,n)$-partition trees (\cref{subsec:facets-operadic-diags}), and their vertices as pattern-avoiding pairs of permutations (\cref{subsec:vertices-operadic-diags}).
Finally, in \cref{sec:shifts}, we show that the geometric $\SU$ diagonal $\SUD$ is a topological enhancement of the original Saneblidze-Umble diagonal~\cite{SaneblidzeUmble} (\cref{subsec:topological-SU}).
In order to prove this result, we define different types of shifts that can be performed on the facets of the $\SU$ diagonal, and give several new equivalent definitions of it. 
These descriptions are directly translated to the $\LA$ diagonal via isomorphism (\cref{subsec:shifts-under-iso}).
Moreover, we observe that the shifts define a natural lattice structure on the set of facets of operadic diagonals, that we call the \emph{shift lattice} (\cref{sec:Shift-lattice}).
Finally, we present the alternative matrix (\cref{subsec:matrix}) and cubical (\cref{sec:Cubical}) descriptions of the $\SU$ diagonal from \cite{SaneblidzeUmble,SaneblidzeUmble-comparingDiagonals}, provide proofs of their equivalence with the other descriptions, and give their $\LA$ counterparts. 

%%%%%%%%%%%%%%%%%%%%%%%%%%%%%%%%%%%%%%

\section{Cellular diagonals}
\label{sec:cellularDiagonals}

%%%%%%%%%%%%%%%

\subsection{Cellular diagonals for polytopes}
\label{subsec:cellularDiagonalsPolytopes}
 
As discussed in the introduction, cellular approximations of the thin diagonal for families of polytopes are of fundamental importance in algebraic topology and geometry.
They allow one to define the cup product and thus define the ring structure on the cohomology groups of a topological space, and combinatorially on the Chow groups of a toric variety. 
We now proceed to define thin, cellular, and geometric diagonals.

\begin{definition} 
\label{def:thinDiagonal}
The \defn{thin diagonal} of a set $X$ is the map~$\delta : X \to X \times X$ defined by $\delta(x) \eqdef (x,x)$ for all $x \in X$.
See \cref{fig:examplesDiagonals1}\,(left).
\end{definition}

\begin{definition} 
\label{def:cellularDiagonal}
A \defn{cellular diagonal} of a $d$-dimensional polytope $P$ is a continuous map ${\Delta : P \to P \times P}$ such that
\begin{enumerate}
\item its image is a union of $d$-dimensional faces of $P\times P$ (\ie it is \defn{cellular}),
\item it agrees with the thin diagonal of~$P$ on the vertices of $P$, and
\item it is homotopic to the thin diagonal of~$P$, relative to the image of the vertices of~$P$. 
\end{enumerate}
See \cref{fig:examplesDiagonals1}\,(middle left).
A cellular diagonal is said to be \defn{face coherent} if its restriction to a face of $P$ is itself a cellular diagonal for that face. 
\end{definition}

A powerful geometric technique to define face coherent cellular diagonals on polytopes first appeared in~\cite{FultonSturmfels}, was presented in~\cite{MasudaThomasTonksVallette}, and was fully developed in~\cite{LaplanteAnfossi}.
We provide in \cref{thm:diagonal} the precise (but slightly technical) definition of these diagonals, even though we will only use the characterization of the faces in their image provided in \cref{thm:universalFormula}.

The key idea is that any vector $\b{v}$ in generic position with respect to $P$ defines a cellular diagonal of $P$.
For $\b{z}$ a point of $P$, we denote by $\rho_{\b{z}} P \eqdef 2\b{z}-P$ the reflection of $P$ with respect to the point~$\b{z}$.

\begin{definition}
	The \defn{fundamental hyperplane arrangement}~$\mathcal{H}_P$ of a polytope~$P \subset \R^d$ is the set of all linear hyperplanes of~$\R^d$ orthogonal to the edges of $P\cap \rho_{\b{z}} P$ for all $\b{z} \in P$. 
	See \cref{fig:examplesHyperplanes}.
\end{definition}

\begin{figure}[p]
	\centerline{
		\includegraphics[scale=.38]{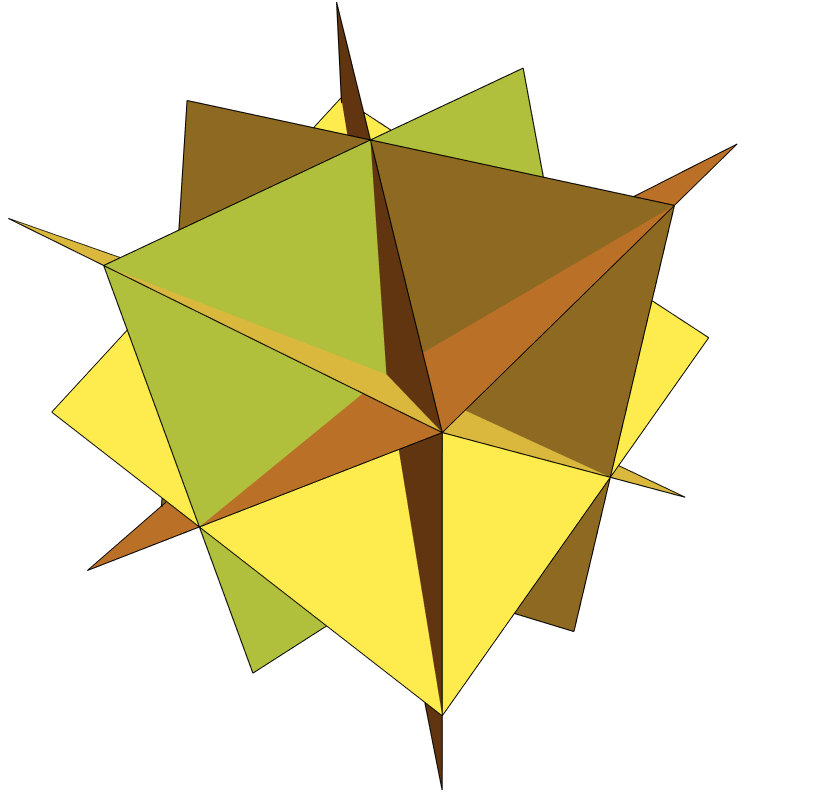} \hspace{-.6cm}
		\includegraphics[scale=.44]{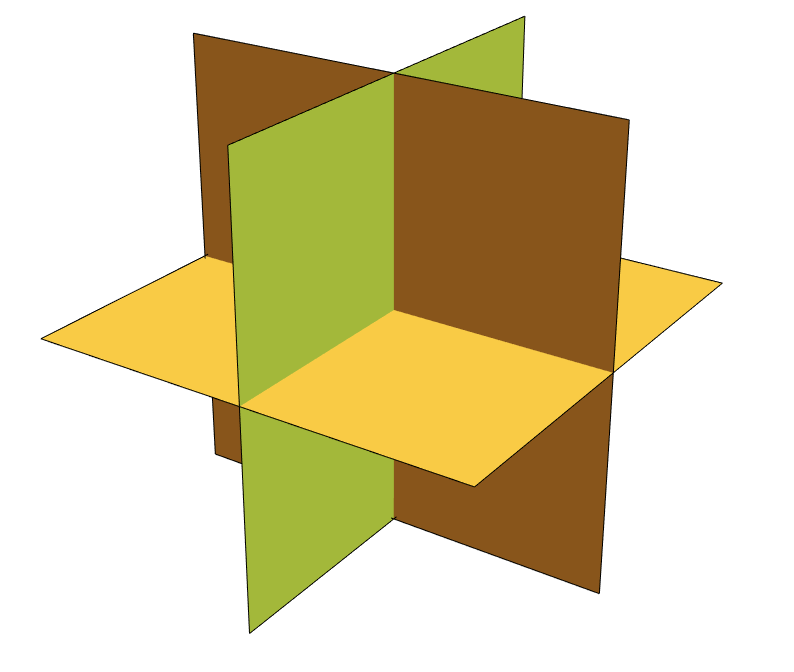} \hspace{-.3cm}
		\includegraphics[scale=.38]{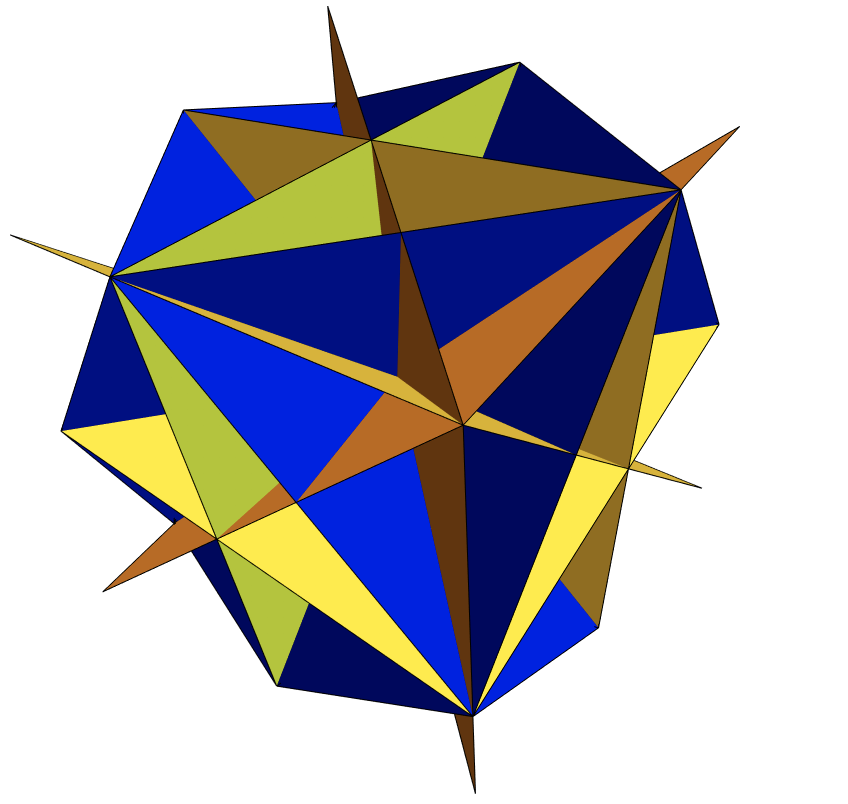}
	}
	\caption{The fundamental hyperplane arrangements of the $3$-dimensional simplex (left), cube (middle), and permutahedron (right). The hyperplanes perpendicular 
	to edges of some intersection $P\cap \rho_z P$, which are \emph{not} edges of the polytope~$P$, are colored in blue. Left and rightmost illustrations from~\cite[Fig.~12]{LaplanteAnfossi}.}
	\label{fig:examplesHyperplanes}
\end{figure}

A vector is \defn{generic with respect to~$P$} if it does not belong to the union of the hyperplanes of the fundamental hyperplane arrangement~$\mathcal{H}_P$.
In particular, such a vector is not perpendicular to any edge of $P$, and we denote by~$\min_{\b{v}}(P)$ (\resp $\max_{\b{v}}(P)$) the unique vertex of~$P$ which minimizes (\resp maximizes) the scalar product with~$\b{v}$. 
Note that the datum of a polytope~$P$ together with a vector~$\b{v}$ generic with respect to~$P$ was called \defn{positively oriented polytope} in~\cite{MasudaThomasTonksVallette,LaplanteAnfossi,LaplanteAnfossiMazuir}.

\pagebreak

\begin{theorem}
\label{thm:diagonal}
For any vector $\b{v} \in \R^d$ generic with respect to $P$, the tight coherent section~$\triangle_{(P,\b{v})}$ of the projection $P \times P \to P, (\b{x}, \b{y}) \mapsto (\b{x}+\b{y})/2$ selected by the vector~$(-\b{v}, \b{v})$ defines a cellular diagonal of~$P$.
More precisely, $\triangle_{(P,\b{v})}$ is given by the formula
\begin{align*}
	\begin{array}{rlcl}
		\triangle_{(P,\b{v})}\ : & P & \to & P\times P \\
		& \b{z} & \mapsto & \bigl( \min_{\b{v}}(P\cap \rho_{\b{z}} P), \, \max_{\b{v}}(P\cap \rho_{\b{z}} P) \bigr) .
	\end{array}
\end{align*}
\end{theorem}

\begin{definition}
\label{def:geometricDiagonal}
A \defn{geometric diagonal} of a polytope~$P$ is a diagonal of the form~$\triangle_{(P,\b{v})}$ for some vector~$\b{v} \in \R^d$ generic with respect to~$P$.
\end{definition}

Note that the geometric diagonal~$\triangle_{(P,\b{v})}$ only depends on the region of~$\mathcal{H}_P$ containing~$\b{v}$, see~\cite[Prop.~1.23]{LaplanteAnfossi}.

\begin{figure}[p]
	\centerline{\scalebox{1.25}{
	\begin{tabular}{c@{\hspace{-.2cm}}c@{\hspace{-.2cm}}c@{\hspace{-.2cm}}c}
		\includegraphics[scale=1.2]{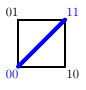} &
		\includegraphics[scale=1.2]{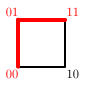} &
		\raisebox{-.2cm}{\includegraphics[scale=1]{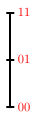}} &
		\raisebox{-.2cm}{\includegraphics[scale=1]{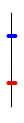}}
		\\
		\includegraphics[scale=.9]{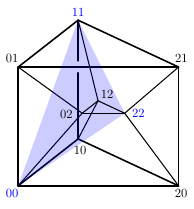} &
		\includegraphics[scale=.9]{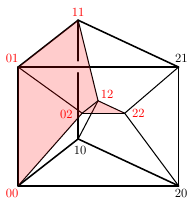} &
		\includegraphics[scale=.6]{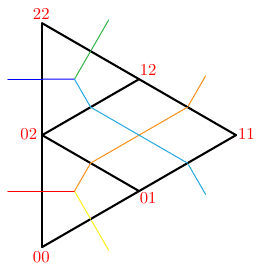} &
		\includegraphics[scale=.6]{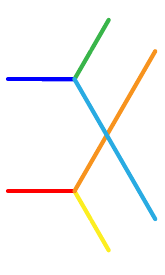}	
		\\
		\includegraphics[scale=.9]{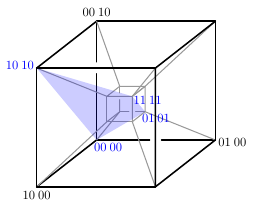} &
		\includegraphics[scale=.9]{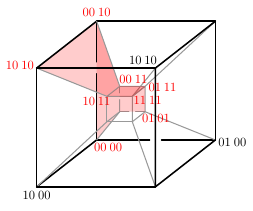} &
		\includegraphics[scale=.6]{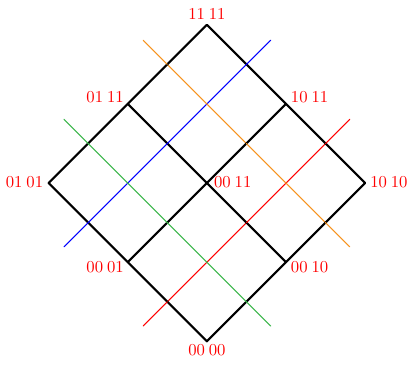} &
		\includegraphics[scale=.6]{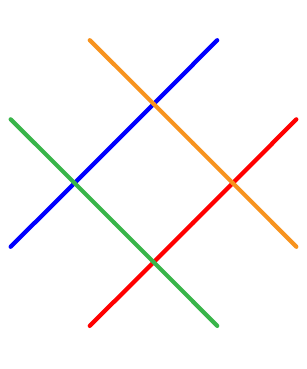}
	\end{tabular}
	}}
	\caption{Cellular diagonals of the segment (top), the triangle (middle) and the square (bottom). For each of them, we have represented the thin diagonal of~$P$ (left, in blue), a cellular diagonal of~$P$ (middle left, in red) both in~$P \times P$, the associated polytopal subdivision of~$P$ (middle right) and the common refinement of the two copies of the normal fan of~$P$ (right) both in~$P$.}
	\label{fig:examplesDiagonals1}
\end{figure}

\begin{figure}
	\centerline{
		\includegraphics[scale=.3]{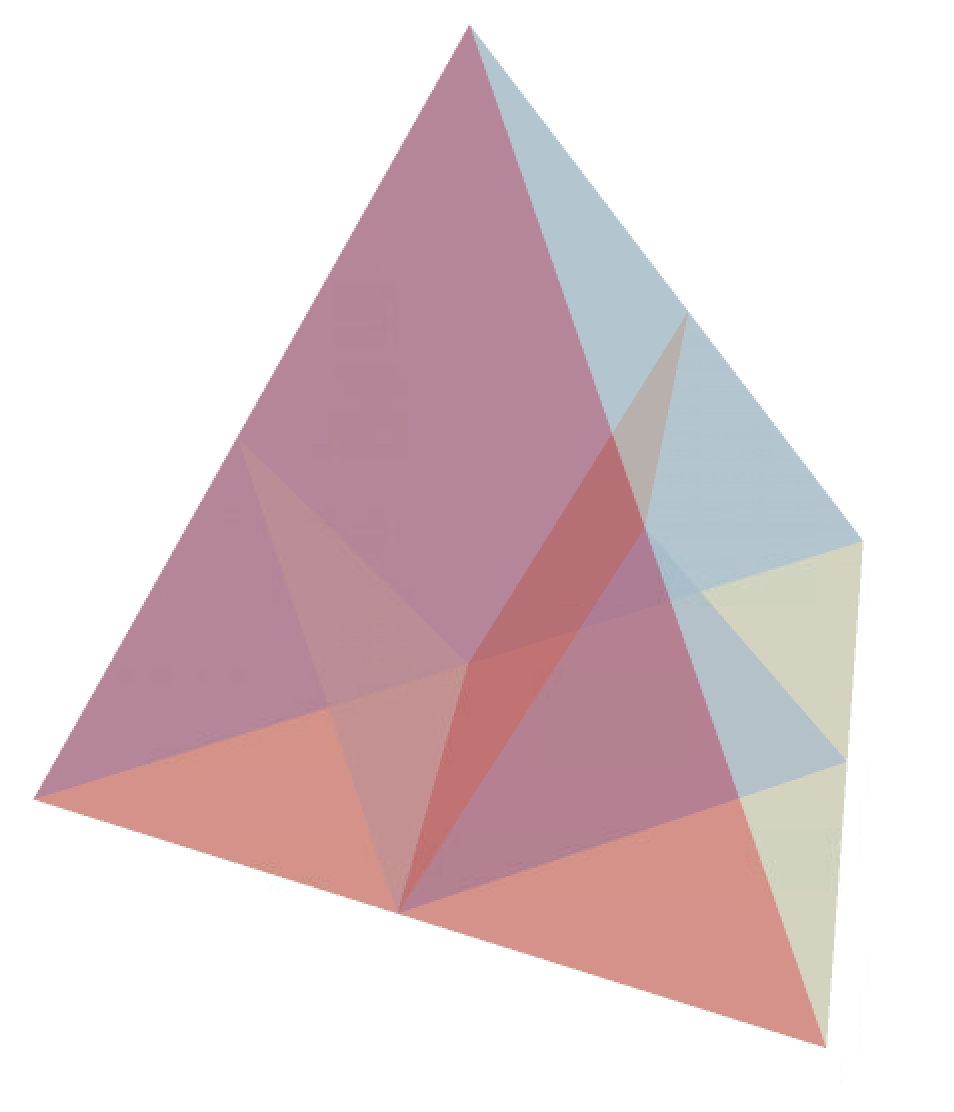}
		\includegraphics[scale=.35]{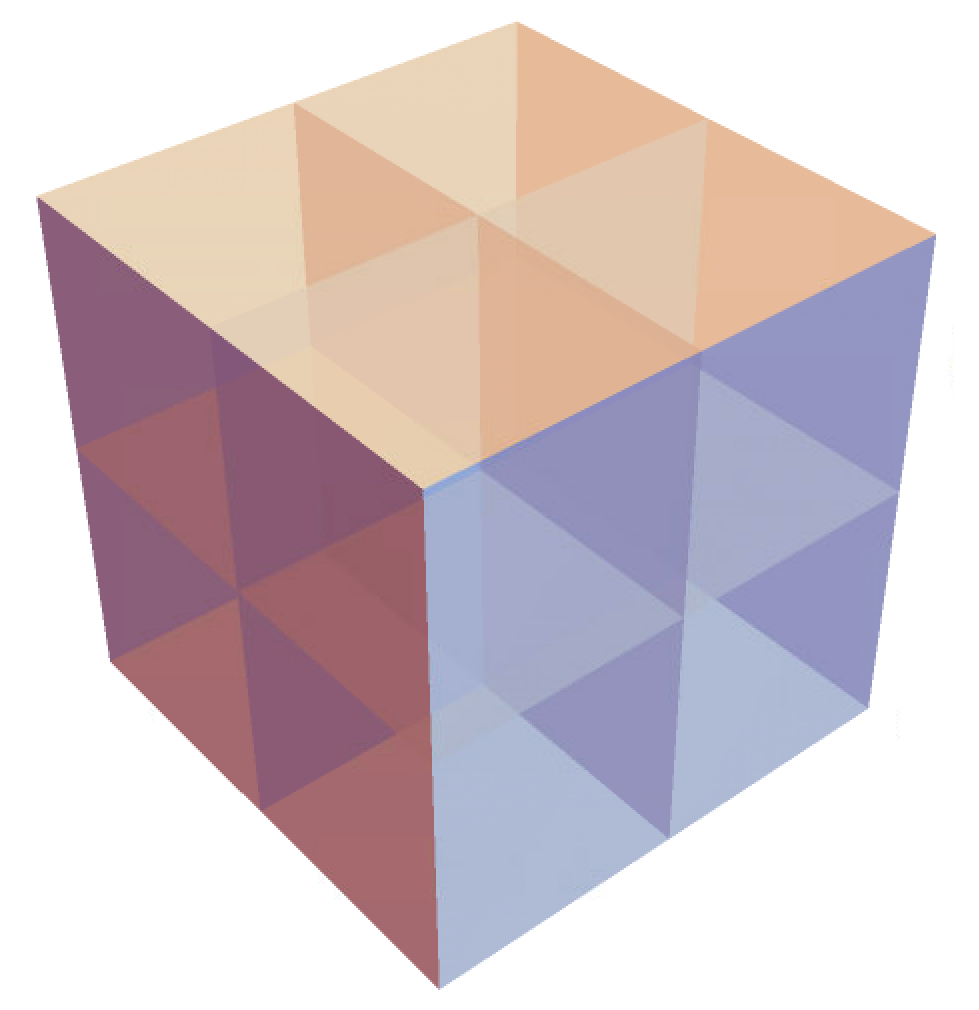}
		\includegraphics[scale=.3]{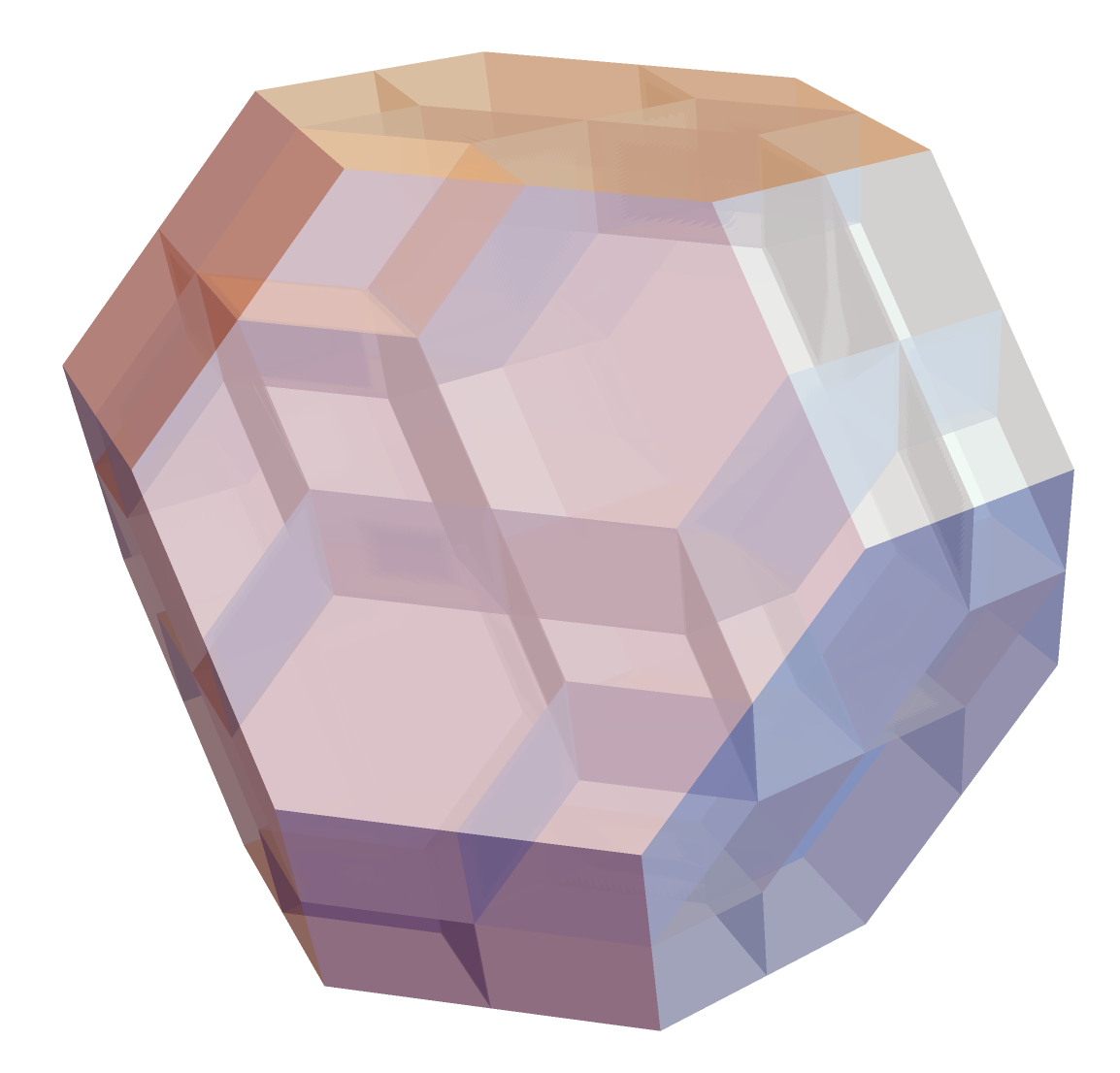}
	}
	\caption{The subdivisions induced by cellular diagonals of the $3$-dimensional simplex (left), cube (middle), and permutahedron (right). Right illustration from~\cite[Fig.~13]{LaplanteAnfossi}.}
	\label{fig:examplesDiagonals2}
\end{figure}

Now the following \defn{universal formula}~\cite[Thm.~1.26]{LaplanteAnfossi} expresses combinatorially the faces in the image of the geometric diagonal~$\triangle_{(P,\b{v})}$.
Recall that the \defn{normal cone} of a face~$F$ of a polytope~$P$ in~$\R^d$ is the cone of directions~$\b{c} \in \R^d$ such that the maximum of the scalar product~$\dotprod{\b{c}}{\b{x}}$ over~$P$ is attained for some~$\b{x}$ in~$F$.

\begin{theorem}[{\cite[Thm.~1.26]{LaplanteAnfossi}}]
\label{thm:universalFormula}
Fix a vector $\b{v} \in \R^d$ generic with respect to $P$.
For each hyperplane~$H$ of the fundamental hyperplane arrangement~$\mathcal{H}_P$, denote by~$H^{\b{v}}$ the open half space defined by~$H$ and containing~$\b{v}$.
The faces of~$P \times P$ in the image of the geometric diagonal~$\triangle_{(P,\b{v})}$ are the faces~$F \times G$ where~$F$ and~$G$ are faces of~$P$ such that either the normal cone of~$F$ intersects~$H^{-\b{v}}$ or the normal cone of~$G$ intersects~$H^{\b{v}}$, for each~$H \in \mathcal{H}_P$.
\end{theorem}

The image of $\triangle_{(P,\b{v})}$ is a union of pairs of faces $F \times G$ of the Cartesian product~$P \times P$.
By drawing the polytopes~${(F+G)/2}$ for all pairs of faces $(F,G) \in \Ima \triangle_{(P,\b{v})}$, we can visualize~$\triangle_{(P,\b{v})}$ as a polytopal subdivision of $P$.
See \cref{fig:examplesDiagonals1}\,(middle right) and \cref{fig:examplesDiagonals2}.

It turns out that the dual of this complex is just the common refinement of two translated copies of the normal fan of~$P$.
See \cref{fig:examplesDiagonals1}\,(right).
Recall that the \defn{normal fan} of~$P$ is the fan formed by the normal cones of all faces of~$P$.
We thus obtain the following statement.

\begin{proposition}[{\cite[Coro.~1.4]{LaplanteAnfossi}}]
\label{prop:diagonalCommonRefinement}
The inclusion poset on the faces in the image of the diagonal~$\triangle_{(P,\b{v})}$ is isomorphic to the reverse inclusion poset on the faces of the common refinement of two copies of the normal fan of~$P$, translated from each other by the vector~$\b{v}$. 
\end{proposition}

Finally, the following statement relates the image of the diagonal~$\triangle_{(P, \b{v})}$ to the intervals of the poset obtained by orienting the skeleton of~$P$ in direction~$\b{v}$.

\begin{proposition}[{\cite[Prop. 1.17]{LaplanteAnfossi}}]
\label{prop:magicalFormula}
For any polytope $P$ and any generic vector $\b{v}$, we have
\begin{equation}
\label{eq:magique}
\Ima\triangle_{(P, \b{v})} \quad \subseteq \bigcup_{\substack{F,G \text{ faces of } P \\ \max_{\b{v}}(F) \, \le \, \min_{\b{v}}(G)}} F \times G .
\end{equation}
\end{proposition}

\begin{remark}
\label{rem:magicalFormula}
For some polytopes such as the simplices~\cite{EilenbergMacLane}, the cubes~\cite{Serre}, the freehedra~\cite{Saneblidze-freeLoopFibration}, and the associahedra~\cite{MasudaThomasTonksVallette}, the reverse inclusion also holds (in the case of the simplices and the cubes, the diagonals are known as the \emph{Alexander--Whitney} map~\cite{EilenbergMacLane} and \emph{Serre} map~\cite{Serre}).
According to~\cite{MasudaThomasTonksVallette}, the resulting equality enhancing~(\ref{eq:magique}) was called \defn{magical formula} by J.-L. Loday.
This equality simplifies the computation of the $f$-vectors of the diagonals.
For instance, the number of $k$-dimensional faces in the diagonal of the $(n-1)$-dimensional simplex, cube, and associahedron are respectively given by
\begin{alignat*}{2}
f_k(\triangle_{\Simplex}) & = (k+1) \binom{n+1}{k+2} && \text{\OEIS{A127717}},
% choose k+2 points of [n] where 2 consecutive are distinguished and might be equal, and the rest is distinct
% this is the same as choosing k+2 points of [n+1], and the position of the consecutive pair of distinguished points among them
% hence (k+1) \binom{n+1}{k+2}
\\
f_k(\triangle_{\Cube}) & = \binom{n-1}{k} 2^k 3^{n-1-k} && \text{\OEIS{A038220},}
% choose a < b <= c < d in the boolean lattice such that |b-a| + |d-c| = k.
% choose the positions of the ones in (b-a) + (d-c) => \binom{n}{k}
% choose whether each of this ones is in (b-a) or in (d-c) => 2^k
% choose the values of b and c on the remaining n-k positions to be either 00, 01 or 11 => 3^(n-k)
\\
f_k(\triangle_{\Asso}) & = \frac{2}{(3n+1)(3n+2)} \binom{n-1}{k} \binom{4n+1-k}{n+1} \qquad && \text{\cite{BostanChyzakPilaud}.}
\end{alignat*}
Polytopes of greater complexity such as the multiplihedra~\cite{LaplanteAnfossiMazuir} or the operahedra~\cite{LaplanteAnfossi}, which include the permutahedra, do not possess this exceptional property, and the $f$-vectors of their diagonals are harder to compute.
\end{remark}

\begin{remark}
\label{rem:Fulton--Sturmfels}
This is in fact precisely the content of the \emph{Fulton--Sturmfels formula} for the intersection product on toric varieties~\cite[Thm.~4.2]{FultonSturmfels}, where the definition of cellular diagonals as tight coherent sections first appeared.
\end{remark}

To conclude, we define the opposite of a geometric diagonal.

\begin{definition}
\label{def:oppositeDiagonal}
The \defn{opposite} of the geometric diagonal~$\triangle \eqdef \triangle_{(P, \b{v})}$ for a vector~$\b{v} \in \R^d$ generic with respect to $P$ is the geometric diagonal~$\triangle^{\op} \eqdef \triangle_{(P, -\b{v})}$ for the vector~$-\b{v}$.
Observe that
\[{
F \times G \in \Ima \triangle} \qquad\iff\qquad G \times F \in \Ima \triangle^{\op}.
\]
\end{definition}

%%%%%%%%%%%%%%%

\subsection{Cellular diagonals for the permutahedra}
\label{sec:cellularDiagonalsPermutahedra}

We now specialize the statements of \cref{subsec:cellularDiagonalsPolytopes} to the standard permutahedron.
We first recall its definition.

\begin{definition}
\label{def:permutahedron}
The \defn{permutahedron}~$\Perm$ is the polytope in~$\R^n$ defined equivalently as
\begin{itemize}
\item the convex hull of the points~$\sum_{i \in [n]} i \, \b{e}_{\sigma(i)}$ for all permutations~$\sigma$ of~$[n]$, see~\cite{Schoute},
\item the intersection of the hyperplane~$\smash{\bigset{\b{x} \in \R^n}{\sum_{i \in I} x_i = \binom{n+1}{2}}}$ with the affine halfspaces $\smash{\bigset{\b{x} \in \R^n}{\sum_{i \in I} x_i \ge \binom{\card{I}+1}{2}}}$ for all~${\varnothing \ne I \subsetneq [n]}$, see~\cite{Rado}.
\end{itemize}
The normal fan of the permutahedron~$\Perm$ is the fan defined by the braid arrangement~$\braidArrangement$.
In particular, the faces of~$\Perm$ correspond to the ordered partitions of~$[n]$.
Moreover, when oriented in a generic direction, the skeleton of the permutahedron~$\Perm$ is isomorphic to the Hasse diagram of the classical weak order on permutations of~$[n]$.
% when oriented in the direction~${\b{\omega} \eqdef (n,\dots,1) - (1,\dots,n) = \sum_{i \in [n]} (n+1-2i) \, \b{e}_i}$
See \cref{fig:permutahedron}.
\begin{figure}
	\centerline{
		\includegraphics[scale=.8]{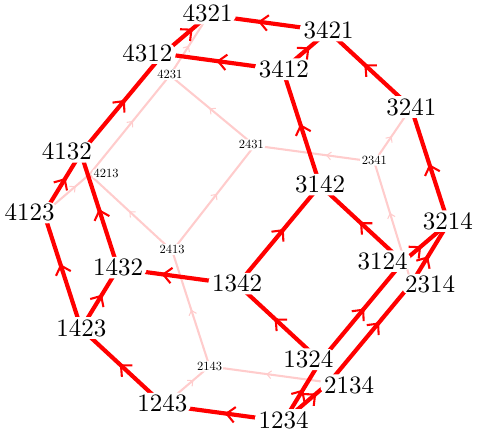}
		\qquad
		\includegraphics[scale=.7]{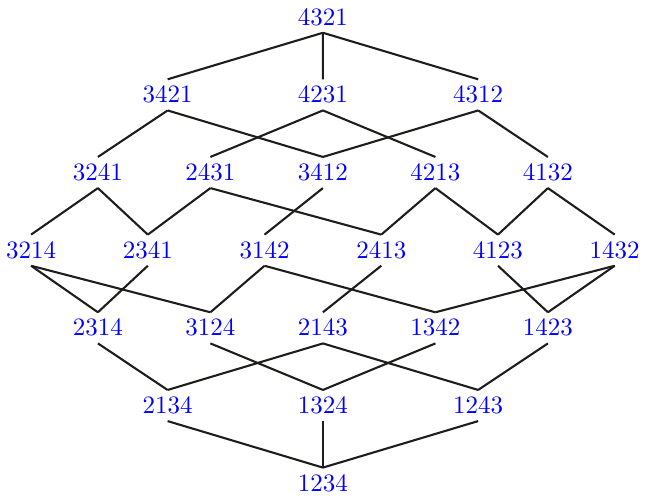}
	}
	\caption{The permutahedron~$\Perm[4]$ (left) and the weak order on permutations of~$[4]$ (right).}
	\label{fig:permutahedron}
\end{figure}
\end{definition}

The fundamental hyperplane arrangement of the permutahedron~$\Perm$ was described in~\cite[Sect.~3.1]{LaplanteAnfossi}.
As we will use the following set throughout the paper, we embed it in a definition.

\begin{definition}
\label{def:Un}
For~$n \in \N$, we define
\[
\Un(n) \eqdef \bigset{\{I,J\}}{I,J \subset [n] \text{ with } \card{I}=\card{J} \text{ and } I \cap J = \varnothing}.
\]
An \defn{ordering} of~$\Un(n)$ is a set containing exactly one of the two ordered pairs $(I,J)$ or $(J,I)$ for each $\{I,J\} \in \Un(n)$.
\end{definition}

\begin{proposition}[{\cite[Sect.~3.1]{LaplanteAnfossi}}]
The fundamental hyperplane arrangement of the permutahedron~$\Perm$ is given by the hyperplanes~$\bigset{\b{x} \in \R^n}{\sum\limits_{i \in I} x_i = \sum\limits_{j \in J} x_j}$ for all~$\{I,J\} \in \Un(n)$.
\end{proposition}

For a vector~$\b{v}$ generic with respect to~$\Perm$, we denote by~$\Or(\b{v})$ the ordering of~$\Un(n)$ such that~$\sum_{i \in I} v_i > \sum_{j \in J} v_j$ for all~$(I,J) \in \Or(\b{v})$.
Applying \cref{thm:universalFormula}, we next describe the faces in the image of the geometric diagonal~$\triangle_{(\Perm, \b{v})}$.
For this, the following definition will be convenient.

\begin{definition}
	\label{def:domination}
For $I,J \subseteq [n]$ and an ordered partition $\sigma$ of~$[n]$ into $k$ blocks, we say that $I$~\defn{dominates} $J$ in~$\sigma$ if for all $\ell \in [k]$, the first~$\ell$ blocks of~$\sigma$ contain at least as many elements of~$I$ than of~$J$.
\end{definition}

\begin{theorem}[{\cite[Thm.~3.16]{LaplanteAnfossi}}]
\label{thm:IJ-description}
A pair~$(\sigma, \tau)$ of ordered partitions of~$[n]$ corresponds to a face in the image of the geometric diagonal~$\triangle_{(\Perm, \b{v})}$ if and only if, for all~$(I,J) \in \Or(\b{v})$, $J$ does not dominate~$I$ in~$\sigma$ or $I$ does not dominate~$J$ in~$\tau$.
\end{theorem}

As the normal fan of the permutahedron~$\Perm$ is the fan defined by the braid arrangement~$\braidArrangement$, we obtain by \cref{prop:diagonalCommonRefinement} the following connection between the diagonal of~$\Perm$ and the $(2,n)$-braid arrangement~$\multiBraidArrangement[2][n]$ studied in \cref{part:multiBraidArrangements}.
This connection is illustrated in \cref{fig:diagonalPermutahedron1}.
\begin{figure}
	\centerline{\includegraphics[scale=.7]{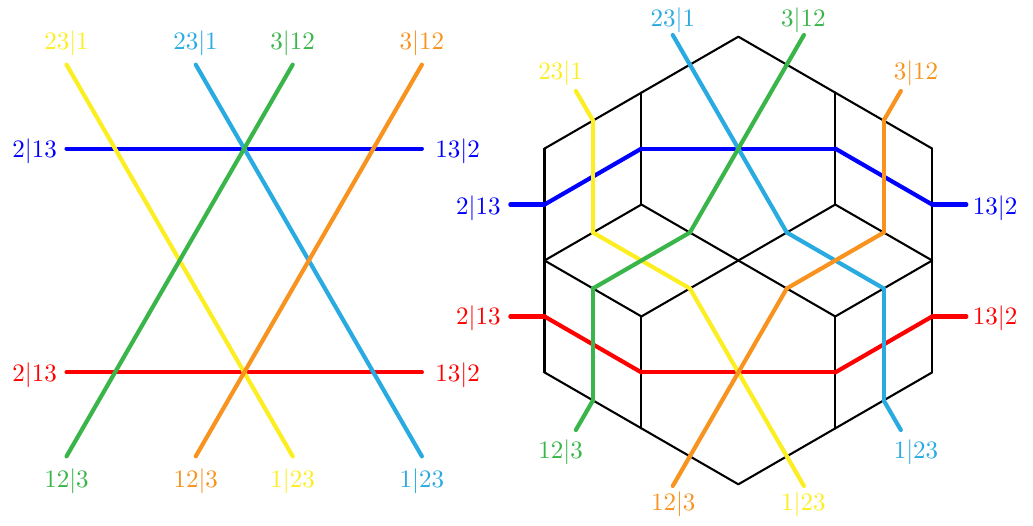}}
	\caption{The duality between the $(2,3)$-braid arrangement~$\multiBraidArrangement[3][2]$ (left) and the diagonal of the permutahedron~$\Perm[3]$ (right).}
	\label{fig:diagonalPermutahedron1}
\end{figure}

\begin{proposition}
\label{prop:diagonalPermutahedraMultiBraidArrangements}
The inclusion poset on the faces in the image of the diagonal~$\triangle_{(\Perm, \b{v})}$ is isomorphic to the reverse inclusion poset on the faces of the $(2,n)$-braid arrangement~$\multiBraidArrangement[2][n]$. 
\end{proposition}

\begin{remark}
\label{rem:translationMatrix}
To be more precise, the diagonal~$\triangle_{(\Perm, \b{v})}$ is dual to the $\b{a}$-braid arrangement~$\multiBraidArrangement[2][n](\b{a})$ where the translation matrix~$\b{a}$ has two rows, with first row~${\b{a}_{1,j} = 0}$ and second row~$\b{a}_{2,j} = v_j-v_{j+1}$ for all~$j \in [n-1]$.
Since~$\b{v}$ is generic with respect to~$\Perm$, we have~$\sum_{i \in I} v_i \ne \sum_{j \in J} v_j$ for all~$(I,J) \in \Un(n)$, so that~$\b{a}$ is indeed generic.
\end{remark}

\begin{remark}
Similarly, the combinatorics of the $(\ell-1)$\ordinalst{} iteration of a diagonal of the permutahedron~$\Perm$ is given by the combinatorics of the $(\ell,n)$-braid arrangement.
\end{remark}

Finally, the permutahedron~$\Perm$ is not magical in the sense of \cref{prop:magicalFormula}.
See also \cref{exm:intervalsWeakOrderNotDiagonals}.

\begin{proposition}[{\cite[Sect.~3]{LaplanteAnfossi}}]
For any~$n > 3$ and any generic vector~$\b{v}$, the diagonal~$\triangle_{(\Perm, \b{v})}$ of the permutahedron~$\Perm$ does \emph{not} satisfy the magical formula.
Namely, we have the strict inclusion	  
\begin{equation*}
\Ima\triangle_{(\Perm, \b{v})} \quad \subsetneq \bigcup_{\substack{F,G \text{ faces of } \Perm \\ \max_{\b{v}}(F) \, \le \, \min_{\b{v}}(G)}} F \times G .
\end{equation*}
\end{proposition}

\begin{remark}
\label{rem:oppositeDiagonals}
Note that opposite diagonals have opposite orderings.
Namely, $\Or(-\b{v}) = \Or(\b{v})^{\op}$ where~$\Or^{\op} \eqdef \set{(J,I)}{(I,J) \in \Or}$.
\end{remark}

%%%%%%%%%%%%%%%

\subsection{Enumerative results on cellular diagonals of the permutahedra} 
\label{subsec:enumerationDiagonalPermutahedra}

Relying on \cref{prop:diagonalPermutahedraMultiBraidArrangements}, we now specialize the results of \cref{part:multiBraidArrangements} to the case~$\ell = 2$ to derive enumerative results on the diagonals of the permutahedra.

Observe that one can easily compute the full M\"obius polynomials of the $(2,n)$-braid arrangements~$\multiBraidArrangement[n][2]$ from \cref{thm:MobiusPolynomialMultiBraidArrangement}:
\begin{align*}
\allowdisplaybreaks
\mobPol[{\multiBraidArrangement[1][2]}] & = 1 , \\
\mobPol[{\multiBraidArrangement[2][2]}] & = x y - 2 x + 2 , \\
\mobPol[{\multiBraidArrangement[3][2]}] & = x^2 y^2 - 6 x^2 y + 10 x^2 + 6 x y - 18 x + 8 , \\
\mobPol[{\multiBraidArrangement[4][2]}] & = x^3 y^3 - 12 x^3 y^2 + 52 x^3 y - 84 x^3 + 12 x^2 y^2 - 96 x^2 y + 216 x^2 + 44 x y - 182 x + 50 , \\
\mobPol[{\multiBraidArrangement[5][2]}] & = x^4 y^4 - 20 x^4 y^3 + 160 x^4 y^2 - 620 x^4 y + 1008 x^4 \\ & \quad+ 20 x^3 y^3 - 300 x^3 y^2 + 1640 x^3 y - 3360 x^3 \\ & \quad+ 140 x^2 y^2 - 1430 x^2 y + 4130 x^2 + 410 x y - 2210 x + 432 .
\end{align*}
%which can be encoded in matrices as
%{\small
%\[
%\left(\begin{array}{r}
%1
%\end{array}\right)
%\left(\begin{array}{rr}
%2 & -2 \\
%0 & 1
%\end{array}\right)
%\left(\begin{array}{rrr}
%8 & -18 & 10 \\
%0 & 6 & -6 \\
%0 & 0 & 1
%\end{array}\right)
%\left(\begin{array}{rrrr}
%50 & -182 & 216 & -84 \\
%0 & 44 & -96 & 52 \\
%0 & 0 & 12 & -12 \\
%0 & 0 & 0 & 1
%\end{array}\right)
%\left(\begin{array}{rrrrr}
%432 & -2210 & 4130 & -3360 & 1008 \\
%0 & 410 & -1430 & 1640 & -620 \\
%0 & 0 & 140 & -300 & 160 \\
%0 & 0 & 0 & 20 & -20 \\
%0 & 0 & 0 & 0 & 1
%\end{array}\right)
%%\left(\begin{array}{r}
%%1
%%\end{array}\right)
%%\left(\begin{array}{rr}
%%2 & 0 \\
%%-2 & 1
%%\end{array}\right)
%%\left(\begin{array}{rrr}
%%8 & 0 & 0 \\
%%-18 & 6 & 0 \\
%%10 & -6 & 1
%%\end{array}\right)
%%\left(\begin{array}{rrrr}
%%50 & 0 & 0 & 0 \\
%%-182 & 44 & 0 & 0 \\
%%216 & -96 & 12 & 0 \\
%%-84 & 52 & -12 & 1
%%\end{array}\right)
%%\left(\begin{array}{rrrrr}
%%432 & 0 & 0 & 0 & 0 \\
%%-2210 & 410 & 0 & 0 & 0 \\
%%4130 & -1430 & 140 & 0 & 0 \\
%%-3360 & 1640 & -300 & 20 & 0 \\
%%1008 & -620 & 160 & -20 & 1
%%\end{array}\right)
%%\left(\begin{array}{rrrrrr}
%%4802 & 0 & 0 & 0 & 0 & 0 \\
%%-31922 & 4732 & 0 & 0 & 0 & 0 \\
%%82560 & -23100 & 1830 & 0 & 0 & 0 \\
%%-104400 & 41560 & -6210 & 340 & 0 & 0 \\
%%64800 & -32760 & 6960 & -720 & 30 & 0 \\
%%-15840 & 9568 & -2580 & 380 & -30 & 1
%%\end{array}\right)
%\]
%}

We now focus on the number of vertices, regions, and bounded regions of the $(2,n)$-braid arrangement~$\multiBraidArrangement[n][2]$, to obtain the number of facets, vertices, and internal vertices of the diagonal of the permutahedron~$\Perm$.
The first few values are gathered in \cref{table:enumerationDiagonalPermutahedra1,table:enumerationDiagonalPermutahedra2}.

\afterpage{
\begin{table}
	\centerline{\scalebox{1}{
		\begin{tabular}[t]{c|ccccccccccc}
			$n$ & $1$ & $2$ & $3$ & $4$ & $5$ & $6$ & $7$ & $8$ & $9$ & $\dots$ & OEIS \\
			\hline
			facets & $1$ & $2$ & $8$ & $50$ & $432$ & $4802$ & $65536$ & $1062882$ & $20000000$ & $\dots$ & \OEIS{A007334} \\
			vertices & $1$ & $3$ & $17$ & $149$ & $1809$ & $28399$ & $550297$ & $12732873$ & $343231361$ & $\dots$ & \OEIS{A213507} \\
			int. verts. & $1$ & $1$ & $5$ & $43$ & $529$ & $8501$ & $169021$ & $4010455$ & $110676833$ & $\dots$ & \OEIS{A251568}
		\end{tabular}
	}}
%	\vspace{.3cm}
	\caption{The numbers of facets, vertices, and internal vertices of the diagonal of the permutahedron~$\Perm$ for~$n \in [9]$.}
	\label{table:enumerationDiagonalPermutahedra1}
\end{table}
}

\afterpage{
\begin{table}
	\centerline{
		\begin{tabular}{c@{\quad}c@{\quad}c@{\quad}c}
			$n = 1$ & $n = 2$ & $n = 3$ & $n = 4$
			\\
			\begin{tabular}[t]{r|c}
				\textbf{dim} & \textbf{0}  \\
				\hline
				\textbf{0} & 1  
			\end{tabular}
			&
			\begin{tabular}[t]{r|cc}
				\textbf{dim} & \textbf{0} & \textbf{1}  \\
				\hline
				\textbf{0} & 3 & 1 \\
				\textbf{1} & 1 &  
			\end{tabular}
			&
			\begin{tabular}[t]{r|ccc}
				\textbf{dim} & \textbf{0} & \textbf{1} & \textbf{2}  \\
				\hline
				\textbf{0} & 17 & 12 & 1 \\
				\textbf{1} & 12 &  6 & \\
				\textbf{2} & 1 &  & 
			\end{tabular}
			&
			\begin{tabular}[t]{r|cccc}
				\textbf{dim} & \textbf{0} & \textbf{1} & \textbf{2} & \textbf{3} \\
				\hline
				\textbf{0} & 149 & 162 & 38 & 1 \\
				\textbf{1} & 162 & 150 & 24 & \\
				\textbf{2} & 38 & 24 & & \\
				\textbf{3} & 1 & & &
			\end{tabular}
		\end{tabular}
	}
	\vspace{.3cm}
	\centerline{
		\begin{tabular}{c@{\quad}c}
			$n = 5$ & $n = 6$
			\\
			\begin{tabular}[t]{r|ccccc}
				\textbf{dim} & \textbf{0} & \textbf{1} & \textbf{2} & \textbf{3} & \textbf{4} \\
				\hline
				\textbf{0} & 1809 & 2660 & 1080 & 110 & 1 \\
				\textbf{1} & 2660 & 3540 & 1200 & 80 & \\
				\textbf{2} & 1080 & 1200 & 270 & & \\
				\textbf{3} & 110 & 80 & && \\
				\textbf{4} & 1 & & & &
			\end{tabular}
			&
			\begin{tabular}[t]{r|cccccc}
				\textbf{dim} & \textbf{0} & \textbf{1} & \textbf{2} & \textbf{3} & \textbf{4} & \textbf{5} \\
				\hline
				\textbf{0} & 28399 & 52635 & 30820 & 6165 & 302 & 1 \\
				\textbf{1} & 52635 & 90870 & 67580 & 7785 & 240 & \\
				\textbf{2} & 30820 & 47580 & 20480 & 2160 & & \\
				\textbf{3} & 6165 & 7785 & 2160 & && \\
				\textbf{4} & 302 & 240 & & &&\\
				\textbf{5} & 1 & & & &&
			\end{tabular}
		\end{tabular}
	}
	\caption{Number of pairs of faces in the cellular image of the diagonal of the permutahedron~$\Perm$ for~$n \in [6]$.}
	\label{table:enumerationDiagonalPermutahedra2}
\end{table}
}

\begin{corollary}
\label{coro:enumerationDiagonalPermutahedra}
The diagonal of the permutahedron~$\Perm$ has 
\begin{itemize}
\item $2 (n + 1)^{n-2}$ facets,
\item $n \binom{n-1}{k_1} (n-k_1)^{k_1-1} (n-k_2)^{k_2-1}$ facets corresponding to pairs~$(F_1, F_2)$ of faces of the permutahedron~$\Perm$ with~$\dim(F_1) = k_1$ and~$\dim(F_2) = k_2$ (thus~$k_1 + k_2 = n-1$),
\item $\displaystyle n! \, [z^n] \exp \bigg( \sum_{m \ge 1} \frac{C_m \, z^m}{m} \bigg)$ vertices,
\item $\displaystyle (n-1)! \, [z^{n-1}] \exp \bigg( \sum_{m \ge 1} C_m \, z^m \bigg)$ internal vertices,
\end{itemize}
where~$\displaystyle C_m \eqdef \frac{1}{m+1} \binom{2m}{m}$ denotes the $m$\ordinal{} Catalan number.
\end{corollary}

\begin{proof}
Use the duality between the $(2,n)$-braid arrangement~$\multiBraidArrangement[n][2]$ and the diagonal of the permutahedron~$\Perm$ (see \cref{prop:diagonalPermutahedraMultiBraidArrangements,fig:diagonalPermutahedron1}), and specialize \cref{thm:verticesMultiBraidArrangement,thm:verticesRefinedMultiBraidArrangement,thm:regionsMultiBraidArrangement} to the case~$\ell = 2$.
\end{proof}

\begin{remark}
For completeness, we provide an alternative simpler proof of the first point of \cref{coro:enumerationDiagonalPermutahedra}.
By \cref{prop:flatPosetMultiBraidArrangement}, we just need to count the $(2,n)$-partition trees.
Consider a $(2,n)$-partition tree~$\b{F} \eqdef (F_1,F_2)$ (hence~$\card{F_1} + \card{F_2} = n + 1$).
Consider the intersection tree~$T$ of~$\b{F}$ with vertices labeled by the parts of~$F_1$ and of~$F_2$ and edges labeled by~$[n]$, root~$T$ at the part of~$F_1$ containing vertex~$1$, forget the vertex labels of~$T$, and send each edge label of~$T$ to the next vertex away from the root, and label the root by~$0$.
See \cref{fig:tree}.
The result is a spanning tree of the complete graph~$K_{n+1}$ on~$\{0, \dots, n\}$ which must contain the edge~$(0,1)$ (because we have chosen the root to be the part of~$F_1$ containing~$1$).
\afterpage{
\begin{figure}
	\centerline{\includegraphics[scale=.9]{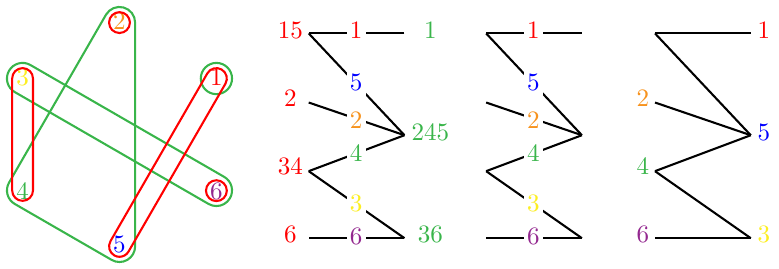}}
	\caption{The bijection from rooted $(\ell,n)$-partition trees (left) to spanning trees of~$K_{n+1}$ containing the edge~$(0,1)$ (right).}
	\label{fig:tree}
\end{figure}
}
Finally, by double counting the pairs~$(T,e)$ where $T$ is a spanning tree of~$K_{n+1}$ and $e$ is an edge of~$T$, we see that $n$ times the number of spanning trees of~$K_{n+1}$ equals $\binom{n+1}{2}$ times the number of spanning trees of~$K_{n+1}$ containing~$(0,1)$.
Hence, by Cayley's formula for spanning trees of~$K_{n+1}$, we obtain that the number of $(2,n)$-partition trees~is
\[
\frac{2n}{n(n+1)} (n+1)^{n-1} = 2 (n + 1)^{n-2}.
\]
\end{remark}

%%%%%%%%%%%%%%%%%%%%%%%%%%%%%%%%%%%%%%

\pagebreak
\section{Operadic diagonals}
\label{sec:operadicDiagonals}

This section is devoted to the combinatorics of two specific diagonals of the permutahedra, the $\LA$ and $\SU$ diagonals, which are shown to be the only operadic geometric diagonals of the permutahedra.

%%%%%%%%%%%%%%%

\subsection{The $\LA$ and $\SU$ diagonals}
\label{subsec:LASUdiagonal}

Recall from \cref{sec:cellularDiagonalsPermutahedra} that a geometric diagonal of the permutahedron~$\Perm$ gives a choice of ordering of the sets~$\Un(n)$ (see \cref{def:Un}).
As we consider in this section diagonals for all permutahedra, we now consider families of orderings.

\begin{definition}
An \defn{ordering} of $\Un \eqdef \{\Un(n)\}_{n \ge 1}$ is a family~$\Or \eqdef \{\Or(n)\}_{n \ge 1}$ where~$\Or(n)$ is an ordering of~$\Un(n)$ for each~$n \ge 1$.
\end{definition}

We will be focusing on the following two orderings and their corresponding diagonals.

\begin{definition}
The \defn{$\LA$} and~\defn{$\SU$ orderings} are defined by
\begin{itemize} 
	\item $\LA(n) \eqdef \set{(I,J)}{\{I,J\} \in \Un(n) \text{ and } \min(I\cup J) = \min I}$ and
	\item $\SU(n) \eqdef \set{(I,J)}{\{I,J\} \in \Un(n) \text{ and } \max(I\cup J) = \max J}$.
\end{itemize}
\end{definition}

\begin{definition}
\label{def:LA-and-SU}
The \defn{$\LA$ diagonal} $\LAD$ (\resp \defn{$\SU$ diagonal} $\SUD$) of the permutahedron~$\Perm$ is the geometric diagonal~$\triangle_{(\Perm,\b{v})}$ given by any vector $\b{v} \in \R^n$ satisfying  
\[
\sum_{i \in I} v_i > \sum_{j \in J} v_j,
\]
for all~$(I,J) \in \LA(n)$ (\resp $(I,J) \in \SU(n)$).
\end{definition}

First note that this definition makes sense, since vectors $\b{v}=(v_1,\ldots,v_n)$ such that $\Or(\b{v})$ is the $\LA$ or $\SU$ order do exist: take for instance $v_i\eqdef 2^{-i+1}$ for $\LAD$ and $v_i \eqdef 2^n - 2^{i-1}$ for $\SUD$.
Second, note that the $\LA$ and $\SU$ diagonals coincide up to dimension $2$, but differ in dimension~$\ge 3$.
The former is illustrated in \cref{fig:LUSAdiagonals}, with the faces labeled by ordered $(2,3)$-partition forests.

\begin{figure}
	\centerline{\includegraphics[scale=.85]{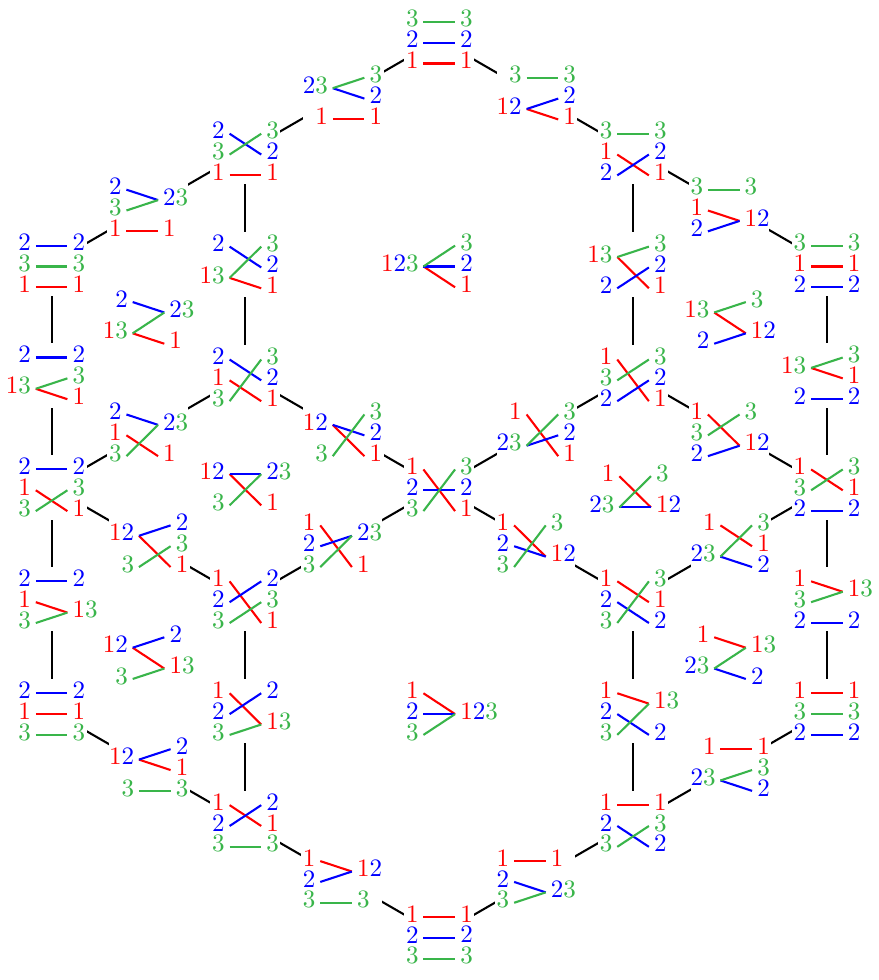}}
	\caption{The $\LA$ (and~$\SU$) diagonal of~$\Perm[3]$ with faces labeled by ordered $(2,3)$-partition forests. (See also \cref{fig:B23a} for the dual hyperplane arrangement.)}
	\label{fig:LUSAdiagonals}
\end{figure}

We will prove in \cref{thm:recover-SU} that $\SUD$ is a topological enhancement of the Saneblidze--Umble diagonal from~\cite{SaneblidzeUmble}.
The faces of the $\LA$ and $\SU$ diagonals are described by the following specialization of \cref{thm:IJ-description}.

\begin{theorem}
\label{thm:minimal}
A pair $(\sigma,\tau)$ of ordered partitions of $[n]$ is not a face of the $\LA$ diagonal~$\LAD$ if and only if there exists~$(I,J) \in \LA(n)$ such that~$J$ dominates~$I$ in~$\sigma$ and~$I$ dominates~$J$ in~$\tau$.
%\[
%(\sigma,\tau) \notin \LAD \iff \text{there is } (I,J) \in \LA(n) \text{ such that } J \text{ dominates } I \text{ in } \sigma \text{ and } I \text{ dominates } J \text{ in } \tau.
%\]
The same holds for the $\SU$ diagonal by replacing $\LAD$ by $\SUD$ and $\LA(n)$ by $\SU(n)$.
\end{theorem}

%%%%%%%%%%%%%%%

\subsection{The operadic property}
\label{subsec:operadicProperty}

The goal of this section is to prove that the $\LA$ and $\SU$ diagonals are the only two operadic diagonals which induce the weak order on the vertices of the permutahedra (\cref{thm:unique-operadic}). 
We start by properly defining operadic diagonals.

Let~$A \sqcup B = [n]$ be a partition of~$[n]$ where~$A \eqdef \{a_1,\dots,a_p\}$ and~$B \eqdef \{b_1,\dots,b_q\}$.
Recall that the ordered partition $B | A$ corresponds to a facet of the permutahedron~$\Perm$ defined by the inequality~$\sum_{b \in B} x_b \ge \binom{\card{B}+1}{2}$.
This facet is isomorphic to the Cartesian product
\[
(\Perm[p] +q \one_{[p]}) \times \Perm[q]
\]
of lower dimensional permutahedra, where the first factor is translated by~$q \one_{[p]} \eqdef q \sum_{i \in [p]} \b{e}_i$, via the permutation of coordinates
\begin{equation*}
	\begin{matrix}
		\Theta & : & \R^{p} \times \R^{q} & \overset{\cong}{\longrightarrow} & \R^{n} \\
		 & & (x_1,\ldots,x_p) \times (x_{p+1}, \ldots, x_{n})  & \longmapsto & (x_{\sigma^{-1}(1)},\ldots,x_{\sigma^{-1}(n)}) \ , 
	\end{matrix}
\end{equation*}
where $\sigma$ is the $(p,q)$-shuffle defined by $\sigma(i) \eqdef a_i$ for~$i \in [p]$, and $\sigma(p+j)\eqdef b_j$ for~$j \in [q]$.
Note that this map is a particular instance of the eponym map introduced in Point (5) of~\cite[Prop.~2.3]{LaplanteAnfossi}.

\begin{definition}
\label{def:operadicDiagonal}
A family of geometric diagonals $\triangle \eqdef \{\triangle_n : \Perm \to \Perm \times \Perm\}_{n\geq 1}$ of the permutahedra is \defn{operadic} if for every face $A_1 | \ldots | A_k$ of the permutahedron $\Perm[\card{A_1}+\cdots + \card{A_k}]$, the map $\Theta$ induces a topological cellular isomorphism
\[
\triangle_{\card{A_1}} \times \ldots \times \triangle_{\card{A_k}} \cong \triangle_{\card{A_1} + \ldots + \card{A_k}} .
\]
\end{definition}

In other words, we require that the diagonal $\triangle$ commutes with the map $\Theta$, see~\cite[Sect.~4.2]{LaplanteAnfossi}.
At the algebraic level, this property is called ``comultiplicativity" in~\cite{SaneblidzeUmble}.
Note that in particular, such an isomorphism respects the poset structures. 

We will now translate this operadic property for geometric diagonals to the corresponding orderings~$\Or$ of~$\Un$.
We need the following standardization map (this map is classical for permutations or words, but we use it here for pairs of sets).

\begin{definition}
The \defn{standardization} of a pair~$(I,J)$ of disjoint subsets of~$[n]$ is the only partition~$\std(I,J)$ of~$[\card{I}+\card{J}]$ where the relative order of the elements is the same as in~$(I,J)$.
More precisely, it is defined recursively by
\begin{itemize}
\item $\std(\varnothing, \varnothing) \eqdef (\varnothing, \varnothing)$, and
\item if~$k \eqdef \card{I}+\card{J}$ and~$\ell \eqdef \max(I \cup J)$ belongs to~$I$ (resp. $J$), then~$\std(I,J) = (U \cup \{k\}, V)$ (resp. $\std(I,J) \! = \! (U, V \cup \{k\})$) where~$(U,V) \! = \! \std(I \ssm \{\ell\}, J)$ (resp. ${(U,V) \! = \! \std(I, J \ssm \{\ell\})}$).
\end{itemize}
\end{definition}

For example, $\std(\{5,9,10\},\{6,8,12\}) = (\{1,4,5\},\{2,3,6\})$.

\begin{definition}
\label{def:operadicOrdering}
An ordering $\Or$ of $\Un$ is \defn{operadic} if every $\{I,J\} \in \Un$ satisfies the following two conditions
\begin{enumerate}
\item $\std(I,J) \in \Or$ implies~$(I,J) \in \Or$,
\item if there exist~$I'\subset I, J'\subset J$ such that $(I',J') \in \Or$ and $(I\ssm I',J\ssm J') \in \Or$, then~$(I,J) \in \Or$.
\end{enumerate}
We call \defn{indecomposables} of $\Or$ the pairs $(I,J) \in \Or$ for which the only subpair $(I',J')$ satisfying Condition (2) above is the pair $(I,J)$ itself.
\end{definition}

\begin{proposition}
\label{prop:equiv-operadic}
A geometric diagonal~$\triangle \eqdef \bigl( \triangle_{(\Perm, \b{v}_n)} \bigr)_{n \in \N}$ of the permutahedra is operadic if and only if its associated ordering~$\Or \eqdef \bigl( \Or(\b{v}_n) \bigr)_{n \in \N}$ is operadic. 
\end{proposition}

\begin{proof}
Every operadic diagonal satisfies~\cite[Prop.~4.14]{LaplanteAnfossi}, which amounts precisely to an operadic ordering of $\Un$ in the sense of \cref{def:operadicOrdering}.
\end{proof}

We now turn to the study of the $\LA$ and $\SU$ orderings.

\begin{lemma} 
\label{lem:operadic-ordering}
The $\LA$ and $\SU$ orderings are operadic, and their indecomposables are the pairs whose standardization are $(\{1\},\{2\})$ and
\begin{align}
(\{1,k+2,k+3,\dots,2k-1,2k\}, \{2,3,\dots,k+1\}) \tag{$\LA$} \label{eq:std-LA} \\
(\{k,k+1,\dots,2k-1\},\{1,2,3,\dots,k-1,2k\}) \tag{$\SU$}
\end{align} 
for all~$k\geq 1$. 
The opposite orderings $\LA^{op}$ and $\SU^{op}$ are also operadic, and generated by the opposite pairs.
\end{lemma}

\begin{proof}
We present the proof for the $\LA$ ordering, the proofs for the $\SU$ and opposite orders are similar.
First, we verify that $\LA$ is operadic. 
Condition (1) follows from the fact that standardizing a pair preserves its minimal element.
Condition (2) also holds, since whenever $(I',J')$ and its complement are in $\LA$, we have $\min(I)=\min\{\min(I'),\min(I\ssm I')\}$, and thus $(I,J)$ itself is in~$\LA$.

Second, we compute the indecomposables. 
Let $(I,J)$ be a pair with standardization (\ref{eq:std-LA}).
If we try to decompose $(I,J)$ as a non-trivial union, there is always one pair $(I',J')$ in this union for which~$\min (I\cup J) \notin I'$, so we have $\min ( I' \cup J') = \min J'$, which implies that $(I',J') \notin \LA$.
Thus, the pair~$(I,J)$ is indecomposable. 

It remains to show that any pair $(I,J)$ in $\LA$ whose standardization is not of the form (\ref{eq:std-LA}) can be decomposed as a union of such pairs. 
Let us denote by $(I_k,J_k)$ the standard form (\ref{eq:std-LA}) and let $(I,J)\in \LA$ be such that $\std(I,J) \neq (I_k,J_k)$.
Then there exists $i_2 \in I\ssm \min I$ such that $i_2 < \max J$.
This means that $(I,J)$ can be decomposed as a union: if we write it as $(\{i_1,\dots,i_k\},\{j_1,\dots,j_k\})$, where each set ordered smallest to largest, then we must have $\min (I\cup J)=i_1<i_2<j_k$, in which case $(\{i_2\},\{j_k\})$ and $(\{i_1,i_3,\dots,i_k\},\{j_1,\dots,j_{k-1}\})$ are both smaller $\LA$ pairs.
Then it must be the case that $\std((\{i_2\},\{j_k\})) = (\{1\},\{2\})$, and $\std((\{i_1,i_3,\dots,i_k\},\{j_1,\dots,j_{k-1}\}))$ is either $(I_{k-1},J_{k-1})$, or we can repeat this decomposition.
%This process must eventually terminate with the right-hand side reducing to $(I_l,J_l)$ for some $1 \leq l \leq k-1$.
%In other words, any ${(I,J) = (\{i_1,\dots,i_k\}, \{j_1,\dots,j_k\})}$ decomposes as
%\begin{align*}
%	(I,J) = (\{i_2\},\{j_k\}) \sqcup (\{i_3\},\{j_{k-1}\}) \sqcup \dots \sqcup (\{i_{l+1} \},\{j_{k-l-1} \}) \sqcup (I',J')
%\end{align*}
%where $\std((I',J')) = (I_l,J_l)$, and $1\leq l \leq k$.
\end{proof}

\begin{remark}
We note that the decomposition in \cref{lem:operadic-ordering} is one of potentially many different decompositions of the pair $(I,J)$. 
However, by definition of the $\LA$ order, for any decomposition $(I,J) = (\bigsqcup_{a\in A} I_a, \bigsqcup_{a \in A} J_a)$, we have $\std(I_a, J_a) \in \LA$ for all~$a \in A$.
As such, all decompositions of a pair $(I,J)$, order it the same way.
\end{remark}

\begin{proposition}
\label{prop:operadicOrdering}
The only operadic orderings of $\Un=\{\Un(n)\}_{n\geq 1}$ are the $\LA,\SU,\LA^{\op}$ and $\SU^{\op}$ orderings.
\end{proposition}

\begin{proof}
We build operadic orderings inductively, showing that the choices for $\Un(n)$, $n\leq 4$ determine higher ones. 
We prove the statement for the $\LA$ order, the $\SU$ and opposite orders are similar. 
First, we decide to order the unique pair of $\Un(2)$ as $(\{1\},\{2\})$.
The operadic property then determines the orders of all the pairs of $\Un(3)$ and $\Un(4)$, except the pair $\{\{1,4\},\{2,3\}\}$, for which we choose the order $(\{1,4\},\{2,3\})$.
Now, we claim that all the higher choices are forced by the operadic property, and lead to the $\LA$ diagonal. 
Starting instead with the orders $(\{1\},\{2\})$ and $(\{2,3\},\{1,4\})$ would give the $\SU$ diagonal, and reversing the pairs would give the opposite orders. 

Let $\ell \geq 2$ and suppose that for all $k\leq \ell$, we have given the pair \[\{I_k,J_k\} \eqdef \{\{1,k+2,{k+3},\dots,{2k-1},2k\},\{2,3,\dots,k+1\}\}\] the $\LA$ ordering $(I_k,J_k)$.
Then from \cref{lem:operadic-ordering}, we know that the only $\{I,J\}$ pair of order $\ell+1$ that will not decompose, and hence be specified by the already chosen conditions, is~$\{I_{\ell+1},J_{\ell+1}\}$.
As such, the only way we can vary from $\LA$ is to order this element in the opposite direction $(J_{\ell+1},I_{\ell+1})$. 
We now consider a particular decomposable pair $\{I'_m,J'_m\}$ where ${I'_m \eqdef I_m \sqcup \{3\} \ssm \{m+2\}}$ and~$J'_m \eqdef J_m \sqcup \{m+2\} \ssm \{3\}$, of order $m \eqdef \ell+2$, that will lead to a contradiction (see \cref{ex:Non-coherent order contradiction}).
On the one side, we can decompose $\{I'_m,J'_m\} = \{I_a \cup I_b, J_a \cup J_b\}$ with $I_a \eqdef  \{1,m+3,\dots,2m\}, J_a \eqdef  \{4,5,\dots,m+2\}, I_b \eqdef \{3\}$ and $J_b \eqdef  \{2\}$. 
By hypothesis, we have the orders $(J_a,I_a)$ and $(J_b,I_b)$ which imply the order $(J'_m,I'_m)$ since our ordering is operadic. 
On the other side, we can decompose $\{I'_m,J'_m\} = \{I_c \cup I_d, J_c \cup J_d\}$, where $I_c \eqdef  \{1,m+3,\dots,  2m-1\},$ $J_c \eqdef \{2,5,\dots,m+1\}$, $I_d \eqdef \{3, 2m\}$ and $J_d \eqdef  \{4, m+2\}$.
By hypothesis, we have the orders $(I_c,J_c)$ and $(I_d,J_d)$, which imply that $(I'_m,J'_m)$ since our ordering is operadic.
We arrived at a contradiction.
Thus, the only possible operadic choice of ordering for $\{I_{\ell+1},J_{\ell+1}\}$ is the $\LA$ ordering, which finishes the proof. 
\end{proof}

\begin{example} 
\label{ex:Non-coherent order contradiction}
To illustrate our proof of \cref{prop:operadicOrdering}, consider an operadic ordering $\Or$ for which the $\LA$ ordering holds for pairs of order $1$ and $2$, but is reversed for pairs of order $3$, \ie
\begin{align*}
(\{1\}, \{2\}) \in \Or, \quad (\{1,4\}, \{2,3\}) \in \Or, \text{ and } (\{2, 3, 4\}, \{1, 5, 6 \}) \in \Or.
\end{align*}
Then, the pair $\{I'_4,J'_4\}=\{\{1, 3, 7, 8\}, \{2, 4, 5, 6\}\}$ admits two different orientations.
In particular, 
\begin{align*}
(\{ 4, 5, 6 \} , \{1, 7, 8\}) \in \Or \text{ and } (\{2\}, \{3\}) \in \Or & \Longrightarrow (\{2, 4, 5, 6\}, \{1, 3, 7, 8\} ) \in \Or \ \text{and} \\
(\{1, 7\}, \{2, 5\}) \in \Or \text{ and } (\{3, 8\}, \{4, 6\}) \in \Or & \Longrightarrow (\{1, 3, 7, 8\}, \{2, 4, 5, 6\}) \in \Or.
\end{align*}
\end{example}

Combining \cref{prop:equiv-operadic} with \cref{prop:operadicOrdering}, we get the main result of this section.
Recall that a vector induces the weak order on the vertices of the standard permutahedron if and only if it has strictly decreasing coordinates (see \cref{def:permutahedron}).

\begin{theorem}
\label{thm:unique-operadic}
There are exactly four operadic geometric diagonals of the permutahedra, given by the $\LA$ and $\SU$ diagonals, and their opposites~$\LA^{\op}$ and $\SU^{\op}$. 
The $\LA$ and $\SU$ diagonals are the only operadic geometric diagonals which induce the weak order on the vertices of the permutahedron.
\end{theorem}

\begin{remark}
This answers by the negative a conjecture regarding unicity of diagonals on the permutahedra, raised at the beginning of~\cite[Sect.~3]{SaneblidzeUmble}, and could be seen as an alternative statement. 
See \cref{sec:shifts} where we prove that the geometric $\SU$ diagonal is a topological enhancement of the $\SU$ diagonal from~\cite{SaneblidzeUmble}.
\end{remark}

%%%%%%%%%%%%%%%

\subsection{Isomorphisms between operadic diagonals}
\label{subsec:isos-LA-SU}

Let $P$ be a centrally symmetric polytope, and let $s : P \to P$ be its involution, given by taking a point $\b{p} \in P$ to its reflection $s(\b{p})$ with respect to the center of symmetry of $P$. 
Note that this map is cellular, and induces an involution on the face lattice of $P$. 
For the permutahedron~$\Perm$, this face poset involution is given in terms of ordered partitions of~$[n]$ by the map~$A_1 | \cdots | A_k \mapsto A_k | \cdots | A_1$. 

The permutahedron~$\Perm$ possesses another interesting symmetry, namely, the cellular involution ${r : \Perm \to \Perm}$ which sends a point $p \in \Perm$ to its reflection with respect to the hyperplane of equation \[ \sum_{i=1}^{\lfloor n/2 \rfloor}x_i = \sum_{i=1}^{\lfloor n/2 \rfloor}x_{n-i+1} . \]
This involution also respects the face poset structure. 
In terms of ordered partitions, it replaces each block $A_j$ in $A_1 | \cdots | A_k$ by the block $r(A_j) \eqdef \set{r(i)}{i \in A_j}$ where~$r(i) \eqdef n-i+1$.

The involution $t : P \times P \to P \times P$, sending $(x,y)$ to~$(y,x)$, takes a cellular diagonal to another cellular diagonal.
As we have already remarked in \cref{def:oppositeDiagonal}, this involution sends a geometric diagonal~$\triangle$ to its opposite~$\triangle^{\op}$.
%As we have already remarked in \cref{subsec:LASUdiagonal}, in the case of the permutahedron it sends $\LAD$ to $(\LAD)^\op$ and $\SUD$ to $(\SUD)^\op$.

\pagebreak
\begin{theorem}
\label{thm:bijection-operadic-diagonals}
For the permutahedron~$\Perm$, the involutions $t$ and $rs \times rs$ are cellular isomorphisms between the four operadic diagonals, through the following commutative diagram
\begin{center}
	\begin{tikzcd}
		\LAD \arrow[r,"t"] \arrow[d,"rs \times rs"']&
		(\LAD)^{\op}\arrow[d,"rs \times rs"]\\
		\SUD \arrow[r,"t"'] &
		(\SUD)^\op
	\end{tikzcd}
\end{center}
Moreover, they induce poset isomorphisms at the level of the face lattices. 
\end{theorem}

\begin{proof}
The result for the transpositions $t$ and the commutativity of the diagram are straightforward, so we prove that $rs\times rs$ is a cellular isomorphism respecting the poset structure. 
First, since $r(\min(I\cup J))=\max(r(I) \cup r(J))$, we observe that the map $(I,J) \mapsto (r(J),r(I))$ defines a bijection between $\LA(n)$ and $\SU(n)$.
Then, if $\sigma$ is an ordered partition such that $I$ dominates $J$ in $\sigma$ (\cref{def:domination}), it must be the case that $J$ dominates $I$ in $s\sigma$, and consequently $rJ$ dominates $rI$ in $rs\sigma$.
As such, the domination characterization of the diagonals (\cref{thm:minimal}), tells us~$(\sigma,\tau) \in \LAD$ if and only if $(rs(\sigma),rs(\tau)) \in \SUD$.
Finally, since both $t,r$ and $s$ preserve the poset structures, so does their composition, which finishes the proof.
\end{proof}

\begin{remark} \label{rem:Alternate Isomorphism}
There is a second, distinct isomorphism given by $t(r\times r)$.
This follows from the fact that $s\times s:\LAD \to (\LAD)^{\op}$ is an isomorphism (and also for  $s\times s:\SUD \to (\SUD)^{\op}$ ), as such the composite 
\begin{center}
\begin{tikzcd}
t(r\times r):\LAD \arrow[r,"s\times s"] &
(\LAD)^{\op}\arrow[r,"rs \times rs"] &
(\SUD)^\op \arrow[r,"t"] & \SUD
\end{tikzcd}
\end{center}
is also an isomorphism.
This second isomorphism has the conceptual benefit of sending left shift operators to left shift operators (and right to right), see \cref{prop:trr is an isomorphism of shifts}.
\end{remark}

%%%%%%%%%%%%%%%

\subsection{Facets of operadic diagonals}
\label{subsec:facets-operadic-diags}

We now aim at describing the facets of the $\LA$ and $\SU$ diagonals. 
We have seen in \cref{subsec:multiBraidArrangement} that facets of any diagonal of the permutahedron~$\Perm$ are in bijection with $(2,n)$-partition trees, that is, pairs of unordered partitions of $[n]$ whose intersection graph is a tree.
These pairs of partitions were first studied and called \defn{essential complementary partitions} in~\cite{Chen, ChenGoyal, KajitaniUenoChen}.
Specializing \cref{prop:PFtoOPF1}, we now explain how to order these pairs of partitions to get the facets of $\LAD$ and $\SUD$. 

\begin{theorem}
\label{thm:facet-ordering}
Let $(\sigma,\tau)$ be a pair of ordered partitions of $[n]$ whose underlying intersection graph is a $(2,n)$-partition tree.
The pair~$(\sigma,\tau)$ is a facet of the $\LA$ (\resp $\SU$) diagonal if and only if the minimum (\resp maximum) of every directed path between two consecutive blocks of $\sigma$ or $\tau$ is oriented from $\sigma$ to $\tau$ (\resp from $\tau$ to $\sigma$).
\end{theorem}

\begin{proof}
We specialize \cref{prop:PFtoOPF1} to the $\LA$ diagonal, the proof for the $\SU$ diagonal is similar. 
Let $\b{v}$ be a vector inducing the $\LA$ diagonal as in \cref{def:LA-and-SU}.
Without loss of generality, we place the first copy of the braid arrangement centered at $0$, and the second copy centered at $\b{v}$.
From \cref{def:multiBraidArrangementPrecise,rem:translationMatrix}, we get that $A_{1,s,t}=0$ and $A_{2,s,t}=v_s-v_t$ for any edges $s,t$.
We treat the case of two consecutive blocks $A$ and $B$ of $\sigma$, the analysis for $\tau$ is similar. 
The directed path between $A$ and $B$ is a sequence of edges $r_0,r_1,\dots,r_q$.  
Let us denote by $I \eqdef \{r_0,r_2,\dots\}$ de set of edges directed from $\sigma$ to $\tau$, and by $J \eqdef \{r_1,r_3,\dots\}$ its complement. 
According to \cref{prop:PFtoOPF1}, the order between $A$ and $B$ is determined by the sign of $A_{1,r_0,r_q}- \sum_{p \in [q]} A_{i_p,r_{p-1},r_p}$, where $i_p$ is the copy of the block adjacent to both edges $r_{p-1}$ and $r_p$. 
Thus, the order between $A$ and $B$ is determined by the sign of $\sum_{i \in I} v_i - \sum_{j \in J} v_j$, which according to the definition of $\LAD$ is positive whenever $\min(I\cup J) \in I$. 
Thus, the pair $(\sigma,\tau) \in \LAD$ if and only if the minimum of every directed path between two consecutive blocks of $\sigma$ or $\tau$ is oriented from $\sigma$ to $\tau$. 
\end{proof}

In the rest of the paper, we shall represent ordered $(2,n)$-partition trees $(\sigma,\tau)$ of facets by drawing $\sigma$ on the left, and $\tau$ on the right, with their blocks from top to bottom. 
The conditions ``oriented from $\sigma$ to $\tau$" in the preceding Theorem then reads as ``oriented from left to right", an expression we shall adopt from now on. 

\begin{example}
\label{ex:ECbijection}
Below are the two orderings of the $(2,7)$-partition tree $\{15,234,6,7\} \times \{13,2,46,57\}$ giving facets of the $\LA$ (left) and $\SU$ (right) diagonals, obtained via \cref{thm:facet-ordering}.
Note that ordered partitions should be read from top to bottom.
\begin{center}
	\begin{tikzpicture}[scale=.7]  
		\node[anchor=east] (1) at (-1.5, -1) {$15$};
		\node[anchor=east] (2) at (-1.5, -2) {$7$};
		\node[anchor=east] (3) at (-1.5, -3) {$234$};
		\node[anchor=east] (4) at (-1.5, -4) {$6$};
		\node[anchor=west] (5) at (1.5, -1) {$57$};
		\node[anchor=west] (6) at (1.5, -2) {$46$};
		\node[anchor=west] (7) at (1.5, -3) {$13$};
		\node[anchor=west] (8) at (1.5, -4) {$2$};
		\draw[thick] (1.east) -- (5.west); 
		\draw[thick] (1.east) -- (7.west);
		\draw[thick] (2.east) -- (5.west); 
		\draw[thick] (3.east) -- (6.west); 
		\draw[thick] (3.east) -- (7.west); 
		\draw[thick] (3.east) -- (8.west); 
		\draw[thick] (4.east) -- (6.west);
	\end{tikzpicture}
	$\quad$
	\begin{tikzpicture}[scale=.7]  
		\node[anchor=east] (7) at (-1.5, -4) {$7$};
		\node[anchor=east] (6) at (-1.5, -3) {$6$};
		\node[anchor=east] (234) at (-1.5, -2) {$234$};
		\node[anchor=east] (15) at (-1.5, -1) {$15$};
		\node[anchor=west] (2) at (1.5, -4) {$2$};
		\node[anchor=west] (13) at (1.5, -3) {$13$};
		\node[anchor=west] (46) at (1.5, -2) {$46$};
		\node[anchor=west] (57) at (1.5, -1) {$57$};
		\draw[thick] (2.west) -- (234.east);
		\draw[thick] (13.west) -- (15.east); 
		\draw[thick] (13.west) -- (234.east);
		\draw[thick] (46.west) -- (234.east);
		\draw[thick] (46.west) -- (6.east); 
		\draw[thick] (57.west) -- (15.east);
		\draw[thick] (57.west) -- (7.east);
	\end{tikzpicture}
\end{center}
In the $\LA$ facet (left), we have $7 < 234$, since in the path between the two vertices $7 \xrightarrow{7} 57 \xrightarrow{5} 15 \xrightarrow{1} 13 \xrightarrow{3} 234$, the minimum $1$ is oriented from left to right. 
In this case, we have $(I,J)=(\{1,7\},\{3,5\})$.
Similarly, the path $57 \xrightarrow{5} 15 \xrightarrow{1} 13 \xrightarrow{3} 234 \xrightarrow{4} 46$ imposes that $57 < 46$, for $(I,J)=(\{1,4\},\{3,5\})$.
\end{example}

\begin{remark}
Note that forgetting the order in a facet of $\LAD$, and then ordering again the $(2,n)$-partition tree to obtain a facet of $\SUD$, provides a bijection between facets that was not considered in \cref{subsec:isos-LA-SU}.
However, this map is not defined on the other faces.
\end{remark}

\begin{example}
The isomorphism $rs\times rs$ from \cref{thm:bijection-operadic-diagonals} sends the $\LA$ facet from \cref{ex:ECbijection}\,(left) to the following $\SU$ facet\,(right).
\begin{center}
	\begin{tikzpicture}[scale=.7]  
		\node[anchor=east] (1) at (-1.5, -1) {$15$};
		\node[anchor=east] (2) at (-1.5, -2) {$7$};
		\node[anchor=east] (3) at (-1.5, -3) {$234$};
		\node[anchor=east] (4) at (-1.5, -4) {$6$};
		\node[anchor=west] (5) at (1.5, -1) {$57$};
		\node[anchor=west] (6) at (1.5, -2) {$46$};
		\node[anchor=west] (7) at (1.5, -3) {$13$};
		\node[anchor=west] (8) at (1.5, -4) {$2$};
		\draw[thick] (1.east) -- (5.west); 
		\draw[thick] (1.east) -- (7.west);
		\draw[thick] (2.east) -- (5.west); 
		\draw[thick] (3.east) -- (6.west); 
		\draw[thick] (3.east) -- (7.west); 
		\draw[thick] (3.east) -- (8.west); 
		\draw[thick] (4.east) -- (6.west);
	\end{tikzpicture}
	\raisebox{3.4em}{$\xrightarrow{rs\times rs}$}
	\begin{tikzpicture}[scale=.7]  
		\node[anchor=east] (5) at (-1.5, -4) {$37$};
		\node[anchor=east] (6) at (-1.5, -3) {$1$};
		\node[anchor=east] (7) at (-1.5, -2) {$456$};
		\node[anchor=east] (8) at (-1.5, -1) {$2$};
		\node[anchor=west] (1) at (1.5, -4) {$13$};
		\node[anchor=west] (2) at (1.5, -3) {$24$};
		\node[anchor=west] (3) at (1.5, -2) {$57$};
		\node[anchor=west] (4) at (1.5, -1) {$6$};
		\draw[thick] (1.west) -- (5.east);
		\draw[thick] (1.west) -- (6.east); 
		\draw[thick] (2.west) -- (7.east);
		\draw[thick] (2.west) -- (8.east);
		\draw[thick] (3.west) -- (5.east); 
		\draw[thick] (3.west) -- (7.east);
		\draw[thick] (4.west) -- (7.east);
	\end{tikzpicture}
\end{center}
Moreover, it sends the path $57 \xrightarrow{5} 15 \xrightarrow{1} 13 \xrightarrow{3} 234 \xrightarrow{4} 46 $ to the path $24 \xrightarrow{4} 456 \xrightarrow{5} 57 \xrightarrow{7} 37 \xrightarrow{3} 13$, where the maximum $7$ is oriented from right to left, witnessing the fact that $24 < 13$. 
The associated $(I,J)=(\{1,4 \},\{3,5\})$ becomes $(r(J),r(I))=(\{3,5 \},\{4,7\})$.
\end{example}

%%%%%%%%%%%%%%%

\subsection{Vertices of operadic diagonals}
\label{subsec:vertices-operadic-diags}

We are now interested in characterizing the vertices that occur in an operadic diagonal as pattern-avoiding pairs of partitions of $[n]$. 
These pairs form a strict subset of the weak order intervals. 
We first need the following Lemma, which follows directly from the definition of domination (\cref{def:domination}).

\begin{lemma}
\label{lem:Coherent Domination}
Let $\sigma$ be an ordered partition of $[n]$ and let $I,J \subseteq [n]$ be such that $I$ dominates $J$ in~$\sigma$. 
For an element $i$ in $I$ or $J$, we denote by $\sigma^{-1}(i)$ the index of the block containing it in $\sigma$. 
Then, for any $i \in I,j \in J$, we have $I\ssm i$ dominates $J\ssm j$ in $\sigma$ if and only if
\begin{enumerate}
\item either~$\sigma^{-1}(j) < \sigma^{-1}(i)$ (meaning that $j$ comes strictly before $i$ in $\sigma$),
\item or for all $k$ between $\sigma^{-1}(i)$ and $\sigma^{-1}(j)$, the first $k$ blocks of $\sigma$ contain strictly more elements of $I$ than of $J$.
\end{enumerate}
\end{lemma}

% As before, we represent a permutation $\sigma: [n] \to [n]$ by the non-commutative word~$\sigma(1)\cdots \sigma(n)$.

\begin{definition}
For $k\leq n$, a permutation $\tau$ of $[k]$ is a \defn{pattern} of a permutation $\sigma$ of $[n]$ if there exists a subset $I \eqdef \{i_1 < \dots < i_k\} \subset [n]$ so that the permutation~$\tau$ gives the relative order of the entries of~$\sigma$ at positions in~$I$, \ie $\tau_u < \tau_v \eqdef \sigma_{i_u} < \sigma_{i_v}$.
We say that $\sigma$ \defn{avoids} $\tau$ if $\tau$ is not a pattern~of~$\sigma$. 
\end{definition}

\begin{example}
The pairs of permutations $(\sigma,\tau)$ avoiding the patterns $(21,12)$ are precisely the intervals of the weak order. 
\end{example}

\begin{theorem}
\label{thm:patterns}
A pair of permutations of $[n]$ is a vertex of the $\LA$ (\resp $\SU$) diagonal if and only if for any~$k\geq 1$ and for any $(I,J) \in \LA(k)$ (\resp $\SU(k)$) of size $\card{I}=\card{J}=k$ it avoids~the~patterns 
\begin{align}
(j_1 i_1 j_2 i_2 \cdots j_k i_k,\ i_2 j_1 i_3 j_2 \cdots i_k j_{k-1} i_1 j_k), \tag{LA} \label{eq:pattern} \\
\text{ \resp } (j_1 i_1 j_2 i_2 \cdots j_k i_k, \ i_1 j_k i_2 j_1 \cdots i_{k-1} j_{k-2}i_k j_{k-1}), \tag{SU}
\end{align}
where $I=\{i_1,\dots,i_k\}$, $J=\{j_1,\dots,j_k\}$ and $i_1=1$ (\resp $j_k=2k$).
\end{theorem}

\begin{example}\label{exm:intervalsWeakOrderNotDiagonals}
\enlargethispage{.5cm}
For each $k \ge 1$, there are $\binom{2k-1}{k-1,k}(k-1)!k!$ avoided standardized patterns.
For~$k=1$, both diagonals avoid $(21,12)$, so the vertices are intervals of the weak order.
For~$k=2$, 
\begin{itemize}
    \item $\LA$ avoids 
    $(3142,2314), (4132,2413),
    (2143,3214), (4123,3412),
    (2134,4213), (3124,4312)$.
    \item $\SU$ avoids 
    $(1243,2431),(1342,3421),
    (2143,1432),(2341,3412),
    (3142,1423),(3241,2413)$.
\end{itemize}
As such, the vertices of $\LA$ and $\SU$ are a strict subset of all intervals of the weak order.
Here is an illustration of the patterns avoided for $k=4$.
The $\LA$ pattern is drawn on the left, the $\SU$ pattern is drawn on the right, and they are in bijection under the isomorphism $t(r\times r)$ (\cref{subsec:isos-LA-SU}), where the bijection between elements is $(i_1,i_2,i_3,i_4)\mapsto (j'_4,j'_1,j'_2,j'_3)$ and $(j_1,j_2,j_3,j_4)\mapsto (i'_1,i'_2,i'_3,i'_4)$.
\begin{center}
\begin{tikzpicture}[scale=.65]  
\node[anchor=east] (0) at (-1.5, 0) {$j_1$};
\node[anchor=east] (1) at (-1.5, -1) {$i_1$};
\node[anchor=east] (2) at (-1.5, -2) {$j_2$};
\node[anchor=east] (3) at (-1.5, -3) {$i_2$};
\node[anchor=east] (4) at (-1.5, -4) {$j_3$};
\node[anchor=east] (5) at (-1.5, -5) {$i_3$};
\node[anchor=east] (6) at (-1.5, -6) {$j_4$};
\node[anchor=east] (7) at (-1.5, -7) {$i_4$};
\node[anchor=west] (8) at (1.5, 0) {$i_2$};
\node[anchor=west] (9) at (1.5, -1) {$j_1$};
\node[anchor=west] (10) at (1.5, -2) {$i_3$};
\node[anchor=west] (11) at (1.5, -3) {$j_2$};
\node[anchor=west] (12) at (1.5, -4) {$i_4$};
\node[anchor=west] (13) at (1.5, -5) {$j_3$};
\node[anchor=west] (14) at (1.5, -6) {$i_1$};
\node[anchor=west] (15) at (1.5, -7) {$j_4$};
\draw[thick, green] (0.east) -- (9.west);
\draw[thick, blue, dashed] (1.east) -- (14.west);
\draw[thick, green] (2.east) -- (11.west);
\draw[thick, blue] (3.east) -- (8.west);
\draw[thick, green] (4.east) -- (13.west);
\draw[thick, blue] (5.east) -- (10.west);
\draw[thick, green] (6.east) -- (15.west);
\draw[thick, blue] (7.east) -- (12.west);
\end{tikzpicture}
\raisebox{7em}{$\xrightarrow{t(r\times r)}$}
\begin{tikzpicture}[scale=.65]  
\node[anchor=east] (0) at (-1.5, 0) {$j'_1$};
\node[anchor=east] (1) at (-1.5, -1) {$i'_1$};
\node[anchor=east] (2) at (-1.5, -2) {$j'_2$};
\node[anchor=east] (3) at (-1.5, -3) {$i'_2$};
\node[anchor=east] (4) at (-1.5, -4) {$j'_3$};
\node[anchor=east] (5) at (-1.5, -5) {$i'_3$};
\node[anchor=east] (6) at (-1.5, -6) {$j'_4$};
\node[anchor=east] (7) at (-1.5, -7) {$i'_4$};
\node[anchor=west] (8) at (1.5, 0) {$i'_1$};
\node[anchor=west] (9) at (1.5, -1) {$j'_4$};
\node[anchor=west] (10) at (1.5, -2) {$i'_2$};
\node[anchor=west] (11) at (1.5, -3) {$j'_1$};
\node[anchor=west] (12) at (1.5, -4) {$i'_3$};
\node[anchor=west] (13) at (1.5, -5) {$j'_2$};
\node[anchor=west] (14) at (1.5, -6) {$i'_4$};
\node[anchor=west] (15) at (1.5, -7) {$j'_3$};
\draw[thick, green] (0.east) -- (11.west);
\draw[thick, blue] (1.east) -- (8.west);
\draw[thick, green] (2.east) -- (13.west);
\draw[thick, blue] (3.east) -- (10.west);
\draw[thick, green] (4.east) -- (15.west);
\draw[thick, blue] (5.east) -- (12.west);
\draw[thick, green, dashed] (6.east) -- (9.west);
\draw[thick, blue] (7.east) -- (14.west);
\end{tikzpicture}
\end{center}
The alternate isomorphism $t(s\times s)$, provides an alternate way to establish a bijection between these two patterns.
The avoided patterns for higher $k$ extend this crisscrossing shape.

\end{example}

\begin{proof}[Proof of \cref{thm:patterns}]
We give the proof for the $\LA$ diagonal, the one for the $\SU$ diagonal is similar, and can be obtained by applying either of the two isomorphisms of \cref{subsec:isos-LA-SU}.
According to \cref{thm:IJ-description}, we have $(\sigma,\tau) \notin \LAD$ if and only if there is an $(I,J)$ such that~$J$ dominates~$I$ in $\sigma$ and~$I$ dominates $J$ in $\tau$.
It is clear that if a pair of permutations $(\sigma,\tau)$ contains a pattern of the form~(\ref{eq:pattern}), then the associated $(I,J)$ satisfies the domination condition.
Thus, we just need to show the reverse implication. 
We proceed by induction on the cardinality $\card{I} = \card{J}$. 
The case of cardinality $1$ is clear. 
Now suppose that for all $(I,J)$ of size $\card{I} = \card{J} \le m-1$, we have if~$J$ dominates $I$ in $\sigma$ and $I$ dominates $J$ in $\tau$, then $(\sigma,\tau)$ contains a pattern of the form~(\ref{eq:pattern}).

We need to show that this is still true for the pairs $(I,J)$ of size $\card{I} = \card{J} = m$.
Firstly, we need only consider standardized $I,J$ conditions, and pairs of permutations of $[2m]$.
Indeed, if we define $(\sigma',\tau') \eqdef \std(\sigma \cap (I\cup J),\tau \cap (I\cup J))$, and $(I',J')\eqdef\std(I',J')$, 
then if $(\sigma,\tau)$ satisfies the $(I,J)$ domination condition, this implies $(\sigma',\tau')$ satisfies $(I',J')$.
Conversely, if $(\sigma',\tau')$ has a pattern, then this implies $(\sigma,\tau)$ has the same pattern, which yields the indicated simplification.

Let $(\sigma,\tau)$ be such a pair of permutations.
Suppose the leftmost element $i_1$ of $I$ in $\sigma$ is not $1$, and let us write $j_1$ for the leftmost element of $J$ in $\tau$.
Consider the pair ${(I',J') \eqdef (I\ssm \{i_1\}, J \ssm \{j_1\})}$.
Using both cases of \cref{lem:Coherent Domination}, we have $J'$ dominates $I'$ in $\sigma$, and $I'$ dominates $J'$ in $\tau$, and we can thus conclude by induction that $(\sigma,\tau)$ contains a smaller pattern.

So, we assume the leftmost element of $I$ in $\sigma$ is $i_1=1$, and for $n\geq1$ we prove that,
\begin{enumerate}[(a)]
\item If $(\sigma,\tau)=(j_1 i_1 j_2 i_2 \dots j_{n-1} i_{n-1} w i_{n} w', \ i_2 j_1 i_3 j_2 \dots i_{n}j_{n-1}w'') $, and $j_n$ is the leftmost element of $J\ssm \{j_1, \dots, j_{n-1}\}$ in $\tau$, then either $w = j_n$, or it matches a smaller pattern.
\item If $(\sigma,\tau)=(j_1 i_1 j_2 i_2 \dots j_{n-1} i_{n-1} w'', \ i_2 j_1 i_3 j_2 \dots i_{n-1}j_{n-2} w j_{n-1} w'')$, and $i_n$ is the leftmost element of ${I\ssm \{i_1, \dots, i_{n-1}\}}$ in $\sigma$, then either $w = i_n$, or it matches a smaller pattern.
\end{enumerate}
We prove (a), the proof of (b) proceeds similarly. 
Let $j_n$ be the leftmost element of $J\ssm \{j_1, \dots, j_{n-1} \}$ in $\tau$.
If $w\neq j_n$, then either (i) $w$ consists of multiple elements of $J$ including $j_n$, or (ii) $j_n$ comes after $i_n$ in $\sigma$.
Now consider the pair $(I',J') \eqdef (I\ssm \{i_n\}, J \ssm \{j_n\})$.
As was the case for the proof that $i_1=1$, we have $I'$ dominates $J'$ in $\tau$.
To prove that $J'$ dominates $I'$ in $\sigma$, we split by the cases.
In case (i), we may apply condition $(2)$ of \cref{lem:Coherent Domination}.
In case (ii), either $i_n$ comes before~$j_n$, in which case we meet condition $(1)$ of \cref{lem:Coherent Domination}, or $j_n$ comes before $i_n$.
In this last situation, we have condition $(2)$ of \cref{lem:Coherent Domination} holds as $i_n$ is the leftmost element of~${I\ssm \{i_1, \dots, i_{n-1}\}}$ in $\sigma$.
Thus, if $w \neq j_n$, by the inductive hypothesis, we match a smaller pattern.

Finally, using statements (a) and (b) above, we can inductively generate $(\sigma,\tau)$, determining $j_1$ via (a), then $i_2$ via (b), then $j_2$ via (a), and so on.
This inductive process fully generates $\sigma$, and places all elements of $\tau$ except $i_1$, yielding $\tau=i_2 j_1 i_3 j_2 \cdots i_k j_{k-1} w j_k w''$.
However, as $j_k$ must be dominated by an element of $I$, this forces $w = i_1$ and $w'' =\varnothing$, completing the proof.
\end{proof}

\subsection{Relation to the facial weak order}
\label{sec:facial-weak-order}

There is a natural lattice structure on all ordered partitions of~$[n]$ which extends the weak order on permutations of~$[n]$.
This lattice was introduced in~\cite{KrobLatapyNovelliPhanSchwer}, where it is called \emph{pseudo-permutahedron} and defined on packed words rather than ordered partitions.
It was later generalized to arbitrary finite Coxeter groups in~\cite{PalaciosRonco, DermenjianHohlwegPilaud}, where it is called \emph{facial weak order} and expressed in more geometric terms.
We now recall a definition of the facial weak order on ordered partitions, and use the vertex characterization of the preceding section to show all faces of the  $\LA$ and $\SU$ diagonal are intervals of this order.

\begin{definition}[\cite{KrobLatapyNovelliPhanSchwer,PalaciosRonco,DermenjianHohlwegPilaud}]
The \defn{facial weak order} on ordered partitions is the transitive closure of the relations
\begin{align}
    \sigma_1|\dots|\sigma_k < \sigma_1|\cdots|\sigma_i \sqcup \sigma_{i+1}|\cdots|\sigma_k \quad & \text{for any } \sigma_1|\dots|\sigma_k \text{ with } \max \sigma_i < \min \sigma_{i+1},  \label{eq:facial weak 1}\\
    \sigma_1|\cdots|\sigma_i \sqcup \sigma_{i+1}|\cdots|\sigma_k < \sigma_1|\dots|\sigma_k \quad & \text{for any } \sigma_1|\dots|\sigma_k \text{ with } \min \sigma_i > \max \sigma_{i+1}. \label{eq:facial weak 2}
\end{align}
\end{definition}

The facial weak order recovers the weak order on permutations as illustrated in \cref{fig:Hasse diagram Perm3}.
\begin{figure}
\begin{tikzpicture}[xscale=1.3, yscale=1.5]
\node (1) at (-1, -.5) {$3|12$};
\node (2) at (-2, -1) {$3|1|2$};
\node (3) at (-2, -2) {$13|2$};
\node (4) at (-2, -3) {$1|3|2$};
\node (5) at (-1, -3.5) {$1|23$};
\node (6) at (1, -.5) {$23|1$};
\node (7) at (2, -1) {$2|3|1$};
\node (8) at (2, -2) {$2|13$};
\node (9) at (2, -3) {$2|1|3$};
\node (10) at (1, -3.5) {$12|3$};
\node (11) at (0, 0) {$3|2|1$};
\node (12) at (0, -2) {$123$};
\node (13) at (0, -4) {$1|2|3$};
\draw[thick] (1) -- (2); 
\draw[thick] (2) -- (3); 
\draw[thick] (3) -- (4); 
\draw[thick] (4) -- (5);
\draw[thick] (6) -- (7); 
\draw[thick] (7) -- (8); 
\draw[thick] (8) -- (9); 
\draw[thick] (9) -- (10);
\draw[thick] (1) -- (11);
\draw[thick] (6) -- (11);
\draw[thick] (1) -- (12);
\draw[thick] (5) -- (12);
\draw[thick] (6) -- (12);
\draw[thick] (10) -- (12);
\draw[thick] (5) -- (13);
\draw[thick] (10) -- (13);
\end{tikzpicture}
\caption{The Hasse diagram of the facial weak order for $\Perm[3]$.}
\label{fig:Hasse diagram Perm3}
\end{figure}
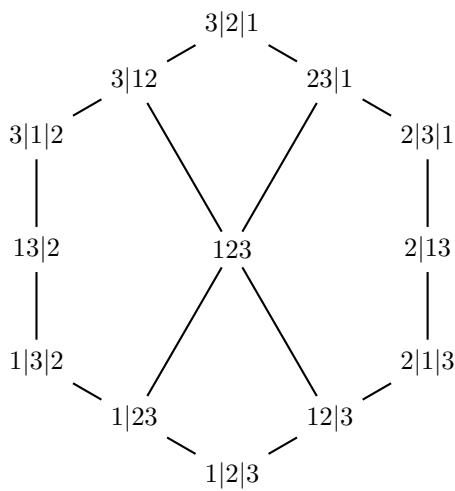

\begin{proposition}
If $(\sigma,\tau)\in \LAD$, or $(\sigma,\tau)\in \SUD$, then $\sigma \leq \tau$ under the facial weak order.
\end{proposition}

\begin{proof}
By \cref{prop:magicalFormula}, the faces $(\sigma,\tau)$ satisfy $\max_{\mathbf{v}} \sigma \leq \min_{\mathbf{v}} \tau$ under the weak order.
Thus, if we can show that $\sigma\leq \max_{\mathbf{v}} \sigma$ and $\min_{\mathbf{v}} \tau \leq \tau$ under the facial weak order, then the result immediately follows.
If $\sigma$ is a face of the permutahedra, then under both the $\LA$, and $\SU$ orientation vectors, the vertex $\max_{\mathbf{v}} \sigma$ is given by writing out each block of $\sigma$ in decreasing order, and the vertex $\min_{\mathbf{v}} \sigma$ is given by writing out each  block of $\sigma$ in increasing order.
Then under the facial weak order
$\sigma \leq \max_{\mathbf{v}} \sigma$, as repeated applications of \cref{eq:facial weak 2} shows that a block of elements is smaller than those same elements arranged in decreasing order.
Similarly $\min_{\mathbf{v}} \sigma \leq \sigma$, as repeated applications of \cref{eq:facial weak 1} shows that a sequence of increasing elements is smaller than those same elements in a block. 
\end{proof}

\begin{example}
The facet $13|24|57|6 \times 3|17|456|2 \in \SUD$, satisfies the inequality through the vertices
\begin{align*}
    13|24|57|6 <  3|1|4|2|7|5|6 <  3|1|7|4|5|6|2 <  3|17|456|2 \, .
\end{align*}
\end{example}

%%%%%%%%%%%%%%%%%%%%%%%%%%%%%%%%%%%%%%%%%

\newpage
\section{Shift lattices}
\label{sec:shifts}

In this section, we prove that the geometric $\SU$ diagonal $\SUD$ agrees at the cellular level with the original Saneblidze-Umble diagonal defined in~\cite{SaneblidzeUmble}.
This involves some shift operations on facets of the diagonal, which are interesting on their own right, and lead to lattice and cubic structures.
The proof is technical and proceeds in several steps: we introduce two additional combinatorial descriptions of the diagonal, that we call the $1$-shift and $m$-shift $\SU$ diagonals, and show the sequence of equivalences
\begin{center}
\begin{tikzcd}
\text{original $\SUD$} \arrow[r,leftrightarrow,"\ref{prop:iso-original-shift-diagonals}"]&
\text{$1$-shift $\SUD$} \arrow[r,leftrightarrow,"\ref{prop:iso-1-to-m-shift}"]&
\text{$m$-shift $\SUD$} \arrow[r,leftrightarrow,"\ref{prop:iso-shift-IJ-diagonals}"]&
\text{geometric $\SUD$} .
\end{tikzcd}
\end{center}
Throughout this section, we borrow notation from~\cite{SaneblidzeUmble-comparingDiagonals}.

%%%%%%%%%%%%%%%

\subsection{Topological enhancement of the original $\SU$ diagonal}
\label{subsec:topological-SU}

We proceed to introduce different versions of the $\SU$ diagonal, and to prove that all these notions coincide.

%%%%%%%%%%%%%%%

\subsubsection{Strong complementary pairs}

We start by the following definition.

\begin{definition}
\label{def:strong-complementary-pairs}
A \defn{strong complementary pair}, or~\defn{$\SCP$} for short, is a pair $(\sigma,\tau)$ of ordered partitions of $[n]$ where $\sigma$ is obtained from a permutation of $[n]$ by merging the adjacent elements which are decreasing, and $\tau$ is obtained from the same permutation by merging the adjacent elements which are increasing.
\end{definition}

We denote by $\SCP(n)$ the set of $\SCP$s of $[n]$, which is in bijection with the set of permutations of $[n]$.
As is clear from the definition, the intersection graph of a $\SCP$ is a $(2,n)$-partition tree.

\begin{example}
\label{ex:strong-complementary}
The $\SCP$ associated to the permutation $3|1|7|4|2|5|6$ is
\begin{center}
	\begin{tikzpicture}[scale=.7]  
		\node (p) at (0, 0) {$13|247|5|6 \times 3|17|4|256$};
		\node[anchor=east] (1) at (-1.5, -1) {$13$};
		\node[anchor=east] (2) at (-1.5, -2) {$247$};
		\node[anchor=east] (3) at (-1.5, -3) {$5$};
		\node[anchor=east] (4) at (-1.5, -4) {$6$};
		\node[anchor=west] (5) at (1.5, -1) {$3$};
		\node[anchor=west] (6) at (1.5, -2) {$17$};
		\node[anchor=west] (7) at (1.5, -3) {$4$};
		\node[anchor=west] (8) at (1.5, -4) {$256$};
		\draw[thick] (1.east) -- (5.west); 
		\draw[thick] (1.east) -- (6.west); 
		\draw[thick] (2.east) -- (6.west); 
		\draw[thick] (2.east) -- (7.west); 
		\draw[thick] (2.east) -- (8.west); 
		\draw[thick] (3.east) -- (8.west); 
		\draw[thick] (4.east) -- (8.west);
	\end{tikzpicture}
\end{center}
Observe that the permutation can be read off the intersection graph of the SCP by a vertical down slice through the edges.
\end{example}

\begin{notation}
For a $(2,n)$-partition tree $(\sigma,\tau)$, we denote by $\sigma_{i,j}$ (\resp $\tau_{i,j}$) the unique oriented path between blocks $\sigma_{i}$ and $\sigma_j$ (\resp $\tau_{i}$ and $\tau_j$).
Note that we make a slight abuse in notation, as the path $\sigma_{i,j}$ also depends on $\tau$.
\end{notation}

We can immediately characterize the paths between adjacent blocks in $\SCP$s.

\begin{proposition} 
\label{lem:SCP-path-desc}
For any $\SCP$ $(\sigma,\tau)$, we have
\begin{enumerate}
\item $ \sigma_{i,i+1} = ( \min \sigma_i, \max \sigma_{i+1} )$ and $\min \sigma_i< \max \sigma_{i+1}$, and
\item $  \tau_{i,i+1} =  ( \max \tau_i, \min \tau_{i+1} )$ and $\min \tau_{i+1}< \max \tau_{i}$.
\end{enumerate}
As a consequence, all $\SCP$s are in both $\LAD$ and $\SUD$, and constitute precisely the set of facets~$(\sigma,\tau)$ of these diagonals such that $\max_{\b{v}}(\sigma) = \min_{\b{v}}(\tau)$.
\end{proposition}
\begin{proof}
First, the path description of $(\sigma,\tau)$ is a straightforward observation. 
Second, since the minima of these paths are traversed from left to right, and the maxima from right to left, \cref{thm:facet-ordering} implies that $\SCP$s are in both geometric operadic diagonals~$\LAD$ and~$\SUD$.
Third, the fact that these constitute the facets $(\sigma,\tau) \in \LAD$ or $\SUD$ satisfying $\max_{\b{v}}(\sigma) = \min_{\b{v}}(\tau)$ can be seen as follows.
The maximal (\resp minimal) vertex of a face $\sigma$ of the permutahedron with respect to the weak order, is obtained by ordering the elements of each block of $\sigma$ in decreasing (\resp increasing) order. 
Thus, it is clear that the original permutation giving rise to a $\SCP$ $(\sigma,\tau)$ is the permutation $\max_{\b{v}}(\sigma)=\min_{\b{v}}(\tau)$, for any vector $\b{v}$ inducing the weak order.
Since both diagonals agree with this order on the vertices, we have that $\SCP$s are indeed facets of $\LAD$ and $\SUD$ with the desired property. 
The fact that these are \emph{all} the facets with this property follows from the bijection between $\SCP(n)$ and the permutations of~$[n]$.
\end{proof}

%%%%%%%%%%%%%%%

\subsubsection{Shifts and the $\SU$ diagonals}
\label{subsec:SU-shifts}

We recall the definition of the original $\SU$ diagonal of~\cite{SaneblidzeUmble}, based on the exposition given in~\cite{SaneblidzeUmble-comparingDiagonals}.
We then introduce two variants of this definition, the $1$-shift and $m$-shift $\SU$ diagonals, which will be shown to be equivalent to the original one.

\begin{definition}
\label{def:subset shifts}
Let $\sigma=\sigma_1| \cdots |\sigma_k$ be an ordered partition, and let~$M \subsetneq \sigma_{i}$ be a non-empty subset of the block $\sigma_i$.
For $m\geq 1$, the \defn{right $m$-shift} $R_M$, moves the subset $M$, $m$ blocks to the right, \ie
\begin{align*}
R_M(\sigma) \eqdef \sigma_1 | \cdots | \sigma_i \ssm M | \cdots | \sigma_{i+m} \cup M | \cdots | \sigma_k
\end{align*}
while the \defn{left $m$-shift} $L_M$, moves the subset $M$, $m$ blocks to the left, \ie
\begin{align*}
L_M(\sigma) \eqdef \sigma_1 | \cdots | \sigma_{i-m} \cup M | \cdots | \sigma_{i} \ssm M | \cdots | \sigma_k .
\end{align*}
\end{definition}

\begin{definition}
\label{def:movable-subsets}
Let $\sigma$ denote either one of the two ordered partitions of $[n]$ in an ordered \mbox{$(2,n)$-partition} tree, and let $M \subsetneq \sigma_i$.
The right $m$-shift $R_M$ (\resp the left $m$-shift $L_M$) is 
\begin{enumerate}
\item \defn{block-admissible} if $\min \sigma_i \notin M$ and $\min M > \max \sigma_{i+m}$ (\resp $\min M > \max \sigma_{i-m}$),
\item \defn{path-admissible} if $\min M > \max \sigma_{i,i+m}$ (\resp $\min M > \max \sigma_{i,i-m}$).
\end{enumerate}
\end{definition}

\begin{remark}
\label{rem:inverses}
Observe that for a given subset $M$, an inverse to the right $m$-shift $R_M$ (\resp left $m$-shift $L_M$) is given by the left $m$-shift $L_M$ (\resp right $m$-shift $R_M$).
Moreover, one $m$-shift is path-admissible if and only if its inverse is. 
\end{remark}

\begin{example}\label{ex:1-shift example}
Performing the $1$-shifts $R_7$ and $L_{56}$ (they are both block and path admissible) of the $\SCP$~$(\sigma,\tau)$ of \cref{ex:strong-complementary}, one obtains the pair~$R_{7}(\sigma) \times L_{56}(\tau)$, as illustrated below.
\begin{center}
	\begin{tikzpicture}[scale=.7]  
		\node (p) at (0, 0) {$13|247|5|6 \times 3|17|4|256$};
		\node[anchor=east] (1) at (-1.5, -1) {$13$};
		\node[anchor=east] (2) at (-1.5, -2) {$247$};
		\node[anchor=east] (3) at (-1.5, -3) {$5$};
		\node[anchor=east] (4) at (-1.5, -4) {$6$};
		\node[anchor=west] (5) at (1.5, -1) {$3$};
		\node[anchor=west] (6) at (1.5, -2) {$17$};
		\node[anchor=west] (7) at (1.5, -3) {$4$};
		\node[anchor=west] (8) at (1.5, -4) {$256$};
		\draw[thick] (1.east) -- (5.west); 
		\draw[thick] (1.east) -- (6.west); 
		\draw[thick] (2.east) -- (6.west); 
		\draw[thick] (2.east) -- (7.west); 
		\draw[thick] (2.east) -- (8.west); 
		\draw[thick] (3.east) -- (8.west); 
		\draw[thick] (4.east) -- (8.west);
	\end{tikzpicture}
	\quad
	\raisebox{3.4em}{$\xrightarrow{R_{7} \times L_{56}}$}
	\quad
	\begin{tikzpicture}[scale=.7]  
		\node (p) at (0, 0) {$13|24|57|6 \times 3|17|456|2$};
		\node[anchor=east] (1) at (-1.5, -1) {$13$};
		\node[anchor=east] (2) at (-1.5, -2) {$24$};
		\node[anchor=east] (3) at (-1.5, -3) {$57$};
		\node[anchor=east] (4) at (-1.5, -4) {$6$};
		\node[anchor=west] (5) at (1.5, -1) {$3$};
		\node[anchor=west] (6) at (1.5, -2) {$17$};
		\node[anchor=west] (7) at (1.5, -3) {$456$};
		\node[anchor=west] (8) at (1.5, -4) {$2$};
		\draw[thick] (1.east) -- (5.west); 
		\draw[thick] (1.east) -- (6.west); 
		\draw[thick] (3.east) -- (6.west); 
		\draw[thick] (2.east) -- (7.west); 
		\draw[thick] (2.east) -- (8.west); 
		\draw[thick] (3.east) -- (7.west); 
		\draw[thick] (4.east) -- (7.west);
	\end{tikzpicture}
\end{center}
\end{example}

We shall concentrate on three families of shifts: block-admissible $1$-shifts of subsets of various sizes, path-admissible $1$-shifts of singletons, and path-admissible $m$-shifts of singletons, for various $m\geq 1$,
and show that specific sequences generate the same diagonal. 

\begin{definition}
\label{def:SU-admissible}
Let $\sigma$ denote either one of the two ordered partitions of $[n]$ in an ordered \mbox{$(2,n)$-parti}\-tion tree, and let~$\b{M} = (M_1,\dots,M_p)$ with $M_1 \subsetneq \sigma_{i_1}, \dots, M_p \subsetneq \sigma_{i_p}$ for some~$p \ge 1$.
Then the sequence of right shifts~$R_\mathbf{M}(\sigma) \eqdef R_{M_p} \dots R_{M_1}(\sigma)$ is 
\begin{enumerate}
\item \defn{block-admissible} if we have $1\leq i_1 < i_2 < \dots < i_p \leq k-1$, and each $R_{M_j}$ is a block-admissible $1$-shift,
\item \defn{path-admissible} if each $R_{M_j}$ is a path-admissible $m_j$-shift, for some $m_j \geq 1$.
\end{enumerate}
Admissible sequences of left shifts are defined similarly and are denoted by $L_\mathbf{M}(\tau)$.
\end{definition}

By convention, we declare the empty sequence of shift operators to be admissible, and to act by the identity, \ie we have $R_{\mathbf{\varnothing}}(\sigma)  \eqdef  \sigma$ and $L_{\mathbf{\varnothing}} (\tau) \eqdef  \tau$.

\begin{definition}
\label{def:classical-SU}
The facets of the \defn{original $\SU$ diagonal}, the \defn{$1$-shift $\SU$ diagonal} and the \defn{$m$-shift $\SU$ diagonal} are defined by the formula
\begin{align*}
\SUD([n]) = \bigcup_{(\sigma,\tau)} \bigcup_{\mathbf{M}, \mathbf{N}} R_\mathbf{M}(\sigma)\times L_\mathbf{N}(\tau)
\end{align*}
where the unions are taken over all $\SCP$s $(\sigma, \tau)$ of $[n]$, and respectively over all block-admissible sequences of subset $1$-shifts $\b{M},\b{N}$, over all path-admissible sequences of singleton $1$-shifts, and over all path-admissible sequences of singleton $m_j$-shifts, for various $m_j \ge 1$.
\end{definition}

\begin{remark}
Observe that the left (\resp right) shifts acts on the right (\resp left) ordered partition.
In the analogous description for the $\LA$ diagonal $\LAD$ obtained at in \cref{subsec:shifts-under-iso}, left and right shifts act on the left and right ordered partitions, respectively.
\end{remark}

%%%%%%%%%%%%%%%

\subsubsection{First isomorphism between $\SU$ diagonals}

We start by analyzing the original $\SU$ diagonal.

\begin{proposition} 
\label{prop:SU-preserves-max}
Block-admissible sequences of subset $1$-shifts, defining the original $\SU$ diagonal, conserve 
\begin{enumerate}
\item the maximal element of any path between two blocks of the same ordered partition,
\item the direction in which this element is traversed. 
\end{enumerate}
In particular, for a pair of ordered partitions $(R_{\mathbf{M}}(\sigma),L_{\mathbf{N}}(\tau))$ obtained via a block-admissible sequence of $1$-shifts from a $\SCP$~$(\sigma,\tau)$, we have
\begin{align}
\label{eq:max-1}
\max R_{\mathbf{M}}(\sigma)_{i,j} = \max \sigma_{i,j} \qquad \text{and} \qquad \max L_{\mathbf{N}}(\tau)_{i,j} = \max \tau_{i,j} \tag{P}
\end{align}
and consequently
\begin{align}
\label{eq:max-2}
\max R_{\mathbf{M}}(\sigma)_{i,i+1} = \max \sigma_{i+1} \qquad \text{and}\qquad \max L_{\mathbf{N}}(\tau)_{i,i+1} = \max \tau_{i} . \tag{B}
\end{align}
\end{proposition}

Note that in \cref{eq:max-1}, the maxima $\max \sigma_{i,j}$ and $\max \tau_{i,j}$ are maxima of \emph{paths}, while in \cref{eq:max-2} the maxima $\max \sigma_{i+1}$ and $\max \tau_{i}$ are maxima of \emph{blocks}.

\begin{proof}
We consider the right shift operator; the left shift operator proceeds similarly.
Let us start with Point (1).
As $\SCP$s trivially meet the above conditions, we will prove the result inductively by assuming that the result holds for $(R_{\mathbf{M}}(\sigma),L_{\mathbf{N}}(\tau))$, and then showing that applying a block-admissible operator $R_{M_k}$, for $M_k\subsetneq \sigma_k$,
conserves the maximal elements of paths.
By the inductive hypothesis, we know that $\max R_{\mathbf{M}}(\sigma)_{k,k+1}= \max \sigma_{k+1}$.
As $R_{M_k}$ is an admissible operator, we know two things: firstly that $\min M_k > \max R_{\mathbf{M}}(\sigma)_{k+1}$, and secondly that $\max \sigma_{k+1} = \max R_{\mathbf{M}}(\sigma)_{k+1}$, as $k$ is greater than the maximal index used by $\mathbf{M}$.
So combining these, we know that
\begin{align}
\label{eq:shift-subset-dominates}
\min M_k > \max R_{\mathbf{M}}(\sigma)_{k+1} = \max \sigma_{k+1} = \max R_{\mathbf{M}}(\sigma)_{k,k+1} .  \tag{W}
\end{align}
A key consequence of this inequality is that the intersection graph of $(R_{M_k}R_{\mathbf{M}}(\sigma),L_{\mathbf{N}}(\tau))$ is a bipartite tree conditional on $(R_{\mathbf{M}}(\sigma),L_{\mathbf{N}}(\tau))$ being a bipartite tree: the shift will not disconnect the graph as none of the shifted elements are in the path $R_{\mathbf{M}}(\sigma)_{k,k+1}$. 
So, it is legitimate to speak of unique paths between blocks. 

We now explicitly explore how the shift operator $R_{M_k}$ alters paths. 
Throughout the rest of this proof, we use the following shorthand: we denote by $\delta_{k,k+1} \eqdef  R_{\mathbf{M}}(\sigma)_{k,k+1}$ the path between the $k$\ordinal{} and $(k+1)$\ordinalst{} blocks in $(R_{\mathbf{M}}(\sigma),L_{\mathbf{N}}(\tau))$, and by $\delta_{k+1,k}$ the same path reversed.
Let~$\gamma$ be any path between two blocks of $(R_{\mathbf{M}}(\sigma),L_{\mathbf{N}}(\tau))$, and let~$\gamma'$ be the path between the same blocks, by indices, in $(R_{M_k}R_{\mathbf{M}}(\sigma),L_{\mathbf{N}}(\tau))$. 
There are four cases to consider. 
\begin{enumerate}[i)]
\item the path $\gamma$ does not contain an element of $M_k$.
In this case, it is unaffected by the shift, so $\gamma' = \gamma$.
We note that in light of \cref{eq:shift-subset-dominates}, the path $\delta_{k,k+1}$ meets this case.

\item the path $\gamma$ contains one element $m\in M_k$, \ie it is of the form $\gamma = \alpha m \beta$, with $\alpha$ or $\beta$ possibly empty.
We assume that $m$ is incoming to $R_\mathbf{M}(\sigma)_{k}$, the case when it is outgoing is similar.
We must have $\alpha \cap \delta_{k,k+1}=\varnothing$, since otherwise there would be an oriented loop from $R_\mathbf{M}(\sigma)_{k}$ to itself.
For the same reason, $\beta \cap \delta_{k,k+1}$ must be connected and starting at $R_\mathbf{M}(\sigma)_{k}$.
Then there are two cases to consider:
	\begin{enumerate}
	\item the path $\beta$ does not use any steps of $\delta_{k,k+1}$, in which case $\gamma' = \alpha m \delta_{k+1,k} \beta$.
	This is a path in the tree of $(R_{M_k}R_{\mathbf{M}}(\sigma),L_{\mathbf{N}}(\tau))$ with no repeated steps, as such it must be the unique minimal path. 

	\item the path $\beta$ uses steps of $\delta_{k,k+1}$, in which case $\gamma' = \alpha m (\delta_{k+1,k}\ssm \beta)(\beta \ssm \delta_{k+1,k})$.
	This follows, as we know that $\beta$ must follow the path $\delta_{k,k+1}$ for some time before diverging ($\beta$ could also be a subset of $\delta_{k,k+1}$, in which case it will never diverge).
	As such, the path $(\delta_{k+1,k}\ssm \beta)$ reaches the point of divergence from $R_{\mathbf{M}}(\sigma)_{k+1}$ instead of $R_{\mathbf{M}}(\sigma)_{k}$, and the path $(\beta \ssm \delta_{k+1,k})$ completes the rest of the route unchanged.
	\end{enumerate}

\item the path $\gamma$ contains two elements of $M_k$. 
In this case we still have $\gamma'=\gamma$ (in path elements) but $\gamma'$ will step through the $(k+1)$\ordinalst{} block instead of $\gamma$ stepping through the $k$\ordinal{} block.

\item the path $\gamma$ contains more than two elements of $M_k$. 
This is impossible, as $\gamma$ would not be a minimal path on a tree.

\end{enumerate}
Observe that all (non-trivial or non-contradictory) paths $\gamma'$ contain $m \geq \min M_k$ and either some addition or deletion by $\delta_{k,k+1}$.
It thus follows from \cref{eq:shift-subset-dominates} that $\max \gamma' = \max \gamma$, since in each case, the maximal element will either be $m$, or in $\alpha$, or in $\beta$.
This finishes the proof of Point~(1).

For Point (2), we need to see that the maximal element $\max \gamma = \max \gamma'$ is traversed in the same direction.
It is immediate for cases (i), (ii.a) and (iii); the condition is empty in the case~(iv).
For the case (ii.b) it follows from the observation that the number of steps of $\beta \cap \delta_{k,k+1}$ and~$\delta_{k,k+1} \ssm \beta$ have the same parity: since~$\delta_{k,k+1}$ has an even number of steps, either they both have an even number of steps, or an odd number, which completes the proof.
\end{proof}

\begin{corollary} 
\label{cor:SU-1-shift-preserves-max}
Path-admissible sequences of singleton $1$-shifts, defining the $1$-shift $\SU$ diagonal, conserve 
\begin{enumerate}
\item the maximal element of any path between two blocks of the same ordered partition,
\item the direction in which this element is traversed. 
\end{enumerate}
\end{corollary}

\begin{proof}
By the definition of path-admissible $1$-shifts (\cref{def:SU-admissible}), the outer inequality of \cref{eq:shift-subset-dominates} holds by assumption.
As such, can run the proof of \cref{prop:SU-preserves-max} \emph{mutatis mutandis}.
\end{proof}

We are now in position to show that the $1$-shift and $m$-shift descriptions are equivalent. 

%Combining \cref{cor:SU-1-shift-preserves-max} and \cref{prop:iso-1-to-m-shift}

\begin{proposition}
\label{prop:iso-1-to-m-shift}
The $1$-shift and $m$-shift $\SU$ diagonals coincide.
\end{proposition}

\begin{proof}
It is clear that any path-admissible sequence of $1$-shifts is a path-admissible sequence of $m$-shifts, and thus that the facets of the $1$-shift $\SU$ diagonal are facets of the $m$-shift $\SU$ diagonal. 
For the reverse inclusion, we need to show that any $m$-shift can be re-expressed as a path-admissible sequence of $1$-shifts. 
We proceed by induction, and consider only the case of right shifts, the case of left shifts is similar.
For right $1$-shifts, there is nothing to prove.
Let $(\sigma,\tau)$ be a pair of ordered partitions which has been generated only by $k$-shifts, for~$k<m$. 
We wish to show that any path-admissible right $m$-shift $R_\rho^m$ on $\sigma$ can be decomposed as a path-admissible $1$-shift $R_\rho^{1}$ followed by a path-admissible $(m-1)$-shift $R_\rho^{m-1}$, which yields the result by induction.
As $R_\rho^{m}$ is path-admissible, we know that $\rho > \max \sigma_{i,i+m}$, and we want to show that 
$\rho>\max \sigma_{i,i+m}\geq \max \sigma_{i,i+1}, \max R_\rho^{1}(\sigma)_{i+1,i+m}$.
\begin{center}
\begin{tikzpicture}[scale=0.7]
\node[anchor=east] (1) at (-0.5,0) {$\sigma_i$};
\coordinate (2) at (1,0);
\node[anchor=south] (3) at (1,1.5) {$\sigma_{i+1}$};
\node[anchor=west] (4) at (2.5,0) {$\sigma_{i+m}$}; 
\draw[thick,->] (1.east) -- node[below] {$\alpha$} (2);
\draw[thick,->] (2) -- node[left] {$\beta$} (3.south);
\draw[thick,->] (2) -- node[below] {$\gamma$} (4.west);
\end{tikzpicture}
\end{center}
We define the oriented paths $\alpha \eqdef \sigma_{i,i+1}\setminus \sigma_{i+1,i+m}$, $\beta \eqdef \sigma_{i,i+1}\setminus \sigma_{i,i+m}$ and $\gamma \eqdef \sigma_{i,i+m}\setminus \sigma_{i,i+1}$, as illustrated above. 
Suppose that $\beta$ is not empty, and moreover that $\max\sigma_{i+1,i+m}>\max\sigma_{i,i+m}$ or $\max \sigma_{i,i+1}> \max \sigma_{i,i+m}$.
Then we must have that $\max \beta> \max \alpha, \max \gamma$.
However, $\max\beta$ cannot be the maximum of both $\sigma_{i,i+1}$ and $\sigma_{i+1,i+m}$.
Indeed, in this case, it would be traversed in two opposite directions, which is impossible since by the induction hypothesis $(\sigma,\tau)$ can be generated by $1$-shifts, and by \cref{cor:SU-1-shift-preserves-max} these conserve the maximal elements of paths and their direction.
We thus have $\max \sigma_{i,i+m}\geq \max\sigma_{i,i+1}, \max \sigma_{i+1,i+m}$, and applying \cref{cor:SU-1-shift-preserves-max} again yields $\max \sigma_{i,i+m}\geq \max R_\rho^{1}(\sigma)_{i+1,i+m}$ as required.
\end{proof}

We stress that \cref{eq:shift-subset-dominates} holds for $m$-shifts without us needing to perform shifts in increasing order, or requiring $\min M_k > \max R_{\mathbf{M}}(\sigma)_{k+1}$.
We are now ready to prove that the $m$-shift description is equivalent to the original one. 

\begin{theorem}
\label{prop:iso-original-shift-diagonals}
The original and $m$-shift $\SU$ diagonals coincide.
\end{theorem}

\begin{proof}
Since $m$-shift and $1$-shift diagonals are equivalent (\cref{prop:iso-1-to-m-shift}), it suffices to show that the $1$-shift and original $\SU$ diagonals coincide. 
We analyze the right shift operator, the case of the left shift is similar. 
First, we observe that any block-admissible right shift $R_{M}(\sigma)$, for~$M\subsetneq\sigma_k$, can be decomposed into a series of singleton right $1$-shifts; since $\min M > \max \sigma_{k,k+1}$ by the proof of \cref{prop:SU-preserves-max}, we can shift the elements of $M$ to the right, one after the other (in any order!).
This shows that any facet of the original $\SU$ diagonal is also a facet in the $1$-shift $\SU$ diagonal.

For the reverse inclusion, we proceed by induction. 
We are required to show that if we apply a right $1$-shift to $(R_{\mathbf{M}}(\sigma),L_{\mathbf{N}}(\tau))$, say $(R_{\rho}R_{\mathbf{M}}(\sigma),L_{\mathbf{N}}(\tau))$, then this can be re-expressed as a well-defined subset shift operation $(R_{\mathbf{M'}}(\sigma),L_{\mathbf{N}}(\tau))$. 
Suppose that prior to the $1$-shift, the element $\rho$ lives in block $\ell$, then we must have
\begin{align*}
1 \leq i_1 < \cdots < i_j \leq \ell < i_{j+1} <\cdots< i_p \leq k-1
\end{align*}
for some $j$. 
If $i_j < \ell$, then we have $R_{\rho}R_{\mathbf{M}}(\sigma) = R_{M_{p}}\cdots R_{\{\rho\}}R_{M_{j}}\cdots R_{M_{1}}(\sigma)$, and we are done.
Otherwise, if $i_j = \ell$, we set $M'_{j} \eqdef M_{j} \cup \{\rho \}$. 
It is clear that ${R_{\rho}R_{\mathbf{M}}(\sigma) = R_{M_{p}}\cdots R_{M'_{j}}\cdots R_{M_{1}}(\sigma)}$, however, we need to check that $R_{M'_{j}}$ is block-admissible, \ie that $\min M_{j}' > \max (R_{M_{j-1}}\cdots R_{M_{1}}(\sigma))_{i_j+1}$.
If we have~$\rho > \min M_{j}$, then we are done since in this case $\min M_{j}'=\min M_{j}$ and $R_{M_{j}}$ is block-admissible.
Otherwise, we have $\rho=\min M_{j}'$. 
Since by definition block-admissible shift operators do not move the minimal element of a block, we have $\rho > \min R_{\mathbf{M}}(\sigma)_{i_j}= \min (R_{M_{j-1}}\cdots R_{M_{1}}(\sigma))_{i_j}$.
Then, by induction, \cref{prop:SU-preserves-max} shows that $\rho>\max \sigma_{i_j+1} = \max (R_{M_{j-1}}\cdots R_{M_{1}}(\sigma))_{i_j + 1}$, where the equality follows as $i_1<\cdots<i_j$. 
This proves that $R_{M'_{j}}$ is block-admissible, completing the inductive proof.
\end{proof}

%%%%%%%%%%%%%%%

\subsubsection{Inversions}
\label{subsec:inversions}

Our next goal is to prove the equivalence between the $m$-shift and geometric $\SU$ diagonals (\cref{prop:iso-shift-IJ-diagonals}).
As a tool for this proof, we now study \emph{inversions}, or crossings in the partition trees of the geometric $\SU$ diagonal.

\begin{definition}\label{def:inverions}
Let $\sigma$ be an ordered partition.
\begin{itemize}
    \item The \defn{inversions} of an ordered partition are $I(\sigma):= \{(i,j): i<j \land \sigma^{-1}(j)<\sigma^{-1}(i) \}$.
    \item The \defn{anti-inversions} of an ordered partition are $J(\sigma):= \{(i,j): i<j \land \sigma^{-1}(i)< \sigma^{-1}(j) \}$.
\end{itemize}
We then define the \defn{inversions of an ordered partition pair} $I((\sigma,\tau)):=I(\tau)\cap J(\sigma)$. 
\end{definition}
In words, the inversions of an ordered partition pair are those $i<j$ pairs in which $j$ comes in an earlier block than $i$ in $\tau$, and $i$ comes in an earlier block than $j$ in $\sigma$. 

\begin{proposition}
\label{p:crossings}
The set of inversions of a facet of the geometric $\SU$ diagonal is in bijection with its set of edge crossings. 
Moreover, under this bijection, strong complementary pairs correspond to facets with no crossings.
\end{proposition}

\begin{proof}
For the first part of the statement,
we note that crossings are clearly produced by both $I(\tau)\cap J(\sigma)$ and $I(\sigma)\cap J(\tau)$ (i.e. in the second case $j$ appears before $i$ in $\sigma$ and $i$ before $j$ in $\tau$).
However, the later cannot occur in a facet of the geometric $\SUD$; this follows immediately from the $(I,J)$-conditions for $\card{I} = \card{J} = 1$. 
The second part of the statement follows from the fact that facets of the diagonal $\SUD$ with no crossings are in bijection with permutations.
By definition (\cref{def:strong-complementary-pairs}), from a partition one obtains a $\SCP$, which is in the geometric $\SUD$ (\cref{lem:SCP-path-desc}). 
In the other way around, given a $\SCP$, one can read-off the partition in the associated tree, which has no crossings, by a vertical down-slice of edges (\cref{ex:strong-complementary}).
\end{proof}

See \cref{fig: Inversion and lattice counter example} for an example of this bijection.

\begin{definition}
We say that an edge crossing is an \defn{adjacent crossing} if the two crossing elements are in adjacent blocks of the partition tree (\ie they are in blocks of the form $\sigma_i|\sigma_{i+1}$ or $\tau_i|\tau_{i+1}$).
\end{definition}

\begin{lemma}
\label{lem:adjacent-crossing}
A facet $(\sigma,\tau)$ of the geometric $\SU$ diagonal has a crossing if and only if it has an adjacent crossing.
\end{lemma}

\begin{proof}
An adjacent crossing is clearly a crossing. 
In the other direction, suppose there is a crossing between an element of $\sigma_i$ and an element of $\sigma_j$. 
If $\sigma_i$ and $\sigma_j$ are not adjacent, then the ``triangle" produced by the crossing elements encloses another $\sigma_k$ such that $i<k<j$, and this produces other crossings. We may repeat this process until an adjacent crossing is found.
\end{proof}

%%%%%%%%%%%%%%%

\subsubsection{Second isomorphism between $\SU$ diagonals}
\label{sec:Iso m-shifts to IJ}

We now aim at showing that the $m$-shift and the geometric $\SU$ diagonal coincide (\cref{thm:unique-operadic}).
Recall from \cref{rem:inverses} that left and right path-admissible $m$-shifts are inverses to one another. 

\begin{proposition} 
\label{lem:IJ-closed-under-shifts}
Let $(\sigma,\tau)$ be a facet of the geometric $\SU$ diagonal $\SUD$.
Then, any pair of ordered partitions obtained by applying a path-admissible $m$-shift to $(\sigma,\tau)$ is also in the geometric $\SU$ diagonal.
\end{proposition}

\begin{proof}
We consider a right path-admissible shift $R_\rho$, the left shift and dual result proceeds similarly. 
Combining \cref{cor:SU-1-shift-preserves-max} and \cref{prop:iso-1-to-m-shift}, we have that the maxima of paths between consecutive vertices in $(R_{\rho}(\sigma),\tau)$ are the same as the ones in $(\sigma,\tau)$, and are moreover traversed in the same direction.
Thus, all maxima of paths in $(R_{\rho}(\sigma),\tau)$ are traversed from right to left, and hence by \cref{thm:facet-ordering}, we have that $(R_{\rho}(\sigma),\tau)$ is in the geometric $\SU$ diagonal.
\end{proof}

\begin{lemma}
\label{lem:inverse-to-SCP}
Any facet $(\sigma,\tau)$ of the geometric $\SU$ diagonal $\SUD$ is mapped to a $\SCP$ by a path-admissible sequence of inverse $m$-shifts.
\end{lemma}

\begin{proof}
We show that any facet $(\sigma,\tau)$ which has a crossing, and is hence not a $\SCP$, admits an inverse shift operation. 
This shows that a finite number of inverse shift operations converts any facet to a $\SCP$ (if $\sigma$ is an ordered partition of $n$ with $k$ blocks, then clearly less than $n^k$ inverse shifts are possible).
We consider the following partition of the set of facets with crossings, illustrated by example in \cref{ex:proof-inverse-shift}.
\begin{enumerate}
	\item All adjacent blocks are connected by paths of length $2$.
	\item There exist adjacent blocks which are connected by a path of length $2k$ for $k>1$.
	\begin{enumerate}
		\item The maximal step of this path is not the last step.
		\item The maximal step of this path is the last step.
		\begin{enumerate}
			\item The $\tau$ block containing the maximal step is not the greatest block.
			\item The $\tau$ block containing the maximal step is the greatest block.
		\end{enumerate}
	\end{enumerate}
\end{enumerate}
In Case $(1)$, as $(\sigma,\tau)$ has a crossing, it has an adjacent crossing by \cref{lem:adjacent-crossing}. 
This adjacent crossing is of the following form (we illustrate the case where the crossing happens on the left, the case when it happens on the right is similar): 
\begin{center}
\begin{tikzpicture}[scale=.6]  
\node[anchor=east] (1) at (-1.5, 0) {$\sigma_i$};
\node[anchor=east] (2) at (-1.5, -2) {$\sigma_{i+1}$};
\node[anchor=west] (4) at (1.5, 0) {$\tau_j$};
\node[anchor=west] (5) at (1.5, -2) {$\tau_{k}$};
\draw[thick] (1.east) -- (5.west); 
\draw[thick] (2.east) -- (5.west);
\draw[thick] (2.east) -- (4.west);
\node (6) at (1.1, -0.7) {$\rho$};
\node (7) at (0, -1.6) {$b$};
\node (8) at (-1.1, -0.7) {$c$};
\end{tikzpicture}
\quad
\raisebox{1.9em}{$\xrightarrow{R^{-1}_{\rho}}$}
\quad
\begin{tikzpicture}[scale=.6]  
\node[anchor=east] (1) at (-1.5, 0) {$\sigma_i \sqcup \rho$};
\node[anchor=east] (2) at (-1.5, -2) {$\sigma_{i+1}\ssm \rho$};
\node[anchor=west] (4) at (1.5, 0) {$\tau_j$};
\node[anchor=west] (5) at (1.5, -2) {$\tau_{k}$};
\draw[thick] (1.east) -- (4.west); 
\draw[thick] (1.east) -- (5.west);
\draw[thick] (2.east) -- (5.west);
\node (6) at (0.9, -0.4) {$\rho$};
\node (7) at (0, -1.6) {$b$};
\node (8) at (-1.1, -0.7) {$c$};
\end{tikzpicture}
\end{center}
By the path characterization of the geometric $\SU$ diagonal (\cref{thm:facet-ordering}), the fact that $\tau_j < \tau_k$ implies that $\rho > b$, and the fact that $\sigma_i < \sigma_{i+1}$ implies that $b>c$.
Thus, we have $\rho > \max \sigma_{i,i+1}$, which implies that a path-admissible left shift $R_\rho^{-1}$ can be performed.

In Case $(2.a)$, consider the two adjacent blocks say $\sigma_i|\sigma_{i+1}$ (the $\tau$ case is similar), let $\rho \eqdef \max \sigma_{i,i+1}$ denote the maximum of the path between them, and let $m\geq 1$ be such that $\rho$ steps into $\sigma_{i+1+m}$. 
Since $\max \sigma_{i,i+1+m} = \max \sigma_{i,i+1}= \rho$ the definition of the geometric $\SU$ diagonal implies that the block $\sigma_{i+1+m}$ comes after the block $\sigma_{i}$, and by assumption after the block $\sigma_{i+1}$. 
Thus, we know that $m > 1$, and using dashed lines to denote paths of length $\geq 1$ we have the following picture
\begin{center}
\begin{tikzpicture}[scale=.6]  
\node[anchor=east] (1) at (-1.5, 0) {$\sigma_i$};
\node[anchor=east] (2) at (-1.5, -1.5) {$\sigma_{i+1}$};
\node[anchor=east] (3) at (-1.5, -3) {$\sigma_{i+1+m}$};
\node[anchor=west](4) at (1.5, 0) {$\hphantom{.}$};
\node[anchor=west](5) at (1.5, -1.5) {$\hphantom{.}$};
\draw[thick,dashed] (1.east) -- (4.west); 
\draw[thick,dashed] (2.east) -- (5.west);
\draw[thick] (3.east) -- (4.west);
\draw[thick,dashed] (3.east) -- (5.west);
\node (6) at (1.1, -0.9) {$\rho$};
\end{tikzpicture}
\raisebox{2.15em}{$\xrightarrow{R^{-1}_{\rho}}$}
\begin{tikzpicture}[scale=.6]  
\node[anchor=east] (1) at (-1.5, 0) {$\sigma_i$};
\node[anchor=east] (2) at (-1.5, -1.5) {$\sigma_{i+1}\sqcup \rho$};
\node[anchor=east] (3) at (-1.5, -3) {$\sigma_{i+1+m}\ssm \rho$};
\node[anchor=west] (4) at (1.5, 0) {$\hphantom{.}$};
\node[anchor=west] (5) at (1.5, -1.5) {$\hphantom{.}$};
\draw[thick,dashed] (1.east) -- (4.west); 
\draw[thick,dashed] (2.east) -- (5.west);
\draw[thick] (2.east) -- (4.west);
\draw[thick,dashed] (3.east) -- (5.west);
\node (6) at (1.1, -0.7) {$\rho$};
\end{tikzpicture}
\end{center}
As we have $\rho=\max \sigma_{i,i+1} > \max\sigma_{i+1+m,i+1}$, an inverse $m$-shift operation can be performed.

In the Case $(2.b.i)$, there exists $j,m>0$ such that $\tau_{j-m}$ is the block of $\tau$ which contains the maximal element $\rho$, and $\tau_j$ is any greater block on the path from $\sigma_i$ to $\sigma_{i+1}$. 
Then we may apply the following inverse $m$-shift operator.
\begin{center}
\begin{tikzpicture}[scale=.6]  
\node[anchor=east] (1) at (-1.5, 0) {$\sigma_i$};
\node[anchor=east] (2) at (-1.5, -2) {$\sigma_{i+1}$};
\node[anchor=west] (3) at (1.5, 0) {$\tau_{j-m}$};
\node[anchor=west] (4) at (1.5, -2) {$\tau_{j}$};
\draw[thick, dashed] (1.east) -- (4.west); 
\draw[thick] (2.east) -- (3.west);
\draw[thick, dashed] (3.west) -- (4.west);
\node (6) at (-0.6, -1.8) {$\rho$};
\end{tikzpicture}
\raisebox{1.8em}{$\xrightarrow{L_\rho^{-1}}$}
\begin{tikzpicture}[scale=.6]  
\node[anchor=east] (1) at (-1.5, 0) {$\sigma_i$};
\node[anchor=east] (2) at (-1.5, -2) {$\sigma_{i+1}$};
\node[anchor=west] (3) at (1.5, 0) {$\tau_{j-m} \ssm \rho$};
\node[anchor=west] (4) at (1.5, -2) {$\tau_{j} \sqcup \rho$};
\draw[thick, dashed] (1.east) -- (4.west); 
\draw[thick] (2.east) -- (4.west);
\draw[thick, dashed] (3.west) -- (4.west);
\node (6) at (-0.9, -1.6) {$\rho$};
\end{tikzpicture}
\end{center}

There remains to be treated Case $(2.b.ii)$.
We consider the path $\sigma_{i,i+1} \eqdef (i_1,j_1,i_2,\dots,i_{k},\rho)$ to be of length $2k$, $k>1$. 
We denote by $\tau_j$ the block of $\tau$ containing the maximal step $\rho$ of this path, which by hypothesis is the last block of $\tau$. 
Let $\sigma_{i-n}$ be the last block of $\sigma$ which is attained by the path $\sigma_{i,i+1}$ before the block $\sigma_{i+1}$. 
We must have $n>1$, since $\rho$ is the last step and $\sigma_i|\sigma_{i+1}$ are adjacent blocks.
The situation can be pictured as on the left.
\begin{center}
\begin{tikzpicture}[scale=.6]  
\node[anchor=east] (1) at (-1.5, 0) {$\sigma_{i-n}$};
\node[anchor=east] (2) at (-1.5, -1.5) {$\sigma_{i}$};
\node[anchor=east] (3) at (-1.5, -3) {$\sigma_{i+1}$};
\node[anchor=west] (4) at (1.5, 0) {$\tau_{j-m}$};
\node[anchor=west] (5) at (1.5, -1.5) {$\tau_j$};
\draw[thick,dashed] (1.east) -- (4.west); 
\draw[thick] (1.east) -- (5.west);
\draw[thick,dashed] (2.east) -- (4.west);
\draw[thick] (3.east) -- (5.west);
\node (6) at  (1.1, -2.1) {$\rho$};
\node (7) at  (1.1, -0.8) {$i_k$};
\end{tikzpicture}
\raisebox{2.15em}{$\xrightarrow{L_{\rho'}^{-1}}$}
\begin{tikzpicture}[scale=.6]  
\node[anchor=east] (1) at (-1.5, 0) {$\sigma_{i-n}$};
\node[anchor=east] (2) at (-1.5, -1.5) {$\sigma_{i}$};
\node[anchor=east] (3) at (-1.5, -3) {$\sigma_{i+1}$};
\node[anchor=west] (4) at (1.5, 0) {$\tau_{j-m}\ssm \rho'$};
\node[anchor=west] (5) at (1.5, -1.5) {$\tau_j \sqcup \rho'$};
\draw[thick,dashed] (1.east) -- (4.west); 
\draw[thick] (1.east) -- (5.west);
\draw[thick,dashed] (2.east) -- (5.west);
\draw[thick] (3.east) -- (5.west);
\node (6) at  (1.1, -2.1) {$\rho$};
\node (7) at  (1.1, -0.8) {$i_k$};
\end{tikzpicture}
\end{center}
Now let $\rho' \eqdef \max\sigma_{i,i-n}$ be the maximum of the path $\sigma_{i,i-n}=(i_1,j_1,i_2,\dots,i_{k-1},j_{k-1})$, and let~$\tau_{j-m}$, $m\geq 1$ be the block of $\tau$ containing $\rho'$.
We want to show that an inverse left $m$-shift can bring $\rho'$ back to $\tau_j$, \ie we need $\rho' > \max \tau_{j-m,j}$.
As $\rho'\eqdef \max \sigma_{i,i-n}$, and apart from the step $i_k$ we have $\tau_{j-m,j}\subset \sigma_{i,i-n}$, we just need to show that $\rho' > i_k$.
To see that this is indeed the case, it suffices to look at the last steps $$\tau_l \overset{j_{k-1}}{\longrightarrow} \sigma_{i-n} \overset{i_{k}}{\longrightarrow} \tau_j \overset{\rho}{\longrightarrow} \sigma_{i+1}$$ of the path $\sigma_{i,i+1}$. 
By the definition of the geometric $\SU$ diagonal, the fact that $\tau_l < \tau_j$ implies that $j_{k-1} > i_k$, and thus that $\rho' > i_k$, which finishes the proof. 
\end{proof}

\begin{example}
\label{ex:proof-inverse-shift}
Here are examples of each of the cases in the proof of \cref{lem:inverse-to-SCP}. 
We display cases $(1)$, $(2.a)$, $(2.b.i)$ and, $(2.b.ii)$ respectively, reading the diagrams from top-left to bottom-right. 
\begin{center}
\begin{tikzpicture}[xscale=.6,yscale=0.9]  
\node[anchor=east] (1) at (-1.5, -1) {$1$};
\node[anchor=east] (2) at (-1.5, -2) {$23$};
\node[anchor=west] (4) at (1.5, -1) {$3$};
\node[anchor=west] (5) at (1.5, -2) {$12$};
\draw[thick] (1.east) -- (5.west); 
\draw[thick] (2.east) -- (5.west);
\draw[thick] (2.east) -- (4.west);
\end{tikzpicture}
\raisebox{1.3em}{$\xrightarrow{R^{-1}_{3}}$}
\begin{tikzpicture}[xscale=.6,yscale=0.9]  
\node[anchor=east] (1) at (-1.5, -1) {$13$};
\node[anchor=east] (2) at (-1.5, -2) {$2$};
\node[anchor=west] (4) at (1.5, -1) {$3$};
\node[anchor=west] (5) at (1.5, -2) {$12$};
\draw[thick] (1.east) -- (4.west); 
\draw[thick] (1.east) -- (5.west);
\draw[thick] (2.east) -- (5.west);
\end{tikzpicture}
,
\begin{tikzpicture}[xscale=.6,yscale=0.9] 
\node[anchor=east] (1) at (-1.5, -1) {$1$};
\node[anchor=east] (2) at (-1.5, -2) {$2$};
\node[anchor=east] (3) at (-1.5, -3) {$34$};
\node[anchor=west] (4) at (1.5, -1) {$14$};
\node[anchor=west] (5) at (1.5, -2) {$23$};
\draw[thick] (1.east) -- (4.west); 
\draw[thick] (2.east) -- (5.west); 
\draw[thick] (3.east) -- (4.west);
\draw[thick] (3.east) -- (5.west);
\end{tikzpicture}
\raisebox{2.15em}{$\xrightarrow{R^{-1}_{4}}$}
\begin{tikzpicture}[xscale=.6,yscale=0.9]   
\node[anchor=east] (1) at (-1.5, -1) {$1$};
\node[anchor=east] (2) at (-1.5, -2) {$24$};
\node[anchor=east] (3) at (-1.5, -3) {$3$};
\node[anchor=west] (4) at (1.5, -1) {$14$};
\node[anchor=west] (5) at (1.5, -2) {$23$};
\draw[thick] (1.east) -- (4.west); 
\draw[thick] (2.east) -- (4.west);
\draw[thick] (2.east) -- (5.west);
\draw[thick] (3.east) -- (5.west);
\end{tikzpicture}
\end{center}
\begin{center}
\begin{tikzpicture}[xscale=.6,yscale=0.9]     
\node[anchor=east] (1) at (-1.5, -1) {$13$};
\node[anchor=east] (2) at (-1.5, -2) {$2$};
\node[anchor=east] (3) at (-1.5, -3) {$4$};
\node[anchor=west] (4) at (1.5, -1) {$34$};
\node[anchor=west] (5) at (1.5, -2) {$12$};
\draw[thick] (1.east) -- (4.west); 
\draw[thick] (1.east) -- (5.west); 
\draw[thick] (2.east) -- (5.west); 
\draw[thick] (3.east) -- (4.west);
\end{tikzpicture}
\raisebox{2.15em}{$\xrightarrow{L^{-1}_{4}}$}
\begin{tikzpicture}[xscale=.6,yscale=0.9]     
\node[anchor=east] (1) at (-1.5, -1) {$13$};
\node[anchor=east] (2) at (-1.5, -2) {$2$};
\node[anchor=east] (3) at (-1.5, -3) {$4$};
\node[anchor=west] (4) at (1.5, -1) {$3$};
\node[anchor=west] (5) at (1.5, -2) {$124$};
\draw[thick] (1.east) -- (4.west); 
\draw[thick] (1.east) -- (5.west); 
\draw[thick] (2.east) -- (5.west); 
\draw[thick] (3.east) -- (5.west);
\end{tikzpicture}
,
\begin{tikzpicture}[xscale=.6,yscale=0.9]    
\node[anchor=east] (1) at (-1.5, -1) {$12$};
\node[anchor=east] (2) at (-1.5, -2) {$3$};
\node[anchor=east] (3) at (-1.5, -3) {$4$};
\node[anchor=west] (4) at (1.5, -1) {$23$};
\node[anchor=west] (5) at (1.5, -2) {$14$};
\draw[thick] (1.east) -- (4.west); 
\draw[thick] (1.east) -- (5.west);
\draw[thick] (2.east) -- (4.west);
\draw[thick] (3.east) -- (5.west);
\end{tikzpicture}
\raisebox{2.15em}{$\xrightarrow{L^{-1}_{3}}$}
\begin{tikzpicture}[xscale=.6,yscale=0.9]     
\node[anchor=east] (1) at (-1.5, -1) {$12$};
\node[anchor=east] (2) at (-1.5, -2) {$3$};
\node[anchor=east] (3) at (-1.5, -3) {$4$};
\node[anchor=west] (4) at (1.5, -1) {$2$};
\node[anchor=west] (5) at (1.5, -2) {$134$};
\draw[thick] (1.east) -- (4.west); 
\draw[thick] (1.east) -- (5.west);
\draw[thick] (2.east) -- (5.west);
\draw[thick] (3.east) -- (5.west);
\end{tikzpicture}
\end{center}
Note that, as we have chosen minimal illustrative examples, each inverse $m$-shift is an inverse $1$-shift, and after each shift we obtain a $\SCP$. 
This is not typically the case.
\end{example}

\begin{theorem} 
\label{prop:iso-shift-IJ-diagonals}
The $m$-shift and the geometric $\SU$ diagonals coincide.
\end{theorem}

\begin{proof}
We first note that $\SCP$s are known elements of both the $m$-shift and the geometric $\SU$ diagonals (\cref{lem:SCP-path-desc}).
The proof that every facet of the $m$-shift $\SUD$ is in geometric $\SUD$ follows from the closure of $\SUD$ under the shift operators (\cref{lem:IJ-closed-under-shifts}).
The proof that every facet of geometric $\SUD$ is in shift $\SUD$ follows from the closure of $\SUD$ under the inverse shift operator (\cref{lem:IJ-closed-under-shifts}) and the fact that every facet is sent to a $\SCP$ after a finite number of inverse shifts (\cref{lem:inverse-to-SCP}).
In particular, for any given facet in geometric $\SUD$, this provides a $\SCP$ and a sequence of shifts to form it, showing it is a facet of $m$-shift $\SUD$.
\end{proof}

Combining \cref{prop:iso-original-shift-diagonals} and \cref{prop:iso-shift-IJ-diagonals}, we obtain the desired equivalence between the original $\SU$ diagonal from~\cite{SaneblidzeUmble} and the geometric $\SU$ diagonal from \cref{def:LA-and-SU}.

\begin{theorem}
\label{thm:recover-SU}
The original and geometric $\SU$ diagonals coincide.
\end{theorem}

%%%%%%%%%%%%%%%

\subsection{Shifts under the face poset isomorphism}
\label{subsec:shifts-under-iso}

Having proven the equivalence of the original and geometric $\SU$ diagonals, we now use the face poset isomorphisms between the geometric $\LA$ and $\SU$ diagonals from \cref{subsec:isos-LA-SU} to translate results and combinatorial descriptions from one to the other. 
Under the isomorphism $t(r\times r):\LAD\to\SUD$ from \cref{rem:Alternate Isomorphism}, we get a straightforward analogue of \cref{def:SU-admissible} for the $\LA$ diagonal.
Firstly, the morphism $r \times r$ exchanges $\max$ and $\min$, which yields the following ``dual" notions of admissibility.

\begin{definition}
\label{def:LA-admissible}
Let $\sigma$ denote either one of the two ordered partitions of $[n]$ in a $(2,n)$-partition tree, and let $M \subsetneq \sigma_i$.
The right $m$-shift $R_{M}$ (\resp the left $m$-shift $L_{M}$) is 
\begin{enumerate}
\item \defn{block-admissible} if $\max \sigma_i \notin M$ and $\max M < \min \sigma_{i+m}$ (\resp $\max M < \min \sigma_{i-m}$),
\item \defn{path-admissible} if $\max M< \min \sigma_{i,i+m}$ (\resp $\max M < \min \sigma_{i,i-m}$).
\end{enumerate}
\end{definition}

Secondly, as the morphism $t$ switches our ordered partitions, this means that the $\LA$ lefts shifts will act on the left ordered partition, and the $\LA$ right shifts will act on the right ordered partition.
Consequently, admissible sequences of $\LA$ shifts are defined similarly to \cref{def:SU-admissible} (simply replace $\sigma$ with $\tau$).
Which provides an analogue of \cref{def:classical-SU} for the $\LA$ diagonal. 

\begin{definition}
\label{def:classical-LA}
The facets of the \defn{subset shift, $1$-shift and $m$-shift $\LA$ diagonals} are given by the formula
\begin{align*}
\LAD([n]) = \bigcup_{(\sigma,\tau)} \bigcup_{\mathbf{M}, \mathbf{N}} L_\mathbf{M}(\sigma)\times R_\mathbf{N}(\tau)
\end{align*}
where the unions are taken over all $\SCP$s $(\sigma, \tau)$ of $[n]$, and respecitvely over all block-admissible sequences of subset $1$-shifts $\b{M},\b{N}$, over all path-admissible sequences of singleton $1$-shifts, and over all path-admissible sequences of singleton $m_j$-shifts, for various $m_j \ge 1$.
\end{definition}

We now formally verify that the isomorphism $t(r\times r)$ relates these shift definitions as claimed.

\begin{proposition} 
\label{prop:trr is an isomorphism of shifts}
Let $(\sigma,\tau)$ be a facet of $\LAD$.
For each type of $\LA$ shift, let $\b{M},\b{N}$ be admissible sequences of this type, then 
\begin{align*}
t(r,r)(L_\mathbf{N}(\sigma), R_\mathbf{M}(\tau)) = (R_{r\mathbf{M}}(r\tau), L_{r\mathbf{N}}(r\sigma))
\end{align*}
where $r\mathbf{M} \eqdef (rM_{1},\dots,rM_{p})$ and $r\mathbf{N}$ (defined similarly) are admissible sequences of $\SU$ shifts of the same type.
\end{proposition}

\begin{proof}
As reversing the elements then shifting them is the same as shifting the elements then reversing them,
it is clear that the equality holds if $r\mathbf{M},r\mathbf{N}$ are admissible sequences of $\SU$ shifts.
As such, we must simply verify the admissibility, and given the equivalence of the various shift definitions, we just do this for path-admissibility of $1$-shifts.
Consider a right shift $R_{M}$, for $M \in \sigma_i$, which is path-admissible in the $\LA$ sense (\cref{def:LA-admissible}). 
Then, we have $\max M < \min \sigma_{i,i+1}$ which implies that $\min r M > \max r \sigma_{i,i+1}$, and thus the right shift $R_{rM}$ is path-admissible in the $\SU$ sense (\cref{def:SU-admissible}).
Here, $rM \in r\sigma_i$ is interpreted as being in a block of the right partition ($\tau$ in the notation of the definition).
So, if $\mathbf{M} = (M_{i_1},\dots,M_{i_p})$, define $r(\mathbf{M}) \eqdef (rM_{1},\dots,rM_{p})$, and from the prior it is clear this is a path-admissible sequence of $\SU$ shifts, finishing the proof.
\end{proof}
Thus, given $t(r\times r)$ is an isomorphism of the geometric diagonals (\cref{rem:Alternate Isomorphism}), and the geometric $\SU$ diagonal coincides with the shift $\SU$ diagonals (\cref{subsec:SU-shifts}), we immediately obtain the following statement.
\begin{proposition}
The geometric $\LA$ diagonal and all shift $\LA$ diagonals coincide.
\end{proposition}

\begin{remark}
The isomorphism $rs\times rs:\LAD \to \SUD$, identified in \cref{thm:bijection-operadic-diagonals}, also sends shifts operators to shift operators, but it sends left shift operators to right shift operators and vice versa.
The $\LA^{\op}$ and $\SU^{\op}$ diagonals also admit obvious dual shift descriptions.
\end{remark}

We now explore in example how the isomorphism $t(r\times r)$ translates the shift operators.

\begin{example} 
	\label{ex:shift translation by theta}
The isomorphism $t(r\times r)$ sends the $\SCP$ $(\sigma,\tau)  \eqdef  (5|17|4|236,57|146|3|2)$ to the $\SCP$ $(\sigma',\tau')  \eqdef  (13|247|5|6,3|17|4|256)$. 
The corresponding $(2,n)$-partition trees present a clear symmetry
\begin{center}
\begin{tikzpicture}[scale=.7]  
\node[anchor=east] (1) at (-1.5, -1) {$5$};
\node[anchor=east] (2) at (-1.5, -2) {$17$};
\node[anchor=east] (3) at (-1.5, -3) {$4$};
\node[anchor=east] (4) at (-1.5, -4) {$236$};
\node[anchor=west] (5) at (1.5, -1) {$57$};
\node[anchor=west] (6) at (1.5, -2) {$146$};
\node[anchor=west] (7) at (1.5, -3) {$3$};
\node[anchor=west] (8) at (1.5, -4) {$2$};
\draw[thick] (1.east) -- (5.west); 
\draw[thick] (2.east) -- (5.west); 
\draw[thick] (2.east) -- (6.west); 
\draw[thick] (3.east) -- (6.west); 
\draw[thick] (4.east) -- (6.west); 
\draw[thick] (4.east) -- (7.west); 
\draw[thick] (4.east) -- (8.west);
\end{tikzpicture}
\raisebox{3.4em}{$\xrightarrow{t(r\times r)}$}
\begin{tikzpicture}[scale=.7]  
\node[anchor=east] (13) at (-1.5, -1) {$13$};
\node[anchor=east] (247) at (-1.5, -2) {$247$};
\node[anchor=east] (5) at (-1.5, -3) {$5$};
\node[anchor=east] (6) at (-1.5, -4) {$6$};
\node[anchor=west] (3) at (1.5, -1) {$3$};
\node[anchor=west] (17) at (1.5, -2) {$17$};
\node[anchor=west] (4) at (1.5, -3) {$4$};
\node[anchor=west] (256) at (1.5, -4) {$256$};
\draw[thick] (13.east) -- (17.west); 
\draw[thick] (13.east) -- (3.west); 
\draw[thick] (247.east) -- (256.west); 
\draw[thick] (247.east) -- (17.west); 
\draw[thick] (247.east) -- (4.west); 
\draw[thick] (5.east) -- (256.west); 
\draw[thick] (6.east) -- (256.west);
\end{tikzpicture}
\end{center}
We now illustrate all possible path admissible $\LA$ $1$-shifts of $(\sigma,\tau)$ and all possible path admissible $\SU$ $1$-shifts of $(\sigma',\tau')$. 
We first display how the $\LA$ left shifts act on $\sigma$, and the $\SU$ left shifts act on $\tau'$.
The $\LA$ shifts have been drawn so that the leftmost arrow shifts the smallest element, and the $\SU$ shifts have been drawn so that the leftmost arrow shifts the largest element.
As such, the face poset isomorphism $t(r\times r)$ directly translates one diagram to the other.
The specific element being shifted can be inferred by the source and target of the arrow.
\begin{center}
{\small
\begin{tikzcd}
& 5|17|4|236 \arrow[ld] \arrow[d] \arrow[rd]\\
15|7|4|236 \arrow[rd] \arrow[d]&
5|17|24|36  \arrow[dl] \arrow[rd]&
5|17|34|26 \arrow[d] \arrow[dl]\\
15|7|24|36 \arrow[dr]&
15|7|34|26  \arrow[d]&
5|17|234|6  \arrow[dl]\\
&15|7|234|6 
\end{tikzcd}
\qquad
\begin{tikzcd}
& 3|17|4|256 \arrow[ld] \arrow[d] \arrow[rd]\\
37|1|4|256 \arrow[rd] \arrow[d]&
3|17|46|25  \arrow[dl] \arrow[rd]&
3|17|44|26 \arrow[d] \arrow[dl]\\
37|1|46|25 \arrow[dr]&
36|17|44|26  \arrow[d]&
3|17|456|2  \arrow[dl]\\
&37|1|456|2
\end{tikzcd}
}
\end{center}
We now illustrate all the possible $\LA$ right shifts acting on $\tau$
\begin{center}
\begin{tikzcd}
\sigma \times 57|146|3|2 \arrow[r,"\rho=1"] & \sigma \times 57|46|13|2 \arrow[r,"\rho=1"] & \sigma \times 57|46|3|12
\end{tikzcd}
\end{center}
and all possible $\SU$ right shifts acting on $\sigma'$
\begin{center}
\begin{tikzcd}
13|247|5|6 \times \tau' \arrow[r,"\rho=7"] & 13|24|57|6 \arrow[r,"\rho=7"] \times \tau' & 13|24|5|67 \times \tau' .
\end{tikzcd}
\end{center}
No other shifts are possible; observe for instance that we cannot perform the $\LA$ left shift $15|7|234|6 \times 57|46|13|2 \xrightarrow{\rho=2} 15|27|34|6 \times 57|46|13|2$ as the minimal path connecting $234$ and $7$ contains $1$, which is smaller than $2$ (see \cref{ex:ECbijection}).
We shall see in \cref{sec:Shift-lattice}, that these diagrams are the Hasse diagrams of lattices.
\end{example}

\begin{example}
It was observed in~\cite{LaplanteAnfossi} that the $\LA$ and $\SU$ diagonals coincided up until $n=3$, however due to their dual shift structure they generate the non-$\SCP$ pairs in a dual fashion.
In particular, the two center faces of the subdivided hexagon of \cref{fig:LUSAdiagonals} are generated by 
\begin{center}
\begin{tikzpicture}[scale=.6,xscale=0.6]  
\node[anchor=east] (1) at (-1.5, -1) {$13$};
\node[anchor=east] (2) at (-1.5, -2) {$2$};
\node[anchor=west] (4) at (1.5, -1) {$3$};
\node[anchor=west] (5) at (1.5, -2) {$12$};
\draw[thick] (1.east) -- (4.west); 
\draw[thick] (1.east) -- (5.west);
\draw[thick] (2.east) -- (5.west);
\end{tikzpicture}
\quad
\raisebox{1.3em}{$\xrightarrow{R^{\SU}_{3}}$}
\quad
\begin{tikzpicture}[scale=.6,xscale=0.6]  
\node[anchor=east] (1) at (-1.5, -1) {$1$};
\node[anchor=east] (2) at (-1.5, -2) {$23$};
\node[anchor=west] (4) at (1.5, -1) {$3$};
\node[anchor=west] (5) at (1.5, -2) {$12$};
\draw[thick] (1.east) -- (5.west); 
\draw[thick] (2.east) -- (5.west);
\draw[thick] (2.east) -- (4.west);
\end{tikzpicture}
\qquad \qquad
\begin{tikzpicture}[scale=.6,xscale=0.6]  
\node[anchor=east] (1) at (-1.5, -1) {$12$};
\node[anchor=east] (2) at (-1.5, -2) {$3$};
\node[anchor=west] (4) at (1.5, -1) {$2$};
\node[anchor=west] (5) at (1.5, -2) {$13$};
\draw[thick] (1.east) -- (4.west); 
\draw[thick] (2.east) -- (5.west);
\draw[thick] (1.east) -- (5.west);
\end{tikzpicture}
\quad
\raisebox{1.3em}{$\xrightarrow{L^{\SU}_{3}}$}
\quad
\begin{tikzpicture}[scale=.6,xscale=0.6]  
\node[anchor=east] (1) at (-1.5, -1) {$12$};
\node[anchor=east] (2) at (-1.5, -2) {$3$};
\node[anchor=west] (4) at (1.5, -1) {$23$};
\node[anchor=west] (5) at (1.5, -2) {$1$};
\draw[thick] (1.east) -- (4.west); 
\draw[thick] (2.east) -- (4.west);
\draw[thick] (1.east) -- (5.west);
\end{tikzpicture}
\end{center}
and
\begin{center}
\begin{tikzpicture}[scale=.6,xscale=0.6]  
\node[anchor=east] (1) at (-1.5, -1) {$1$};
\node[anchor=east] (2) at (-1.5, -2) {$23$};
\node[anchor=west] (4) at (1.5, -1) {$13$};
\node[anchor=west] (5) at (1.5, -2) {$2$};
\draw[thick] (1.east) -- (4.west); 
\draw[thick] (2.east) -- (4.west);
\draw[thick] (2.east) -- (5.west);
\end{tikzpicture}
\quad
\raisebox{1.3em}{$\xrightarrow{R^{\LA}_{1}}$}
\quad
\begin{tikzpicture}[scale=.6,xscale=0.6]  
\node[anchor=east] (1) at (-1.5, -1) {$1$};
\node[anchor=east] (2) at (-1.5, -2) {$23$};
\node[anchor=west] (4) at (1.5, -1) {$3$};
\node[anchor=west] (5) at (1.5, -2) {$12$};
\draw[thick] (1.east) -- (5.west); 
\draw[thick] (2.east) -- (4.west);
\draw[thick] (2.east) -- (5.west);
\end{tikzpicture}
\qquad \qquad
\begin{tikzpicture}[scale=.6,xscale=0.6]  
\node[anchor=east] (1) at (-1.5, -1) {$2$};
\node[anchor=east] (2) at (-1.5, -2) {$13$};
\node[anchor=west] (4) at (1.5, -1) {$23$};
\node[anchor=west] (5) at (1.5, -2) {$1$};
\draw[thick] (1.east) -- (4.west); 
\draw[thick] (2.east) -- (4.west);
\draw[thick] (2.east) -- (5.west);
\end{tikzpicture}
\quad
\raisebox{1.3em}{$\xrightarrow{L^{\LA}_{1}}$}
\quad
\begin{tikzpicture}[scale=.6,xscale=0.6]  
\node[anchor=east] (1) at (-1.5, -1) {$12$};
\node[anchor=east] (2) at (-1.5, -2) {$3$};
\node[anchor=west] (4) at (1.5, -1) {$23$};
\node[anchor=west] (5) at (1.5, -2) {$1$};
\draw[thick] (1.east) -- (4.west); 
\draw[thick] (2.east) -- (4.west);
\draw[thick] (1.east) -- (5.west);
\end{tikzpicture}
\end{center}
where we have aligned each element, of each diagonal, vertically with its dual element.
\end{example}

%%%%%%%%%%%%%%%

\subsection{Shift lattices}
\label{sec:Shift-lattice}

In this section, we show that the $1$-shifts of the operadic diagonals~$\LAD$ and~$\SUD$, define the covering relations of a lattice structure on the set of facets.
More precisely, we show that each $\SCP$ is the minimal element of a lattice isomorphic to a product of chains, where the partial order is given by shifts.
Given the bijection between $\SCP$s and permutations, and our prior enumeration of the facets of the diagonal, this produces two new statistics on permutations.
In addition, later in \cref{sec:Cubical}, we will use the lattice structure to relate the cubical and shift definitions of the $\SUD$ diagonal.

\begin{definition}
\label{def:shift-lattice}
The $\LA$ \defn{shift poset} on the set of facets $(\sigma,\tau) \in \LAD$ is defined to be the transitive closure of the relations $ (\sigma,\tau) \prec (L_\rho \sigma,\tau)$ and $ (\sigma,\tau) \prec (\sigma,R_\rho \tau)$, for all $\LA$ path-admissible $1$-shifts $L_{\rho}$ and $R_\rho$. 
The $\SU$ \defn{shift poset} is defined similarly.
Then, for each permutation $v$ of $[n]$, we define the subposet \defn{$\hour^\LA$} (\resp \defn{$\hour^\SU$}) to be the set of all admissible $\LA$ (\resp $\SU$) shifts of the associated~$\SCP$.
\end{definition}

Given the shift definitions of $\SUD$ and $\LAD$ (\cref{def:classical-SU,def:classical-LA}), the subposets $\hour^\LA$ and~$\hour^\SU$ are clearly connected components of their posets, with a unique minimal element corresponding to the $\SCP$.
We now aim to prove they are also lattices (\cref{prop:shift lattice}).

\begin{lemma}
\label{lem:m-shifts commute}
The $m$-shift operators defining the facets of an operadic diagonal commute. 
That is, whenever the successive composition of two $m_1$, $m_2$-shifts is defined on a facet of the $\LA$ or $\SU$ diagonal, the reverse order of composition is also defined, and the two yield the same facet. 
\end{lemma}
\begin{proof}
A $\SUD$ (\resp $\LAD$) $m$-shift is defined when $\rho$ is greater (\resp smaller) than the maximal (\resp minimal) element of the connecting path.
Combining \cref{cor:SU-1-shift-preserves-max} and \cref{prop:iso-1-to-m-shift}, we know $m$-shifts conserve the maximal (minimal) elements of paths and their directions.
As such, we can commute any two shift operators.
\end{proof}

\begin{definition}
\label{def:heights}
Let $v$ be a permutation of $[n]$, and $(\sigma,\tau)$ the SCP corresponding to $v$.
For $\rho \in [n]$, we define the $\LA$ \defn{left height} and \defn{right height} of $\rho$ to be 
\begin{align*}
	\ell_v(\rho) \eqdef \max (\{0\}\cup \{m>0: L_{\rho}(\sigma) \text{ is a path admissible $\LA$ $m$-shift} \}), \\
	r_v(\rho) \eqdef \max (\{0\}\cup \{m>0: R_{\rho}(\tau) \text{ is a path admissible $\LA$ $m$-shift} \}).
\end{align*}
The left and right heights for the $\SU$ diagonal are defined similarly. 
\end{definition}

The height of an element $\rho$ in a $\SCP$ can be explicitly calculated as follows. 

\begin{lemma}
\label{prop:maximal m-shift formulae}
Let $(\sigma,\tau)$ be the $\SCP$ corresponding to a permutation $v$.
Then, the right $\LA$ height $r_v(\rho)$ (\resp the left $\LA$ height $\ell_v(\rho)$) of $\rho \in [n]$ is given by the number of consecutive blocks of $\sigma$ (\resp of~$\tau$) to the right (\resp left) of the one containing $\rho$ whose minima are larger than $\rho$.
The $\SU$ heights are obtained similarly by considering blocks whose maxima are smaller than $\rho$. 
\end{lemma}

\begin{proof}
We consider the right $\SU$ height, the other cases are similar. 
As $(\sigma,\tau)$ is a $\SCP$, we have $\max \sigma_{k,k+1} = \max \sigma_{k+1}$ for all $k\geq 1$. 
Moreover, from the equivalence between $1$-shifts and $m$-shifts (\cref{prop:iso-1-to-m-shift}), there exists a $m$-shift of $\rho$ from $\sigma_i$ to $\sigma_{i+m}$ if, and only if, there exists a sequence of $m$ consecutive $1$-shifts, each satisfying $\rho > \max \sigma_{j,j+1}=\max \sigma_{j+1}$.
Thus, these iterated $1$-shifts will be path-admissible until the first failure at $j=r_v(\rho)+1$.
\end{proof}

\begin{remark}
The height calculations can also be reformulated directly in terms of the permutation.
For instance, for the $\SU$ diagonal $r_v(\rho)$ is the number of consecutive descending runs, to the right of the descending run containing $\rho$, whose maximal element is smaller than $\rho$.
\end{remark}

\begin{proposition} 
\label{prop:shift lattice}
The subposets~$\hour^\LA$ and~$\hour^\SU$, are lattices isomorphic to products of chains
\begin{align*}
\hour^\LA \cong \prod_{\rho \in [n]} [0,\ell_v(\rho)] \times \prod_{\rho \in [n]} [0,r_v(\rho)]
\quad \text{ and } \quad
\hour^\SU \cong \prod_{\rho \in [n]} [0,r_v(\rho)] \times \prod_{\rho \in [n]} [0,\ell_v(\rho)] \ ,
\end{align*}
where $[0,k]$ is the chain lattice $0<1<\dots<k$, for $k\geq 0$.
\end{proposition}

\begin{proof}
We denote by $L_\rho^m$ (\resp $R_\rho^m$) a left (right) $m$-shift of $\rho$ for $m>0$, and we let it be the identity if~$m=0$.
By the commutativity of $m$-shift operators (\cref{lem:m-shifts commute}), and the existence of unique heights for each element (\cref{prop:maximal m-shift formulae}), every element of $\hour^\LA$ admits a unique shift description $(L^{\ell_n}_{n} \dots L^{\ell_1}_{1}(\sigma),
R^{r_n}_{n}\cdots R^{r_{1}}_{1}(\tau))$, where $0\leq \ell_\rho\leq \ell_v(\rho)$ and $0\leq r_\rho\leq r_v(\rho)$.
Thus, we identify it with the pair of tuples $(\ell_1,\ldots,\ell_n)\times (r_1,\ldots,r_n)$.
This is clearly an isomorphism of lattices.
\end{proof}

Note that the maximal element of $\hour$ (for either diagonal) is given by shifting each element of $(\sigma,\tau)$ by its maximal shift.
The joins and meets of any two elements are also thus given by isomorphism to the product of chains.
For instance, in the case of meets,
\begin{multline*}
	(\ell_1,\ldots,\ell_n,r_1,\ldots,r_n)\land (\ell'_1,\ldots,\ell'_n,r'_1,\ldots,r'_n) = \\ (\min\{\ell_1,\ell'_1\},\ldots,\min\{\ell_n,\ell'_n\},\min\{r_1,r'_1\},\ldots,\min\{r_n,r'_n\})
\end{multline*}
For clear examples of the Hasse diagrams corresponding to our lattices, we direct the reader to \cref{ex:shift translation by theta}, and \cref{fig: Inversion and lattice counter example}.
We note that \cref{fig: Inversion and lattice counter example} also illustrates that there is no general relation between the shift lattice structure and inversions sets.
In particular, the shift lattices are not sub-lattices of the facial weak order (discussed in \cref{sec:facial-weak-order}), as $24|13$ and $234|1$ are incomparable.

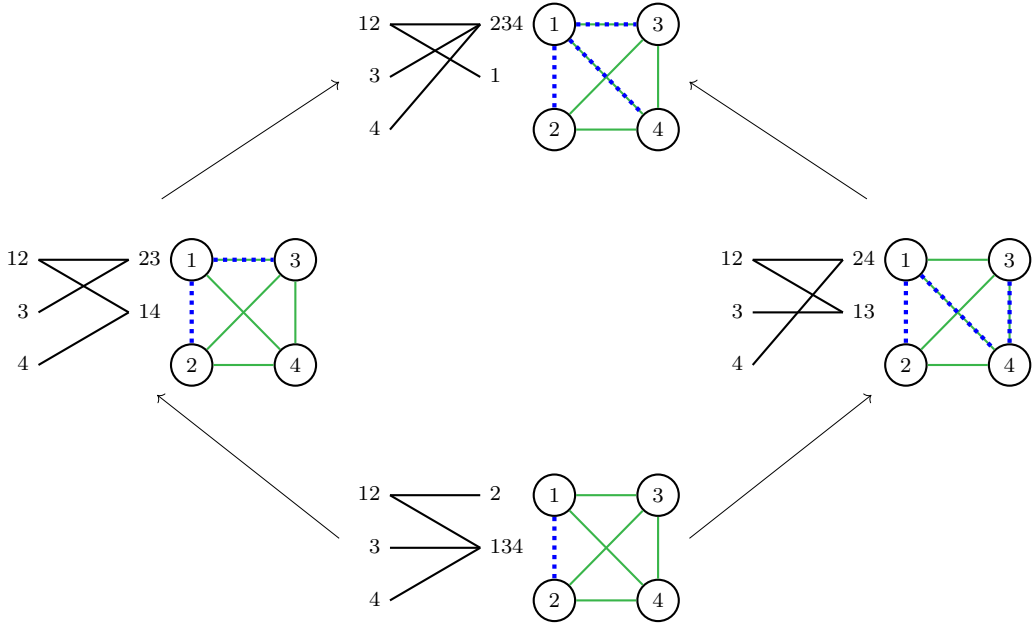
\begin{figure}
{\footnotesize
\begin{tikzcd}[column sep=tiny]
&
\begin{tikzpicture}[xscale=.4,yscale=0.7]   
\node[anchor=east] (1) at (-1.5, -1) {$12$};
\node[anchor=east] (2) at (-1.5, -2) {$3$};
\node[anchor=east] (3) at (-1.5, -3) {$4$};
\node[anchor=west](4) at (1.5, -1) {$234$};
\node[anchor=west](5) at (1.5, -2) {$1$};
\draw[thick] (1.east) -- (4.west);
\draw[thick] (1.east) -- (5.west);
\draw[thick] (2.east) -- (4.west);
\draw[thick] (3.east) -- (4.west);
\end{tikzpicture}
\begin{tikzpicture}[xscale=.4,yscale=0.7] 
\node[anchor=east, circle, draw=black, thick] (1) at (-1, -1) {$1$};
\node[anchor=east, circle, draw=black, thick] (2) at (-1, -3) {$2$};
\node[anchor=west, circle, draw=black, thick] (3) at (1, -1) {$3$};
\node[anchor=west, circle, draw=black, thick]  (4) at (1, -3) {$4$};
\draw[thick, green] (1) -- (3);
\draw[thick, green] (1) -- (4); 
\draw[thick, green] (2) -- (3); 
\draw[thick, green] (2) -- (4); 
\draw[thick, green] (3) -- (4); 
\draw[ultra thick, blue, dotted] (1) -- (2);
\draw[ultra thick, blue, dotted] (1) -- (3);
\draw[ultra thick, blue, dotted] (1) -- (4);
\end{tikzpicture}
\\
\begin{tikzpicture}[xscale=.4,yscale=0.7]   
\node[anchor=east] (1) at (-1.5, -1) {$12$};
\node[anchor=east] (2) at (-1.5, -2) {$3$};
\node[anchor=east] (3) at (-1.5, -3) {$4$};
\node[anchor=west](4) at (1.5, -1) {$23$};
\node[anchor=west](5) at (1.5, -2) {$14$};
\draw[thick] (1.east) -- (4.west);
\draw[thick] (1.east) -- (5.west);
\draw[thick] (2.east) -- (4.west);
\draw[thick] (3.east) -- (5.west);
\end{tikzpicture}
\begin{tikzpicture}[xscale=.4,yscale=0.7] 
\node[anchor=east, circle, draw=black, thick] (1) at (-1, -1) {$1$};
\node[anchor=east, circle, draw=black, thick] (2) at (-1, -3) {$2$};
\node[anchor=west, circle, draw=black, thick] (3) at (1, -1) {$3$};
\node[anchor=west, circle, draw=black, thick]  (4) at (1, -3) {$4$};
\draw[thick, green] (1) -- (3);
\draw[thick, green] (1) -- (4); 
\draw[thick, green] (2) -- (3); 
\draw[thick, green] (2) -- (4); 
\draw[thick, green] (3) -- (4); 
\draw[ultra thick, blue, dotted] (1) -- (2);
\draw[ultra thick, blue, dotted] (1) -- (3);
\end{tikzpicture} 
\arrow[ur]
&
&
\begin{tikzpicture}[xscale=.4,yscale=0.7]  
\node[anchor=east] (1) at (-1.5, -1) {$12$};
\node[anchor=east] (2) at (-1.5, -2) {$3$};
\node[anchor=east] (3) at (-1.5, -3) {$4$};
\node[anchor=west](4) at (1.5, -1) {$24$};
\node[anchor=west](5) at (1.5, -2) {$13$};
\draw[thick] (1.east) -- (4.west);
\draw[thick] (1.east) -- (5.west);
\draw[thick] (2.east) -- (5.west);
\draw[thick] (3.east) -- (4.west);
\end{tikzpicture}
\begin{tikzpicture}[xscale=.4,yscale=0.7] 
\node[anchor=east, circle, draw=black, thick] (1) at (-1, -1) {$1$};
\node[anchor=east, circle, draw=black, thick] (2) at (-1, -3) {$2$};
\node[anchor=west, circle, draw=black, thick] (3) at (1, -1) {$3$};
\node[anchor=west, circle, draw=black, thick]  (4) at (1, -3) {$4$};
\draw[thick, green] (1) -- (3);
\draw[thick, green] (1) -- (4); 
\draw[thick, green] (2) -- (3); 
\draw[thick, green] (2) -- (4); 
\draw[thick, green] (3) -- (4); 
\draw[ultra thick, blue, dotted] (1) -- (2);
\draw[ultra thick, blue, dotted] (1) -- (4);
\draw[ultra thick, blue, dotted] (3) -- (4);
\end{tikzpicture} 
\arrow[ul]
\\
&
\begin{tikzpicture}[xscale=.4,yscale=0.7]  
\node[anchor=east] (1) at (-1.5, -1) {$12$};
\node[anchor=east] (2) at (-1.5, -2) {$3$};
\node[anchor=east] (3) at (-1.5, -3) {$4$};
\node[anchor=west](4) at (1.5, -1) {$2$};
\node[anchor=west](5) at (1.5, -2) {$134$};
\draw[thick] (1.east) -- (4.west);
\draw[thick] (1.east) -- (5.west);
\draw[thick] (2.east) -- (5.west);
\draw[thick] (3.east) -- (5.west);
\end{tikzpicture}
\begin{tikzpicture}[xscale=.4,yscale=0.7] 
\node[anchor=east, circle, draw=black, thick] (1) at (-1, -1) {$1$};
\node[anchor=east, circle, draw=black, thick] (2) at (-1, -3) {$2$};
\node[anchor=west, circle, draw=black, thick] (3) at (1, -1) {$3$};
\node[anchor=west, circle, draw=black, thick]  (4) at (1, -3) {$4$};
\draw[thick, green] (1) -- (3);
\draw[thick, green] (1) -- (4); 
\draw[thick, green] (2) -- (3); 
\draw[thick, green] (2) -- (4); 
\draw[thick, green] (3) -- (4); 
\draw[ultra thick, blue, dotted] (1) -- (2);
\end{tikzpicture}
\arrow[ul, to path= (\tikztostart.210) -- (\tikztotarget.south)]
\arrow[ur, to path= (\tikztostart.330) -- (\tikztotarget.south)]
\end{tikzcd}
}
\caption{The shift lattice $\hour^\SU$ for $v=4|3|1|2$.
Each facet is drawn next to a graph encoding its inversions (\cref{def:inverions}).
If $(i,j) \in J(\sigma)$, then a green edge connects $(i,j)$, and if $(i,j) \in I(\tau)$, then a blue dotted edge connects $(i,j)$.
Consequently, $I((\sigma,\tau))$ is encoded by the presence of both edges, and also the crossings, by \cref{p:crossings}.
}
\label{fig: Inversion and lattice counter example}
\end{figure}

\begin{remark}
As a consequence of \cref{prop:shift lattice}, the facets of the operadic diagonals are disjoint unions of lattices.
However, any lattice $L$ on permutations (such as the weak order) induces a lattice on the facets as follows.
For every $v \in L$, we can substitute the lattice $\hour^\LA$ (or $\hour^\SU$) into the permutation $v$.
In particular, every element which was covered by $v$ is now covered by the minimal element of $\hour^\LA$, and every element which was covering $v$ now covers the maximal element of $\hour^\LA$.
\end{remark}

Given our previously obtained formulae for the number of elements in the diagonal (\cref{subsec:enumerationDiagonalPermutahedra}), and the results of this section, we obtain the following statistics on permutations.

\begin{corollary}
\label{cor:statistics-lattice}
Using the heights of either diagonal,
\begin{align*}
2(n+1)^{n-2} = \sum_{v \in \mathbb{S}_n} \prod_{\rho \in [n]} (\ell_v(\rho)+1)(r_v(\rho)+1)
\end{align*}
Moreover, denoting by $\mathbb{S}_n^{k_1} \subseteq \mathbb{S}_n$ the set of permutations with $k_1$ ascending runs, and consequently $k_2 = n-1-k_1$ descending runs, we have 
\begin{align*}
n \binom{n-1}{k_1} (n-k_1)^{k_1-1} (n-k_2)^{k_2-1} = \sum_{v \in \mathbb{S}_n^{k_1}} \prod_{\rho \in [n]} (\ell_v(\rho)+1)(r_v(\rho)+1).
\end{align*}
\end{corollary}

\begin{proof}
This follows directly from \cref{coro:enumerationDiagonalPermutahedra}, with the observation that shifts conserve the number of blocks, and hence the dimensions of the faces.
\end{proof}

%%%%%%%%%%%%%%%

\subsection{Cubical description}
\label{sec:Cubical}

In this section, we recall the cubical definition of the $\SU$ diagonal from~\cite{SaneblidzeUmble-comparingDiagonals} and explicitly relate it to their shift description, using a new proof that exploits the lattice description of the diagonal (\cref{sec:Shift-lattice}).
Then we construct an analogous cubical definition of the $\LA$ diagonal, transferring the cubical formulae via isomorphism.

\subsubsection{The cubical $\SU$ diagonal}

We define inductively a subdivision $\divcube{n-1}^\SU$ of the $(n-1)$-dimensional cube which is combinatorially isomorphic to the permutahedron $\Perm$ (\cref{prop:subdiv cube Combinatorially Isomorphic to perm}).

\begin{construction}
\label{constr:cubicPermutahedron1}
Given a $(n-k)$-dimensional face $\sigma = \sigma_1| \cdots |\sigma_k$ of the $(n-1)$-dimensional permutahedron $\Perm[n]$, we set $n_j \eqdef \card{\sigma_{k-j+1}\cup\dots\cup \sigma_k}$, and define a subdivision $I_\sigma \eqdef I_1 \cup \dots \cup I_k$ of the interval $[0,1]$ by the following formulas
\[
	I_j \eqdef
	\begin{cases}
		\left[0,1 - 2^{-n_j}\right] & \text{ if } j=1, \\
		\left[1 - 2^{-n_{j-1}}, 1 - 2^{-n_{j}}\right]  & \text{ if } 1 < j < k, \\
		\left[1 - 2^{-n_{j-1}},1\right] & \text{ if } j=k .
	\end{cases}
\]
Let $\divcube{0}^\SU$ be the $0$-dimensional cube (a point), trivially subdivided by the sole element $1$ of~$\Perm[1]$.
Then, assuming we have constructed the subdivision $\divcube{n-1}^\SU$ of the $(n-1)$-cube, we construct $\divcube{n}^\SU$ as the subdivision of $\divcube{n-1}^\SU \times [0,1]$ given, for each face $\sigma$ of $\divcube{n-1}^\SU$, by the polytopal complex $\sigma \times I_\sigma$. 
We label the faces $\sigma \times I$ of the subdivided rectangular prism $\sigma \times I_\sigma$ by the following rule
\begin{align}
\label{eq:sub}
	\sigma \times I \eqdef
	\begin{cases}
		\sigma_1| \cdots |\sigma_k| n+1 & \text{if } I = \{0\}, \\
		\sigma_1| \cdots |\sigma_j| n+1 |\sigma_{j+1}| \cdots |\sigma_k & \text{if } I = I_j \cap I_{j+1} \text{ with } 1 \leq j \leq k-1 , \\
		n+1|\sigma_1| \cdots |\sigma_k  & \text{if } I = \{1\}, \\
		\sigma_1| \cdots |\sigma_j \cup \{n+1\}| \cdots |\sigma_k & \text{if } I = I_j \text{ with }  1\leq j \leq k.
	\end{cases} 
	\tag{I}
\end{align}
This defines a subdivision $\divcube{n}^\SU$ of the $n$-cube.
\end{construction}

\cref{fig:cubicPermutahedron1} illustrates this subdivision for the first few dimensions.
We indicate, in bold, the embedding $\divcube{n-1}^\SU\hookrightarrow \divcube{n}^\SU$ induced by the natural embedding $\R^{n-1} \hookrightarrow \R^n$.
Only the vertices of~$\divcube{3}^\SU$ are labelled, but its edges, facets and outer face are all identified with the expected elements of $\Perm[4]$.
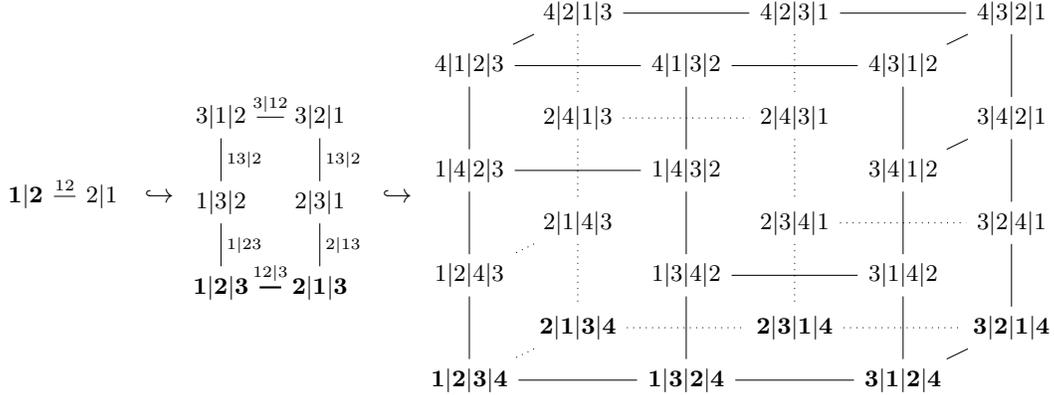
\begin{figure}[h!]
	\begin{center}
	{\small
	\begin{tikzcd}[sep=0.3cm]
	\mathbf{1|2} \arrow[r,dash, "12"] & 2|1
	\end{tikzcd}
	}
	$\hookrightarrow$
	{\small
	\begin{tikzcd}[column sep=0.3cm]
	3|1|2 \arrow[r,dash, "3|12"] \arrow[d,dash, "13|2"]& 3|2|1 \arrow[d,dash, "13|2"]\\
	1|3|2 \arrow[d,dash, "1|23"] & 2|3|1 \arrow[d,dash, "2|13"]\\
	\mathbf{1|2|3} \arrow[r,dash,thick, "12|3"] & \mathbf{2|1|3}
	\end{tikzcd}
	}
	$\hookrightarrow$
	{\small
	\begin{tikzcd}[sep=0.15cm,scale=0.8]
	& 4|2|1|3 \arrow[rr,dash] \arrow[dd,dash,dotted] \arrow[dl,dash] & & 4|2|3|1 \arrow [dd, dash,dotted] \arrow[rr,dash] & & 4|3|2|1 \arrow[dd,dash] \arrow[dl,dash]\\
	4|1|2|3 \arrow [rr,dash] \arrow [dd,dash] & & 4|1|3|2 \arrow[rr,dash] \arrow[dd,dash] & & 4|3|1|2 \arrow[dd,dash]  & \\
	& 2|4|1|3 \arrow [rr,dash,dotted] \arrow[dd,dash,dotted]  & & 2|4|3|1  \arrow[dd,dash,dotted] & & 3|4|2|1 \arrow [dl,dash] \arrow[dd,dash]\\
	1|4|2|3 \arrow [rr,dash] \arrow [dd,dash]& & 1|4|3|2  \arrow[dd,dash] & & 3|4|1|2 \arrow[dd,dash]\\
	& 2|1|4|3 \arrow[dd,dash,dotted] \arrow[dl,dash,dotted] & & 2|3|4|1 \arrow [dd, dash,dotted] \arrow [rr, dash,dotted] & & 3|2|4|1 \arrow[dd,dash] \\
	1|2|4|3 \arrow [dd,dash] & & 1|3|4|2 \arrow[rr,dash] \arrow[dd,dash] & & 3|1|4|2 \arrow[dd,dash] & \\
	& \mathbf{2|1|3|4} \arrow [rr,dash,dotted] \arrow [dl,dash,dotted] & & \mathbf{2|3|1|4} \arrow [rr,dash,dotted] & & \mathbf{3|2|1|4} \arrow [dl,dash] \\
	\mathbf{1|2|3|4} \arrow [rr,dash] & & \mathbf{1|3|2|4} \arrow [rr,dash] & & \mathbf{3|1|2|4} \\
	\end{tikzcd}
	}
	\end{center}
	\caption{Cubical realizations $\divcube{1}^\SU,\divcube{2}^\SU$ and $\divcube{3}^\SU$ of the permutahedra~$\Perm[2]$, $\Perm[3]$ and $\Perm[4]$, respectively, from \cref{constr:cubicPermutahedron1}.}
	\label{fig:cubicPermutahedron1}
\end{figure}

\begin{remark}
\label{rem:coordinates}
A consequence of the construction is that each edge of $\divcube{n}^\SU$ is parallel to one of the canonical basis vectors $e_i$ of $\R^n$, and corresponds to shifting the element $i+1$ of $[n+1]\ssm \{1\}$.
\end{remark}

\begin{proposition}
\label{prop:subdiv cube Combinatorially Isomorphic to perm}
The polytopal complex $\divcube{n}^\SU$ is combinatorially isomorphic to the permutahedron $\Perm[n+1]$.
\end{proposition}

\begin{proof}
By construction it is clear that the faces of $\divcube{n}^\SU$ and $\Perm[n+1]$ are in bijection, and that this bijection preserves the dimension. 
It remains to show that this bijection is a poset isomorphism.
Let $\sigma_1 | \ldots | \sigma_i | \sigma_{i+1} | \ldots | \sigma_k \prec \sigma_1 | \ldots | \sigma_i \cup \sigma_{i+1} | \ldots | \sigma_k$ be a covering relation in the face poset of $\Perm[n+1]$.
We need to see that the corresponding faces $F,G$ of $\divcube{n}^\SU$ satisfy $F \prec G$. 
From \cref{eq:sub} this clearly holds for lines.
Since any face of $\divcube{n}^\SU$ is a product of lines, the result follows by induction on the dimension of the faces.
\end{proof}

We now unpack how certain properties of $\Perm$ have been encoded in the cubical structure of $\divcube{n}^\SU$.
This will later allow us to construct a cubical formula for the diagonal through maximal pairings of $k$-subdivision cubes of $\divcube{n}^\SU$, which we now introduce.

\begin{definition} 
\label{def:Subdivisions}
For $k\geq 0$, a \defn{$k$-subdivision cube} of $\divcube{n}^\SU$ is a union of $k$-faces of $\divcube{n}^\SU$ whose underlying set is a $k$-dimensional rectangular prism.
\end{definition}

An important example of a $k$-subdivision cubes are the $k$-faces of $\divcube{n}^\SU$ (other examples are provided in \cref{ex:subdivision cubes}).

\begin{lemma}
\label{lem:k-subdiv cubes have max/min k faces}
A $k$-subdivision cube has a unique maximal (\resp minimal) vertex with respect to the weak order on permutations.
\end{lemma}

\begin{proof}
By construction, the edges of $\divcube{n}^\SU$ are parallel to the basis vectors of $\R^n$ (\cref{rem:coordinates}), and correspond to inversions on permutations. 
Thus, the vector $\b{v}\eqdef(1,\ldots,1)$ induces the weak order on the vertices of $\divcube{n}^\SU$.  
Since each $k$-subdivision cube is a rectangular prism whose edges are not perpendicular to $\b{v}$, the scalar product with $\b{v}$ is maximized (\resp minimized) at a unique vertex. 
\end{proof}

\begin{definition}
The \defn{maximal (\resp minimal) $k$-face} of a $k$-subdivision cube, with respect to the weak order, is the unique $k$-face in the subdivision cube which contains the maximal (\resp minimal) vertex.
\end{definition}

\begin{construction}
\label{const:unique-sub-cube}
For a $k$-dimensional face $\sigma$ of the cubical permutahedron $\divcube{n}^\SU$, we construct the unique maximal $k$-subdivision cube, with respect to inclusion, whose maximal (\resp minimal) $k$-face with respect to the weak order is $\sigma$.
\end{construction}

We only treat the case for the maximal $k$-face $\sigma$, the minimal face proceeds similarly.
We build the maximal subdivision cubes inductively.  
Let $\sigma$ be an edge of $\divcube{n}^\SU$, and let $v$ be its maximal vertex.
Let $\rho \in [n+1]\ssm\{1\}$ be the element shifted by this edge (\cref{rem:coordinates}).
Shifting this element to the right as far as possible ($\rho$ will be shifted all the way to the right, or be blocked by a larger element), we get the desired $1$-subdivision line.

Suppose that we have constructed maximal subdivision cubes up to dimension $k$, and let $\sigma$ be a $(k+1)$-face of $\divcube{n}^\SU$, with maximal vertex $v$.
Consider the $k+1$ elements of $[n]$, which correspond to dimensions spanned by $\sigma$ (\cref{rem:coordinates}), and let $\rho$ be the largest such element.
Let $\sigma_L$ be the $1$-face, with maximal vertex~$v$, whose only non-trivial dimension corresponds to~$\rho$.
From the initial step above, there is a unique maximal $1$-subdivision line $L$ with maximal $1$-face~$\sigma_L$.
Let $\sigma_C$ be the $k$-face, with maximal vertex $v$, spanned by the complement of $\rho$ in $[n+1] \ssm \{1\}$. 
By induction, there is a maximal $k$-subdivision cube $C$ with maximal $k$-face $\sigma_C$. 
Then, it is clear that the product $L\times C$ defines a $(k+1)$-subdivision cube of $\divcube{n}^\SU$, with maximal vertex~$v$. 
Indeed, as $\rho$ is the maximal element corresponding to dimensions of $L\times C$, the faces of $L\times C$ are the $(k+1)$-dimensional rectangular prisms resulting from shifting $\rho$ through the $k$-faces of $C$, as in \cref{eq:sub}.

Finally, $L\times C$ is maximal under inclusion.
If there was a larger $(k+1)$-subdivision cube enveloping $L\times C$, then one of its projections would be a $1$ or $k$-subdivision cube enveloping $L$ or $C$, contradicting the assumption that they are maximal.
This finishes the construction. 

\begin{definition} 
\label{def:hourglass}
For a vertex $v$ of $\divcube{n}^\SU$, the \defn{hourglass} of $\divcube{n}^\SU$ at $v$ is the maximal pair of subdivision cubes $\hour^\SU$, with respect to inclusion, within the set of all pairs $(C,C')$ of subdivision cubes such that $C$ has maximal vertex $v$, $C'$ has minimal vertex $v$, and $\dim C +\dim C' = n$.
\end{definition}

\begin{figure}[h!]
	\begin{center}
	{\small
	\begin{tikzcd}[sep=0.1cm, scale=0.3]
	& 4|2|1|3 \arrow[rr,dash] \arrow[dd,dash,dotted] \arrow[dl,dash] & & 4|2|3|1 \arrow [dd, dash,dotted] \arrow[rr,dash] & & 4|3|2|1 \arrow[dd,dash] \arrow[dl,dash, near end, "\mathbf{\blue{4|3|12}}"']\\
	4|1|2|3 \arrow [rr,dash] \arrow [dd,dash] & & 4|1|3|2 \arrow[rr,dash] \arrow[dd,dash] & & \mathbf{\blue{4|3|1|2}} \arrow[dd,dash]  & \\
	& \red{\mathbf{14|23}} \arrow [rr,dash,dotted] \arrow[dd,dash,dotted]  & & 2|4|3|1  \arrow[dd,dash,dotted] & & 3|4|2|1 \arrow [dl,dash] \arrow[dd,dash]\\
	1|4|2|3 \arrow [rr,dash] \arrow [dd,dash]& & 1|4|3|2  \arrow[dd,dash] \arrow[rr, phantom, "\blue{\mathbf{134|2}}" description,crossing over] & & 3|4|1|2 \arrow[dd,dash]\\
	& 2|1|4|3 \arrow[dd,dash,dotted] \arrow[dl,dash,dotted] & & 2|3|4|1 \arrow [dd, dash,dotted] \arrow [rr, dash,dotted] & & 3|2|4|1 \arrow[dd,dash] \\
	1|2|4|3 \arrow [dd,dash] \arrow[rr, phantom, "\red{\mathbf{1|234}}" description,crossing over] & & 1|3|4|2 \arrow[rr,dash] \arrow[dd,dash] & & 3|1|4|2 \arrow[dd,dash] & \\
	& 2|1|3|4 \arrow [rr,dash,dotted] \arrow [dl,dash,dotted] & & \red{\mathbf{13|24}} \arrow [rr,dash,dotted] & & 3|2|1|4 \arrow [dl,dash] \\
	1|2|3|4 \arrow [rr,dash] & & 1|3|2|4 \arrow [rr,dash] & & 3|1|2|4 \\
	\end{tikzcd}
	}
	\end{center}
	\caption{The hourglass~$\hour^\SU$ of $\divcube{3}^\SU$ at~$v = \blue{\mathbf{4|3|1|2}}$.}
	\label{fig:hourglass1}
\end{figure}

The following examples are pictured in \cref{fig:hourglass1}.

\begin{example}
\label{ex:subdivision cubes}
The sets of faces~$\{1234\}$, $\{1|234\}$, $\{1|234, 14|23\}$, and $\{1|3|24,1|34|2\}$ are all subdivision cubes of $\divcube{3}$.
In contrast, the sets of faces $\{134|2, 14|23\}$, $\{1|234, 134|2, 14|23\}$, and $\{1|2|34,1|23|4\}$
are not subdivision cubes of $\divcube{3}^\SU$.
For $v = \blue{\mathbf{4|3|1|2}}$, the only $1$-subdivision cube with minimal vertex $v$ is $4|3|12$, and the three $2$-subdivision cubes with maximal vertex $v$ are $\{134|2\}$, $\{13|24, 134|2\}$, $\{ 1|234, 13|24, 14|23, 134|2\}$.
This defines the hourglass $\hour^\SU=\{ 1|234, 13|24, 14|23, 134|2\} \times \{4|3|12\}$ of $\divcube{3}^\SU$ at $v$.

Let us observe that the $\SCP$ corresponding to $v$ is $(\sigma,\tau) \eqdef (\blue{\mathbf{134|2}},\blue{\mathbf{4|3|12}} )$.
The ordered partition $\sigma$ admits three distinct rights shifts, $\red{\mathbf{13|24}}$, $\red{\mathbf{14|23}}$, $\red{\mathbf{1|234}}$, and $\tau$ admits no left shifts.
\cref{thm:cubical-SU} shows that $\hour^\SU$ is generated by all shifts of the $\SCP$ corresponding to $v$.
\end{example}

\begin{theorem}
\label{thm:cubical-SU}
Let $v$ be a vertex of the cubical permutahedron $\divcube{n}^\SU$, and let $(\sigma,\tau)$ be its associated~$\SCP$. 
Then, we have 
\[
\hour^\SU = \bigcup_{\b{M},\b{N}}R_{\b{M}}(\sigma) \times L_{\b{N}}(\tau) \ ,
\]
where the union is taken over all block-admissible sequences of $\SU$ shifts $\b{M},\b{N}$.
\end{theorem}

\begin{proof}
We prove the result inductively from lines, and consider the case of $\sigma$, $\tau$ proceeds similarly.
Combining \cref{const:unique-sub-cube} with \cref{prop:maximal m-shift formulae}, we have that if $L$ is a maximal $1$-subdivision cube with maximal $1$-face $\sigma$, then its faces are generated by the right $1$-shifts $R_\rho^i$, for $i$ between $0$ and its maximal right height~$r_\rho$ (\cref{def:heights}).

Now consider the unique maximal $(k+1)$-subdivision cube of $\divcube{n}^\SU$ with maximal $(k+1)$-face~$\sigma$.
By \cref{const:unique-sub-cube}, this subdivision cube is given by the product $L\times C$ of a line corresponding to the maximal element $\rho$ being shifted, and a $k$-cube $C$ corresponding to all other elements.
By induction hypothesis, both the line $L$ and the cube $C$ are generated by all block-admissible right shifts from their unique maximal faces $\sigma_L$ and $\sigma_C$.
Moreover, if $\rho$ is in the $i$th block of $\sigma_C$, then $\sigma$ is obtained from~$\sigma_C$ by merging the $i$\ordinal{} and $(i+1)$\ordinalst{} blocks.

On the one hand, the right height of any element being shifted in $C$ is the same as its right height in $L\times C$.
This follows from the inductive description of $\divcube{n}^\SU$ (\cref{eq:sub}): every $k$-face in $C$ has $\rho$ in a singleton block, and as $\rho$ is larger than all other elements it blocks all other shifts.

On the other hand, we know from \cref{const:unique-sub-cube} that the $(k+1)$-faces of $L\times C$ are obtained by weaving $\rho$ through the $k$-faces of~$C$.
Indeed, every $k$-face in $C$ has $\rho$ in a fixed singleton set, and every $k$-face on the opposite side of $L\times C$ has $\rho$ in another fixed singleton set; given \cref{eq:sub}, this final singleton block in any $k$-face of $C$ is either the last block, or is followed by a block containing an element larger than $\rho$.
If we translate this observation back to the $(k+1)$-faces of $L\times C$ that are adjacent to the boundary of $L$, then this corresponds to an equivalent calculation of the right height of $\rho$ in $\sigma$.

Given the lattice description of the diagonal (\cref{prop:shift lattice}), we thus have that $L\times C$ is generated by all block-admissible sequences of right shifts of $\sigma$, which concludes the proof.
\end{proof}

This recovers the formulas \cite[Form.~(1)~\&~(3)]{SaneblidzeUmble-comparingDiagonals}.

%%%%%%%%%%%%%%%%%%%%%%%%%%%%

\subsubsection{$\LA$ cubical description}

The $\LA$ diagonal also admits a similar cubical description, which we may quickly induce by isomorphism.
We first define inductively a subdivision $\divcube{n-1}^{\LA}$ of the $(n-1)$-dimensional cube which is combinatorially isomorphic to the permutahedron $\Perm$, analogous to the one from the preceding sections. 

\begin{construction}
\label{const:cubical-LA}
Given a $(n-k)$-dimensional face $\sigma = \sigma_1| \cdots |\sigma_k$ of the $(n-1)$-dimensional permutahedron $\Perm[n]$, we set $n_j \eqdef \card{\sigma_{k-j+1}\cup\dots\cup \sigma_k}$, and define a subdivision $I_\sigma \eqdef I_1 \cup \dots \cup I_k$ of the interval $[0,1]$ by the same formulas as in \cref{constr:cubicPermutahedron1}.

Let $\divcube{0}^\LA$ be the $0$-dimensional cube (a point), trivially subdivided by the sole element $1$ of~$\Perm[1]$.
Then, assuming we have constructed the subdivision $\divcube{n-1}^\LA$ of the $(n-1)$-cube, we construct $\divcube{n}^\LA$ as the subdivision of $\divcube{n-1}^\LA \times [0,1]$ given, for each face $\sigma$ of $\divcube{n-1}^\LA$, by the polytopal complex $\sigma \times I_\sigma$. 
We label the faces $\sigma \times I$ of the subdivided rectangular prism $\sigma \times I_\sigma$ by the following rule
\begin{align*} 
	\sigma \times I \eqdef
	\begin{cases}
		\sigma_1'| \cdots |\sigma_k'| 1 & \text{if } I = \{0\}, \\
		\sigma_1'| \cdots |\sigma_j'| 1 |\sigma_{j+1}'| \cdots |\sigma_k' & \text{if } I = I_j \cap I_{j+1} \text{ with } 1 \leq j \leq k-1 , \\
		1|\sigma_1'| \cdots |\sigma_k'  & \text{if } I = \{1\}, \\
		\sigma_1'| \cdots |\sigma_j' \cup \{1\}| \cdots |\sigma_k' & \text{if } I = I_j \text{ with }  1\leq j \leq k , 
	\end{cases} 
\end{align*}
where each block of $\sigma$ has been renumbered as $\sigma_i' \eqdef \{p+1 \ | \ p \in \sigma_i\}$ for all $1 \leq i \leq k$.
We obtain a subdivision $\divcube{n}^\LA$ of the $n$-cube isomorphic to the permutahedron $\Perm[n+1]$.
\end{construction}

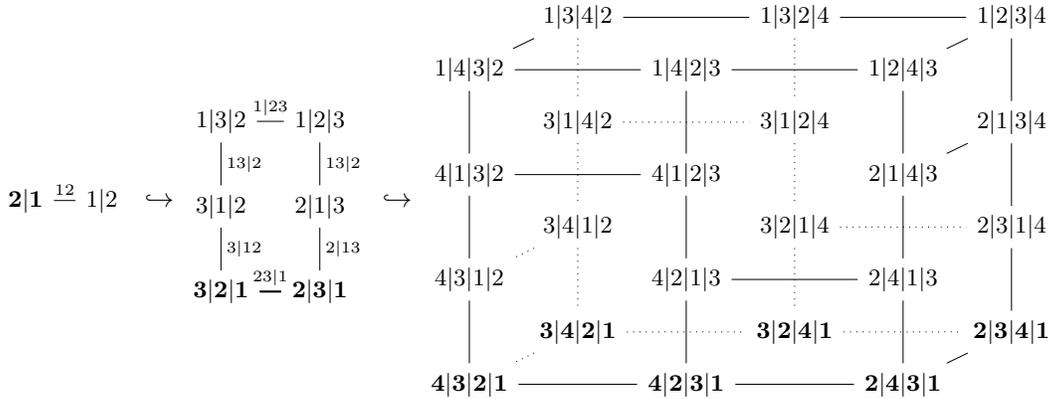
\begin{figure}[h!]
	\begin{center}
	{\small
	\begin{tikzcd}[sep=0.3cm]
	\mathbf{2|1} \arrow[r,dash, "12"] & 1|2
	\end{tikzcd}
	}
	$\hookrightarrow$
	{\small
	\begin{tikzcd}[column sep=0.3cm]
	1|3|2 \arrow[r,dash, "1|23"] \arrow[d,dash, "13|2"]& 1|2|3 \arrow[d,dash, "13|2"]\\
	3|1|2 \arrow[d,dash, "3|12"] & 2|1|3 \arrow[d,dash, "2|13"]\\
	\mathbf{3|2|1} \arrow[r,dash,thick, "23|1"] & \mathbf{2|3|1}
	\end{tikzcd}
	}
	$\hookrightarrow$
	{\small
	\begin{tikzcd}[sep=0.15cm,scale=0.8]
	& 1|3|4|2 \arrow[rr,dash] \arrow[dd,dash,dotted] \arrow[dl,dash] & & 1|3|2|4 \arrow [dd, dash,dotted] \arrow[rr,dash] & & 1|2|3|4 \arrow[dd,dash] \arrow[dl,dash]\\
	1|4|3|2 \arrow [rr,dash] \arrow [dd,dash] & & 1|4|2|3 \arrow[rr,dash] \arrow[dd,dash] & & 1|2|4|3 \arrow[dd,dash]  & \\
	& 3|1|4|2 \arrow [rr,dash,dotted] \arrow[dd,dash,dotted]  & & 3|1|2|4  \arrow[dd,dash,dotted] & & 2|1|3|4 \arrow [dl,dash] \arrow[dd,dash]\\
	4|1|3|2 \arrow [rr,dash] \arrow [dd,dash]& & 4|1|2|3  \arrow[dd,dash] & & 2|1|4|3 \arrow[dd,dash]\\
	& 3|4|1|2 \arrow[dd,dash,dotted] \arrow[dl,dash,dotted] & & 3|2|1|4 \arrow [dd, dash,dotted] \arrow [rr, dash,dotted] & & 2|3|1|4 \arrow[dd,dash] \\
	4|3|1|2 \arrow [dd,dash] & & 4|2|1|3 \arrow[rr,dash] \arrow[dd,dash] & & 2|4|1|3 \arrow[dd,dash] & \\
	& \mathbf{3|4|2|1} \arrow [rr,dash,dotted] \arrow [dl,dash,dotted] & & \mathbf{3|2|4|1} \arrow [rr,dash,dotted] & & \mathbf{2|3|4|1} \arrow [dl,dash] \\
	\mathbf{4|3|2|1} \arrow [rr,dash] & & \mathbf{4|2|3|1} \arrow [rr,dash] & & \mathbf{2|4|3|1} \\
	\end{tikzcd}
	}
	\end{center}
	\caption{Cubical realizations of the permutahedra~$\Perm[2]$, $\Perm[3]$ and $\Perm[4]$ from \cref{const:cubical-LA}.}
	\label{fig:cubicPermutahedron2}
\end{figure}

\cref{fig:cubicPermutahedron2} illustrates $\divcube{n}^{\LA}$ in dimensions $1$ to $3$.
We have indicated in bold the embedding ${\divcube{n-1}^{\LA}\hookrightarrow \divcube{n}^{\LA}}$ induced by the natural inclusions $\R^n \hookrightarrow \R^{n+1}$.

The appropriate base definitions for the $\LA$ diagonal of $k$-subdivision cubes (\cref{def:Subdivisions}) and hourglasses ${\hour}$ (\cref{def:hourglass}) are the same as in the $\SU$ case.
The proof that $\hour^\LA $ is indeed combinatorially isomorphic to $\Perm[n+1]$ proceeds similarly to \cref{prop:subdiv cube Combinatorially Isomorphic to perm}.

Recall from \cref{subsec:isos-LA-SU} the face poset isomorphism of the permutahedra which acts on~$A_1| \cdots |A_k$, by replacing each block $A_j$ by the block $r(A_j)\eqdef\set{n-i+1}{i \in A_j}$.
We have the analogue of \cref{thm:cubical-SU} for the $\LA$ diagonal. 

\begin{theorem}
\label{prop:LA-cubical}
Let $v$ be a vertex of the cubical permutahedron $\divcube{n}$, and let $(\sigma,\tau)$ be its associated~$\SCP$. 
Then, we have 
\[
\hour^\LA = \bigcup_{\b{M},\b{N}}L_{\b{M}}(\sigma) \times R_{\b{N}}(\tau) \ ,
\]
where the union is taken over all block-admissible sequences of $\LA$ shifts $\b{M},\b{N}$.
\end{theorem}

\begin{proof}
It is straightforward to see that the involution $r$ induces a $k$-subdivision cube isomorphism $r:\divcube{n}^{\SU}\to \divcube{n}^{\LA}$ between the cubical subdivisions of \cref{constr:cubicPermutahedron1} and \cref{const:cubical-LA}, which sends the hourglass $\hour^\SU$ to the hourglass $\hour[r(v)]^\LA$.  
By \cref{thm:cubical-SU}, we know that $\hour^\SU$ is generated by $\SU$ shifts; we want to deduce that $\hour[r(v)]^\LA$ is generated by $\LA$ shifts. 
First, we observe that the diagram
\begin{center}
\begin{tikzcd}
\SCP \arrow[r,"t(r\times r)",leftrightarrow] \arrow[d,leftrightarrow] & \SCP \arrow[d,leftrightarrow]\\
\mathbb{S}_n \arrow[r,"r"',leftrightarrow] & \mathbb{S}_n ,
\end{tikzcd}
\end{center}
where $t$ is the permutation of the two factors, and the vertical arrows are the bijection between $\SCP$s and permutations (\cref{def:strong-complementary-pairs}), is commutative.
Thus, if we start from a $\SCP$ and consider its associated $\SU$ hourglass $\hour^\SU$, applying the subdivision cube isomorphism $r$ or applying the map $t(r \times r)$ both give the $\LA$ hourglass $\hour[r(v)]^\LA$.
Combining this with the fact that the map $t(r\times r):\LAD\to \SUD$ is an isomorphism between the $\LA$ and $\SU$ diagonal (\cref{rem:Alternate Isomorphism}) which preserves left and right shifts (\cref{prop:trr is an isomorphism of shifts}), we obtain the desired result. 
\end{proof}

\begin{figure}[h!]
	\centerline{
	{\small
	\begin{tikzcd}[sep=0.1cm, scale=0.3, ampersand replacement=\&]
	\& 4|2|1|3 \arrow[rr,dash] \arrow[dd,dash,dotted] \arrow[dl,dash] \& \& 4|2|3|1 \arrow [dd, dash,dotted] \arrow[rr,dash] \& \& 4|3|2|1 \arrow[dd,dash] \arrow[dl,dash, near end, "\mathbf{\blue{4|3|12}}"']\\
	4|1|2|3 \arrow [rr,dash] \arrow [dd,dash] \& \& 4|1|3|2 \arrow[rr,dash] \arrow[dd,dash] \& \& \mathbf{\blue{4|3|1|2}} \arrow[dd,dash]  \& \\
	\& \red{\mathbf{14|23}} \arrow [rr,dash,dotted] \arrow[dd,dash,dotted]  \& \& 2|4|3|1  \arrow[dd,dash,dotted] \& \& 3|4|2|1 \arrow [dl,dash] \arrow[dd,dash]\\
	1|4|2|3 \arrow [rr,dash] \arrow [dd,dash]\& \& 1|4|3|2  \arrow[dd,dash] \arrow[rr, phantom, "\blue{\mathbf{134|2}}" description,crossing over] \& \& 3|4|1|2 \arrow[dd,dash]\\
	\& 2|1|4|3 \arrow[dd,dash,dotted] \arrow[dl,dash,dotted] \& \& 2|3|4|1 \arrow [dd, dash,dotted] \arrow [rr, dash,dotted] \& \& 3|2|4|1 \arrow[dd,dash] \\
	1|2|4|3 \arrow [dd,dash] \arrow[rr, phantom, "\red{\mathbf{1|234}}" description,crossing over] \& \& 1|3|4|2 \arrow[rr,dash] \arrow[dd,dash] \& \& 3|1|4|2 \arrow[dd,dash] \& \\
	\& 2|1|3|4 \arrow [rr,dash,dotted] \arrow [dl,dash,dotted] \& \& \red{\mathbf{13|24}} \arrow [rr,dash,dotted] \& \& 3|2|1|4 \arrow [dl,dash] \\
	1|2|3|4 \arrow [rr,dash] \& \& 1|3|2|4 \arrow [rr,dash] \& \& 3|1|2|4 \\
	\end{tikzcd}
	}
	$\overset{r}{\longrightarrow}$
	{\small
	\begin{tikzcd}[sep=0.1cm, scale=0.3, ampersand replacement=\&]
	\& 1|3|4|2 \arrow[rr,dash] \arrow[dd,dash,dotted] \arrow[dl,dash] \& \& 1|3|2|4 \arrow [dd, dash,dotted] \arrow[rr,dash] \& \& 1|2|3|4 \arrow[dd,dash] \arrow[dl,dash, near end, "\mathbf{\blue{1|2|34}}"']\\
	1|4|3|2 \arrow [rr,dash] \arrow [dd,dash] \& \& 1|4|2|3 \arrow[rr,dash] \arrow[dd,dash] \& \& \mathbf{\blue{1|2|4|3}} \arrow[dd,dash]  \& \\
	\& \red{\mathbf{14|23}} \arrow [rr,dash,dotted] \arrow[dd,dash,dotted]  \& \& 3|1|2|4  \arrow[dd,dash,dotted] \& \& 2|1|3|4 \arrow [dl,dash] \arrow[dd,dash]\\
	4|1|3|2 \arrow [rr,dash] \arrow [dd,dash]\& \& 4|1|2|3  \arrow[dd,dash] \arrow[rr, phantom, "\blue{\mathbf{124|3}}" description,crossing over] \& \& 2|1|4|3 \arrow[dd,dash]\\
	\& 3|4|1|2 \arrow[dd,dash,dotted] \arrow[dl,dash,dotted] \& \& 3|2|1|4 \arrow [dd, dash,dotted] \arrow [rr, dash,dotted] \& \& 2|3|1|4 \arrow[dd,dash] \\
	4|3|1|2 \arrow [dd,dash] \arrow[rr, phantom, "\red{\mathbf{4|123}}" description,crossing over] \& \& 4|2|1|3 \arrow[rr,dash] \arrow[dd,dash] \& \& 2|4|1|3 \arrow[dd,dash] \& \\
	\& 3|4|2|1 \arrow [rr,dash,dotted] \arrow [dl,dash,dotted] \& \& \red{\mathbf{24|13}} \arrow [rr,dash,dotted] \& \& 2|3|4|1 \arrow [dl,dash] \\
	4|3|2|1 \arrow [rr,dash] \& \& 4|2|3|1 \arrow [rr,dash] \& \& 2|4|3|1 \\
	\end{tikzcd}
	}
	}
	\caption{The isomorphism~$r$ applied to the $\SU$ cubical subdivision from \cref{fig:hourglass1}.}
	\label{fig:hourglass2}
\end{figure}
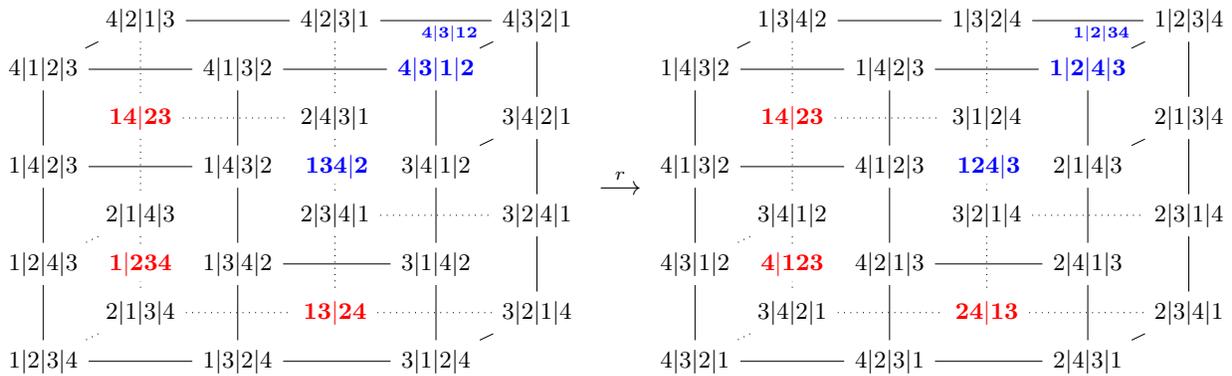

\begin{example}
Applying the isomorphism $r$ to \cref{ex:subdivision cubes} yields the illustration of \cref{fig:hourglass2}.
As $r$ is an isomorphism of $k$-subdivision cubes, the maximal pair of $\SU$ subdivision cubes has been mapped to a maximal pair of $\LA$ subdivision cubes.
Note that the maximal $\SU$ $2$-subdivision cube with maximal vertex $4|3|1|2$ was generated by $\SU$ right shifts.
Its image under $r$ is the maximal $\LA$ $2$-subdivision cube with minimal vertex $1|2|4|3$, and it is generated by $\LA$ right shifts.
\end{example}

%%%%%%%%%%%%%%%%%%%%%%%%%%%%%%%%%%%%%%%%%%%%%%%%%%%%

\subsection{Matrix description}
\label{subsec:matrix}

For completeness, we recall from \cite{SaneblidzeUmble} the matrix description of facets of the $\SU$ diagonal.
We previously saw that $\SCP$s and permutations are in bijection (\cref{def:strong-complementary-pairs}). 
There is also a third equivalent way to encode this data, the \defn{step matrices} of \cite[Def. 6]{SaneblidzeUmble}.
Given a permutation, one defines its associated step matrix by starting in the bottom left corner, writing increasing sequences vertically and decreasing sequences horizontally one after the other, leaving all other entries $0$.
See \cref{fig:step-matrix}.

\begin{figure}[h!]
	\begin{center}
	\raisebox{3em}{$6|5|2|4|7|1|3$}
	$\quad \quad$
	\begin{tikzpicture}[scale=.7]  
	\node[anchor=east] (1) at (-1.5, -1) {$256$};
	\node[anchor=east] (2) at (-1.5, -2) {$4$};
	\node[anchor=east] (3) at (-1.5, -3) {$17$};
	\node[anchor=east] (4) at (-1.5, -4) {$3$};
	\node[anchor=west] (5) at (1.5, -1) {$6$};
	\node[anchor=west] (6) at (1.5, -2) {$5$};
	\node[anchor=west] (7) at (1.5, -3) {$247$};
	\node[anchor=west] (8) at (1.5, -4) {$13$};
	\draw[thick] (1.east) -- (5.west);
	\draw[thick] (1.east) -- (6.west); 
	\draw[thick] (1.east) -- (7.west); 
	\draw[thick] (2.east) -- (7.west); 
	\draw[thick] (3.east) -- (7.west); 
	\draw[thick] (3.east) -- (8.west); 
	\draw[thick] (4.east) -- (8.west);
	\end{tikzpicture}
	$\quad \quad$
	\raisebox{3em}{
	$
	\begin{blockarray}{ccccc}
	  & \sigma_1 & \sigma_2 & \sigma_3 & \sigma_4  \\
	\begin{block}{c[cccc]}
	\tau_4 &   &   & 1 & 3 \\
	\tau_3 & 2 & 4 & 7  &   \\
	\tau_2 & 5 &   &   &   \\
	\tau_1 & 6 &   &   &   \\
	\end{block}
	\end{blockarray}
	$
	}
	\end{center}
	\caption{A permutation, its associated $\SCP$, and their step matrix.}
	\label{fig:step-matrix}
\end{figure}
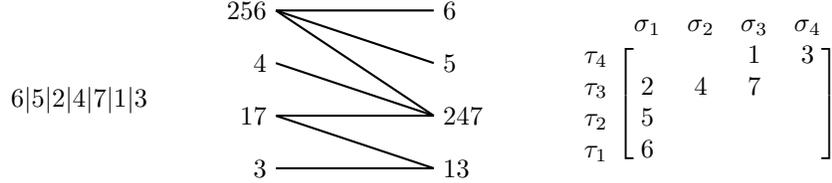
Given a matrix $A$ whose only non-zero entries are the elements $[n]$, let $\sigma_i(A)$ denote the non-zero entries of the $i$\ordinal{} column, and $\tau_j(A)$ the non-zero entries of the $(r-j+1)$\ordinalst{} row, where $r$ is the number of rows of $A$. See the labelling in \cref{fig:step-matrix}.
With this identification, the definitions of the shift operators can be translated directly:
the right shift operator $R_{M}$ shifts the elements of a subset $M \subset \sigma_i(A)$ one column to the right, or one row up, replacing only elements of value $0$, and leaving $0$ elements in their wake, while the left shift operator $L_{M}$ shifts the elements of $M$ to the left, or down one row.

The fact that the shifts avoid collisions with other elements is a consequence of their admissibility.
Recall from \cref{def:movable-subsets}, that a right $\SU$ $1$-shift $R_M$ is \defn{block-admissible} if $\min \sigma_i \notin M$ and $\min M > \max \sigma_{i+1}$, and that admissible sequence of right shifts proceed in increasing order (\cref{def:SU-admissible}).

\begin{proposition}
Admissible sequences of matrix shift operators are well-defined.
\end{proposition}

\begin{proof}
We verify the claim that the admissible sequences of matrix shift operators never replace non-zero elements.
It is straightforward to show that this is true when a shift operator is applied to a $\SCP$ $(\sigma,\tau)$.
From here we proceed inductively.
We assume that all prior shift operators have been well-defined, and we then check that applying another admissible shift operator is also well-defined.
Suppose that an admissible right shift $R_{M}$ is not well-defined, as it tries to move a value $m$ into a non-zero matrix entry $n$.
Then, $n$ must have been placed into that column by a prior left shift operator $L_{N_j}$, and consequently $n>\min N_j > \max \tau_{j-1}> m$.
However, $n\in \sigma_{i+1}$ so $\max \sigma_{i+1}>n>m>\min M_i $ implies that $M$ is not block-admissible, a contradiction.
\end{proof}

\begin{figure}[h!]
	{\footnotesize
\begin{tikzcd}[column sep = 0.5cm]
\begin{bmatrix}
  &   & 1 & 3 \\
2 & 4 & 7 &   \\
5 &   &   &   \\
6 &   &   &   \\
\end{bmatrix}
\arrow[r,"R^{\SU}_{56}"]
&
\begin{bmatrix}
  &   & 1 & 3 \\
2 & 4 & 7 &   \\
  & 5 &   &   \\
  & 6 &   &   \\
\end{bmatrix}
\arrow[r,"L^{\SU}_7"]
&
\begin{bmatrix}
  &   & 1 & 3 \\
2 & 4 &   &   \\
  & 5 & 7 &   \\
  & 6 &   &   \\
\end{bmatrix}
\arrow[r,"R^{\SU}_7"]
&
\begin{bmatrix}
  &   & 1 & 3 \\
2 & 4 &   &   \\
  & 5 &   & 7 \\
  & 6 &   &   \\
\end{bmatrix}
\\
\begin{bmatrix}
  &   &   & 2 \\
  &   &   & 3 \\
  & 1 & 4 & 6 \\
5 & 7 &   &   \\
\end{bmatrix}
\arrow[r,"R^{\LA}_{23}"]
&
\begin{bmatrix}
  &   & 2 &   \\
  &   & 3 &   \\
  & 1 & 4 & 6 \\
5 & 7 &   &   \\
\end{bmatrix}
\arrow[r,"L^{\LA}_1"]
&
\begin{bmatrix}
  &   & 2 &   \\
  & 1 & 3 &   \\
  &   & 4 & 6 \\
5 & 7 &   &   \\
\end{bmatrix}
\arrow[r,"R^{\LA}_1"]
&
\begin{bmatrix}
  &   & 2 &   \\
1 &   & 3 &   \\
  &   & 4 & 6 \\
5 & 7 &   &   \\
\end{bmatrix}
\end{tikzcd}
}
\caption{Matrix shifts under the isomorphism $t(r \times r)$ between the $\LA$ and $\SU$ diagonals.}
\label{fig:matrix-shifts}
\end{figure}
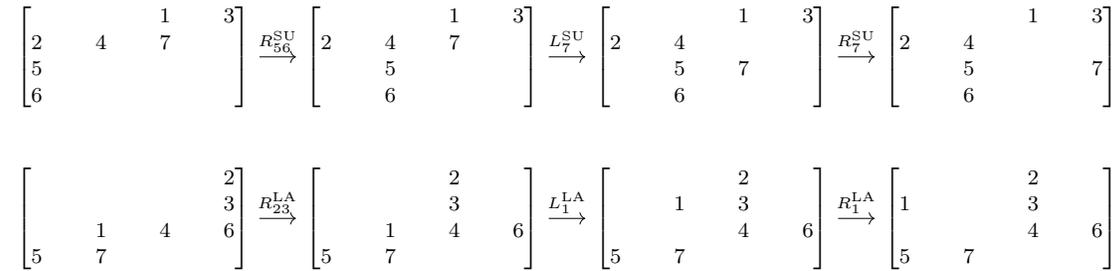

\defn{Configuration matrices} \cite[Def. 7]{SaneblidzeUmble} are the matrices corresponding to $\SCP$s and those generated by admissible sequences of shifts.
Consequently, they are in bijection with the facets of the $\SU$ diagonal.
The translation of these results for the $\LA$ diagonal is clear.
One can use the isomorphism $t(r\times r)$ as in the following example. 

\begin{example}
\label{ex:matrix shifts}
The first row of \cref{fig:matrix-shifts} contains a sequence of admissible $\SU$ subset shifts applied to the matrix encoding of the SCP $256|4|17|3\times 6|5|147|13$. 
The second row is the image of these shifts under the isomorphism $t(r\times r)$.
Note that the shifts of this example are also isomorphic to those of \cref{ex:shift translation by theta}, under the isomorphism $(rs\times rs)$.
\end{example}

%%%%%%%%%%%%%%%%%%%%%%%%%%%%%%%%%%%%%%
%%%%%%%%%%%%%%%%%%%%%%%%%%%%%%%%%%%%%%

\clearpage

\part{Higher algebraic structures}
\label{part:higherAlgebraicStructures}

In this third part, we derive some higher algebraic consequences of the results obtained in \cref{part:diagonalsPermutahedra}.
We first prove in \cref{subsec:top-operadic-structures} that there are exactly two topological operad structures on the family of operahedra (\resp multiplihedra) which are compatible with the generalized Tamari order, and thus two geometric universal tensor products of (non-symmetric non-unital) homotopy operads (\resp $\Ainf$-morphisms).
Then, we show that these topological operad structures are isomorphic (\cref{subsec:iso-top-operads}).
However, these isomorphisms do not commute with the diagonal maps (\cref{ex:iso-not-Hopf,ex:iso-not-Hopf-2}).
Finally, we show that contrary to the case of permutahedra, the faces of the $\LA$ and $\SU$ diagonals of the operahedra (\resp multiplihedra) are in general not in bijection (\cref{subsec:tensor-products}).
However, from a homotopical point of view, the two tensor products of homotopy operads (\resp $\Ainf$-morphisms) that they define are $\infty$-isomorphic (\cref{thm:infinity-iso,thm:infinity-iso-2}). 

%%%%%%%%%%%%%%%%%%%%%%%%%%%%%%%%%%%%%%

\section{Higher tensor products}

%%%%%%%%%%%%%%%

\subsection{Topological operadic structures}
\label{subsec:top-operadic-structures}

The permutahedra are part of a more general family of polytopes called Loday realizations of the \emph{operahedra}~\cite[Def.~2.9]{LaplanteAnfossi}, which encodes the notion of homotopy operad~\cite[Def.~4.11]{LaplanteAnfossi} (we consider here only \emph{non-symmetric non-unital} homotopy operads).
Let $\PT_n$ be the set of planar trees with $n$ internal edges, which are labelled by $[n]$ using the infix order.
For every planar tree $t$, there is a corresponding operahedron $P_t$ whose codimension~$k$ faces are in bijection with nestings of $t$ with $k$ non-trivial nests.

\begin{definition}[{\cite[Def.~2.1 \& 2.22]{LaplanteAnfossi}}]\label{def:nesting}
	A \defn{nest} of $t \in \PT_n$ is a subset of internal edges which induce a subtree, and a \defn{nesting} of $t$ is a family of nests which are either included in one another, or disjoint.
	See \cref{fig:nestings}.
\end{definition}

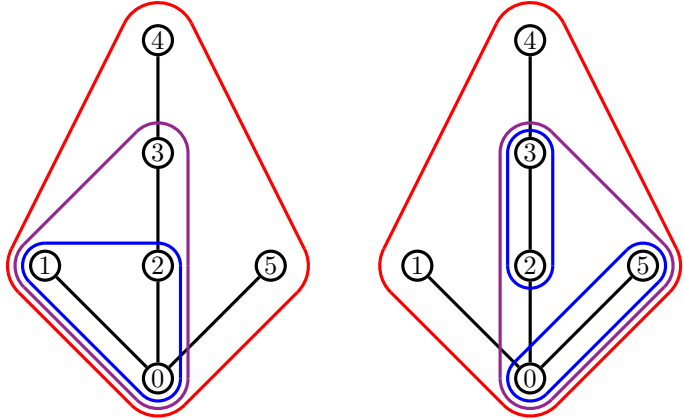
\begin{figure}[h!]
\begin{tikzpicture}[yscale=-1, every node/.style={draw, very thick, circle, inner sep=1pt}, edge from parent path={[very thick, draw] (\tikzparentnode) -- (\tikzchildnode)}]
]
\node(0) {0}
	child{node(1){1}}
	child{node(2){2}
		child{node(3){3}
		child{node(4){4}}
		}}
	child{node(5){5}};
% \begin{pgfonlayer}{bg}    % select the background layer
\hedge[blue, very thick]{0,2,1}{0.3cm}
\hedge[violet, very thick]{0,2,3,1}{0.4cm}
\hedge[red, very thick]{0,5,4,1}{0.5cm}
%\end{pgfonlayer}
\end{tikzpicture}
\qquad
\begin{tikzpicture}[yscale=-1, every node/.style={draw, very thick, circle, inner sep=1pt}, edge from parent path={[very thick, draw] (\tikzparentnode) -- (\tikzchildnode)}]
	]
	\node(0) {0}
		child{node(1){1}}
		child{node(2){2}
			child{node(3){3}
			child{node(4){4}}
			}}
		child{node(5){5}};
	% \begin{pgfonlayer}{bg}    % select the background layer
	\hedge[blue, very thick]{2,3}{0.3cm}
	\hedge[blue, very thick]{0,5}{0.3cm}
	\hedge[violet, very thick]{0,5,3}{0.4cm}
	\hedge[red, very thick]{0,5,4,1}{0.5cm}
	%\end{pgfonlayer}
\end{tikzpicture}
\caption{Two nestings of a tree with $5$ internal edges. These nestings, \cref{def:nesting}, are also $2$-colored, \cref{def:2-Colored Nesting}.}
\label{fig:nestings}
\end{figure}

Since the operahedra are generalized permutahedra~\cite[Coro.~2.16]{LaplanteAnfossi}, a choice of diagonal for the permutahedra induces a choice of diagonal for every operahedron~\cite[Coro.~1.31]{LaplanteAnfossi}.
Every face of an operahedron is isomorphic to a product of lower-dimensional operahedra, via an isomorphism~$\Theta$ which generalizes the one from \cref{subsec:operadicProperty}, see Point (5) of~\cite[Prop.~2.3]{LaplanteAnfossi}.

\begin{definition}
An \emph{operadic diagonal} for the operahedra is a choice of diagonal $\triangle_t$ for each Loday operahedron~$P_t$, such that $\triangle \eqdef \{\triangle_t\}$ commutes with the map $\Theta$, \ie it satisfies~\cite[Prop.~4.14]{LaplanteAnfossi}.
\end{definition}

An operadic diagonal gives rise to topological operad structure on the set of Loday operahedra \cite[Thm 4.18]{LaplanteAnfossi}, and via the functor of cellular chains, to a universal tensor product of homotopy operads \cite[Prop. 4.27]{LaplanteAnfossi}. 
Here, by \defn{universal}, we mean a formula that applies uniformly to \emph{any} pair of homotopy operads. 
Since such an operad structure and tensor product are induced by a geometric diagonal, we shall call them \defn{geometric}.

\begin{theorem}
\label{thm:operahedra}
There are exactly 
\begin{enumerate}
\item two geometric operadic diagonals of the Loday operahedra, the $\LA$ and $\SU$ diagonals,
\item two geometric colored topological cellular operad structures on the Loday operahedra,
\item two geometric universal tensor products of homotopy operads,
\end{enumerate}
which agree with the generalized Tamari order on fully nested trees. 
\end{theorem}

\begin{proof}
Let us first examine Point (1).
By \cref{thm:unique-operadic}, we know that if one of the two choices $\LAD$ or $\SUD$ is made on an operahedron $P_t$, one has to make the same choice on every lower-dimensional operahedron appearing in the decomposition $P_{t_1} \times \cdots \times P_{t_k} \cong F \subset P_t$ of a face $F$ of~$P_t$. 
Now suppose that one makes two distinct choices for two operahedra $P_t$ and $P_{t'}$.
It is easy to find a bigger tree $t''$, of which both $t$ and $t'$ are subtrees.
Therefore, $P_t$ and $P_{t'}$ appear as facets of $P_{t''}$ and by the preceding remark, any choice of diagonal for $P_{t''}$ will then contradict our initial two choices. 
Thus, these had to be the same from the start, which concludes the proof. 

Point (2) then follows from the fact that a choice of diagonal for the Loday realizations of the operahedra \emph{forces} a unique topological cellular colored operad structure on them, see~\cite[Thm.~4.18]{LaplanteAnfossi}.
Since universal tensor products of homotopy operads are induced by a colored operad structure on the operahedra~\cite[Coro.~4.24]{LaplanteAnfossi}, we obtain Point~(3).
Finally, since only vectors with strictly decreasing coordinates induce the generalized Tamari order on the skeleton of the operahedra~\cite[Prop.~3.11]{LaplanteAnfossi}, we get the last part of the statement. 
\end{proof}

This answers a question raised in~\cite[Rem.~3.14]{LaplanteAnfossi}.

\begin{example}
The Loday associahedra correspond to the Loday operahedra associated with linear trees~\cite[Sect. 2.2]{LaplanteAnfossi}, and define a suboperad.
The restriction of the two operad structures of \cref{thm:operahedra} coincide in this case, and both the $\LA$ and $\SU$ diagonals induce the \emph{magical formula} of~\cite{MarklShnider, MasudaThomasTonksVallette, SaneblidzeUmble-comparingDiagonals} defining a universal tensor product of $\Ainf$-algebras. 
\end{example}

\begin{example}
The restriction of \cref{thm:operahedra} to the permutahedra associated with $2$-leveled trees gives two distinct universal tensor products of permutadic $\Ainf$-algebras, as studied in~\cite{LodayRonco-permutads,Markl}.
\end{example}

Two other important families of operadic polytopes are the \emph{Loday associahedra} and \emph{Forcey multiplihedra}, which encode respectively $\Ainf$-algebras and $\Ainf$-morphisms~\cite[Prop.~4.9]{LaplanteAnfossiMazuir}, as well as $\Ainf$-categories and $\Ainf$-functors~\cite[Sect.~4.3]{LaplanteAnfossiMazuir}.
For every linear tree $t \in \PT_n$, there is a corresponding Loday associahedron $\K_n$, whose faces are in bijection with nestings of $t$, and a Forcey multiplihedron $\J_n$ whose faces are in bijection with $2$-colored nestings of $t$.

\begin{definition}[{\cite[Def. 3.2]{LaplanteAnfossiMazuir}}]\label{def:2-Colored Nesting}
	A \defn{$2$-colored nesting} is a nesting where each nest $N$ is either blue, red, or blue and red (purple), and which satisfies that if $N$ is blue or purple (\resp red or purple), then all nests contained in $N$ are blue (\resp all nests that contain $N$ are red). 
 	See \cref{fig:nestings}.
\end{definition}

The Loday associahedra are faces of the Forcey multiplihedra: they correspond to $2$-colored nestings where all the nests are of the same color (either blue or red). 

Forcey realizations of the multiplihedra are not generalized permutahedra, but they are projections of the Ardila--Doker realizations, which are \cite[Prop. 1.16]{LaplanteAnfossiMazuir}.
A choice of diagonal for the permutahedra thus induces a choice of diagonal for every Ardila--Doker multiplihedron, and a subset of these choices (the ones which satisfy \cite[Prop. 2.7 \& 2.8]{LaplanteAnfossiMazuir}) further induce a choice of diagonal for the Forcey multiplihedra.
Every face of a Forcey multiplihedron is isomorphic to a product of a Loday associahedron and possibly many lower-dimensional Forcey multiplihedra, via an isomorphism $\Theta$ similar to the one from \cref{subsec:operadicProperty}, see Point (4) of \cite[Prop. 1.10]{LaplanteAnfossiMazuir}.

\begin{definition}
	An \defn{operadic diagonal} for the multiplihedra is a choice of diagonal $\triangle_n$ for each multiplihedron $\J_n$, such that $\triangle \eqdef \{\triangle_n\}$ commutes with the map $\Theta$.
\end{definition}

An operadic diagonal endows the Loday associahedra with a topological operad structure \cite[Thm.~1]{MasudaThomasTonksVallette}, and the Forcey multiplihedra with a topological operadic bimodule structure over the operad of Loday associahedra \cite[Thm.~1]{LaplanteAnfossiMazuir}.
Via the functor of cellular chains, it defines universal tensor products of $\Ainf$-algebras and $\Ainf$-morphisms \cite[Sec.~4.2.1]{LaplanteAnfossiMazuir}.
Here again, by \emph{universal} we mean a formula that applies uniformly to any pair of $\Ainf$-algebras or $\Ainf$-morphisms.
We shall call such geometrically defined operadic structures and tensor products \emph{geometric}. 

\begin{theorem}
\label{thm:multiplihedra}
There are exactly 
\begin{enumerate}
\item two geometric operadic diagonals of the Forcey multiplihedra, the $\LA$ and $\SU$ diagonals,
\item two geometric topological cellular operadic bimodule structures (over the Loday associahedra) on the Forcey multiplihedra,
\item two compatible geometric universal tensor products of $\Ainf$-algebras and $\Ainf$-morphisms,
\end{enumerate}
which agree with the Tamari-type order on atomic $2$-colored nested linear trees. 
\end{theorem}

\begin{proof}
Let us first examine Point (1).
Consider the vectors $\b{v}_\LA \eqdef (1,2^{-1},2^{-2},\ldots,2^{-n+1})$ and $\b{v}_\SU \eqdef (2^n-1,2^{n}-2,2^n-2^2,\ldots,2^n-2^{n-1})$ in $\R^n$.
As previously observed, they induce the $\LA$ and $\SU$ diagonals on the permutahedra (\cref{def:LA-and-SU}).
One checks directly that both vectors satisfy \cite[Prop. 2.7 \& 2.8]{LaplanteAnfossiMazuir}, and thus define diagonals of the Forcey multiplihedron~$\J_n$ which agree with the Tamari-type order \cite[Prop. 2.10]{LaplanteAnfossiMazuir}.
Moreover, these diagonals commute with the map $\Theta$ for the Forcey multiplihedra \cite[Prop. 2.14]{LaplanteAnfossiMazuir}; this is because after deleting the last coordinate of $\b{v}_\LA$ or $\b{v}_\SU$, and then applying $\Theta^{-1}$, we still have vectors which induce the $\LA$ or $\SU$ diagonal, respectively. 

By \cref{thm:unique-operadic}, we know that if one of the two choices $\LAD$ or $\SUD$ is made on a multiplihedron $\J_n$, one has to make the same choice on every lower-dimensional multiplihedra and associahedra appearing in the product decomposition of any face of~$\J_n$. 
Now suppose that one makes two distinct choices for two multiplihedra $\J_n$ and $\J_{n'}$.
It is easy to find a bigger multiplihedron $\J_{n''}$, for which $\J_n$ and $\J_{n'}$ appear in the product decomposition of a face of $\J_{n''}$ and by the preceding remark, any choice of diagonal for $\J_{n''}$ will then contradict our initial two choices. 
Thus, these had to be the same from the start, which conclude the proof of Point (1). 

Point (2) then follows from the fact that a choice of diagonal for the Loday associahedra and the Forcey multiplihedra \emph{forces} a unique topological cellular colored operad and operadic bimodule structure on them, see~\cite[Thm.~1]{MasudaThomasTonksVallette} and \cite[Thm.~1]{LaplanteAnfossiMazuir}.
Since a universal tensor products of $\Ainf$-algebras, and a compatible universal tensor products of $\Ainf$-morphisms are induced by an operad and operadic bimodule structures on the associahedra and multiplihedra respectively~\cite[Sec.~4.2.3]{LaplanteAnfossiMazuir}, we obtain Point~(3).
Finally, since only vectors with strictly decreasing coordinates induce the Tamari-type order on the skeleton of the Loday multiplihedra~\cite[Prop. 2.10]{LaplanteAnfossiMazuir}, we get the last part of the statement. 
\end{proof}

This answers a question raised in~\cite[Rem.~3.9]{LaplanteAnfossiMazuir}.

\begin{remark}
	Note that in the case of the Loday associahedra, there is only one geometric operadic diagonal which induces the Tamari order atomic $2$-colored nested planar trees (equivalently, binary trees, see \cite[Fig.~6]{LaplanteAnfossiMazuir}).
	Therefore, there is only one geometric topological operad structure, and only one geometric universal tensor product.
	This is because any vector with strictly decreasing coordinates lives in the same chamber of the fundamental hyperplane arrangement of the Loday associahedra (see \cite[Ex.~1.21]{LaplanteAnfossi}).
\end{remark}

\begin{remark}
	Considering all $2$-colored nested trees instead of only linear trees, one should obtain similar results for tensor products of $\infty$-morphisms of homotopy operads.
\end{remark}

We shall see now that the two operad (\resp operadic bimodule) structures on the operahedra (\resp multiplihedra) are related to one another in the strongest possible sense: they are isomorphic as topological cellular colored operads (\resp topological operadic bimodule structure over the associahedra).

%%%%%%%%%%%%%%%

\subsection{Relating operadic structures} 
\label{subsec:iso-top-operads}

Recall that the topological cellular operad structure on the operahedra~\cite[Def.~4.17]{LaplanteAnfossi} is given by a family of partial composition maps 
\[
\vcenter{\hbox{
\begin{tikzcd}[column sep=1cm]
\circ_i^{\LA}\ : \ P_{t'}\times P_{t''}
\arrow[r,  "\tr\times \id"]
& P_{(t',\omega)}\times P_{t''}
 \arrow[r,hookrightarrow, "\Theta"]
&
P_t .
\end{tikzcd}
}}  \]
Here, the map $\tr$ is the \emph{unique} topological cellular map which commutes with the diagonal~$\LAD$, see~\cite[Prop.~7]{MasudaThomasTonksVallette}. 
This partial composition $\circ_i^\LA$ is an isomorphism in the category $\PolySub$~\cite[Def.~4.13]{LaplanteAnfossi} between the product $P_{t'}\times P_{t''}$ and the facet $t' \circ_i t''$ of~$P_t$.
At the level of trees, the composition operation $\circ_i$ is given by \emph{substitution} \cite[Fig.~14]{LaplanteAnfossi}.
Using the $\SU$ diagonal $\SUD$, one can define similarly a topological operad structure via the same formula, but with a different transition map $\tr$, which commutes with $\SUD$.

Recall that a face $F$ of $P_t$ is represented by a nested tree $(t,\mathcal{N})$, which can be written uniquely as a sequence of substitution of trivially nested trees 
$(t,\mathcal{N})=((\cdots((t_1\circ_{i_1} t_2) \circ_{i_2} t_3) \cdots )\circ_{i_k} t_{k+1})$.
Here we use the increasing order on nestings~\cite[Def. 4.5]{LaplanteAnfossi}, and observe that any choice of sequence of $\circ_i$ operations yield the same nested tree, since these form an operad~\cite[Def.~4.7]{LaplanteAnfossi}.
At the geometric level, we have an isomorphism
\[((\cdots((\circ_{i_1}^\LA) \circ_{i_2}^\LA) \cdots) \circ_{i_k}^\LA): P_{t_1} \times P_{t_2} \times \cdots \times P_{t_{k+1}} \overset{\cong}{\longrightarrow} F \subset P_t \]
between a uniquely determined product of lower dimensional operahedra, and the face $F=(t,\mathcal{N})$ of~$P_t$.
Note that any choice of sequence of $\circ_i^\LA$ operations yield the same isomorphism, since they form an operad~\cite[Thm.~4.18]{LaplanteAnfossi}.
The same holds when taking the~$\circ_i^\SU$ operations instead of the $\circ_i^\LA$. 

\begin{construction}
	\label{const:top-iso}
	For any operahedron $P_t$, we define a map $\Psi_t : P_t \to P_t$ 
	\begin{itemize}
		\item on the interior of the top face by the identity $\id : \mathring P_t \to \mathring P_t$, and 
		\item on the interior of the face $F=((\cdots((t_1 \circ_{i_1} t_2) \circ_{i_2} t_3) \cdots )\circ_{i_k} t_{k+1})$ of~$P_t$ by the composition of the two isomorphisms
		\[ 
		((\cdots ((\circ_{i_1}^\SU) \circ_{i_2}^\SU) \cdots) \circ_{i_k}^\SU) \ ((\cdots((\circ_{i_1}^\LA) \circ_{i_2}^\LA) \cdots) \circ_{i_k}^\LA)^{-1}: F \to F . \] 
	\end{itemize}
\end{construction}

\begin{theorem}
\label{thm:top-iso}
The map $\Psi \eqdef \{\Psi_t\}$ is an isomorphism of topological cellular symmetric colored operad between the $\LA$ and $\SU$ operad structures on the operahedra, in the category $\PolySub$.
\end{theorem}

\begin{proof}
By definition, we have that $\Psi$ is an isomorphism in the category $\PolySub$. 
It remains to show that it preserves the operad structures, \ie that the following diagram commutes
\[
\vcenter{\hbox{
\begin{tikzcd}[column sep=2.2cm, row sep=1.3cm]
P_{t'}\times P_{t''}
\arrow[r,  "\circ_i^\LA"] 
\arrow[d,  "\Psi_{t'}\times\Psi_{t''}"']
& P_{t} \arrow[d,  "\Psi_t"] \\
P_{t'}\times P_{t''}  
\arrow[r,  "\circ_i^\SU"']
& P_{t}
\end{tikzcd}
}}\]
For two interior points $(x,y) \in \mathring P_{t'}\times \mathring P_{t''}$, the diagram clearly commutes by definition, since $\Psi_{t'}$ and~$\Psi_{t''}$ are the identity in that case. 
If $x$ is in a face $F=((\cdots((t_1 \circ_{i_1} t_2) \circ_{i_2} t_3) \cdots )\circ_{i_k} t_{k+1})$ of the boundary of~$P_{t'}$, then the lower composite is equal to $\circ_i^\SU (\circ_{i_1}^\SU \circ_{i_2}^\SU \cdots \circ_{i_k}^\SU \times \id)(\circ_{i_1}^\LA \circ_{i_2}^\LA \cdots \circ_{i_k}^\LA \times \id)^{-1}$, and so is the upper composite since $\Psi_t$ starts with the inverse $(\circ_i^\LA)^{-1}$ and the decomposition of~$F$ into $P_{t_1} \times \cdots \times P_{t_{k+1}} \times P_{t''}$ is unique.
The case when $y$ is in the boundary of $P_{t''}$ is similar.  
Finally, the compatibility of $\Psi$ with units and the symmetric group actions are straightforward to check, see~\cite[Def.~4.17 \& Thm.~4.18]{LaplanteAnfossi}.
\end{proof}

\begin{remark}
\cref{const:top-iso} and \cref{thm:top-iso} do not depend on a specific choice of operadic diagonal.
In this case, however, we do not lose any generality by using specifically the $\LA$ and $\SU$ operad structures. 
\end{remark}

\begin{example}
\label{ex:iso-not-Hopf}
Note that $\Psi$ is \emph{not} a morphism of ``Hopf" operads, \ie it does not commute with the respective diagonals $\LAD$ and $\SUD$. 
Consider the two square faces $F \eqdef 12|34$ and $G \eqdef 24|13$ of the $3$-dimensional permutahedron $\Perm[4]$, and choose a point $\b z \in (\mathring F + \mathring G)/2$.
Then, $\LAD(z)$ and~$\SUD(z)$ are two different pair of points on the $1$-skeleton of $\Perm[4]$. 
Since $\circ_i^\LA$ and $\circ_i^\SU$ are the identity both on the interior of $\Perm[4]$ (by \cref{const:top-iso}) and on the $1$-skeleton of $\Perm[4]$ (see the proof of~\cite[Prop. 7]{MasudaThomasTonksVallette}), we directly obtain that
\[
{\LAD(z)=(\Psi \times \Psi)\LAD(z) \neq \SUD \Psi(z)=\SUD(z)}.
\] 
\end{example}

Recall that the topological cellular operadic bimodule structure on the Forcey multiplihedra is given by a family of action-composition maps \cite[Def. 2.13]{LaplanteAnfossiMazuir}
\[
\vcenter{\hbox{
\begin{tikzcd}[column sep = 16pt]
\circ_{p+1}^\LA\ : \ \J_{p+1+r}\times \K_q
\arrow[rr,  "\tr\times \id"]
& & 
\J_{(1,\ldots,q,\ldots,1)}\times \K_q 
\arrow[rr,hookrightarrow, "\Theta_{p,q,r}"]
&  &
\J_{n}\ \ \text{and}
\end{tikzcd}
}}
\]
\[
\vcenter{\hbox{
\begin{tikzcd}[column sep = 16pt]
\gamma_{i_1,\ldots,i_k}^\LA \ : \ \K_{k}\times \J_{i_1} \times \cdots \times \J_{i_k}
\arrow[rr,  "\tr\times \id"]
& &
\K_{(i_1,\ldots,i_k)} \times \J_{i_1} \times \cdots \times \J_{i_k} 
\arrow[rr,hookrightarrow, "\Theta^{i_1, \ldots , i_k}"]
& &
\J_{i_1+\cdots + i_k}\, .
\end{tikzcd}
}}
\]
Here, the map $\tr$ is the \emph{unique} topological cellular map which commutes with the diagonal $\LAD$, see \cite[Prop. 7]{MasudaThomasTonksVallette}. 
These action-composition maps $\circ_{p+1}^\LA$ and $\gamma_{i_1,\ldots,i_k}^\LA$ are isomorphisms in the category $\PolySub$ \cite[Sec.~2.1]{LaplanteAnfossiMazuir} between the products $\J_{p+1+r}\times \K_q$ and $\K_{k}\times \J_{i_1} \times \cdots \times \J_{i_k}$, and corresponding facets of $\J_n$ and $\J_{i_1 + \cdots + i_k}$, respectively.
Using the $\SU$ diagonal $\SUD$, one defines similarly a topological operadic bimodule structure via the same formula, but with a different transition map $\tr$, which commutes with $\SUD$.

There is a bijection between $2$-colored planar trees and $2$-colored nested linear trees \cite[Lem.~3.4 \& Fig.~6]{LaplanteAnfossiMazuir}, which translate grafting of planar trees into substitution at a vertex of nested linear trees.
The indices of the $\circ_{p+1}$ and $\gamma_{i_1,\ldots,i_k}$ operations above refer to grafting. 
Equivalently, a face of $\J_n$ is represented by a $2$-colored nested tree $(t,\mathcal{N})$, which can be written uniquely as a sequence of substitution of trivially nested $2$-colored trees $(t,\mathcal{N})=((\cdots((t_1\circ_{i_1} t_2) \circ_{i_2} t_3) \cdots )\circ_{i_k} t_{k+1})$.
Here we use the left-levelwise order on nestings \cite[Def. 4.12]{LaplanteAnfossiMazuir}, and translate tree grafting operations $\circ_{p+1}$ and $\gamma_{i_1,\ldots,i_k}$ into nested tree substitution $\circ_{i_j}$.
Note that any choice of substitutions yield the same $2$-colored nested tree, since these form an operadic bimodule.

At the geometric level, we have an isomorphism $((\cdots((\circ_{i_1}^\LA) \circ_{i_2}^\LA) \cdots) \circ_{i_k}^\LA)$ between a uniquely determined product of lower dimensional associahedra and multiplihedra, and the face $(t,\mathcal{N})$.
Note that any choice of $\circ_i^\LA$ operations (\ie the $\circ_{p+1}^\LA$ and $\gamma_{i_1,\ldots,i_k}^\LA$ action-composition operations) yield the same isomorphism, since they form an operadic bimodule \cite[Thm.~1]{LaplanteAnfossiMazuir}.
The same holds when taking the $\circ_i^\SU$ (\ie the $\circ_{p+1}^\SU$ and $\gamma_{i_1,\ldots,i_k}^\SU$ action-composition) operations instead.

\begin{construction}
	\label{const:top-iso-2}
	For any Forcey multiplihedron $\J_n$, we define a map $\Psi_n : \J_n \to \J_n$ 
	\begin{itemize}
		\item on the interior of the top face by the identity $\id : \mathring \J_n \to \mathring \J_n$, and 
		\item on the interior of the face $F=((\cdots((t_1 \circ_{i_1} t_2) \circ_{i_2} t_3) \cdots )\circ_{i_k} t_{k+1})$ of~$\J_n$ by the composition of the two isomorphisms
		\[ 
		((\cdots ((\circ_{i_1}^\SU) \circ_{i_2}^\SU) \cdots) \circ_{i_k}^\SU) \ ((\cdots((\circ_{i_1}^\LA) \circ_{i_2}^\LA) \cdots) \circ_{i_k}^\LA)^{-1}: F \to F . \] 
	\end{itemize}
\end{construction}

\begin{theorem}
\label{thm:top-iso-2}
The map $\Psi \eqdef \{\Psi_n\}$ is an isomorphism of topological cellular operadic bimodule structure over the Loday associahedra between the $\LA$ and $\SU$ operadic bimodule structures on the Forcey multiplihedra, in the category $\PolySub$.
\end{theorem}

\begin{proof}
	The proof is the same as the one of \cref{thm:top-iso}, with the multiplihedra $\circ_i^\LA$ and $\circ_i^\SU$ operations (that is, the action-composition maps $\circ_{p+1}^\LA$ and $\gamma_{i_1,\ldots,i_k}^\LA$, and $\circ_{p+1}^\SU$ and $\gamma_{i_1,\ldots,i_k}^\SU$) in place of the operahedra operations. 
\end{proof}

\begin{example}
	\label{ex:iso-not-Hopf-2}
	Note that $\Psi$ does not commute with the respective diagonals $\LAD$ and $\SUD$. 
	Consider the two square faces $F \eqdef \purplea{\bluea{\bullet \bullet \bullet}\bullet}$ and $G \eqdef \reda{\bullet\purplea{\bullet\bullet}\bullet}$ of the $3$-dimensional Forcey multiplihedron $\J_4$, and choose a point $\b z \in (\mathring F + \mathring G)/2$.
	Then, $\LAD(z)$ and~$\SUD(z)$ are two different pair of points on the $1$-skeleton of $\J_4$ (see \cite[Ex.~3.7 \& Fig.~9]{LaplanteAnfossiMazuir}). 
	Since the $\LA$ and $\SU$ action-composition maps are the identity both on the interior of $\J_4$ (by \cref{const:top-iso-2}) and on the $1$-skeleton of $\J_4$ (see the proof of~\cite[Prop. 7]{MasudaThomasTonksVallette}), we directly obtain that
	\[
	{\LAD(z)=(\Psi \times \Psi)\LAD(z) \neq \SUD \Psi(z)=\SUD(z)}.
	\] 
\end{example}

%%%%%%%%%%%%%%%

\subsection{Tensor products}
\label{subsec:tensor-products}
Recall that a homotopy operad $\mathcal{P}$ is a family of vector spaces $\{\mathcal{P}(n)\}_{n \geq 1}$ together with a family of operations $\{\mu_t\}$ indexed by planar trees $t$ \cite[Def. 4.11]{LaplanteAnfossi}.
One can consider the category of homotopy operads with strict morphisms, that is morphisms of the underlying vector spaces which commute strictly with all the higher operations $\mu_t$, or with their $\infty$-morphisms, made of a tower of homotopies controlling the lack of commutativity of their first component with the higher operations \cite[Sec. 10.5.2]{LodayVallette}.

\begin{theorem}
\label{thm:infinity-iso}
For any pair of homotopy operads, the two universal tensor products defined by the $\LA$ and $\SU$ diagonals are not isomorphic in the category of homotopy operads and strict morphisms.
However, they are isomorphic in the category of homotopy operads and their $\infty$-morphisms.
\end{theorem}

\begin{proof}
Since the two morphisms of topological operads $\LAD$ and $\SUD$ do not have the same cellular image, the tensor products that they define are not strictly isomorphic.
However, they are both homotopic to the usual thin diagonal.
Recall that homotopy operads are algebras over the colored operad $\mathcal{O}_\infty$, which is the minimal model of the operad $\mathcal{O}$ encoding (non-symmetric non-unital) operads \cite[Prop. 4.9]{LaplanteAnfossi}. 
Using the universal property of the minimal model $\mathcal{O}_\infty$, one can show that the algebraic diagonals $\LAD,\SUD : \mathcal{O}_\infty \to \mathcal{O}_\infty \otimes \mathcal{O}_\infty$ are homotopic, in the sense of \cite[Sec. 3.10]{MarklShniderStasheff}, see \cite[Prop. 3.136]{MarklShniderStasheff}.
Then, by \cite[Cor.~2]{DotsenkoShadrinVallette} there is an $\infty$-isotopy, that is an $\infty$-isomorphism whose first component is the identity, between the two homotopy operad structures on the tensor product.
\end{proof}

\begin{remark}
Neither of the two diagonals $\LAD$ or $\SUD$ are cocommutative, or coassociative, as they are special cases of $\Ainf$-algebras \cite[Thm. 13]{MarklShnider}. 
\end{remark}

Note that restricting to linear trees, the two tensor products of $\Ainf$-algebras induced by the $\LA$ and $\SU$ diagonals coincide (and are thus strictly isomorphic).
Restricting to $2$-leveled trees, we obtain two tensor product of permutadic $\Ainf$-algebras whose terms are in bijection.
For the operahedra in general, such a bijection does not exist, as the following example demonstrates. 

\begin{example}
\label{ex:operahedra-LA-SU}
The $\LA$ and $\SU$ diagonals of the operahedra associated with trees that have less than $4$ internal edges have the same number of facets. 
However, there are 24 planar trees with $5$ internal edges, such that the number of facets of the $\LA$ and $\SU$ diagonals are distinct, displayed in \cref{fig:trees}.
To compute these numbers, we first computed the facets of the $\LA$ and $\SU$ diagonals of the permutahedra, and then used the projection from the permutahedra to the operahedra described in \cite[Prop. 3.20]{LaplanteAnfossi}.
\end{example}

\begin{remark}
The lack of symmetry in the trees in \cref{fig:trees} arises from the lack of symmetry inherent in the infix order, and in how the $\LA$ and $\SU$ diagonal treat maximal and minimal elements.
A sufficient condition for the diagonals of a tree $t$ to have the same number of facets is to satisfy, $N$ is a nesting of $t$ if and only if $rN$ is a nesting of $t$.
For a tree satisfying this condition, relabelling its edges via the function $r : [n]\to [n]$ defined by $r(i)\eqdef n-i+1$ exchanges the number of facets between the $\LA$ and $\SU$ diagonals.
\end{remark}

We have an analogous result for universal tensor products of $\Ainf$-morphisms. 
Let $\Ainf^2$ denote the $2$-colored operad whose algebras are pairs of $\Ainf$-algebras together with an $\Ainf$-morphism between them \cite[Sec.~4.4.1]{LaplanteAnfossiMazuir}.
The datum of a diagonal of the operad $\Ainf$ encoding $\Ainf$-algebras and a diagonal of the operadic bimodule $\mathrm{M}_\infty$ encoding $\Ainf$-morphisms is equivalent to the datum of a morphism of $2$-colored operads $\Ainf^2 \to \Ainf^2 \otimes \Ainf^2$. 

\begin{theorem}
	\label{thm:infinity-iso-2}
	For any pair of $\Ainf$-morphisms, the two universal tensor products defined by the $\LA$ and $\SU$ diagonals are not isomorphic in the category of $\Ainf^2$-algebras and strict morphisms.
	However, they are isomorphic in the category of $\Ainf^2$-algebras and their $\infty$-morphisms.
\end{theorem}

\begin{proof}
	Since the two morphisms of topological operadic bimodules on the multiplihedra $\LAD$ and $\SUD$ do not have the same cellular image, the tensor products that they define are not strictly isomorphic.
	However, they are both homotopic to the usual thin diagonal.
	Recall that the operad $\Ainf^2$ is the minimal model of the operad $\mathrm{As}^2$, whose algebras are pairs of associative algebras together with a morphism between them \cite[Prop.~4.9]{LaplanteAnfossi}. 
	Using the universal property of the minimal model $\Ainf^2$, one can show that the algebraic diagonals $\LAD,\SUD : \Ainf^2 \to \Ainf^2 \otimes \Ainf^2$ are homotopic, in the sense of \cite[Sec.~3.10]{MarklShniderStasheff}.
	Then, by \cite[Cor.~2]{DotsenkoShadrinVallette} there is an $\infty$-isotopy, that is an $\infty$-isomorphism whose first component is the identity, between the two tensor products of $\Ainf$-morphisms.
\end{proof}

\begin{remark}
	As studied in \cite[Sec.~4.4]{LaplanteAnfossiMazuir}, the above tensor products of $\Ainf$-morphisms are not coassociative, nor cocommutative. 
	Moreover, there \emph{does not exist} a universal tensor product of $\Ainf$-morphisms which is compatible with composition \cite[Prop.~4.23]{LaplanteAnfossiMazuir}.
\end{remark}

\begin{example}
	\label{ex:multiplihedra-LA-SU}
	The $\LA$ and $\SU$ diagonals of the multiplihedra associated with trees that have less than $4$ edges have the same number of facets. 
	However, for linear trees with $5$ and $6$ internal edges, the number of facets of the $\LA$ and $\SU$ diagonals differ, as displayed in \cref{table:multiplihedra}.
	To compute these numbers, we first computed the facets of the $\LA$ and $\SU$ diagonals of the permutahedra, and then used the projection from the permutahedra to the multiplihedra described in the proof of \cite[Thm.~3.3.6]{Doker}.
\end{example}

\begin{figure}[h]
\begin{tabular}{cccccc}
	\imagebot{\begin{tikzpicture}[yscale=-1,every tree node/.style={draw, very thick, circle, inner sep=0.08cm}, level distance=0.7cm,sibling distance=0.3cm, , edge from parent path={[very thick, draw] (\tikzparentnode) -- (\tikzchildnode)}]]
	\Tree 
	[.\node{};
		[.\node{};
			[.\node{};]
		]
		[.\node{};]
		[.\node{};]
		[.\node{};]
	]
	\node(1) at (0,0.5) {$(266,263)$};
	\end{tikzpicture}}
	&
	\imagebot{\begin{tikzpicture}[yscale=-1,every tree node/.style={draw, very thick, circle, inner sep=0.08cm}, level distance=0.7cm,sibling distance=0.3cm, , edge from parent path={[very thick, draw] (\tikzparentnode) -- (\tikzchildnode)}]]
	\Tree 
	[.\node{};
		[.\node{};]
		[.\node{};
			[.\node{};]
		]
		[.\node{};]
		[.\node{};]
	]
	\node(1) at (0,0.5) {$(256,254)$};
	\end{tikzpicture}}
	&
	\imagebot{\begin{tikzpicture}[yscale=-1,every tree node/.style={draw, very thick, circle, inner sep=0.08cm}, level distance=0.7cm,sibling distance=0.3cm, , edge from parent path={[very thick, draw] (\tikzparentnode) -- (\tikzchildnode)}]]
	\Tree 
	[.\node{};
		[.\node{};]
		[.\node{};]
		[.\node{};
			[.\node{};]
		]
		[.\node{};]
	]
	\node(1) at (0,0.5) {$(255,254)$};
	\end{tikzpicture}}
	&
	\imagebot{\begin{tikzpicture}[yscale=-1,every tree node/.style={draw, very thick, circle, inner sep=0.08cm}, level distance=0.7cm,sibling distance=0.3cm, , edge from parent path={[very thick, draw] (\tikzparentnode) -- (\tikzchildnode)}]]
	\Tree 
	[.\node{};
		[.\node{};]
		[.\node{};]
		[.\node{};]
		[.\node{};
			[.\node{};]
		]
	]
	\node(1) at (0,0.5) {$(263,266)$};
	\end{tikzpicture}}
	&
	\imagebot{\begin{tikzpicture}[yscale=-1,every tree node/.style={draw, very thick, circle, inner sep=0.08cm}, level distance=0.7cm,sibling distance=0.3cm, , edge from parent path={[very thick, draw] (\tikzparentnode) -- (\tikzchildnode)}]]
	\Tree 
	[.\node{};
		[.\node{};
			[.\node{};]
				[.\node{};]
		]
		[.\node{};]
		[.\node{};]
	]
	\node(1) at (0,0.5) {$(214,216)$};
	\end{tikzpicture}}
	&
	\imagebot{\begin{tikzpicture}[yscale=-1,every tree node/.style={draw, very thick, circle, inner sep=0.08cm}, level distance=0.7cm,sibling distance=0.3cm, , edge from parent path={[very thick, draw] (\tikzparentnode) -- (\tikzchildnode)}]]
	\Tree 
	[.\node{};
		[.\node{};
			[.\node{};]
		]
		[.\node{};
			[.\node{};]
		]
		[.\node{};]
	]
	\node(1) at (0,0.5) {$(162,161)$};
	\end{tikzpicture}}
	\\
	\imagebot{\begin{tikzpicture}[yscale=-1,every tree node/.style={draw, very thick, circle, inner sep=0.08cm}, level distance=0.7cm,sibling distance=0.3cm, , edge from parent path={[very thick, draw] (\tikzparentnode) -- (\tikzchildnode)}]]
	\Tree 
	[.\node{};
		[.\node{};]
		[.\node{};
			[.\node{};]
			[.\node{};]
		]
		[.\node{};]
	]
	\node(1) at (0,0.5) {$(212,213)$};
	\end{tikzpicture}}
	&	
	\imagebot{\begin{tikzpicture}[yscale=-1,every tree node/.style={draw, very thick, circle, inner sep=0.08cm}, level distance=0.7cm,sibling distance=0.3cm, , edge from parent path={[very thick, draw] (\tikzparentnode) -- (\tikzchildnode)}]]
	\Tree 
	[.\node{};
		[.\node{};]
		[.\node{};
			[.\node{};
				[.\node{};]
			]
		]
		[.\node{};]
	]
	\node(1) at (0,0.5) {$(129,127)$};
	\end{tikzpicture}}
		&
	\imagebot{\begin{tikzpicture}[yscale=-1,every tree node/.style={draw, very thick, circle, inner sep=0.08cm}, level distance=0.7cm,sibling distance=0.3cm, , edge from parent path={[very thick, draw] (\tikzparentnode) -- (\tikzchildnode)}]]
	\Tree 
	[.\node{};
		[.\node{};]
		[.\node{};
			[.\node{};]
		]
		[.\node{};
			[.\node{};]
		]
	]
	\node(1) at (0,0.5) {$(160,161)$};
	\end{tikzpicture}}
	&
	\imagebot{\begin{tikzpicture}[yscale=-1,every tree node/.style={draw, very thick, circle, inner sep=0.08cm}, level distance=0.7cm,sibling distance=0.3cm, , edge from parent path={[very thick, draw] (\tikzparentnode) -- (\tikzchildnode)}]]
	\Tree 
	[.\node{};
		[.\node{};  
			[.\node{};  
				[.\node{};]
			]
			[.\node{};] 
		]
		[.\node{};]
	]
	\node(1) at (0,0.5) {$(142,141)$};
	\end{tikzpicture}}
	&
	\imagebot{\begin{tikzpicture}[yscale=-1,every tree node/.style={draw, very thick, circle, inner sep=0.08cm}, level distance=0.7cm,sibling distance=0.3cm, , edge from parent path={[very thick, draw] (\tikzparentnode) -- (\tikzchildnode)}]]
	\Tree 
	[.\node{};
		[.\node{};  
			[.\node{};]
			[.\node{};] 
		]
		[.\node{};  
			[.\node{};] 
		]
	]
	\node(1) at (0,0.5) {$(141,144)$};
	\end{tikzpicture}}
	&
	\imagebot{\begin{tikzpicture}[yscale=-1,every tree node/.style={draw, very thick, circle, inner sep=0.08cm}, level distance=0.7cm,sibling distance=0.3cm, , edge from parent path={[very thick, draw] (\tikzparentnode) -- (\tikzchildnode)}]]
	\Tree 
	[.\node{};
		[.\node{};  
			[.\node{};  
				[.\node{};] 
			]
		]
		[.\node{};  
			[.\node{};] 
		]
	]
	\node(1) at (0,0.5) {$(91,92)$};
	\end{tikzpicture}}
	\\
	\imagebot{\begin{tikzpicture}[yscale=-1,every tree node/.style={draw, very thick, circle, inner sep=0.08cm}, level distance=0.7cm,sibling distance=0.3cm, , edge from parent path={[very thick, draw] (\tikzparentnode) -- (\tikzchildnode)}]]
	\Tree 
	[.\node{}; 
		[.\node{};  
			[.\node{};] 
		]
		[.\node{};  
			[.\node{};]
			[.\node{};] 
		]
	]
	\node(1) at (0,0.5) {$(155,152)$};
	\end{tikzpicture}}
		&
	\imagebot{\begin{tikzpicture}[yscale=-1,every tree node/.style={draw, very thick, circle, inner sep=0.08cm}, level distance=0.7cm,sibling distance=0.3cm, , edge from parent path={[very thick, draw] (\tikzparentnode) -- (\tikzchildnode)}]]
	\Tree 
	[.\node{}; 
		[.\node{};  
			[.\node{};] 
		]
		[.\node{};  
			[.\node{};  
				[.\node{};] 
			]
		]
	]
	\node(1) at (0,0.5) {$(98,97)$};
	\end{tikzpicture}}
	&
	\imagebot{\begin{tikzpicture}[yscale=-1,every tree node/.style={draw, very thick, circle, inner sep=0.08cm}, level distance=0.7cm,sibling distance=0.3cm, , edge from parent path={[very thick, draw] (\tikzparentnode) -- (\tikzchildnode)}]]
	\Tree 
	[.\node{}; 
		[.\node{};]
		[.\node{};
			[.\node{};]
			[.\node{};]
			[.\node{};]
		]
	]
	\node(1) at (0,0.5) {$(266,263)$};
	\end{tikzpicture}}
	&  
	\imagebot{\begin{tikzpicture}[yscale=-1,every tree node/.style={draw, very thick, circle, inner sep=0.08cm}, level distance=0.7cm,sibling distance=0.3cm, , edge from parent path={[very thick, draw] (\tikzparentnode) -- (\tikzchildnode)}]]
	\Tree 
	[.\node{}; 
		[.\node{};]
		[.\node{};
			[.\node{};  
				[.\node{};] 
			]
			[.\node{};]
		]
	]
	\node(1) at (0,0.5) {$(157,154)$};
	\end{tikzpicture}}
	&
	\imagebot{\begin{tikzpicture}[yscale=-1,every tree node/.style={draw, very thick, circle, inner sep=0.08cm}, level distance=0.7cm,sibling distance=0.3cm, , edge from parent path={[very thick, draw] (\tikzparentnode) -- (\tikzchildnode)}]]
	\Tree 
	[.\node{}; 
		[.\node{};
			[.\node{};
				[.\node{};]
			]
			[.\node{};]
			[.\node{};]
		]
	]
	\node(1) at (0,0.5) {$(256,255)$};
	\end{tikzpicture}}
	&
	\imagebot{\begin{tikzpicture}[yscale=-1,every tree node/.style={draw, very thick, circle, inner sep=0.08cm}, level distance=0.7cm,sibling distance=0.3cm, , edge from parent path={[very thick, draw] (\tikzparentnode) -- (\tikzchildnode)}]]
	\Tree 
	[.\node{}; 
		[.\node{};
			[.\node{};]
			[.\node{};
				[.\node{};]
			]
			[.\node{};]
		]
	]
	\node(1) at (0,0.5) {$(255,254)$};
	\end{tikzpicture}}
	\\
	\imagebot{\begin{tikzpicture}[yscale=-1,every tree node/.style={draw, very thick, circle, inner sep=0.08cm}, level distance=0.7cm,sibling distance=0.3cm, , edge from parent path={[very thick, draw] (\tikzparentnode) -- (\tikzchildnode)}]]
	\Tree 
	[.\node{}; 
		[.\node{};
			[.\node{};]
			[.\node{};]
			[.\node{};
				[.\node{};]
			] 
		]
	]
	\node(1) at (0,0.5) {$(263,266)$};
	\end{tikzpicture}}
	&
	\imagebot{\begin{tikzpicture}[yscale=-1,every tree node/.style={draw, very thick, circle, inner sep=0.08cm}, level distance=0.7cm,sibling distance=0.3cm, , edge from parent path={[very thick, draw] (\tikzparentnode) -- (\tikzchildnode)}]]
	\Tree 
	[.\node{}; 
		[.\node{};
			[.\node{}; 
				[.\node{};]
				[.\node{};]
			] 
			[.\node{};]
		]
	]
	\node(1) at (0,0.5) {$(212,213)$};
	\end{tikzpicture}}  
	&
	\imagebot{\begin{tikzpicture}[yscale=-1,every tree node/.style={draw, very thick, circle, inner sep=0.08cm}, level distance=0.7cm,sibling distance=0.3cm, , edge from parent path={[very thick, draw] (\tikzparentnode) -- (\tikzchildnode)}]]
	\Tree 
	[.\node{}; 
		[.\node{};
			[.\node{}; 
				[.\node{}; 
					[.\node{}; ]
				]
			] 
			[.\node{};]
		]
	]
	\node(1) at (0,0.5) {$(129,127)$};
	\end{tikzpicture}}
	&
	\imagebot{\begin{tikzpicture}[yscale=-1,every tree node/.style={draw, very thick, circle, inner sep=0.08cm}, level distance=0.7cm,sibling distance=0.3cm, , edge from parent path={[very thick, draw] (\tikzparentnode) -- (\tikzchildnode)}]]
	\Tree 
	[.\node{}; 
		[.\node{}; 
			[.\node{};  
				[.\node{};] 
			]
			[.\node{};  
				[.\node{};] 
			]
		]
	]
	\node(1) at (0,0.5) {$(160,161)$};
	\end{tikzpicture}}
	&
	\imagebot{\begin{tikzpicture}[yscale=-1,every tree node/.style={draw, very thick, circle, inner sep=0.08cm}, level distance=0.7cm,sibling distance=0.3cm, , edge from parent path={[very thick, draw] (\tikzparentnode) -- (\tikzchildnode)}]]
	\Tree 
	[.\node{}; 
		[.\node{}; 
			[.\node{};
				[.\node{};]
				[.\node{};]
				[.\node{};] 
			]
		]
	]
	\node(1) at (0,0.5) {$(266,263)$};
	\end{tikzpicture}}
	&
	\imagebot{\begin{tikzpicture}[yscale=-1,every tree node/.style={draw, very thick, circle, inner sep=0.08cm}, level distance=0.7cm,sibling distance=0.3cm, , edge from parent path={[very thick, draw] (\tikzparentnode) -- (\tikzchildnode)}]]
	\Tree 
	[.\node{}; 
		[.\node{}; 
			[.\node{};
				[.\node{};  
					[.\node{};] 
				]
				[.\node{};]
			]
		]
	]
	\node(1) at (0,0.5) {$(154,157)$};
	\end{tikzpicture}}
\end{tabular}
\caption{The $24$ planar trees $t$ with $5$ internal edges for which the number of facets in the $\LA$ diagonal (left) and the $\SU$ diagonal (right) differ.}
\label{fig:trees}
\end{figure}
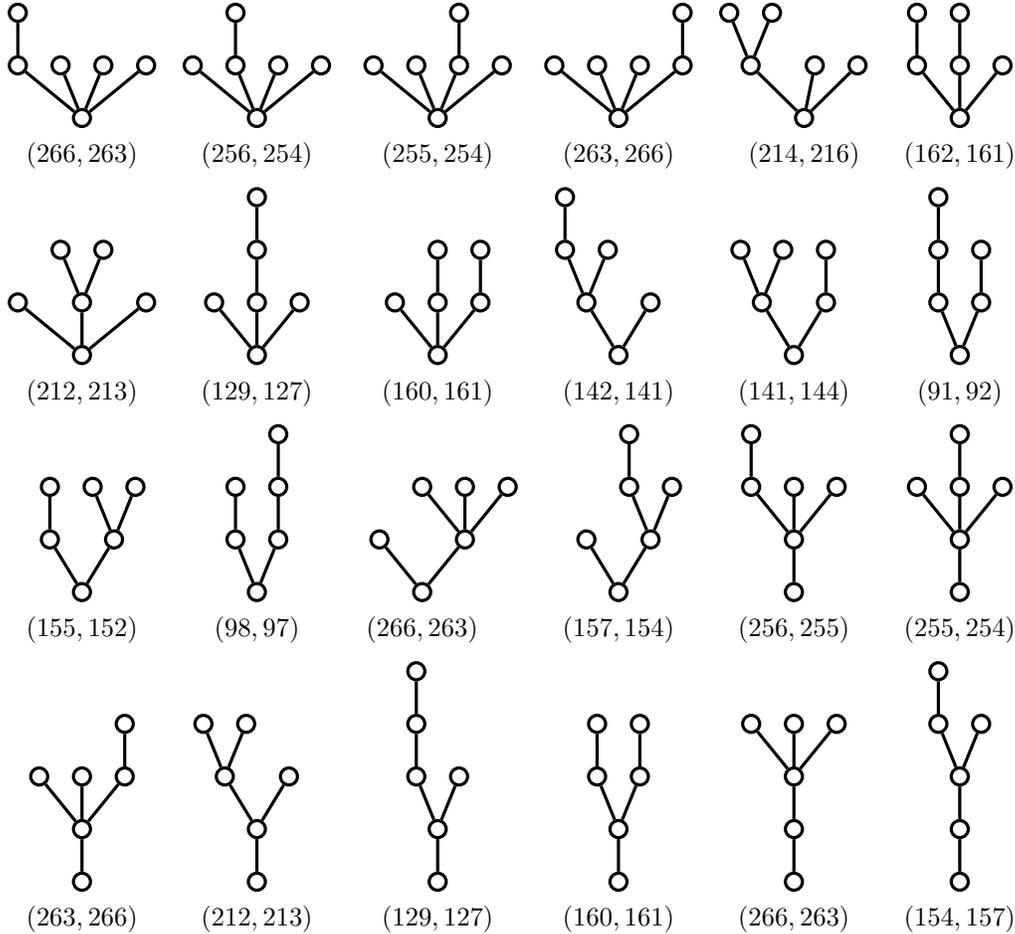

\begin{table}[h]
	\begin{center}
	\begin{tabular}{c|c|c|c|c|c} 
	Internal edges & $\LA$ diagonal & $\LA$ only & Shared & $\SU$ only & $\SU$ diagonal \\
	\hline
	$n=1$ & 2 & 0 & 2 & 0 & 2 \\
	$n=2$ & 8 & 0 & 8 & 0 & 8 \\
	$n=3$ & 42 & 5 & 37 & 5 & 42 \\
	$n=4$ & 254 & 72 & 182 & 72 & 254 \\
	$n=5$ & 1678 & 759 & 919 & 757 & 1676 \\
	$n=6$ & 11790 & 7076 & 4714 & 7024 & 11738 
	\end{tabular}
	\end{center}
	\caption{Number of facets in the $\LA$ and $\SU$ diagonals of the multiplihedra, indexed by linear trees with $n$ internal edges.}
	\label{table:multiplihedra}
\end{table}

%%%%%%%%%%%%%%%%%%%%%%%%%%%%%%%%%%%%%%

\clearpage
\bibliographystyle{alpha}
\bibliography{diagonalsPermutahedra}
\label{sec:biblio}

\end{document}